\documentclass[a4wide,12pt,twoside]{article}
\pdfoutput=1
\usepackage[a4paper]{geometry}
\usepackage{a4wide}

\usepackage[german]{babel}
\usepackage[utf8]{inputenc}
\newcommand{\operator}[1]{\mathcal{L}\left( #1 \right)}

\usepackage{enumerate}

\usepackage{amsmath}
\usepackage{amsfonts} \usepackage{amssymb}

\usepackage{amsthm}
\numberwithin{equation}{section}

\usepackage{graphicx}
\usepackage{xfrac} \usepackage{mathrsfs}  
\usepackage{extarrows} 
\usepackage{listings}
\usepackage[nottoc]{tocbibind}
\usepackage[pdftex, bookmarks=true, bookmarksnumbered=true, hypertexnames=false, breaklinks=true ]{hyperref}
\usepackage[arrow, matrix, curve]{xy} 
\usepackage{tikz} \usepackage{dsfont} 
\usepackage{nomencl} \makenomenclature
\renewcommand{\nomname}{Symbolverzeichnis}
\RequirePackage{ifthen}
\renewcommand{\nomgroup}[1]{\ifthenelse{\equal{#1}{S}}{\item[\textbf{Pseudodifferentialsymbolklassen}]}{\ifthenelse{\equal{#1}{F}}{\item[\textbf{Normierte Räume}]}{
\ifthenelse{\equal{#1}{N}}{\item[\textbf{Normen und Halbnormen}]}{
\ifthenelse{\equal{#1}{M}}{\item[\textbf{Sonstige Mengen}]}{
\ifthenelse{\equal{#1}{T}}{\item[\textbf{Fourier-Transformation}]}{
\ifthenelse{\equal{#1}{O}}{\item[\textbf{Operatorklassen}]}{}
}}}
}
}
}

\usepackage{fancyhdr} 

\fancypagestyle{mystyle}{
\fancyhf{}
\fancyhead[RE]{\textit{\nouppercase{\leftmark}}}
\fancyhead[LO]{\textit{\nouppercase{\rightmark}}}

\fancyhead[RO,LE]{\thepage}
}

\fancypagestyle{plain}{
\fancyhf{}

}
\newcommand*{\blankpage}{\newpage
\pagestyle{empty}
\vspace*{\fill}
\begin{center}
{\sl Diese Seite wurde bewusst freigelassen.}
\end{center}
\vspace{\fill}
\newpage
\pagestyle{fancy}
}

\newcommand{\dslash}{d \hspace{-0.8ex}\rule[1.2ex]{0.8ex}{.1ex}}

\theoremstyle{plain}
\newtheorem{thm}   {Theorem}[section]
\newtheorem{lemma} [thm]{Lemma}

\newtheorem{cor}   [thm]{Folgerung}
\newtheorem{notation}[thm]{Notation}

\theoremstyle{definition}
\newtheorem{definition}   [thm]{Definition}

\newtheorem{konv}[thm]{Konvention}
\newtheorem{satz}[thm]{Satz}

\newtheorem{remark}  [thm]{Bemerkung}
\newtheorem{example} [thm]{Beispiel}

\newcommand{\norm}[2]{\left\|#1\right\|_{#2}}
\newcommand{\linearop}[2]{\mathcal{L}\left(#1;#2\right)}
\newcommand{\opnorm}[1]{\left\|#1\right\|_{\linearop{\R^n}{\R^n}}}

\newcommand{\menge}[1]{\left\{#1\right\}}
\newcommand{\folge}[2]{\left\{#1\right\}_{#2}}
\newcommand{\R}{\mathbb{R}}
\newcommand{\C}{\mathbb{C}}
\newcommand{\N}{\mathbb{N}}
\newcommand{\Z}{\mathbb{Z}}
\newcommand{\hold}{\cdot}
\newcommand{\StripSum}[1]{\mathcal{G}\left(#1\right)}
\newcommand{\Fourier}[3]{\mathcal{F}\left[#1 \mapsto #3\right]\left(#2\right)}
\newcommand{\Fouriers}[1]{\mathcal{F}\left[#1 \right]}
\newcommand{\Fourieri}{\mathcal{F}}
\newcommand{\FourierB}[3]{\mathcal{F}^{-1}\left[#1 \mapsto #3\right]\left(#2\right)}
\newcommand{\Stetig}[2]{\mathscr{C}^{#1}\left(#2\right)}
\newcommand{\Stetigi}[1]{\mathscr{C}^{#1}\left(\R^n\right)}
\newcommand{\Schwarz}[1]{\mathscr{S}\hspace{-0.4ex}\left(#1\right)}
\newcommand{\SchwarzDual}[1]{\mathscr{S}'\hspace{-0.4ex}\left(#1\right)}
\newcommand{\Schwarzi}{\Schwarz{\R^n}}
\newcommand{\SchwarziDual}{\mathscr{S'}(\R^n)}
\newcommand{\SchwarzNorm}[2]{\abs{#1}_{\mathcal{S}, #2}}
\newcommand{\SchwarzzNorm}[2]{\abs{#1}'_{\mathcal{S}, #2}}
\newcommand{\supp}{\mathrm{ supp }\,}
\newcommand{\gauss}[1]{\lfloor #1 \rfloor}
\newcommand{\piket}[2]{\left<#1\right>^{#2}}
\newcommand{\piketi}[1]{\left<#1\right>}
\newcommand{\abs}[1]{\left|#1\right|}
\newcommand{\hnorm}[2]{\left[#1\right]_{#2}}
\newcommand{\lilwood}[1]{\varphi_{#1}}
\newcommand{\lilwoodt}[2]{\tilde{\varphi}_{#1}\left(#2\right)}

\newcommand{\interpolC}[2]{\left(#1\right)_{\left[#2\right]}}

\newcommand{\OP}[1]{\mathrm{ OP}#1}
\newcommand{\OPi}{\mathrm{ OP}}
\newcommand{\OPSym}[1]{\mathrm{ OP}\left(#1\right)}
\newcommand{\OPMfg}[3]{#3 C^\tau S^{#1}_{#2}}
\newcommand{\OPMfgInf}{M C^\tau S^{-\infty}}
\newcommand{\OPMfgN}[2]{#2 C^{#1}_{\restOP} S^0_{1,0}}

\newcommand{\vol}{\mathrm{ vol }\,}
\newcommand{\dist}[2]{\mathrm{ dist}\left(#1, #2\right)}
\newcommand{\Beschr}[2]{C^{#1}\hspace{-0.4ex}\left(#2\right)}
\newcommand{\Beschri}{C^{\infty}\hspace{-0.4ex}\left(\R^n\right)}
\newcommand{\BeschrZ}{C^0\hspace{-0.4ex}\left(\R^n\right)}
\newcommand{\Oszillatory}[2]{\mathscr{A}_{#1}^{#2}\left( \R^n \times \R^n \right)}
\newcommand{\Oszillatorye}[3]{\mathscr{A}_{#1}^{#2}\left( #3 \right)}
\newcommand{\Oszillatoryi}{\mathscr{A}\left( \R^n \times \R^n \right)}
\newcommand{\OszillatoryNorm}[4]{\abs{#1}_{\mathscr{A}_{#2}^{#3}, #4}}

\newcommand{\osint}{\mathrm{O}_{\mathrm{s}} - \iint}
\newcommand{\os}[3]{\osint e^{-i#2\cdot#1} #3 d#2 \dslash#1} \newcommand{\subsubset}{\subset\subset}
\newcommand{\Landau}[1]{\mathcal{O}\left(#1\right)}
\newcommand{\Landaui}{\mathcal{O}}
\newcommand{\PSDO}{ $\Psi\textrm{DO}$ }
\newcommand{\PSDOS}{ $\Psi\textrm{DOs}$ }
\newcommand{\derive}[2]{\mathcal{D} #1\left(#2\right)}
\newcommand{\derivei}[1]{\mathcal{D} #1}
\newcommand{\deriveii}{\mathcal{D}}
\newcommand{\restOP}{\mathcal{R}}
\newcommand{\restOPM}{\mathcal{R}(M)}
\newcommand{\distrodual}[1]{\mathscr{D}'\left(#1\right)}

\newcommand{\Streifen}{\mathbb{S}}
\newcommand{\Real}{\operatorname{Re}}
\newcommand{\Imag}{\operatorname{Im}}
\newcommand{\Image}{\operatorname{Bild}}
\newcommand{\id}{\operatorname{id}}
\newcommand{\characteristic}[1]{\mathds{1}_{#1}}
\newcommand{\einheitsmatrix}{\mathds{1}}

\newcommand{\simh}{\scalebox{2}[1]{$\sim$}}
\newcommand{\simv}{\rotatebox{90}{\scalebox{2}[1]{$\sim$}}}

\newcommand{\dualskp}[4]{\left< #1, #2 \right>_{#3,#4}}
\newcommand{\skp}[3]{\left( #1, #2 \right)_{#3}}

\newcommand{\quotientenraum}[2]{
\raisebox{1ex}{\ensuremath{#1}}
\ensuremath{\mkern-3mu}\bigg/\ensuremath{\mkern-3mu}
\raisebox{-1ex}{\ensuremath{#2}}}

\newcommand{\abgabedatum}{1. Dezember 2011}

\usepackage[T1]{fontenc}

\usepackage{verbatim}

\author{Dominik Köppl}

\begin{document}

\begin{titlepage}
\begin{center}
  \bigskip
  \Large{\textsc{Diplomarbeit}}
  \\
  \vspace{1cm}
  \huge{\textsf{Pseudodifferentialoperatoren mit nichtglatten Koeffizienten auf Mannigfaltigkeiten}}
  \\
  \vspace{3.5 cm}
\textsf{
    \begin{large}
    \begin{tabular}{ll}
      Diplomand: & Dominik K\"{o}ppl \\
      Fachbereich: & Mathematik \\
      Betreuer: & Prof. Dr. rer.nat. Helmut Abels \\
      Abgabedatum: & \abgabedatum
    \end{tabular}
  \end{large}
  }
\end{center}
\end{titlepage}
\clearpage{}
\blankpage
\pagenumbering{roman}
\pagestyle{plain}
\section*{Vorwort}
\addcontentsline{toc}{section}{Vorwort}
Pseudodifferentialoperatoren (\PSDOS) gehören zu den wichtigsten Hilfsmitteln der Analysis von linearen partiellen Differentialgleichungen.
Diese Gleichungen dienen als Basis mathematischer Modelle, um verschiedene Phänomene in der modernen Physik, Biologie, Wirtschaft und anderen
wissenschaftlichen Gebieten zu beschreiben.
Die frühen Errungenschaften in den 60'igern durch Hörmander brachten ein komplett neues Verständnis wichtiger Probleme der Analysis und mathematischen Physik zum Vorschein.
Das Kalkül mit Pseudodifferentialoperatoren kann heutzutage als klassisch betrachtet werden.
Zahlreiche Monographien haben sich dieser Thematik gewidmet, viele prägen den Begriff eines \PSDOS im ``glatten'' Fall - also mit glatten Symbolen.
In den letzten Jahren hat sich ein neues bemerkenswertes Interesse für Pseudodifferentialoperatoren mit weniger regulierenden Eigenschaften entwickelt.
Tatsächlich treten in allen oben erwähnten Bereichen viele Probleme auf, die mit \PSDOS im ``nicht-glatten'' Fall zusammenhängen.
Die Studie von Pseudodifferentialoperatoren mit nicht-glatten Symbolen wurde von Marschall und Taylor in den letzten Jahrzehnten ausgearbeitet.
Das zugehörige Problem mit \PSDOS auf Mannigfaltigkeiten wurde aber noch nicht detailliert durchdacht.
Deswegen widmet sich diese Arbeit genau dieser Thematik.
Die hier erarbeiteten Resultate formen den Kern der Theorie der Pseudodifferentialoperatoren mit nicht-glatten-Symbolen, die in dieser Arbeit eingeführt werden.
Anstatt einen kompletten Überblick über alle bekannten Resultate über \PSDOS mit nicht-glatten Symbolen auszuschöpfen,
werden einige Leitideen aus einem weiten Themengebiet aufgegriffen.
Einer der Hauptschwierigkeiten in der Ausarbeitung war es,
dass in vielen Bereichen der Mathematik verschiedene Notationen für ein und das selbe Objekt gebräuchlich sind, wie die Notation einer Symbolklasse.
Andererseits war es eine schwere Aufgabe, die Arbeit mit einer konsistenten Notation durchzustrukturieren,
sodass wichtige Objekte in allen Teilen der Arbeit mit den gleichen Symbol identifiziert werden.
Ein zweites Problem bestand darin, dass durch den enormen Fortschritt in der Forschung im Bereich der partiellen Differentialgleichungen
die historische Entwicklung nur schwer nachvollziehbar ist, und ich deswegen hoffe, in den Literaturangaben alle Mitwirkende zu nennen, die
einen (indirekten) Beitrag für Entstehung dieser Arbeit geliefert haben.
Eine weitere Schwierigkeit ergibt sich aus der Abhängigkeit der Definition der Pseudodifferentialoperatoren von den oszillatorischen Integralen,
eine Verallgemeinerung der Lebesgue-Integrale für nicht-absolut konvergente Integrale.
Die Theorie der oszillatorischen Integrale stellt sich nämlich als sehr technisch heraus.
Zudem kann die Verwendung der Transformationsformel und partiellen Integration für Integranden, die nicht in $L^1(\R^n)$ liegen,
für Verwirrung stiften.
Besonderes Augenmerk wurde auf die vollständige Beweisführung und die detaillierte Referenzierung ähnlicher Resultate in der Literatur gelegt.
Die Arbeit entstand im Zusammenhang mit der Vorlesung ``Pseudodifferentialoperatoren'' von Prof. Dr. Helmut Abels im Wintersemester 2009/10 an der Universität Regensburg.

Die 2011 veröffentlichte Arbeit wurde von Prof. Dr. Helmut Abels und Prof. Dr. Bernd Ammann begutachtet und kommentiert.
Die kleinen, jedoch zahlreichen Verbesserungsvorschläge floßen 2012 in das Manuskript ein, sodass diese Monographie die zweite Auflage der Arbeit bildet.
 
\section*{Danksagung}

An erster Stelle möchte ich mich bei Herrn Prof. Dr. Helmut Abels bedanken, der mir die Erstellung dieser Arbeit ermöglichte.
Sein persönliches Engagement, sowie zahlreiche Anregungen und Tipps, sowie die ständige Diskussionsbereitschaft waren mir eine wertvolle Hilfe.

Danken möchte ich auch meinen Onkel Günter Markowski für die orthographische Korrekturarbeit und meinen Eltern für die kulinarische Versorgung während des Schreibens dieser Arbeit.

Daneben auch herzlichstes Dank an Herrn Prof. Dr. Bernd Ammann für seine detaillierte Begutachtung.
 
\clearpage
\pagestyle{mystyle}
\tableofcontents
\blankpage
\pagenumbering{arabic}
\pagestyle{mystyle}
\section{Einleitung}
\markboth{Einleitung}{Einleitung}

Es ist eine bekannte Tatsache, dass sich eine lineare partielle Differentialgleichung $P(D) u = f$ mit konstanten Koeffizienten in $\R^n$
durch eine Fourier-Transformation als Gleichung $P(\xi) \hat{u}(\xi) = \hat{f}(\xi)$ darstellen lässt,
sodass wir die Gleichung auf die Division durch das Polynom $P(\xi)$ zurückführen können.
Die Technik der Pseudodifferentialoperatoren verallgemeinert dieses praktische formale Beispiel für partielle Differentialgleichungen mit variablen Koeffizienten.
Beim Studium von nicht-linearen partiellen Differentialgleichungen stellt sich jedoch heraus, dass die Symbole meist nicht glatt sein werden, 
da die Koeffizienten von der (im Allgemeinen nicht-glatten) Lösung der Gleichung abhängen.
In Hinblick darauf wollen wir in Abschnitt \ref{sec:psdo} die Errungenschaften von \cite{Kumanogo} und \cite{Hormander} für die Pseudodifferentialklasse $S^m_{\rho,\delta}(\R^n\times\R^n)$ auf
Operatoren mit Hölder-stetigen Koeffizienten übertragen.
Beim Übergang von klassischen Symbolen zu Symbolen mit nicht-glatten Koeffizienten stellen sich aber sofort grundlegende Fragen:
Sind die Operatoren der Symbole auf den gewöhnlichen Funktionenräumen stetig?
Wie verhalten sich Pseudodifferentialoperatoren, wenn sie verkettet werden?
Können Operatoren dieser Art auf Mannigfaltigkeiten definiert werden?
Die ersten beiden Fragestellungen wollen wir in den Unterabschnitten \ref{sec:psdointro} bzw. \ref{sec:composition} genauer erläutern.
Voraussetzung für letzteres bildet der Unterabschnitt \ref{sec:simplify}, 
in dem wir die Komposition eines Operators mit glatten Symbol und eines Operators mit nicht-glatten Symbol betrachten,
und die Untersuchungen von \cite{AbelsBoundary} für Pseudodifferentialoperatoren auf Besov-Räumen.
Dafür werden wir das Mittel der Symbolglättung benötigen, welches uns ein nicht-glattes Symbol
in ein glattes Symbol gleicher Ordnung und ein nicht-glattes Symbol geringerer Ordnung aufspaltet.
Dieser Thematik ist der Unterabschnitt \ref{sec:smoothing} gewidmet, für den die wesentlichen Gesichtspunkte aus \cite{Nonlinear} übernommen wurden.
Zu allererst wollen wir aber in Abschnitt \ref{sec:funktionsklassen} die notwendigen Funktionsklassen einführen.
Wir erklären die Fourier-Transformation auf dem $\R^n$ mit einem besonderen Augenmerk auf den Schwartz-Raum $\Schwarzi$,
der sich unter einer Fourier-Transformation isometrisch nach Plancherels Theorem verhält.
Die notwendige Theorie zur Definition der Pseudodifferentialoperatoren findet sich in den technischen Beweisen der oszillatorische Integrale,
die wir in Unterabschnitt \ref{sec:oszi} ausgiebig diskutieren.
\PSDOS werden naturgemäß unter mikrolokalisierbaren Räumen wie den Bessel-Potential-Raum betrachtet.
In Unterabschnitt \ref{sec:hoelder} bzw. \ref{sec:bessel} führen wir deswegen die Klassen $C^\tau_*(\R^n)$ bzw. $H^s_q(\R^n)$ ein,
wobei wir die Bessel-Potentialräume in Zusammenhang mit der komplexen Interpolationstheorie (\cite{Lunardi},\cite{Bergh}) einführen werden.
Einer der wesentliche Ziele wird es sein, die Invarianz unter Kartenwechsel von Pseudodifferentialoperatoren auf Mannigfaltigkeiten zu definieren.
Grundlage dafür bildet die Kerndarstellung von \cite{Stein}, die wir für Hölder-stetige Symbole in der sog. $(x,\xi,y)$-Form im Unterabschnitt \ref{sec:kern} verallgemeinern.
Für glatte Mannigfaltigkeiten werden wir die von \cite{Kumanogo} erarbeitete Theorie für glatte Symbole in den Unterabschnitten \ref{sec:smoothkoord} und \ref{sec:smoothmanifold} übernehmen.
Die Situation stellt sich komplizierter dar, wenn wir den Sachverhalt auf nicht-glatten Mannigfaltigkeiten in Abschnitt \ref{sec:nonsmooth} studieren.
Um Aussagen treffen zu können, werden wir uns auf spezielle Symbole mit nicht-positiver Ordnung in $(x,y,\xi)$-Form zurückziehen.
Hier liefert uns \cite{Tools} essentielle Eigenschaften, die wir in Unterabschnitt \ref{sec:x,y,xi} analysieren werden.
Der Abschluß der Arbeit bildet eine Aussage über die Transformationseigenschaft von Operatoren mit Symbolen dieser Klasse.

\section{Funktionenräume und Fourier-Transformationen}\label{sec:funktionsklassen}
\subsection{Wiederholung analytischer Grundbegriffe}
\nomenclature[m]{$\N$}{Menge der natürlichen Zahlen ohne der Null}
\nomenclature[m]{$\N_0$}{Menge der natürlichen Zahlen einschließlich der Null}

Zu Beginn wollen wir bekannte
Resultate der grundlegenden Analysis im $\R^n$ zusammengefasst wiederholen.
Soweit nicht anders erwähnt, werden wir die kanonische Basis des $\R^n$ mit der Menge $\menge{e_j}_{j=1,\ldots,n}$ darstellen.
Wir werden durchgehend mit der fest gewählten Ganzzahl $n \in \N$ die Dimension einer offenen Teilmenge $\Omega \subseteq \R^n$ oder einer Mannigfaltigkeit bezeichnen.
Wir schreiben oft für eine Funktion $f$ auch $f(x)$ falls eine Verwirrung mit dem Wert $f(x)$ bei $x$ ausgeschlossen werden kann.

\begin{definition}
\begin{enumerate}[(a)]
\item
Sei $\Omega \subseteq \R^n$ offen, $f : U \rightarrow \C$.
Wir sagen, dass $f$ $k$-mal differenzierbar in $U$ ist (für $k \in \N_0$), falls
die partiellen Ableitungen $\partial^\alpha f$ existieren und stetig auf $U$ sind für jedes $\abs{\alpha} \le k$.
Wir verwenden hierfür die Notation $f \in \Stetig{k}{U}$.
Insbesondere ist $f \in \Stetig{0}{U}$ falls $f$ stetig ist.
Falls $f \in \Stetig{\infty}{U}$, nennen wir $f$ glatt.
\item
Eine Funktion $f : U \subseteq \R^n \rightarrow V \subseteq \R^n$ gehört der Klasse $\Stetig{k}{\R^n}^n$ an,
falls $f_j := \pi_j \circ f \in \Stetig{k}{\R^n}$ für jedes $j = 1,\ldots,n$, wobei $\pi_j$ die kanonische Projekion auf die $j$.te Koordinate des $\R^n$ bezeichnet.
\item
Das Bild des Nabla-Operators $\nabla : \Stetig{1}{\R^n} \rightarrow \Stetig{0}{\R^n}^n$ ist der Vektor der transponierten Jacobimatrix einer Funktion $f \in \Stetig{1}{\R^n}$.
\end{enumerate}
\end{definition}
\nomenclature[m]{$\Stetig{\infty}{U}$}{Glatte Funktionen in $U$}
\nomenclature[m]{$\Stetig{k}{U}$}{$k$-mal stetig differenzierbare Funktionen in $U$}
\nomenclature{$\nabla$}{Nabla-Operator}
\nomenclature[f]{$C^0(\R^n)$}{Raum der beschränkten stetigen Funktionen}

\begin{definition}
Sei $\Omega \subseteq \R^n$ offen.
Der Raum $\Beschr{m}{\Omega}$ bezeichne die Menge aller Funktionen aus dem Raum $\Stetig{m}{\Omega}$,
deren Ableitungen der Ordnung kleiner gleich $m \in \N_0$ beschränkt sind in $\Omega$, wir schreiben also
$$ \Beschr{m}{\Omega} := \menge{ f \in C^m(\Omega) : \norm{ \partial^\alpha f }{\infty} \le C_\alpha, \forall \alpha\in\N^n_0,\abs{\alpha}\le m \textrm{ mit } C_\alpha > 0}.$$
Setze $ \Beschr{\infty}{\Omega} := \bigcap_{m=0}^\infty \Beschr{m}{\Omega}$.

Für einen Banachraum $Y$ sagen wir, dass eine Funktion $f : \R^n \rightarrow Y$ in $\Beschr{0}{\R^n,Y}$ enthalten ist,
falls $f$ stetig ist und beschränkt durch $\sup_{x \in \R^n} \norm{ f(x) }{Y} < \infty$.
\end{definition}
\nomenclature[f]{$C^\infty(U)$}{beschränkte glatte Funktionen in $U$, deren (mehrfachen) Ableitungen wiederum beschränkt sind}
\nomenclature[f]{$C^\infty_0(U)$}{glatte Funktionen in $U$ mit kompakten Träger}

\begin{lemma}[Mittelwertsatz im Mehrdimensionalen]\label{mittelwertsatz}
Sei $f \in \Stetig{1}{\R^n}^n$ und $x,y\in\R^n$.
Ist $p \mapsto \derive{f}{p} \in \linearop{\R^n}{\R^n} $ für alle $p \in \menge{ \lambda x + (1-\lambda) y : \lambda \in [0,1] }$, 
so ist die lineare Abbildung 
$$ \Xi_{f,p}(x,y) := \int_0^1 \derive{f}{x + t(y-x)} dt$$ wohldefiniert und es gilt
$$ f(y) - f(x) = \Xi_{f,p}(x,y) (y-x).$$
\begin{proof} Setzen wir $g_i(t) := f_i(x+t(y-x))$, so erhalten wir sofort mit dem Hauptsatz der Integralrechnung
\begin{align*}
f_i(y) - f_i(x) 
&= g_i(1) - g_i(0) = \int_0^1 d_t g_i(t) dt = \int_0^1 \sum_{j=1}^n \partial_j f_i(x+t(y-x)) (y-x)_j dt  \\
&= \int_0^1 \derive{f_i}{x+t(y-x)} dt (y-x)
.\end{align*}
\end{proof}
\end{lemma}

\nomenclature{$\derive{f}{p}$}{Jacobimatrix der Funktion $f$ im Punkt $p$}
\nomenclature[f]{$\linearop{U}{V}$}{Raum der linearen beschränkten Abbildungen $U \rightarrow V$ bzgl. der Vektorräume $U$ und $V$}
\nomenclature{$\Xi_{f,p}$}{Operator definiert in Lemma \ref{mittelwertsatz}}

\begin{thm}[Taylorreihe mit Restglied] \label{remind:taylorreihe}
Sei $f \in \Stetig{\infty}{\R^n}$. Dann gilt für jede Ganzzahl $N\in\N$, dass
$$ f(\xi+\eta) = \sum_{\abs{\alpha} < N} \frac{\eta^\alpha}{\alpha!} f^{(\alpha)}(\xi) + N \sum_{\abs{\gamma} = N} \frac{\eta^\gamma}{\gamma!} \int_0^1 (1-\theta)^{N-1} f^{(\gamma)}(\xi+\theta\eta)d\theta .$$
\begin{proof} Siehe \cite{Koenigsberger2} \end{proof}
\end{thm}

\begin{notation}
Seien $f,g : U \subseteq \R^n \rightarrow \R_+$.
Wir sagen $f(x) \le \Landau{g(x)}$, falls es eine Konstante $C > 0$ unabhängig von $x \in U$ gibt, sodass
$f(x) \le C g(x)$ für jedes $x \in U$.
\end{notation}

\begin{definition}\label{def:multiindex}
Ein Multiindex $\nu \in \N^n_0$ ist ein $n$-dimensionaler Vektor der Gestalt $\nu = (\nu_1,\ldots,\nu_n)$.
Insbesondere ist jeder Vektor der kanonische Basis $\menge{e_j}_{j=1,\ldots,n}$ des $\R^n$ ein Multiindex.
Für einen Multiindex $\nu \in \N^n_0$ definieren wir die Länge $\abs{\nu} := \sum_{j=1}^n \nu_i$ und
die Fakulät $\nu! := \prod_{j=1}^n \nu_j$.
Multiindizes bestimmen die Exponenten von Polynomen durch $x^\nu := \prod_{j=1}^n x_j^{\nu_j}$ für $x  = (x_1,\ldots,x_n) \in \R^n$.
Für $\nu = 0$ setzen wir $0^0 := 1$.
Außerdem können wir für $\abs{\nu} > 0$ multivariante Ableitungen mit 
$\partial_x^\nu := \partial_{x_1}^{\nu_1} \ldots \partial_{x_n}^{\nu_n} := \frac{\partial^{\abs{\nu}}}{\partial x_1^{\nu_1} \cdot\cdot\cdot \partial x_n^{\nu_n}}$
darstellen, 
unter $\partial_x^0$ verstehen wir die Identitätsabbildung.
Für einen weiteren Multiindex $\mu = (\mu_1,\ldots,\mu_n) \in \N^n_0$ schreiben wir $\mu \le \nu$ falls
$\mu_j \le \nu_j$ für jedes $j=1,\ldots,n$. In diesem Fall setzen wir
$$ {\nu \choose \mu } := \prod_{j=1}^n {\nu_j \choose \mu_j} = \frac{\nu!}{\mu!(\nu-\mu)!}
.$$
Schließlich wollen wir eine lineare Ordnung auf $\N^n_0$ einführen. Wir sagen $\mu \prec \nu$ falls eine der folgenden Aussagen zutrifft:
\begin{enumerate}[(i)]
	\item $\abs{\mu} < \abs{\nu}$,
	\item $\abs{\mu} = \abs{\nu}$ und $\mu_1 < \nu_1$, oder
	\item $\abs{\mu} = \abs{\nu}, \mu_1 = \nu_1, \ldots, \mu_k = \nu_k$ und $\mu_{k+1} < \nu_{k+1}$ für ein $1 \le k < n$.
\end{enumerate}
\end{definition}
\nomenclature[m]{$\N^n_0$}{Menge der Multiindizes - siehe Definition \ref{def:multiindex}}

\begin{definition}\label{def:multiindexspec}
Definiere für ein $N \in \N_0$ die Menge
$ \N^n_{0,N} := \menge{\mu \in \N^n_0 : 1 \le \abs{\mu} \le N}$.
Insbesondere ist $\N^n_{0,0} = \emptyset$
\end{definition}
\nomenclature[m]{$\N^n_{0,N}$}{Teilmenge von $\N^n_0$ - siehe Definition \ref{def:multiindexspec}}

\begin{lemma}[Formel von Fa\`{a} di Bruno] \label{faaDiBruno}
Für zwei offene Teilmengen $U \subset \R^n$ und V $\subset \R^m$ seien zwei $\Stetigi{k}$-Funktionen $f : U \rightarrow V, g : V \rightarrow \R$ gegeben ($k\in\N$), also 
$ \begin{xy} \xymatrix{ U \ar[r]^f  & V \ar[r]^g & \R } \end{xy} $.
Wir wählen die Standardkoordinaten $(x_1,\ldots,x_n)$ in $U$ und $(y_1,\ldots,y_m)$ in $V$, und stellen unter diesen Koordinaten $f = (f_1,\ldots,f_m)$ durch die natürliche Projektion $\pi_{y_j} : V \rightarrow \R$ mit $f_j = \pi_{y_j} \circ f$ dar. 
Für einen fest gewählten Punkt $x=(x_1,\ldots,x_n) \in U, y=(y_1,\ldots,y_m) := f(x)$ und einen Multiindex $\alpha \in \N^n_0$ mit $1 \le \abs{\alpha} \le k$ hält die Formel
$$
\partial^\alpha_x (g \circ f) (x) = \alpha! \sum_{\sigma \in \N^m_{0,\abs{\alpha}}} \partial_y^\sigma g(y) 
\sum_{\substack{(\nu^1,\ldots,\nu^{\abs{\alpha}}; \mu^1,\ldots,\mu^{\abs{\alpha}}) \\ \in \Sigma(\alpha,\sigma)}}
\prod_{j=1}^{\abs{\alpha}}
\frac{1}{\nu^j! (\mu^j!)^{\abs{\nu^j}}} \left( \partial_x^{\mu^j} f(x) \right)^{\nu^j}
,$$
wobei
\begin{align*}
\Sigma(\alpha,\sigma) &:= \Bigg\{ (\nu^1,\ldots,\nu^{\abs{\alpha}}; \mu^1,\ldots,\mu^{\abs{\alpha}}) : \nu^j \in \N^m_0, \mu^j \in \N^n_0 \textrm{ für jedes } j = 1,\ldots,\abs{\alpha} \textrm{ und } \\
& \textrm{ es ein } 1 \le s \le n \textrm{ gibt, } \textrm{ sodass } \nu^j = 0 \textrm{ und } \mu^j = 0 \textrm{ für } 1 \le j \le \abs{\alpha} - s \\
& \textrm{ sowie } \abs{\nu^j} > 0 \textrm{ für } \abs{\alpha}-s+1 \le j \le n \\
& \textrm{ und } 0 \prec \mu^{\abs{\alpha}-s+1} \prec \ldots \prec \mu^{\abs{\alpha}} \textrm{ sowie } 
  \sum_{j=1}^{\abs{\alpha}} \nu^j = \sigma ;\quad \sum_{j=1}^{\abs{\alpha}} \abs{\nu^j} \mu^j = \alpha \Bigg\}
,\end{align*}
und $\partial_x^{\mu} f(x) = \menge{ \partial_x^{\mu} f_1(x), \ldots, \partial_x^{\mu} f_m(x) }$ für $\mu \in \N^n_0$.
\begin{proof} Siehe \cite[Theorem 2.1]{faadibruno}
\end{proof}
\end{lemma}

\begin{lemma}[Regel von Leibniz]
Für zwei Funktionen $f,g \in \Stetig{k}{U}$ mit $U \subset \R^n$ offen, $k\in\N$ und einem Multiindex $\nu\in\N^n_0$ ist für ein beliebiges $x\in U$
$$ \partial_x^\nu (fg)(x) = \sum_{0\le\mu\le\nu} {\nu \choose \mu} D_x^\mu f(x) D_x^{\nu-\mu} g(x) .$$
\begin{proof} Siehe \cite[Lemma 2.6]{faadibruno}
\end{proof}
\end{lemma}

\begin{notation}
Im Folgenden sei $D_{x_j} := \frac{1}{i} \partial_{x_j}$ für $x = (x_1,\ldots,x_n) \in \R^n$ und 
$D_x = (D_{x_1},\ldots,D_{x_n})$. 
Für einen Multiindex $\alpha \in \N^n_0$ haben wir 
$D_x^\alpha := D_{x_1}^{\alpha_1} \ldots D_{x_n}^{\alpha_n}$.
\end{notation}

\subsubsection{Maßtheorie}

\begin{definition}[$L^p$-Räume]
Sei $\Omega \subseteq \R^n$.
Unter dem $L^p(\Omega)$-Raum für $1 \le p < \infty$ verstehen wir den Raum aller Lebesgue-messbaren Funktionen
$f: \Omega \rightarrow \C$, versehen mit der Norm
$$ \norm{f}{L^p(\Omega)} := \left( \int_{\Omega} \abs{f}^{p} dx \right)^{\frac{1}{p}} .$$
Für zwei Funktionen $f,g \in L^p(\Omega)$ sagen wir $f = g$ fast überall, falls es eine Lebesgue-Nullmenge $N \subset \Omega$ gibt, sodass
$f(x) = g(x)$ für alle $x \in \Omega \setminus N$.
\end{definition}

\begin{thm}[Lebesgue's Satz der dominanten Konvergenz]\label{thm:masstheorieDOM}
Sei $ \menge{f_j}_{j \in \N}$ eine Folge Lebesgue-messbarer Funktionen $f_j : \R^n \rightarrow \C$.
Falls es eine messbare Funktion $f : \R^n \rightarrow \C$ und ein $g \in L^1(\R^n), g : \R^n \rightarrow \R$ gibt mit
\begin{itemize}
\item $\abs{f_j} \le g$ fast überall für jedes $j \in \N$ und
\item $f_j \rightarrow f$ punktweise fast überall für $j \rightarrow \infty$,
\end{itemize}
dann sind $f_j, f \in L^1(\R^n)$ und es konvergiert $f_j \rightarrow f$ in $L^1(\R^n)$, d.h.
$$ \lim_{j \rightarrow \infty} \int_{\R^n} f_j(x) dx = \int_{\R^n} f(x) dx .$$
\begin{proof} Siehe \cite[A 1.21]{Alt}
\end{proof}
\end{thm}

\begin{thm}[Fubini]\label{thm:massthreorieFubini}
Sei $f \in L^1(\R^n \times \R^n)$. 
Dann ist $x \mapsto f(x,y) \in L^1(\R^n)$ fast überall für alle $y \in \R^n$ 
und $F : y \mapsto \int_{\R^n} f(x,y) dy$ existiert fast überall.
Es ist $F \in L^1(\R^n)$ mit
$$ \int_{\R^n} F(y) dy = \int_{\R^n\times\R^n} f(x,y) d(x,y) .$$
Vertauschen wir die Variablen $x$ und $y$, so erhalten wir die analoge Aussage
$$ \int_{\R^n} \left( \int_{\R^n} f(x,y) dx \right) dy = \int_{\R^n\times\R^n} f(x,y) d(x,y) = \int_{\R^n} \left( \int_{\R^n} f(x,y) dy \right) dx .$$
\begin{proof} Siehe \cite[A 4.10]{Alt}
\end{proof}
\end{thm}

\begin{thm}[Transformationsformel]\label{thm:masstheorieTransformation}
Seien $X, Y \subseteq \R^n$ offen, und eine Abbildung $h : Y \rightarrow X$ gegeben.
Falls $h$ fast überall ein $\Stetigi{1}$-Diffeomorphmus ist,
d.h. es gibt Lebesgue-Nullmengen $N_X \subset X, N_Y \subset Y$, sodass $h : Y \setminus N_Y \rightarrow X \setminus N_X$ ein Diffeomorphismus ist,
dann gilt 
$$ \int_X f(x) dx = \int_Y f(h(y)) \abs{ \det \derive{h}{y} } dy \textrm{ für alle } f \in L^1(\R^n) ,$$
wobei $f$ genau dann Lebesgue-integrierbar ist, wenn $f \circ h \abs{ \det \derivei{h} }$ Lebesgue-integrierbar ist.
\begin{proof} Siehe \cite[4.10]{Elstrodt}
\end{proof}
\end{thm}

\begin{satz}\label{satz:masstheorieStetig}
Sei $y_0 \in \R^n$ ein Punkt und $f : \R^n \times \R^n \rightarrow \C$ eine Funktion mit den folgenden Eigenschaften:
\begin{itemize}
\item Für jedes $x \in \R^n$ ist $y \mapsto f(x,y)$ stetig in $y_0$.
\item Für jedes $y \in \R^n$ ist $x \mapsto f(x,y)$ über $\R^n$ integrierbar.
\item Es gibt eine integrierbare Funktion $F \in L^1(\R^n)$, sodass $\abs{f(x,y)} \le F(x)$ für alle $y \in \R^n$ und für fast alle $x \in \R^n$.
\end{itemize}
Dann ist die durch $g(y) := \int_{\R^n} f(x,y) dx$ definierte Funktion $g : \R^n \rightarrow \C$ im Punkt $y_0$ stetig.
\begin{proof} Siehe \cite[\S 11 Satz 1]{Forster3}
\end{proof}
\end{satz}

\begin{satz}\label{satz:masstheorieAbl}
Sei $f : \R^n \times \R \rightarrow \C$ eine Funktion mit den folgenden Eigenschaften:
\begin{itemize}
\item Für jedes $x \in \R^n$ ist die Funktion $t \mapsto f(x,t)$ differenzierbar auf $\R$.
\item Für jedes $t \in \R$ ist die Funktion $x \mapsto f(x,t)$ über $\R^n$ integrierbar.
\item Es gibt eine integrierbare reelle Funktion $F \in L^1(\R^n)$ mit $\abs{ \partial_t f(x,t) } \le F(x)$ für alle $(x,t) \in \R^n \times \R$.
\end{itemize}
Dann ist die durch $g(t) := \int_{\R^n} f(x,t) dx$ definierte Funktion $g : \R \rightarrow \C$ differenzierbar.
Für jedes $t \in \R$ ist die Funktion $x \mapsto \partial_t f(x,t)$ über $\R^n$ integrierbar und es gilt, dass
$ g'(t) = \int_{\R^n} \partial_t f(x,t) dx $.
\begin{proof} Siehe \cite[\S 11 Satz 2]{Forster3}
\end{proof}
\end{satz}

\begin{notation}
Wir werden den Integrationsbereich des Lebesgue-Integrals weglassen, falls wir auf dem gesamten $\R^n$ integrieren, 
d.h. wir schreiben $\int$ für $\int_{\R^n}$.
\end{notation}

\begin{definition}[Japanische Klammer]\label{definition:japanese}
Für $\xi \in \R^n$ setzen wir $\piketi{\xi} := \left( 1 + \abs{\xi}^2 \right)^{\frac{1}{2}}$.
Die Klammerung $\piketi{\hold}$ wird auch japanische Klammer genannt.
Offensichtlich ist $\piketi{\xi} - \abs{\xi} > 0$ für alle $\xi \in \R^n$.
Wir verwenden die japanische Klammer, 
da sie im Gegensatz zur Norm für negative Exponenten bei Null keine Singularität aufweist, und wir somit ein Abschätzungskriterium für ``gutartige'' Funktionen erhalten.
\end{definition}
\nomenclature{$\piketi{\hold}$}{Japanische Klammer - Definition \ref{definition:japanese}}

\begin{lemma}\label{lemma:japanese}
Für $\epsilon > 0$ gibt es eine Konstante $C > 0$, sodass $C^{-1} \piketi{x} \le \abs{x} \le C \piketi{x}$ für alle $x \in \R^n$ mit $\abs{x} > \epsilon$.
\begin{proof}
Setze $C := \abs{\epsilon}^{-1} \piketi{\epsilon} > 1$. 
Offensichtlich ist $\R_+ \rightarrow \R, r \mapsto r^{-1} \piketi{r}$ monoton fallend, also $\abs{x}^{-1} \piketi{x} \le C$ für alle $x \in \R^n$ mit $\abs{x} > \epsilon$.
Damit folgt $C^{-1} \piketi{x} \le \frac{\abs{x}}{\piketi{x}} \piketi{x} = \abs{x}$ und nach Definition \ref{definition:japanese} ist
$\abs{x} \le \abs{x} \frac{\piketi{x}}{\abs{x}} C = C \piketi{x}$.
\end{proof}
\end{lemma}

\begin{lemma}[Ungleichung von Peetre] \label{ungl:peetre}
Für $s\in\R$ und $\xi,\eta\in\R^n$ gilt
$$\piket{\xi}{s} \le 2^{\abs{s}} \piket{\xi - \eta}{\abs{s}} \piket{\eta}{s} .$$
\begin{proof} Siehe \cite{Treves}
\end{proof}
\end{lemma}

\begin{lemma}\label{lemma:masstheoriePiket}
Für $s \in \R_+$ mit $s > n$ ist $\piket{x}{-s} \in L^1(\R^n)$.
\begin{proof} Siehe \cite[Lemma 1.3]{Raymond}
\end{proof}
\end{lemma}

\begin{lemma}\label{lemma:piketTech} 
Sei $m \in \R$ gegeben. 
Dann gibt es zu jedem Multiindex $\alpha \in \N^n_0$ eine Konstante $C_{m,\alpha} > 0$, sodass
$$ \abs{\partial_\xi^\alpha \piket{\xi}{m} } \le  C_{m,\alpha} \piket{\xi}{m-\abs{\alpha}} .$$
\begin{proof}
Betrachte $p : \R^{n+1} \setminus \menge{0} \rightarrow \R, (\lambda,\xi) \mapsto \left(\lambda^2 + \abs{\xi}^2\right)^{\frac{m}{2}}$.
$p$ ist eine homogene Funktion vom Grad $m$, da für jedes $r \in \R \setminus \menge{0}$
$$
p(r\lambda, r\xi) 
= \left( r^2\lambda^2 + r^2\abs{\xi}^2\right)^{\frac{m}{2}} 
= r^m p(\lambda,\xi) \textrm{ für alle } \lambda \in \R, \xi \in \R^n \textrm{ mit } \lambda,\xi \neq 0
.
$$
Sei nun $\alpha \in \N^n_0$. 
Ohne Einschränkungen sei $\partial_\xi^\alpha p(\lambda,\xi) \not\equiv 0$.
Da $p$ homogen zum Grad $m$ ist, ist $\partial_\xi^\alpha p(\lambda,\xi)$ homogen zum Grad $m-\abs{\alpha}$.
Für $\abs{\xi} \ge 1$ haben wir
$$
\abs{ \partial_\xi^\alpha p(\lambda,\xi)}
\le \sup_{\substack{\eta \in \R^n \\ \abs{\eta} = 1}} \abs{ \partial_\xi^\alpha p(\abs{\xi}\lambda, \abs{\xi}\eta) } 
= \abs{\xi}^{m-\abs{\alpha}} \sup_{\abs{\eta} = 1} \abs{ \partial_\xi^\alpha p(\lambda, \eta) }
\le C_{\lambda} \piket{\xi}{m-\abs{\alpha}} 
$$
mit $C_{\lambda} := \sup_{\abs{\eta} = 1} \abs{ \partial_\xi^\alpha p(\lambda, \eta) }$ unabhängig von $\xi \in \R^n$.
Für $\abs{\xi} < 1$ folgt sofort $ \abs{ \partial_\xi^\alpha p(\lambda,\xi)} \le C_{\lambda}$.
Schließlich erhalten wir die Behauptung, wenn wir $\lambda = 1$ setzen.
\end{proof}
\end{lemma}

\subsubsection{Fourier-Transformation}

\begin{definition}
Die Fourier-Transformation einer Funktion $f \in L^1(\R^n)$ bei $\xi \in \R^n$ ist definiert durch
$$ \hat{f}(\xi) := \Fourier{x}{\xi}{f(x)} := \int_{\R^n} e^{-ix\cdot\xi} f(x) dx .$$
Wir schreiben auch kurz $\Fouriers{f} := \hat{f}$.
Da $\abs{ e^{-ix\cdot\xi} f(x) } = \abs{f(x)}$ ist $\Fourieri : L^1(\R^n) \rightarrow L^1(\R^n)$.
Die Fourier-Rücktransformation $\Fourieri^{-1}$ definieren wir durch
$$ \check{f}(x) := \FourierB{\xi}{x}{f(\xi)} := \int_{\R^n} e^{ix\cdot\xi} f(\xi) \dslash\xi ,$$
wobei $\dslash\xi := \frac{d\xi}{(2\pi)^n}$. Es gilt also $\hat{f}(\xi) = (2\pi)^n \check{f}(-\xi)$.
\end{definition}
\nomenclature[t]{$\hat{f}$}{Fourier-Transformierte der Funktion $f$}
\nomenclature[t]{$\Fourieri$}{Fourier-Transformationsoperator}
\nomenclature[t]{$\check{f}$}{Fourier-Rücktransformierte der Funktion $f$}
\nomenclature[t]{$\Fourieri^{-1}$}{Fourier-Rücktransformationsoperator}
\nomenclature{$\dslash\xi$}{Kurzschreibweise für $(2\pi)^{-n} d\xi$}

\begin{cor}
Die lineare Abbildung $\Fourieri$ hält die Abbildungseigenschaft $\Fourieri : L^1(\R^n) \rightarrow \Stetigi{0}$.
\begin{proof}
Für $f \in L^1(\R^n)$ ist zunächst
$$\norm{\Fouriers{f}}{\infty} = \sup_{\xi\in\R^n} \abs{\hat{f}(\xi)} \le \sup_{\xi\in\R^n} \int_{\R^n} \abs{ e^{-ix\cdot\xi} f(x) } dx = \norm{f}{L^1(\R^n)} .$$
Also ist $\hat{f}$ beschränkt.
Setzen wir $g(x,\xi) := e^{-ix\cdot\xi} f(x)$, so ist
$\xi \mapsto g(x,\xi) \in \Stetigi{0}$ für $x \in \R^n$ fest, 
$x \mapsto g(x,\xi) \in L^1(\R^n)$ für jedes $\xi \in \R^n$ und
$\abs{g(x,\xi)} \le \abs{f(x)} \in L^1(\R^n)$.
Also erfüllt $g$ die Voraussetzungen von Satz \ref{satz:masstheorieStetig}, $\hat{f}$ ist demnach stetig.
\end{proof}
\end{cor}

\begin{remark}\label{remark:fourier}
Sei $f \in L^1(\R^n)$. Es gelten die folgenden Eigenschaften:
\begin{itemize}
\item (Translationseigenschaft) Für $y \in \R^n$ folgt mit der Transformationsformel
$$\Fourier{x}{\xi}{ f(x+y) } = \int e^{-ix\cdot\xi} f(x+y) dx = e^{iy\cdot\xi} \int e^{-ix\cdot\xi} f(x) dx = e^{iy\cdot\xi} \hat{f}(\xi) .$$
\item (Dilatationseigenschaft) Für $\lambda \in \R_+$ ist wiederum mit Transformationsformel
$$ \Fourier{x}{\xi}{ f(\lambda x) } = \int e^{-ix\cdot\xi} f(\lambda x) dx = \lambda^{-n} \int e^{-ix\cdot\frac{\xi}{\lambda}} f(x) dx = \lambda^{-n} \hat{f}(\lambda^{-1} \xi) .$$
\item Falls zusätzlich $f \in \Stetigi{1}$ mit $\partial_j f \in L^1(\R^n)$ für ein $j=1,\ldots,n$ ist, folgt mit partieller Integration
$$\Fouriers{D_j f} = \int e^{-ix\cdot\xi} D_{x_j} f(x) dx = \left. \frac{e^{-ix\cdot\xi} f(x)}{-ix} \right|^{x=\infty}_{x=-\infty} + \int e^{-ix\cdot\xi} \xi f(x) dx =  \xi_j \Fouriers{f} =\xi_j \hat{f} .$$
\item Falls zusätzlich $x_j f \in L^1(\R^n)$ ist, dann ist $\hat{f}(\xi)$ stetig differenzierbar in $\xi_j$, da mit Satz \ref{satz:masstheorieAbl}
$$ D_{\xi_j} \hat{f}(\xi) = \int D_{\xi_j} e^{-ix\cdot\xi} f(x) dx = - \int e^{-ix\cdot\xi} x_j f(x) dx = - \Fouriers{x_j f(x)} .$$
\end{itemize}
\end{remark}

\begin{thm}[Plancherel]
$\Fourieri : L^2(\R^n) \xlongrightarrow{\sim} L^2(\R^n)$ ist ein isometrischer Isomorphismus,
d.h. für jedes $f \in L^2(\R^n)$ ist $\hat{f} \in L^2(\R^n)$ mit $\norm{f}{L^2(\R^n)} = \norm{\hat{f}}{L^2(\R^n)}$.
\begin{proof} Siehe \cite[ VI \S 2]{Yosida}
\end{proof}
\end{thm}

\subsubsection{Schwartz-Raum}

\begin{definition}[Schwartz-Raum]
Eine glatte Funktion $f : \R^n \rightarrow \C$ ist Element des Raumes $\Schwarzi$, falls 
$$ \lim_{\abs{x} \rightarrow \infty} \piket{x}{k} \partial_x^\alpha f(x) = 0 \textrm{ für jedes } k \in \N_0 \textrm{ und } \alpha \in \N^n_0 .$$
$\Schwarzi$ wird mit den Halbnormen
\begin{equation}\label{equ:SchwarzNorm}
\SchwarzNorm{ f }{m} := \max_{l + \abs{\alpha} \le m} \sup_{x \in \R^n} \left( \piket{x}{l} \abs{\partial_x^\alpha f(x) } \right) 
\end{equation}
für $m \in \N_0$ zum Fréchet-Raum. 
\end{definition}
\nomenclature[f]{$\Schwarzi$}{Schwartz-Raum}
\nomenclature[f]{$\SchwarziDual$}{Dualraum des Schwartz-Raums}
\nomenclature[n]{$\SchwarzNorm{ \hold }{m}$}{$m$.te Schwartz-Raum-Halbnorm}

\begin{thm}\label{thm:SchwartzFourierNorm}
Die Fourier-Transformation $\Fourieri : \Schwarzi \rightarrow \Schwarzi$ ist eine stetige Bijektion, und für jede natürliche Zahl $m \in \N_0$ existiert eine
Konstante $C > 0$, sodass
$$ \SchwarzNorm{ \Fourier{x}{\hold}{f(x)} }{m} \le C \SchwarzNorm{ f }{m+n+1} \textrm{ für alle } f \in \Schwarzi .$$
\begin{proof} Siehe \cite[\S 1 Theorem 3.2]{Kumanogo}
\end{proof}
\end{thm}

\begin{lemma}
Für $s \in \R_+$ mit $s > \frac{n}{2}$ gibt es eine Konstante $c > 0$, sodass für alle $u \in \Schwarzi$
\begin{equation}\label{equ:SchwartzInftyL2}
\norm{u}{L^\infty(\R^n)} \le c \norm{x \mapsto \piket{D_x}{s} u(x) }{L^2(\R^n)} ,
\end{equation}
und
\begin{equation}\label{equ:SchwartzL2Infty}
\norm{u}{L^2(\R^n)} \le c \norm{x \mapsto \piket{x}{s} u(x)}{L^\infty(\R^n)} .
\end{equation}
\begin{proof}
Wir erhalten mit Plancherel zunächst
$ \begin{xy} \xymatrix{ \Schwarzi \ar[r]_\simh^{\piket{D_x}{s}} & \Schwarzi \ar[r]_\simh^{\Fourieri} & \Schwarzi } \end{xy} $.
Desweiteren ist für ein $f \in L^1(\R^n)$
$$\norm{\check{f}}{L^\infty(\R^n)} 
= \sup_{x\in\R^n} \abs{\check{f}(x)} 
\le \sup_{x\in\R^n} \int_{\R^n} \abs{e^{-ix\cdot\xi} f(\xi)} \dslash\xi 
= \norm{f}{L^1(\R^n)} .$$
Damit, und mit der Ungleichung von Hölder erhalten wir die elementaren Abschätzungen
\begin{align*}
\norm{u}{L^\infty(\R^n)} 
&= \norm{ \piket{D_x}{-s} ( \piket{D_x}{s} u ) }{L^\infty(\R^n)} \\
&= \norm{ \FourierB{\xi}{\hold}{ \piket{\xi}{-s} \Fourier{x}{\xi}{ \piket{D_x}{s} u(x)} }}{L^\infty(\R^n)} \\
&\le \frac{1}{(2\pi)^n} \norm{ \piket{\xi}{-s} \Fourier{x}{\xi}{ \piket{D_x}{s} u(x)}}{L^1(\R^n)} \\
&\le \frac{1}{(2\pi)^n} \norm{ \piket{\xi}{-s}}{L^2(\R^n)} \norm{ \Fourier{x}{\hold}{ \piket{D_x}{s} u(x)}}{L^2(\R^n)} \\
&= \norm{ \piket{\xi}{-2s}}{L^1(\R^n)}^{\frac{1}{2}} \norm{\piket{D_x}{s} u}{L^2(\R^n)} \\
&\le c \norm{\piket{D_x}{s} u}{L^2(\R^n)}
\end{align*}
mit $c := \norm{ \piket{\xi}{-2s}}{L^1(\R^n)}^{\frac{1}{2}} < \infty$ wegen $-2s < -n$.

Für die zweite Ungleichung betrachten wir
\begin{align*}
\norm{u}{L^2(\R^n)}^2
&= \abs{ \int u(x) \piket{x}{s} \overline{u(x)} \piket{x}{-s} dx } \\
&\le \norm{ \piket{\hold}{s} u}{L^\infty(\R^n)} \abs{ \int \overline{u(x)} \piket{x}{s} \piket{x}{-2s} dx} \\
&\le \norm{ \piket{\hold}{s} u}{L^\infty(\R^n)} \norm{ x \mapsto \piket{x}{s} \overline{u(x)}}{L^\infty(\R^n)} \abs{ \int \piket{x}{-2s} dx} \\ 
&\le c^2 \norm{ \piket{\hold}{s} u}{L^\infty(\R^n)}^2  
.\end{align*}
\end{proof}
\end{lemma}

\begin{cor}\label{cor:SchwartzEquivNorm}
Für $u \in \Schwarzi$ definieren wir
$$\SchwarzzNorm{u}{k} := \sup_{\abs{\alpha+\beta} \le k} \norm{x \mapsto \piket{x}{\abs{\alpha}} D_x^\beta u(x)}{L^2(\R^n)} \textrm{ für } k \in \N_0 .$$
$\SchwarzzNorm{\hold}{k}$ ist eine bzgl. $k$ aufsteigende Folge von Halbnormen, die zu der Halbnormfamilie $\SchwarzNorm{\hold}{k}$ äquivalent ist.
Genauer ist für jedes $k\in \N_0$
$$ \SchwarzzNorm{\hold}{k} \le C_k \SchwarzNorm{\hold}{k+2n} \textrm{ und } \SchwarzNorm{\hold}{k} \le C_k \SchwarzzNorm{\hold}{k+2n} .$$
\begin{proof}
Der Operator $\piket{D_x}{2n} = (1-\triangle)^n$ ist ein Differentialoperator der Ordnung $2n$. 
Mit (\ref{equ:SchwartzInftyL2}) erhalten wir für jedes $u \in \Schwarzi$
$$
\SchwarzNorm{u}{k}
= \sup_{l+\abs{\alpha} \le k} \norm{\piket{x}{l} D_x^\alpha u}{L^\infty(\R^n)}
\le C \sup_{l+\abs{\alpha} \le k} \norm{ \piket{D_x}{2n} \piket{x}{l} D_x^\alpha u}{L^2(\R^n)} 
\le C_k \SchwarzzNorm{u}{k+2n}
.$$
Andererseits liefert (\ref{equ:SchwartzL2Infty}) mit dem Polynom $\piket{x}{2n} = \left(1 + \abs{x}^2\right)^n$ vom Grad $2n$ die Abschätzung
$$
\SchwarzzNorm{u}{k} 
= \sup_{l+\abs{\alpha} \le k} \norm{\piket{x}{l} D_x^\alpha u}{L^2(\R^n)}
\le \sup_{l+\abs{\alpha} \le k} \norm{ \piket{x}{2n+l} D_x^\alpha
u}{L^\infty(\R^n)} \le C_k \SchwarzNorm{u}{k+2n}
.$$
\end{proof}
\end{cor}
\nomenclature[n]{$\SchwarzzNorm{\hold}{m}$}{$m$.te alternative Schwartz-Raum-Halbnorm}
\nomenclature{$\triangle_x$}{Laplace-Operator $\sum_{j=1}^n \partial_{x_j}^2$}

\subsection{Oszillatorische Integrale}\label{sec:oszi}

\begin{definition}
Sei $m \in \R, 0 \le \delta \le 1, 0 \le \tau$.
Eine glatte Funktion $a(\eta,y) \in C^\infty(\R^n\times\R^n)$ gehört der Klasse $\Oszillatory{\delta,\tau}{m}$ an, 
falls es für alle Multiindizes $\alpha,\beta \in \N^n_0$ eine Konstante $C_{\alpha,\beta} > 0$ gibt, sodass
$$\abs{\partial_\eta^\alpha \partial_y^\beta a(\eta,y) } \le C_{\alpha,\beta} \piket{\eta}{m+\delta\abs{\beta}} \piket{y}{\tau}.$$
Wir setzen $\Oszillatoryi := \bigcup_{0 \le \delta < 1} \bigcup_{m \in \R} \bigcup_{0 \le \tau} \Oszillatory{\delta,\tau}{m}$.
Für $a(\eta,y) \in \Oszillatory{\delta,\tau}{m}$ sei die Familie von Seminormen $\menge{ \OszillatoryNorm{a}{\delta,\tau}{m}{l}}_{l\in\N}$ gegeben durch
$$ \OszillatoryNorm{a}{\delta,\tau}{m}{l} = \max_{\abs{\alpha+\beta}\le l} \sup_{\eta,y\in\R} \abs{\partial_\eta^\alpha \partial_y^\beta a(\eta,y) } \piket{\eta}{-m-\delta\abs{\beta}} \piket{y}{-\tau}.$$
Diese Familie induziert den Fréchet-Raum $\Oszillatory{\delta,\tau}{m}$. Insbesondere ist für $a(\eta,y) \in \Oszillatory{\delta,\tau}{m}$
$$\abs{\partial_\eta^\alpha \partial_y^\beta a(\eta,y) } \le \OszillatoryNorm{a}{\delta,\tau}{m}{\abs{\alpha+\beta}} \piket{\eta}{m+\delta\abs{\beta}} \piket{y}{\tau} .$$
\end{definition}
\nomenclature[f]{$\Oszillatory{\delta,\tau}{m}$}{Klasse oszillatorischer Funktionen zum Grad $m$ mit Parameter $\delta,\tau$}
\nomenclature[f]{$\Oszillatoryi$}{Vereinigung aller Klassen oszillatorischer Funktionen}
\nomenclature[n]{$\OszillatoryNorm{\hold}{\delta,\tau}{m}{l}$}{$l$.te Halbnorm der Klasse $\Oszillatory{\delta,\tau}{m}$}

\begin{lemma}\label{lemma:osziChi}
Sei  $\chi(\eta,y) \in \Schwarz{\R^n\times\R^n}$ mit $\chi(0,0) = 1$.
Dann gilt für $\epsilon \rightarrow 0$, dass
\begin{equation}\label{equ:oszichia}
\begin{array}{l}
\chi(\epsilon x) \rightarrow 1 \textrm{ gleichmäßig in } \R^n \textrm{ auf einer kompakten Menge, } \\
\textrm{ und } \partial_x^\alpha \chi(\epsilon x) \rightarrow 0 \textrm{ gleichmäßig für alle } \alpha \in \N^n_0, \alpha \neq 0 .
\end{array}
\end{equation}
Außerdem gibt es zu jedem Multiindex $\alpha \in \N^n_0$ eine Konstante $C_\alpha > 0$, die unabhängig von $\epsilon \in (0,1)$ die Ungleichung
\begin{equation}\label{equ:oszichib}
\abs{\partial_x^\alpha \chi(\epsilon x) } \le C_\alpha \epsilon^\sigma \piket{x}{-(\abs{\alpha}-\sigma)} \textrm{ für beliebiges } \sigma \in \left[0,\abs{\alpha}\right] 
\end{equation}
erfüllt.
\begin{proof}
Wegen $\partial_x^\alpha \chi(\epsilon x) = \epsilon^{\abs{\alpha}} \partial_y^\alpha \chi(y) |_{y=\epsilon x}$ ist (\ref{equ:oszichia}) klar.
Die Gleichheit 
$$\abs{ \partial^\alpha_x \chi(\epsilon x)} \piket{x}{\abs{\alpha}-\sigma} = \epsilon^\sigma \epsilon^{\abs{\alpha} - \sigma} \abs{\partial^\alpha_y \chi(y) |_{y=\epsilon x}} \piket{x}{\abs{\alpha}-\sigma} \le C_\alpha \epsilon^\sigma$$
löst die Ungleichung (\ref{equ:oszichib}) für $\abs{x} \le 1$.
Für $\abs{x} > 1$ verwende
$$ \abs{\partial_x^\alpha \chi(\epsilon x) } \epsilon^{-\sigma} =  \epsilon^{\abs{\alpha}-\sigma} \abs{ \partial_y^\alpha \chi(y) |_{y=\epsilon x} } = \left( \abs{y}^{\abs{\alpha}-\sigma} \abs{\partial_y^\alpha \chi(y)}\right) |_{y=\epsilon x} \abs{x}^{-(\abs{\alpha}-\sigma)} \le C_\alpha \piket{x}{-(\abs{\alpha}-\sigma)}$$
für beliebiges $\sigma \in \left[0,\abs{\alpha}\right]$.
\end{proof}
\end{lemma}

\begin{definition}\label{def:osziOsziInt}
Sei ein $a(\eta,y) \in \Oszillatoryi$ gegeben.
Wir definieren das oszillatorische Integral $\os{\eta}{y}{a(\eta,y)}$ durch
$$\os{\eta}{y}{a(\eta,y)} := \lim_{\epsilon\rightarrow 0} \iint e^{-iy\cdot\eta} \chi(\epsilon\eta,\epsilon y) a(\eta,y) dy \dslash\eta$$
für ein $\chi(\eta,y) \in \Schwarz{\R^n \times \R^n}$ mit $\chi(0,0) = 1$.
\end{definition}
\nomenclature{$\osint$}{Oszillatorisches Integral - Definition \ref{def:osziOsziInt}}

\begin{thm}\label{thm:osziTransform}
Für $a(\eta,y) \in \Oszillatoryi$ ist das oszillatorische Integral 
$$\os{\eta}{y}{a(\eta,y)}$$
wohldefiniert, d.h. unabhängig von der Wahl eines $\chi(\eta,y) \in \Schwarz{\R^n \times \R^n}$ mit $\chi(0,0) = 1$.

Ist $a \in \Oszillatory{\delta,\tau}{m}$, so ist 
$$ \abs{ \piket{y}{-2l'} \piket{D_\eta}{2l'} \left( \piket{\eta}{-2l} \piket{D_y}{2l} a(\eta,y) \right) } \in L^1(\R^n\times\R^n)$$
für $l,l' \in \N$ mit
\begin{equation}\label{equ:osziLL}
-2l(1-\delta) + m < -n \textrm{ und } -2l' + \tau < -n .
\end{equation}
Dementsprechend lässt sich das oszillatorische Integral umschreiben in
\begin{equation}\label{equ:osziUmschreiben}
\os{\eta}{y}{a(\eta,y)} = \iint  e^{-iy\cdot\eta} \piket{y}{-2l'} \piket{D_\eta}{2l'} \left( \piket{\eta}{-2l} \piket{D_y}{2l} a(\eta,y) \right) dy \dslash\eta .
\end{equation}
Insbesondere gibt es eine Konstante $C > 0$ unabhängig von $a(\eta,y) \in \Oszillatory{\delta,\tau}{m}$, sodass
\begin{equation}\label{equ:osziL0}
\abs{\os{\eta}{y}{ a(\eta,y)}} \le C \OszillatoryNorm{a}{\delta,\tau}{m}{l_0} \textrm{ für } l_0 := 2(l+l') .
\end{equation}

\begin{proof}
Unter Verwendung der Identitäten 
$$D_y^\alpha e^{-iy\cdot\eta} = (-\eta)^\alpha e^{-iy\cdot\eta} \textrm{ sowie } D_\eta^\beta e^{-iy\cdot\eta} = (-y)^\beta e^{-iy\cdot\eta}$$
bekommen wir
$ \piket{\eta}{-2l} \piket{D_y}{2l} e^{-iy\cdot\eta} = e^{-iy\cdot\eta}$ bzw. $\piket{y}{-2l'} \piket{D_\eta}{2l'} e^{-iy\cdot\eta} = e^{-iy\cdot\eta}$.
Da $\chi(\epsilon\eta,\epsilon y) \in \Schwarz{\R^n \times \R^n}$ erhalten wir mit partieller Integration 
\begin{align*}
 I_\epsilon :&= \iint e^{-iy\cdot\eta} \chi(\epsilon\eta, \epsilon y) a(\eta,y) dy \dslash\eta  \\
 &= \iint e^{-iy\cdot\eta} \piket{\eta}{-2l} \piket{D_y}{2l} \left( \chi(\epsilon\eta, \epsilon y) a(\eta,y) \right) dy \dslash\eta \\
&= \iint e^{-iy\cdot\eta} \piket{y}{-2l'} \piket{D_\eta}{2l'} \left( \piket{\eta}{-2l} \piket{D_y}{2l} \left( \chi(\epsilon\eta, \epsilon y) a(\eta,y) \right) \right) dy \dslash\eta
.\end{align*}
Nach Lemma $\ref{lemma:osziChi}$ ist $\menge{\chi(\epsilon\eta,\epsilon y)}_{\epsilon\in(0,1)}$ in $\Oszillatory{0,0}{0}$ beschränkt.
Also gibt es zu allen Multiindizes $\alpha,\beta\in\N^n_0$ eine Konstante $C_{\alpha,\beta}>0$ unabhängig von $\epsilon\in(0,1)$ und $a \in \Oszillatory{\delta,\tau}{m}$, sodass
$$\abs{\partial_\eta^\alpha \partial_y^\beta \left[ \chi(\epsilon\eta, \epsilon y) a(\eta,y) \right] } \le C_{\alpha,\beta} \OszillatoryNorm{a}{\delta,\tau}{m}{\abs{\alpha+\beta}} \piket{\eta}{m+\delta\abs{\beta}} \piket{y}{\tau} .$$

Unter Verwendung des Lemmas \ref{lemma:piketTech} finden wir eine Konstante $C_{l,\alpha}$, die die Ungleichung
$$ \abs{ \partial_\eta^\alpha \left( \piket{\eta}{-2l} \piket{D_y}{2l}  \left( \chi(\epsilon\eta,\epsilon y) a(\eta,y) \right) \right) } \le
C_{l,\alpha} \OszillatoryNorm{a}{\delta,\tau}{m}{2l+\abs{\alpha}} \piket{\eta}{m-2l(1-\delta)} \piket{y}{\tau}$$
erfüllt.
Dementsprechend gibt es eine Konstante $C_{l,l'} > 0$, die von $\epsilon\in(0,1)$ und $a(\eta,y) \in \Oszillatory{\delta,\tau}{m}$ unabhängig ist, sodass
\begin{align}\label{equ:osziBeweisL0}
\abs{ \piket{y}{-2l'} \piket{D_\eta}{2l'} \left( \piket{\eta}{-2l} \piket{D_y}{2l} \left( \chi(\epsilon\eta, \epsilon y) a(\eta,y) \right) \right) } 
\le C_{l,l'} \OszillatoryNorm{a}{\delta,\tau}{m}{2(l+l')} \piket{\eta}{m-2l(1-\delta)} \piket{y}{-2l'+\tau}
.\end{align}
Wegen der Bedingung (\ref{equ:osziLL}) ist $(\eta,y) \mapsto \abs{a}_{2(l+l')} \piket{\eta}{m-2l(1-\delta)} \piket{y}{-2l'+\tau} \in L^1(\R^n\times\R^n)$ nach Lemma \ref{lemma:masstheoriePiket}.
Mit Lemma \ref{lemma:osziChi} und der Lebesgue'schen dominanten Konvergenz bekommen wir
$$ \os{\eta}{y}{ a(\eta,y)} = \lim_{\epsilon\rightarrow 0} I_\epsilon =
\iint e^{-iy\cdot\eta} \piket{y}{-2l'} \piket{D_\eta}{2l'} \left( \piket{\eta}{-2l} \piket{D_y}{2l} a(\eta,y) \right) dy \dslash\eta ,$$
und damit (\ref{equ:osziUmschreiben}).
Da die Gleichung (\ref{equ:osziUmschreiben}) das oszillatorische Integral ohne $\chi$ beschreibt, 
ist der Wert von $\os{\eta}{y}{a(\eta,y)}$ unabhängig von der Wahl eines $\chi(\eta,y) \in \Schwarz{\R^n\times\R^n}$ mit $\chi(0,0) = 1$.
Mit (\ref{equ:osziBeweisL0}) und (\ref{equ:osziUmschreiben}) bekommen wir für $\epsilon \rightarrow 0$ schließlich (\ref{equ:osziL0}).
\end{proof}
\end{thm}

\begin{cor}\label{cor:osziLebesgue}
Für $a(\eta,y)\in L^1(\R^n\times\R^n)$ ist
$$\osint e^{-iy\cdot\eta} a(\eta,y) dy \dslash\eta = \iint e^{-iy\cdot\eta} a(\eta,y) dy \dslash\eta .$$
\end{cor}

\begin{lemma}\label{tech:osziAblAbsch}
Für $I := [0,1] \subset \R$ gibt es eine Konstante $M>0$, sodass für alle $f \in C^2(I)$ gilt
$$ \max_{t\in I} \abs{f'(t)}^2 \le M  \max_{t\in I} \abs{f(t)} \left( \max_{t\in I} \abs{f(t)} + \max_{t\in I} \abs{f''(t)} \right) .$$
\begin{proof}
Sei ohne Einschränkungen $f : I \rightarrow \R, f \neq 0$.
Wähle $t \in [0,\frac{1}{2}]$.
Für ein $\epsilon \in (0,\frac{1}{4}]$ existiert nach dem Mittelwertsatz ein $t_1 \in (t+\epsilon,t+2\epsilon)$, sodass
$$ \frac{1}{\epsilon} \left( f(t+2\epsilon) - f(t+\epsilon) \right) = f'(t_1) .$$
Andererseits erhalten wir nach dem Hauptsatz der Integralrechnung
$$ f'(t_1) - f'(t) = \int_{t}^{t_1} f''(\tau) d\tau .$$
Mit $0 < t_1 - t < 2\epsilon$ erhalten wir
$$ \abs{f'(t)} \le \abs{\int_{t}^{t_1} f''(\tau) d\tau} + \abs{\frac{1}{\epsilon} \left( f(t+2\epsilon) - f(t+\epsilon) \right)} \le 2\epsilon \max_{t\in I}\abs{f''(t)} + \frac{2}{\epsilon} \max_{t\in I}\abs{f(t)} .$$
Für $t\in[\frac{1}{2},1]$ erhält man analog für $\epsilon \in [0,\frac{1}{4}]$ ein $t_1 \in (t-\epsilon,t-2\epsilon)$, sodass
$$ \frac{1}{\epsilon} \left( f(t-2\epsilon) - f(t-\epsilon) \right) = f'(t_1) ,$$
und
$$f'(t_1) - f'(t) = \int_t^{t_1} f''(\tau) d\tau .$$
Wiederum folgt mit $0 < t - t_1 < 2\epsilon$
$$\abs{f'(t)} \le 2\epsilon \max_{t\in I}\abs{f''(t)} + \frac{2}{\epsilon} \max_{t\in I}\abs{f(t)}.$$
Setzen wir 
$$\epsilon := \frac{1}{4} \sqrt{ \frac{ \max_{t\in I}\abs{f(t)} }{ \max_{t\in I}\abs{f(t)} + \max_{t\in I}\abs{f''(t)} } } ,$$
so erhalten wir für beide Fälle aus der obigen Ungleichung
$$ \abs{f'(t)} \le 9  \sqrt{ \max_{t\in I} \abs{f(t)} \left( \max_{t\in I} \abs{f(t)} + \max_{t\in I} \abs{f''(t)} \right) } .$$
\end{proof}
\end{lemma}

\begin{definition}
Wir bezeichnen eine Untermenge $B$ von $\Oszillatoryi$ als beschränkte Untermenge von $\Oszillatoryi$,
falls es $m \in \R, 0 \le \delta \le 1, 0 \le \tau$ gibt, sodass
$$ B \subset \Oszillatory{\delta,\tau}{m} \textrm{ und } \sup_{a(\eta,y)\in B} \OszillatoryNorm{a}{\delta,\tau}{m}{l} < \infty \textrm{ für jedes } l \in \N .$$
\end{definition}

\begin{thm}\label{thm:osziKonvergenz}
Sei $\folge{a_j(\eta,y)}{j\in\N}$ eine beschränkte Untermenge und Folge in $\Oszillatoryi$.
Nehme an, es gebe einen Grenzwert $a(\eta,y) \in \Oszillatoryi$, sodass
$$ a_j(\eta,y) \rightarrow a(\eta,y) \textrm{ für } j \rightarrow \infty$$
gleichmäßig in $\R^n\times\R^n$ auf jeder kompakten Teilmenge konvergiert.
Dann haben wir
$$ \lim_{j\rightarrow\infty} \os{\eta}{y}{a_j(\eta,y)} = \os{\eta}{y}{a(\eta,y)} .$$
\begin{proof}
Bezeichne mit $\menge{e_k}_{k=1,\ldots,n}$ die kanonische Basis des $\R^n$.
Definiere für $\eta,y\in\R^n$ fest gewählt die Folge $\folge{f_j(t)}{j\in\N}$ durch
$$f_j(t) := a_j(\eta + te_k,y) - a(\eta+te_k,y) .$$
Dann ist
$$\partial_t f_j(t) = \partial_{x_k} a_j(x,y) |_{x=\eta+te_k} - \partial_{x_k} a(x,y) |_{x=\eta+te_k} .$$
Da $\max_{t\in[0,1]} \abs{f_j(t)} \rightarrow 0$ und die Folge
$\folge{ \max_{t\in[0,1]} \abs{f_j(t)} + \max_{t\in[0,1]} f_j''(t)}{j\in\N}$ beschränkt ist,
folgt nach Lemma \ref{tech:osziAblAbsch}, dass
$$ f_j'(0) = \partial_{\eta_k} a_j(\eta,y) - \partial_{\eta_k} a(\eta,y) \rightarrow 0 \textrm{ für } j \rightarrow \infty$$
gleichmäßig auf jeder kompakten Teilmenge von $\R^n\times\R^n$ konvergiert.
Wir können also induktiv mit dem gleichen Vorgehen die Funktion $t \mapsto \partial_\eta^\alpha \partial_y^\beta a_j(\eta + te_k,y) - \partial_\eta^\alpha \partial_y^\beta a(\eta+te_k,y)$ für $\alpha,\beta\in\N^n_0$ betrachten,
mit der wiederum aus Lemma \ref{tech:osziAblAbsch} folgt, dass 
$$\partial_{\eta_k} \partial_\eta^\alpha \partial_y^\beta a_j(\eta,y) - \partial_{\eta_k} \partial_\eta^\alpha \partial_y^\beta   a(\eta,y) \rightarrow 0 \textrm{ für } j \rightarrow \infty \textrm{ für jedes } k = 1,\ldots,n \textrm{ und }$$
$$\partial_{y_k} \partial_\eta^\alpha \partial_y^\beta a_j(\eta,y) - \partial_{y_k} \partial_\eta^\alpha \partial_y^\beta   a(\eta,y) \rightarrow 0 \textrm{ für } j \rightarrow \infty \textrm{ für jedes } k = 1,\ldots,n .$$
Da  $\menge{a_j(\eta,y)}_{j\in\N}$ eine beschränkte Untermenge ist, gibt es ein $m \in \R, 0 \le \delta \le 1, 0 \le \tau$, 
sodass $\menge{a_j(\eta,y)}_{j\in\N} \subset \Oszillatory{\delta,\tau}{m}$ und wegen der gleichmäßigen Konvergenz auch $a(\eta,y) \in \Oszillatory{\delta,\tau}{m}$.
Also folgt die Behauptung mit der dominanten Konvergenz von Lebesgue, wenn wir Theorem \ref{thm:osziTransform} anwenden.
\end{proof}
\end{thm}

\begin{remark}
Das Theorem kann als Analogon zum Satz von Lebesgue verstanden werden.
Statt der $L^2$-Beschränktheit haben wir die Beschränktheit von $\menge{a_j}_{j\in\N} \subset \Oszillatoryi$ vorausgesetzt.
\end{remark}

\begin{thm}\label{thm:osziPartielle}
Für $a \in \Oszillatoryi$ und jeden Multiindex $\alpha \in \N^n_0$ gilt
\begin{align*}
\os{\eta}{y}{y^\alpha a(\eta,y)} 
&= \osint (-D_\eta)^\alpha e^{-iy\cdot\eta} a(\eta,y) dy \dslash\eta \\
&=  \os{\eta}{y}{D_\eta^\alpha a(\eta,y)}
,\end{align*}
und
\begin{align*}
\os{\eta}{y}{\eta^\alpha a(\eta,y)} 
&= \osint (-D_y)^\alpha e^{-iy\cdot\eta} a(\eta,y) dy \dslash\eta \\
&=  \os{\eta}{y}{D_y^\alpha a(\eta,y)}
.\end{align*}
\begin{proof}
Unter Ausnutzung der Identität
$$D_\eta^\alpha e^{-iy\cdot\eta} = (-y)^\alpha e^{-iy\cdot\eta}$$
ist für $\chi(\eta,y) \in \Schwarz{\R^n\times\R^n}$ mit $\chi(0,0) = 1$
\begin{align*}
\os{\eta}{y}{y^\alpha a(\eta,y)} 
&= \lim_{\epsilon \rightarrow 0} \iint (-D_\eta)^\alpha e^{-iy\cdot\eta} \chi(\epsilon\eta, \epsilon y) a(\eta,y) dy \dslash\eta \\
&= \lim_{\epsilon \rightarrow 0} \iint e^{-iy\cdot\eta} D_\eta^\alpha \left[ \chi(\epsilon\eta, \epsilon y) a(\eta,y) \right] dy \dslash\eta  \\
&= \lim_{\epsilon \rightarrow 0} \os{\eta}{y}{ D_\eta^\alpha \left[ \chi(\epsilon\eta, \epsilon y) a(\eta,y) \right] }
.\end{align*}
Finden wir $m\in\R, \delta \in (0,1), \tau > 0$, sodass $a \in \Oszillatory{\delta,\tau}{m}$, dann ist
$$ \menge{  D_\eta^\alpha \left[ \chi(\epsilon\eta, \epsilon y) a(\eta,y) \right] }_{\epsilon \in (0,1)} \subset \Oszillatory{\delta,\tau}{m}$$
beschränkt. Für jede kompakte Teilmenge $K \subset \R^n\times\R^n$ konvergiert
$$\menge{ (\eta,y) \rightarrow D_\eta^\alpha \left[ \chi(\epsilon\eta, \epsilon y) a(\eta,y) \right] } \rightarrow D_\eta^\alpha a(\eta,y) \textrm{ für } \epsilon \rightarrow 0 \textrm{ auf } (\eta,y) \in K \textrm{ gleichmäßig } .$$
Mit Theorem \ref{thm:osziKonvergenz} folgt schließlich
$$\lim_{\epsilon \rightarrow 0} \os{\eta}{y}{ D_\eta^\alpha \left[ \chi(\epsilon\eta, \epsilon y) a(\eta,y) \right] } = \os{\eta}{y}{D_\eta^\alpha a(\eta,y)} .$$
Die zweite Gleichung folgt analog mit $D_y^\alpha e^{-iy\cdot\eta} = (-\eta)^\alpha e^{-iy\cdot\eta}$
\begin{align*}
\os{\eta}{y}{\eta^\alpha a(\eta,y)} 
&= \lim_{\epsilon \rightarrow 0} \iint (-D_y)^\alpha e^{-iy\cdot\eta} \chi(\epsilon\eta, \epsilon y) a(\eta,y) dy \dslash\eta  \\
&= \lim_{\epsilon \rightarrow 0} \iint e^{-iy\cdot\eta} D_y^\alpha \left[ \chi(\epsilon\eta, \epsilon y) a(\eta,y) \right] dy \dslash\eta  \\
&= \lim_{\epsilon \rightarrow 0} \os{\eta}{y}{ D_y^\alpha \left[ \chi(\epsilon\eta, \epsilon y) a(\eta,y) \right] }
.\end{align*}
Mit gleicher Argumentation folgt also wieder mit Theorem \ref{thm:osziKonvergenz}, dass
$$\lim_{\epsilon \rightarrow 0} \os{\eta}{y}{ D_y^\alpha \left[ \chi(\epsilon\eta, \epsilon y) a(\eta,y) \right] } = \os{\eta}{y}{D_y^\alpha a(\eta,y)} .$$
\end{proof}
\end{thm}

\begin{thm}\label{thm:osziFubini}
Sei $a \in \Oszillatory{\delta,\tau}{m}$.
\begin{enumerate}[(a)]
\item \label{item:osziFubiniA} Falls es ein $0 \le f \in L^1(\R^n)$ gibt mit $\abs{a(\eta,y)} \le \piket{\eta}{m} f(y)$ für alle $(\eta,y) \in \R^{2n}$, 
dann gilt für alle $\chi \in \Schwarzi$ mit $\chi(0) = 1$, dass
\begin{equation}\label{equ:osziFubini1}
\os{\eta}{y}{a(\eta,y)} = \lim_{\epsilon \rightarrow 0} \iint e^{-iy\cdot\eta} \chi(\epsilon\eta) a(\eta,y) dy \dslash\eta .
\end{equation}
\item \label{item:osziFubiniB} Falls es ein $0 \le g \in L^1(\R^n)$ gibt mit $\abs{a(\eta,y)} \le \piket{y}{\tau} g(\eta)$ für alle $(\eta,y) \in \R^{2n}$, 
dann gilt für alle $\chi \in \Schwarzi$ mit $\chi(0) = 1$, dass
\begin{equation}\label{equ:osziFubini2}
\os{\eta}{y}{a(\eta,y)} = \lim_{\epsilon \rightarrow 0} \iint e^{-iy\cdot\eta} \chi(\epsilon y) a(\eta,y) dy \dslash\eta .
\end{equation}
\item \label{item:osziFubiniC} Falls zusätzlich zu den Bedingungen von (\ref{item:osziFubiniA}) gilt, dass
$\eta \mapsto \int e^{-iy\cdot\eta} a(\eta,y) dy$ integrabel ist, 
dann können wir schreiben
$$ \os{\eta}{y}{a(\eta,y)} = \int \left( \int e^{-iy\cdot\eta} a(\eta,y) dy \right) \dslash\eta .$$
\item \label{item:osziFubiniD} Falls zusätzlich zu den Bedingungen von (\ref{item:osziFubiniB}) gilt, dass
$y \mapsto \int e^{-iy\cdot\eta} a(\eta,y) \dslash\eta$ integrabel ist, dann können wir schreiben
$$ \os{\eta}{y}{a(\eta,y)} =  \int \left( \int e^{-iy\cdot\eta} a(\eta,y) \dslash\eta \right) dy .$$
\end{enumerate}
\begin{proof}
\begin{enumerate}[(a)]
\item Wir sehen sofort, dass $y \mapsto a(\eta,y) \in L^1(\R^n)$.
Folgerung \ref{cor:osziLebesgue} liefert für $\chi(\epsilon\eta) a(\eta,y) \in L^1(\R^n\times\R^n)$, dass
$$\lim_{\epsilon \rightarrow 0} \iint e^{-iy\cdot\eta} \chi(\epsilon\eta) a(\eta,y) dy \dslash\eta = \lim_{\epsilon \rightarrow 0} \os{\eta}{y}{\chi(\epsilon\eta)a(\eta,y)} .$$
Da $ \menge{ \chi(\epsilon\eta) a(\eta,y) }_{\epsilon \in (0,1)}$ eine beschränkte Teilmenge von $\Oszillatory{\delta,\tau}{m}$ ist, folgt mit Theorem \ref{thm:osziKonvergenz}, 
dass 
$$\lim_{\epsilon \rightarrow 0} \os{\eta}{y}{\chi(\epsilon\eta)a(\eta,y)} = \os{\eta}{y}{a(\eta,y)} .$$
\item Folgt analog wegen $\eta \mapsto a(\eta,y) \in L^1(\R^n)$.
\item Wir haben nach (\ref{equ:osziFubini1})
$$ \os{\eta}{y}{a(\eta,y)} = \lim_{\epsilon \rightarrow 0} \int \chi(\epsilon\eta) \int e^{-iy\cdot\eta} a(\eta,y) dy \dslash\eta .$$
Nach Voraussetzung erhalten wir eine integrable Dominante, sodass wir mit dem Satz von Lebesgue erhalten
$$ \int \chi(\epsilon\eta) \int e^{-iy\cdot\eta} a(\eta,y) dy \dslash\eta \rightarrow \iint e^{-iy\cdot\eta} a(\eta,y) dy \dslash\eta \textrm{ für } \epsilon \rightarrow 0 .$$
\item Verwende (\ref{equ:osziFubini2}) unter Ausnutzung des Satzes von Fubini, sodass
$$ \os{\eta}{y}{a(\eta,y)} = \lim_{\epsilon \rightarrow 0} \int \chi(\epsilon y) \int  e^{-iy\cdot\eta} a(\eta,y) \dslash\eta dy .$$
Analog wie in (\ref{item:osziFubiniC}) folgt mit dem Satz von Lebesgue die Behauptung.
\end{enumerate}
\end{proof}
\end{thm}

\begin{cor}
Für $a \in \Beschri$ ist $\mathcal{F}^{-1} \circ \mathcal{F} \left( a(x) \right) = a(x)$ für alle $x \in \R^n$.
D.h. wir haben formal ``$\mathcal{F}^{-1} \circ \mathcal{F} \mid_{\Beschri} = \id_{\Beschri}$''.
\begin{proof}
Für ein $a \in \Beschri$ ist $(y,\eta) \rightarrow e^{ix\cdot\eta} a(y) \in \Oszillatory{0,0}{0}$ für beliebiges $x\in\R^n$.
Demnach können wir die Fourier-Transformation als oszillatorisches Integral
$$ \FourierB{\eta}{x}{ \Fourier{y}{\eta}{a(y)}} = \os{y}{\eta}{e^{ix\cdot\eta} a(y)} $$
auffassen.
Sei $\chi \in \Schwarzi$ mit $\chi(0) = 1$ und setze $\tilde{\chi}(\eta,y) := \chi(\eta) \chi(y)$.
Mit Theorem \ref{thm:osziFubini}.\ref{item:osziFubiniC} und Theorem \ref{thm:osziFubini}.\ref{item:osziFubiniD} folgt
\begin{align*}
\FourierB{\eta}{x}{ \Fourier{y}{\eta}{a(y)}} 
&= \lim_{\epsilon \rightarrow 0} \iint e^{i(x-y)\eta} \tilde{\chi}(\epsilon\eta, \epsilon y) a(y) \dslash\eta dy \\
&= \lim_{\epsilon \rightarrow 0} \int a(y) \chi(\epsilon y) \left( \int e^{i(x-y)\cdot\eta} \chi(\epsilon\eta) \dslash\eta \right) dy
.\end{align*}
Wenn wir zuerst mit $\eta \mapsto \frac{\eta}{\epsilon}$ und danach mit $y \mapsto \frac{x-y}{\epsilon} =: y'$ transformieren, erhalten wir
\begin{align*}
\int a(y) \chi(\epsilon y) \left( \int e^{i(x-y)\cdot\eta} \chi(\epsilon\eta) \dslash\eta \right) dy
=& \epsilon^{-n} \int a(y) \chi(\epsilon y) \left( \int e^{i\frac{(x-y)}{\epsilon}\cdot\eta} \chi(\eta) \dslash\eta \right) dy \\
=& \epsilon^{-n} \int \chi(\epsilon y) a(y) \FourierB{\eta}{\frac{x-y}{\epsilon}}{\chi(\eta)} dy \\
=& \int \chi(\epsilon(x-\epsilon y')) a(x-\epsilon y') \FourierB{\eta}{y'}{\chi(\eta)} dy'
.\end{align*}
Dominante Konvergenz liefert schließlich für $\epsilon \rightarrow 0$
\begin{align*}
\int \chi(\epsilon(x-\epsilon y')) a(x-\epsilon y') \FourierB{\eta}{y'}{\chi(\eta)} dy' 
&\rightarrow a(x) \int \FourierB{\eta}{y'}{\chi(\eta)} dy' \\
&= a(x) \chi(0) = a(x) 
.\end{align*}
\end{proof}
\end{cor}

\begin{cor}
Für $a \in \Beschri$ gilt die Identität
\begin{equation}\label{equ:inversionFormula} a(x) = \osint e^{-iy\cdot\eta} a(x+y) dy\dslash\eta \textrm{ für alle } x \in \R^n .
\end{equation}
\begin{proof}
Wir haben für ein $\chi \in \Schwarz{\R^n\times\R^n}$ mit $\chi(0) = 1$
\begin{align*}
a(x) &= \mathcal{F}^{-1} \circ \mathcal{F} a(x)
= \iint e^{i(x-y)\cdot\eta} \chi(\epsilon\eta,\epsilon y) a(y) dy \dslash\eta \\
&= \iint e^{-iy\cdot\eta} \chi(\epsilon\eta, \epsilon (x+y)) a(x+y) dy \dslash\eta  
= \osint  e^{-iy\cdot\eta} a(x+y) dy\dslash\eta 
.\end{align*}
\end{proof}
\end{cor}

\subsection{Littlewood-Paley-Partitionen}

\begin{definition}[Littlewood-Paley Partition]\label{def:littlewoodpaley}
Eine Familie von glatten Funktionen 
$$\menge{\lilwood{j}}_{j \in \N_0}, \lilwood{j} : \R^n \rightarrow \R$$
heißt \emph{Littlewood-Paley-Partition}, wenn

\begin{equation}
\lilwood{0}(\xi) = 
\begin{cases}
1 & \textrm{ für alle } \abs{\xi} \le 1,\\
0 & \textrm{ für alle } \abs{\xi} \ge 2,
\end{cases}
\end{equation}
und
\begin{equation}\label{equ:littlewoodRekursion}
\lilwood{j}(\xi) = \lilwood{0}(2^{-j} \xi) - \lilwood{0}(2^{-j+1} \xi)\; \textrm{ für jedes } j \in \N.
\end{equation}
Offensichtlich ist für jedes $j\in\N_0$ der Träger von $\lilwood{j}$ in dem dyadischen Ring 
\begin{equation}
D_j := 
\begin{cases}
B_2(0) & \textrm{ für } j = 0, \\
\menge{ \xi \in \R^n : 2^{j-1} \le \abs{\xi} \le 2^{j+1} } & \textrm{ falls } j \in \N\\
\end{cases}
\end{equation}
enthalten.
\end{definition}
\nomenclature{$\menge{\lilwood{j}}_{j \in \N_0}$}{Vertreter einer Littlewood-Paley-Partition}
\nomenclature[m]{$D_j$}{$j$.ter Dyadischer Ring}

\begin{figure}[ht]
	\caption{Beispiel einer Littlewood-Paley Partition}
	\centering
	\setlength{\unitlength}{0.240900pt}
\ifx\plotpoint\undefined\newsavebox{\plotpoint}\fi
\sbox{\plotpoint}{\rule[-0.200pt]{0.400pt}{0.400pt}}\begin{picture}(1500,900)(0,0)
\sbox{\plotpoint}{\rule[-0.200pt]{0.400pt}{0.400pt}}\put(110,147){\makebox(0,0)[r]{ 0}}
\put(130.0,147.0){\rule[-0.200pt]{4.818pt}{0.400pt}}
\put(110,276){\makebox(0,0)[r]{ 0.2}}
\put(130.0,276.0){\rule[-0.200pt]{4.818pt}{0.400pt}}
\put(110,406){\makebox(0,0)[r]{ 0.4}}
\put(130.0,406.0){\rule[-0.200pt]{4.818pt}{0.400pt}}
\put(110,535){\makebox(0,0)[r]{ 0.6}}
\put(130.0,535.0){\rule[-0.200pt]{4.818pt}{0.400pt}}
\put(110,665){\makebox(0,0)[r]{ 0.8}}
\put(130.0,665.0){\rule[-0.200pt]{4.818pt}{0.400pt}}
\put(110,794){\makebox(0,0)[r]{ 1}}
\put(130.0,794.0){\rule[-0.200pt]{4.818pt}{0.400pt}}
\put(130,41){\makebox(0,0){-10}}
\put(130.0,82.0){\rule[-0.200pt]{0.400pt}{4.818pt}}
\put(457,41){\makebox(0,0){-5}}
\put(457.0,82.0){\rule[-0.200pt]{0.400pt}{4.818pt}}
\put(785,41){\makebox(0,0){ 0}}
\put(785.0,82.0){\rule[-0.200pt]{0.400pt}{4.818pt}}
\put(1112,41){\makebox(0,0){ 5}}
\put(1112.0,82.0){\rule[-0.200pt]{0.400pt}{4.818pt}}
\put(1439,41){\makebox(0,0){ 10}}
\put(1439.0,82.0){\rule[-0.200pt]{0.400pt}{4.818pt}}
\put(130.0,82.0){\rule[-0.200pt]{0.400pt}{187.179pt}}
\put(130.0,82.0){\rule[-0.200pt]{315.338pt}{0.400pt}}
\put(1279,819){\makebox(0,0)[r]{$\varphi_0$}}
\put(1299.0,819.0){\rule[-0.200pt]{24.090pt}{0.400pt}}
\put(130,147){\usebox{\plotpoint}}
\multiput(659.58,147.00)(0.493,2.677){23}{\rule{0.119pt}{2.192pt}}
\multiput(658.17,147.00)(13.000,63.450){2}{\rule{0.400pt}{1.096pt}}
\multiput(672.58,215.00)(0.493,9.338){23}{\rule{0.119pt}{7.362pt}}
\multiput(671.17,215.00)(13.000,220.721){2}{\rule{0.400pt}{3.681pt}}
\multiput(685.58,451.00)(0.494,9.017){25}{\rule{0.119pt}{7.129pt}}
\multiput(684.17,451.00)(14.000,231.204){2}{\rule{0.400pt}{3.564pt}}
\multiput(699.58,697.00)(0.493,3.827){23}{\rule{0.119pt}{3.085pt}}
\multiput(698.17,697.00)(13.000,90.598){2}{\rule{0.400pt}{1.542pt}}
\put(130.0,147.0){\rule[-0.200pt]{127.436pt}{0.400pt}}
\multiput(857.58,781.20)(0.493,-3.827){23}{\rule{0.119pt}{3.085pt}}
\multiput(856.17,787.60)(13.000,-90.598){2}{\rule{0.400pt}{1.542pt}}
\multiput(870.58,667.41)(0.494,-9.017){25}{\rule{0.119pt}{7.129pt}}
\multiput(869.17,682.20)(14.000,-231.204){2}{\rule{0.400pt}{3.564pt}}
\multiput(884.58,420.44)(0.493,-9.338){23}{\rule{0.119pt}{7.362pt}}
\multiput(883.17,435.72)(13.000,-220.721){2}{\rule{0.400pt}{3.681pt}}
\multiput(897.58,205.90)(0.493,-2.677){23}{\rule{0.119pt}{2.192pt}}
\multiput(896.17,210.45)(13.000,-63.450){2}{\rule{0.400pt}{1.096pt}}
\put(712.0,794.0){\rule[-0.200pt]{34.930pt}{0.400pt}}
\put(910.0,147.0){\rule[-0.200pt]{127.436pt}{0.400pt}}
\put(1279,778){\makebox(0,0)[r]{$\varphi_1$}}
\multiput(1299,778)(20.756,0.000){5}{\usebox{\plotpoint}}
\put(1399,778){\usebox{\plotpoint}}
\put(130,147){\usebox{\plotpoint}}
\put(130.00,147.00){\usebox{\plotpoint}}
\put(150.76,147.00){\usebox{\plotpoint}}
\put(171.51,147.00){\usebox{\plotpoint}}
\put(192.27,147.00){\usebox{\plotpoint}}
\put(213.02,147.00){\usebox{\plotpoint}}
\put(233.78,147.00){\usebox{\plotpoint}}
\put(254.53,147.00){\usebox{\plotpoint}}
\put(275.29,147.00){\usebox{\plotpoint}}
\put(296.04,147.00){\usebox{\plotpoint}}
\put(316.80,147.00){\usebox{\plotpoint}}
\put(337.55,147.00){\usebox{\plotpoint}}
\put(358.31,147.00){\usebox{\plotpoint}}
\put(379.07,147.00){\usebox{\plotpoint}}
\put(399.82,147.00){\usebox{\plotpoint}}
\put(420.58,147.00){\usebox{\plotpoint}}
\put(441.33,147.00){\usebox{\plotpoint}}
\put(462.09,147.00){\usebox{\plotpoint}}
\put(482.84,147.00){\usebox{\plotpoint}}
\put(503.60,147.00){\usebox{\plotpoint}}
\put(524.35,147.00){\usebox{\plotpoint}}
\multiput(540,148)(8.490,18.940){2}{\usebox{\plotpoint}}
\multiput(553,177)(3.033,20.533){4}{\usebox{\plotpoint}}
\multiput(566,265)(2.386,20.618){6}{\usebox{\plotpoint}}
\multiput(580,386)(2.065,20.652){6}{\usebox{\plotpoint}}
\multiput(593,516)(2.130,20.646){6}{\usebox{\plotpoint}}
\multiput(606,642)(2.650,20.586){5}{\usebox{\plotpoint}}
\multiput(619,743)(5.533,20.004){3}{\usebox{\plotpoint}}
\put(650.23,794.00){\usebox{\plotpoint}}
\multiput(659,794)(3.897,-20.386){3}{\usebox{\plotpoint}}
\multiput(672,726)(1.142,-20.724){12}{\usebox{\plotpoint}}
\multiput(685,490)(1.179,-20.722){12}{\usebox{\plotpoint}}
\multiput(699,244)(2.757,-20.572){4}{\usebox{\plotpoint}}
\put(717.56,147.00){\usebox{\plotpoint}}
\put(738.32,147.00){\usebox{\plotpoint}}
\put(759.07,147.00){\usebox{\plotpoint}}
\put(779.83,147.00){\usebox{\plotpoint}}
\put(800.58,147.00){\usebox{\plotpoint}}
\put(821.34,147.00){\usebox{\plotpoint}}
\put(842.10,147.00){\usebox{\plotpoint}}
\multiput(857,147)(2.757,20.572){5}{\usebox{\plotpoint}}
\multiput(870,244)(1.179,20.722){12}{\usebox{\plotpoint}}
\multiput(884,490)(1.142,20.724){11}{\usebox{\plotpoint}}
\multiput(897,726)(3.897,20.386){4}{\usebox{\plotpoint}}
\put(929.90,792.03){\usebox{\plotpoint}}
\multiput(937,790)(5.533,-20.004){2}{\usebox{\plotpoint}}
\multiput(950,743)(2.650,-20.586){5}{\usebox{\plotpoint}}
\multiput(963,642)(2.130,-20.646){6}{\usebox{\plotpoint}}
\multiput(976,516)(2.065,-20.652){7}{\usebox{\plotpoint}}
\multiput(989,386)(2.386,-20.618){5}{\usebox{\plotpoint}}
\multiput(1003,265)(3.033,-20.533){5}{\usebox{\plotpoint}}
\put(1023.10,161.15){\usebox{\plotpoint}}
\put(1035.32,147.51){\usebox{\plotpoint}}
\put(1056.06,147.00){\usebox{\plotpoint}}
\put(1076.81,147.00){\usebox{\plotpoint}}
\put(1097.57,147.00){\usebox{\plotpoint}}
\put(1118.32,147.00){\usebox{\plotpoint}}
\put(1139.08,147.00){\usebox{\plotpoint}}
\put(1159.83,147.00){\usebox{\plotpoint}}
\put(1180.59,147.00){\usebox{\plotpoint}}
\put(1201.35,147.00){\usebox{\plotpoint}}
\put(1222.10,147.00){\usebox{\plotpoint}}
\put(1242.86,147.00){\usebox{\plotpoint}}
\put(1263.61,147.00){\usebox{\plotpoint}}
\put(1284.37,147.00){\usebox{\plotpoint}}
\put(1305.12,147.00){\usebox{\plotpoint}}
\put(1325.88,147.00){\usebox{\plotpoint}}
\put(1346.63,147.00){\usebox{\plotpoint}}
\put(1367.39,147.00){\usebox{\plotpoint}}
\put(1388.14,147.00){\usebox{\plotpoint}}
\put(1408.90,147.00){\usebox{\plotpoint}}
\put(1429.66,147.00){\usebox{\plotpoint}}
\put(1439,147){\usebox{\plotpoint}}
\sbox{\plotpoint}{\rule[-0.400pt]{0.800pt}{0.800pt}}\sbox{\plotpoint}{\rule[-0.200pt]{0.400pt}{0.400pt}}\put(1279,737){\makebox(0,0)[r]{$\varphi_2$}}
\sbox{\plotpoint}{\rule[-0.400pt]{0.800pt}{0.800pt}}\put(1299.0,737.0){\rule[-0.400pt]{24.090pt}{0.800pt}}
\put(130,147){\usebox{\plotpoint}}
\put(289,146.84){\rule{3.132pt}{0.800pt}}
\multiput(289.00,145.34)(6.500,3.000){2}{\rule{1.566pt}{0.800pt}}
\multiput(303.41,150.00)(0.509,0.574){19}{\rule{0.123pt}{1.123pt}}
\multiput(300.34,150.00)(13.000,12.669){2}{\rule{0.800pt}{0.562pt}}
\multiput(316.41,165.00)(0.509,1.153){19}{\rule{0.123pt}{1.985pt}}
\multiput(313.34,165.00)(13.000,24.881){2}{\rule{0.800pt}{0.992pt}}
\multiput(329.41,194.00)(0.509,1.675){21}{\rule{0.123pt}{2.771pt}}
\multiput(326.34,194.00)(14.000,39.248){2}{\rule{0.800pt}{1.386pt}}
\multiput(343.41,239.00)(0.509,2.188){19}{\rule{0.123pt}{3.523pt}}
\multiput(340.34,239.00)(13.000,46.688){2}{\rule{0.800pt}{1.762pt}}
\multiput(356.41,293.00)(0.509,2.477){19}{\rule{0.123pt}{3.954pt}}
\multiput(353.34,293.00)(13.000,52.794){2}{\rule{0.800pt}{1.977pt}}
\multiput(369.41,354.00)(0.509,2.601){19}{\rule{0.123pt}{4.138pt}}
\multiput(366.34,354.00)(13.000,55.410){2}{\rule{0.800pt}{2.069pt}}
\multiput(382.41,418.00)(0.509,2.684){19}{\rule{0.123pt}{4.262pt}}
\multiput(379.34,418.00)(13.000,57.155){2}{\rule{0.800pt}{2.131pt}}
\multiput(395.41,484.00)(0.509,2.438){21}{\rule{0.123pt}{3.914pt}}
\multiput(392.34,484.00)(14.000,56.876){2}{\rule{0.800pt}{1.957pt}}
\multiput(409.41,549.00)(0.509,2.560){19}{\rule{0.123pt}{4.077pt}}
\multiput(406.34,549.00)(13.000,54.538){2}{\rule{0.800pt}{2.038pt}}
\multiput(422.41,612.00)(0.509,2.394){19}{\rule{0.123pt}{3.831pt}}
\multiput(419.34,612.00)(13.000,51.049){2}{\rule{0.800pt}{1.915pt}}
\multiput(435.41,671.00)(0.509,2.063){19}{\rule{0.123pt}{3.338pt}}
\multiput(432.34,671.00)(13.000,44.071){2}{\rule{0.800pt}{1.669pt}}
\multiput(448.41,722.00)(0.509,1.446){21}{\rule{0.123pt}{2.429pt}}
\multiput(445.34,722.00)(14.000,33.959){2}{\rule{0.800pt}{1.214pt}}
\multiput(462.41,761.00)(0.509,0.905){19}{\rule{0.123pt}{1.615pt}}
\multiput(459.34,761.00)(13.000,19.647){2}{\rule{0.800pt}{0.808pt}}
\multiput(474.00,785.40)(0.737,0.516){11}{\rule{1.356pt}{0.124pt}}
\multiput(474.00,782.34)(10.186,9.000){2}{\rule{0.678pt}{0.800pt}}
\put(487,791.84){\rule{3.132pt}{0.800pt}}
\multiput(487.00,791.34)(6.500,1.000){2}{\rule{1.566pt}{0.800pt}}
\put(130.0,147.0){\rule[-0.400pt]{38.303pt}{0.800pt}}
\put(527,791.84){\rule{3.132pt}{0.800pt}}
\multiput(527.00,792.34)(6.500,-1.000){2}{\rule{1.566pt}{0.800pt}}
\multiput(541.41,784.76)(0.509,-1.153){19}{\rule{0.123pt}{1.985pt}}
\multiput(538.34,788.88)(13.000,-24.881){2}{\rule{0.800pt}{0.992pt}}
\multiput(554.41,740.69)(0.509,-3.594){19}{\rule{0.123pt}{5.615pt}}
\multiput(551.34,752.34)(13.000,-76.345){2}{\rule{0.800pt}{2.808pt}}
\multiput(567.41,646.47)(0.509,-4.573){21}{\rule{0.123pt}{7.114pt}}
\multiput(564.34,661.23)(14.000,-106.234){2}{\rule{0.800pt}{3.557pt}}
\multiput(581.41,520.96)(0.509,-5.331){19}{\rule{0.123pt}{8.200pt}}
\multiput(578.34,537.98)(13.000,-112.980){2}{\rule{0.800pt}{4.100pt}}
\multiput(594.41,391.98)(0.509,-5.166){19}{\rule{0.123pt}{7.954pt}}
\multiput(591.34,408.49)(13.000,-109.491){2}{\rule{0.800pt}{3.977pt}}
\multiput(607.41,272.37)(0.509,-4.132){19}{\rule{0.123pt}{6.415pt}}
\multiput(604.34,285.68)(13.000,-87.685){2}{\rule{0.800pt}{3.208pt}}
\multiput(620.41,185.16)(0.509,-1.898){19}{\rule{0.123pt}{3.092pt}}
\multiput(617.34,191.58)(13.000,-40.582){2}{\rule{0.800pt}{1.546pt}}
\put(632,147.34){\rule{3.000pt}{0.800pt}}
\multiput(632.00,149.34)(7.773,-4.000){2}{\rule{1.500pt}{0.800pt}}
\put(500.0,794.0){\rule[-0.400pt]{6.504pt}{0.800pt}}
\put(923,147.34){\rule{3.000pt}{0.800pt}}
\multiput(923.00,145.34)(7.773,4.000){2}{\rule{1.500pt}{0.800pt}}
\multiput(938.41,151.00)(0.509,1.898){19}{\rule{0.123pt}{3.092pt}}
\multiput(935.34,151.00)(13.000,40.582){2}{\rule{0.800pt}{1.546pt}}
\multiput(951.41,198.00)(0.509,4.132){19}{\rule{0.123pt}{6.415pt}}
\multiput(948.34,198.00)(13.000,87.685){2}{\rule{0.800pt}{3.208pt}}
\multiput(964.41,299.00)(0.509,5.166){19}{\rule{0.123pt}{7.954pt}}
\multiput(961.34,299.00)(13.000,109.491){2}{\rule{0.800pt}{3.977pt}}
\multiput(977.41,425.00)(0.509,5.331){19}{\rule{0.123pt}{8.200pt}}
\multiput(974.34,425.00)(13.000,112.980){2}{\rule{0.800pt}{4.100pt}}
\multiput(990.41,555.00)(0.509,4.573){21}{\rule{0.123pt}{7.114pt}}
\multiput(987.34,555.00)(14.000,106.234){2}{\rule{0.800pt}{3.557pt}}
\multiput(1004.41,676.00)(0.509,3.594){19}{\rule{0.123pt}{5.615pt}}
\multiput(1001.34,676.00)(13.000,76.345){2}{\rule{0.800pt}{2.808pt}}
\multiput(1017.41,764.00)(0.509,1.153){19}{\rule{0.123pt}{1.985pt}}
\multiput(1014.34,764.00)(13.000,24.881){2}{\rule{0.800pt}{0.992pt}}
\put(1029,791.84){\rule{3.132pt}{0.800pt}}
\multiput(1029.00,791.34)(6.500,1.000){2}{\rule{1.566pt}{0.800pt}}
\put(646.0,147.0){\rule[-0.400pt]{66.729pt}{0.800pt}}
\put(1069,791.84){\rule{3.132pt}{0.800pt}}
\multiput(1069.00,792.34)(6.500,-1.000){2}{\rule{1.566pt}{0.800pt}}
\multiput(1082.00,791.08)(0.737,-0.516){11}{\rule{1.356pt}{0.124pt}}
\multiput(1082.00,791.34)(10.186,-9.000){2}{\rule{0.678pt}{0.800pt}}
\multiput(1096.41,777.29)(0.509,-0.905){19}{\rule{0.123pt}{1.615pt}}
\multiput(1093.34,780.65)(13.000,-19.647){2}{\rule{0.800pt}{0.808pt}}
\multiput(1109.41,750.92)(0.509,-1.446){21}{\rule{0.123pt}{2.429pt}}
\multiput(1106.34,755.96)(14.000,-33.959){2}{\rule{0.800pt}{1.214pt}}
\multiput(1123.41,708.14)(0.509,-2.063){19}{\rule{0.123pt}{3.338pt}}
\multiput(1120.34,715.07)(13.000,-44.071){2}{\rule{0.800pt}{1.669pt}}
\multiput(1136.41,655.10)(0.509,-2.394){19}{\rule{0.123pt}{3.831pt}}
\multiput(1133.34,663.05)(13.000,-51.049){2}{\rule{0.800pt}{1.915pt}}
\multiput(1149.41,595.08)(0.509,-2.560){19}{\rule{0.123pt}{4.077pt}}
\multiput(1146.34,603.54)(13.000,-54.538){2}{\rule{0.800pt}{2.038pt}}
\multiput(1162.41,532.75)(0.509,-2.438){21}{\rule{0.123pt}{3.914pt}}
\multiput(1159.34,540.88)(14.000,-56.876){2}{\rule{0.800pt}{1.957pt}}
\multiput(1176.41,466.31)(0.509,-2.684){19}{\rule{0.123pt}{4.262pt}}
\multiput(1173.34,475.15)(13.000,-57.155){2}{\rule{0.800pt}{2.131pt}}
\multiput(1189.41,400.82)(0.509,-2.601){19}{\rule{0.123pt}{4.138pt}}
\multiput(1186.34,409.41)(13.000,-55.410){2}{\rule{0.800pt}{2.069pt}}
\multiput(1202.41,337.59)(0.509,-2.477){19}{\rule{0.123pt}{3.954pt}}
\multiput(1199.34,345.79)(13.000,-52.794){2}{\rule{0.800pt}{1.977pt}}
\multiput(1215.41,278.38)(0.509,-2.188){19}{\rule{0.123pt}{3.523pt}}
\multiput(1212.34,285.69)(13.000,-46.688){2}{\rule{0.800pt}{1.762pt}}
\multiput(1228.41,227.50)(0.509,-1.675){21}{\rule{0.123pt}{2.771pt}}
\multiput(1225.34,233.25)(14.000,-39.248){2}{\rule{0.800pt}{1.386pt}}
\multiput(1242.41,185.76)(0.509,-1.153){19}{\rule{0.123pt}{1.985pt}}
\multiput(1239.34,189.88)(13.000,-24.881){2}{\rule{0.800pt}{0.992pt}}
\multiput(1255.41,160.34)(0.509,-0.574){19}{\rule{0.123pt}{1.123pt}}
\multiput(1252.34,162.67)(13.000,-12.669){2}{\rule{0.800pt}{0.562pt}}
\put(1267,146.84){\rule{3.132pt}{0.800pt}}
\multiput(1267.00,148.34)(6.500,-3.000){2}{\rule{1.566pt}{0.800pt}}
\put(1042.0,794.0){\rule[-0.400pt]{6.504pt}{0.800pt}}
\put(1280.0,147.0){\rule[-0.400pt]{38.303pt}{0.800pt}}
\sbox{\plotpoint}{\rule[-0.500pt]{1.000pt}{1.000pt}}\sbox{\plotpoint}{\rule[-0.200pt]{0.400pt}{0.400pt}}\put(1279,696){\makebox(0,0)[r]{$\varphi_3$}}
\sbox{\plotpoint}{\rule[-0.500pt]{1.000pt}{1.000pt}}\multiput(1299,696)(20.756,0.000){5}{\usebox{\plotpoint}}
\put(1399,696){\usebox{\plotpoint}}
\put(130,752){\usebox{\plotpoint}}
\put(130.00,752.00){\usebox{\plotpoint}}
\put(143.10,768.09){\usebox{\plotpoint}}
\put(158.87,781.43){\usebox{\plotpoint}}
\put(177.75,789.98){\usebox{\plotpoint}}
\put(197.98,794.00){\usebox{\plotpoint}}
\put(218.74,794.00){\usebox{\plotpoint}}
\put(239.49,794.00){\usebox{\plotpoint}}
\put(260.25,794.00){\usebox{\plotpoint}}
\put(281.00,794.00){\usebox{\plotpoint}}
\put(301.43,791.13){\usebox{\plotpoint}}
\multiput(315,776)(8.490,-18.940){2}{\usebox{\plotpoint}}
\multiput(328,747)(6.166,-19.819){2}{\usebox{\plotpoint}}
\multiput(342,702)(4.858,-20.179){3}{\usebox{\plotpoint}}
\multiput(355,648)(4.326,-20.300){3}{\usebox{\plotpoint}}
\multiput(368,587)(4.132,-20.340){3}{\usebox{\plotpoint}}
\multiput(381,523)(4.011,-20.364){3}{\usebox{\plotpoint}}
\multiput(394,457)(4.370,-20.290){4}{\usebox{\plotpoint}}
\multiput(408,392)(4.195,-20.327){3}{\usebox{\plotpoint}}
\multiput(421,329)(4.466,-20.269){3}{\usebox{\plotpoint}}
\multiput(434,270)(5.127,-20.112){2}{\usebox{\plotpoint}}
\multiput(447,219)(7.013,-19.535){2}{\usebox{\plotpoint}}
\put(465.08,172.78){\usebox{\plotpoint}}
\put(476.16,155.50){\usebox{\plotpoint}}
\put(494.55,147.42){\usebox{\plotpoint}}
\put(515.29,147.00){\usebox{\plotpoint}}
\put(536.04,147.00){\usebox{\plotpoint}}
\put(556.80,147.00){\usebox{\plotpoint}}
\put(577.56,147.00){\usebox{\plotpoint}}
\put(598.31,147.00){\usebox{\plotpoint}}
\put(619.07,147.00){\usebox{\plotpoint}}
\put(639.82,147.00){\usebox{\plotpoint}}
\put(660.58,147.00){\usebox{\plotpoint}}
\put(681.33,147.00){\usebox{\plotpoint}}
\put(702.09,147.00){\usebox{\plotpoint}}
\put(722.84,147.00){\usebox{\plotpoint}}
\put(743.60,147.00){\usebox{\plotpoint}}
\put(764.35,147.00){\usebox{\plotpoint}}
\put(785.11,147.00){\usebox{\plotpoint}}
\put(805.87,147.00){\usebox{\plotpoint}}
\put(826.62,147.00){\usebox{\plotpoint}}
\put(847.38,147.00){\usebox{\plotpoint}}
\put(868.13,147.00){\usebox{\plotpoint}}
\put(888.89,147.00){\usebox{\plotpoint}}
\put(909.64,147.00){\usebox{\plotpoint}}
\put(930.40,147.00){\usebox{\plotpoint}}
\put(951.15,147.00){\usebox{\plotpoint}}
\put(971.91,147.00){\usebox{\plotpoint}}
\put(992.67,147.00){\usebox{\plotpoint}}
\put(1013.42,147.00){\usebox{\plotpoint}}
\put(1034.18,147.00){\usebox{\plotpoint}}
\put(1054.93,147.00){\usebox{\plotpoint}}
\put(1075.67,147.51){\usebox{\plotpoint}}
\put(1093.84,156.20){\usebox{\plotpoint}}
\put(1104.52,173.84){\usebox{\plotpoint}}
\multiput(1108,180)(7.013,19.535){2}{\usebox{\plotpoint}}
\multiput(1122,219)(5.127,20.112){2}{\usebox{\plotpoint}}
\multiput(1135,270)(4.466,20.269){3}{\usebox{\plotpoint}}
\multiput(1148,329)(4.195,20.327){3}{\usebox{\plotpoint}}
\multiput(1161,392)(4.370,20.290){4}{\usebox{\plotpoint}}
\multiput(1175,457)(4.011,20.364){3}{\usebox{\plotpoint}}
\multiput(1188,523)(4.132,20.340){3}{\usebox{\plotpoint}}
\multiput(1201,587)(4.326,20.300){3}{\usebox{\plotpoint}}
\multiput(1214,648)(4.858,20.179){3}{\usebox{\plotpoint}}
\multiput(1227,702)(6.166,19.819){2}{\usebox{\plotpoint}}
\put(1245.88,757.88){\usebox{\plotpoint}}
\put(1254.59,776.68){\usebox{\plotpoint}}
\put(1268.76,791.41){\usebox{\plotpoint}}
\put(1289.22,794.00){\usebox{\plotpoint}}
\put(1309.97,794.00){\usebox{\plotpoint}}
\put(1330.73,794.00){\usebox{\plotpoint}}
\put(1351.48,794.00){\usebox{\plotpoint}}
\put(1372.24,794.00){\usebox{\plotpoint}}
\put(1392.39,789.54){\usebox{\plotpoint}}
\put(1411.23,780.89){\usebox{\plotpoint}}
\put(1426.68,767.16){\usebox{\plotpoint}}
\put(1439,752){\usebox{\plotpoint}}
\sbox{\plotpoint}{\rule[-0.600pt]{1.200pt}{1.200pt}}\sbox{\plotpoint}{\rule[-0.200pt]{0.400pt}{0.400pt}}\put(1279,655){\makebox(0,0)[r]{$\varphi_4$}}
\sbox{\plotpoint}{\rule[-0.600pt]{1.200pt}{1.200pt}}\put(1299.0,655.0){\rule[-0.600pt]{24.090pt}{1.200pt}}
\put(130,189){\usebox{\plotpoint}}
\multiput(132.24,181.62)(0.501,-0.575){16}{\rule{0.121pt}{1.777pt}}
\multiput(127.51,185.31)(13.000,-12.312){2}{\rule{1.200pt}{0.888pt}}
\multiput(143.00,170.26)(0.489,-0.501){14}{\rule{1.600pt}{0.121pt}}
\multiput(143.00,170.51)(9.679,-12.000){2}{\rule{0.800pt}{1.200pt}}
\multiput(156.00,158.26)(0.931,-0.505){4}{\rule{2.700pt}{0.122pt}}
\multiput(156.00,158.51)(8.396,-7.000){2}{\rule{1.350pt}{1.200pt}}
\put(170,149.01){\rule{3.132pt}{1.200pt}}
\multiput(170.00,151.51)(6.500,-5.000){2}{\rule{1.566pt}{1.200pt}}
\put(183,145.51){\rule{3.132pt}{1.200pt}}
\multiput(183.00,146.51)(6.500,-2.000){2}{\rule{1.566pt}{1.200pt}}
\put(1373,145.51){\rule{3.132pt}{1.200pt}}
\multiput(1373.00,144.51)(6.500,2.000){2}{\rule{1.566pt}{1.200pt}}
\put(1386,149.01){\rule{3.132pt}{1.200pt}}
\multiput(1386.00,146.51)(6.500,5.000){2}{\rule{1.566pt}{1.200pt}}
\multiput(1399.00,156.24)(0.931,0.505){4}{\rule{2.700pt}{0.122pt}}
\multiput(1399.00,151.51)(8.396,7.000){2}{\rule{1.350pt}{1.200pt}}
\multiput(1413.00,163.24)(0.489,0.501){14}{\rule{1.600pt}{0.121pt}}
\multiput(1413.00,158.51)(9.679,12.000){2}{\rule{0.800pt}{1.200pt}}
\multiput(1428.24,173.00)(0.501,0.575){16}{\rule{0.121pt}{1.777pt}}
\multiput(1423.51,173.00)(13.000,12.312){2}{\rule{1.200pt}{0.888pt}}
\put(196.0,147.0){\rule[-0.600pt]{283.539pt}{1.200pt}}
\sbox{\plotpoint}{\rule[-0.200pt]{0.400pt}{0.400pt}}\put(130.0,82.0){\rule[-0.200pt]{0.400pt}{187.179pt}}
\put(130.0,82.0){\rule[-0.200pt]{315.338pt}{0.400pt}}
\end{picture}
 
\end{figure}
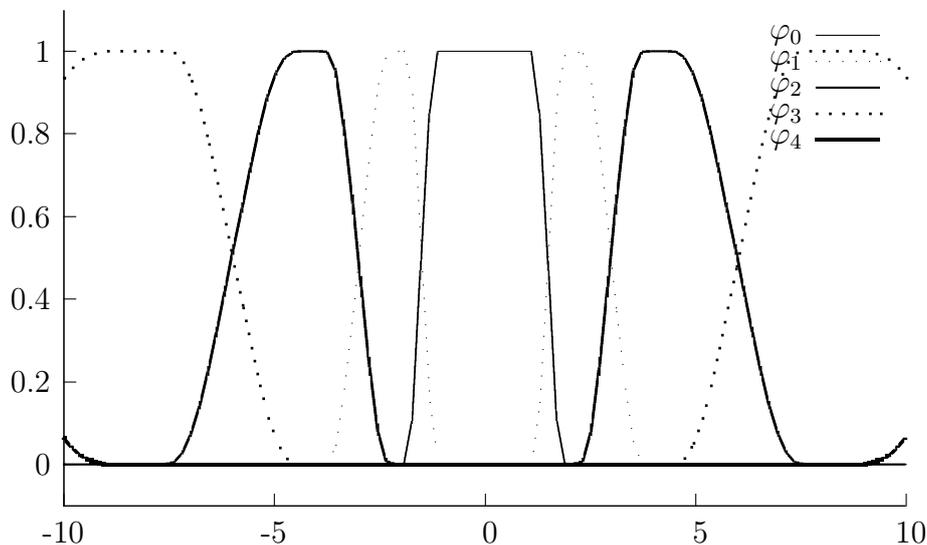

\begin{example} Wir wollen in diesem Beispiel eine Littlewood-Paley Partiton konstruieren.
Definiere analog zu \cite[Lemma 1.10]{Warner} die Funktion
$$f : \R \rightarrow \R, 
f(t) := 
\begin{cases}
e^{-\frac{1}{t}} & \textrm{ für } t > 0 ,\\
0 & \textrm{ für } t \le 0 .
\end{cases}
$$ 
$f$ ist glatt, da nach Formel $\ref{faaDiBruno}$
$$ \frac{d^k}{dt^k} e^{-\frac{1}{t}} = 
\sum_{(m_1,\ldots,m_k)\in M} \frac{k!}{\prod_{i=1}^k m_i! (i!)^{m_i}} e^{-\frac{1}{t}} 
\prod_{j=1}^k \left( - \frac{d^j}{dt^j} \frac{1}{t} \right)^{m_j}
,$$
wobei $M := \menge{ (m_1,\ldots,m_k) \in \N^k_0 : \sum_{i=1}^k i m_i = k }$.
Wir folgern $ - \frac{d^j}{dt^j} \frac{1}{t} = \left( \frac{-1}{t} \right)^{j+1}$.
Da 
$$\lim_{t\rightarrow 0} e^{-\frac{1}{t}} t^{(j+1)N} = 0 \textrm{ für alle } N \in \N ,$$
ist auch
$\lim_{t\rightarrow 0} \frac{d^j}{dt^j} e^{-\frac{1}{t}} = 0$ für alle $j \in \N$, und damit
ist $f$ glatt.
Wir definieren weiterhin 
$$ g : \R \rightarrow \R, g(t) := \frac{f(t)}{f(t)+f(1-t)} = 
\begin{cases}
0 & t \le 0, \\
\frac{e^{-\frac{1}{t}}}{e^{-\frac{1}{t}}+e^{-\frac{1}{1-t}}} & 0 < t < 1 , \\
1 & t \ge 1 .
\end{cases}
$$
Offensichtlich ist $g(t)$ glatt.
Definiere schließlich 
$$h(t) := g(2+t) g(2-t) 
\begin{cases}
 = 0 & t \ge 2 , \\
 = 1 & -1 \le t \le 1 , \\
 = 0 & t \le -2 , \\
 \in (0,1) & \textrm{sonst.}
\end{cases}
$$
$h(\abs{t})$ ist wiederum wegen der Achsensymmetrie $h(t) = h(-t)$ glatt.
Setze nun $\varphi_0(\xi) := h(\abs{\xi})$ und erhalte mit der Rekursionsformel (\ref{equ:littlewoodRekursion}) eine Littlewood-Paley-Partition.
\end{example}

\begin{konv}
Falls nicht anders erwähnt, wollen wir mit $\lilwood{j}$ stets das $j$.te Glied einer gewählten Littlewood-Paley-Partition, und mit $D_j$ den Dyadischen Ring bezeichnen.
\end{konv}

\begin{definition}
Jedes $\lilwood{j}$ definiert wegen $\lilwood{j} \in C^\infty_0(\R^n)$ einen Operator 
\begin{align*}
\lilwood{j}(D) :
& L^2(\R^n) \rightarrow L^2(\R^n) , \\
& \Schwarzi \rightarrow \Schwarzi ,
\end{align*}
mit
$$ \lilwood{j}(D_x) f(x) := \int_{\R^n} e^{ix\cdot\xi} \lilwood{j}(\xi) \hat{f}(\xi) \dslash\xi = (\check{\lilwood{j}} * f)(x) .$$
\end{definition}

\begin{cor}\label{cor:introLilwoodOP}
Zu $\xi \in \R^n$ gibt es ein $j \in \N_0$, sodass $\xi \in \supp \lilwood{j}$ und 
$\lilwood{j-1}(\xi) + \lilwood{j}(\xi) + \lilwood{j+1}(\xi) = 1$ mit $\lilwood{-1} = 0$.
Insbesondere gilt für $f \in L^2(\R^n)$ die Identität
$$ \left( \lilwood{j-1}(D_x) + \lilwood{j}(D_x) + \lilwood{j+1}(D_x) \right) \lilwood{j}(D_x) f(x) = \lilwood{j}(D_x) f(x) .$$
\begin{proof}
Klar, da für jedes $j \in \N$ per Definition \ref{def:littlewoodpaley} gilt, dass 
$\supp \lilwood{j} \cap \bigcup_{k \in \N} \supp \lilwood{k} 
= \supp \lilwood{j-1} \cup \supp \lilwood{j} \cup \supp \lilwood{j+1}$.
Desweiteren ist also
\begin{align*}
\lilwood{j}(D_x) f(x)
&= \int_{\R^n} e^{i x\cdot\xi}\lilwood{j}(\xi) \hat{f}(\xi) \dslash\xi \\
&= \int_{\R^n} e^{i x\cdot\xi} \left( \lilwood{j-1}(\xi) + \lilwood{j}(\xi) + \lilwood{j+1}(\xi) \right) \lilwood{j}(\xi) \hat{f}(\xi) \dslash\xi \\
&= \left( \lilwood{j-1}(D_x) + \lilwood{j}(D_x) + \lilwood{j+1}(D_x) \right) \lilwood{j}(D_x) f(x)
.\end{align*}
\end{proof}
\end{cor}

\begin{lemma}
Die Littlewood-Paley Partition ist eine Partition der Eins mit
\begin{equation}\label{lilwood:abl}
\sum^\infty_{j=0} \partial^\alpha_\xi \lilwood{j}(\xi) = 0\; \textrm{ für jeden Multiindex } \alpha \neq 0 
.\end{equation}
\begin{proof}
Für jedes $\xi \in \R^n$ setze $N := \max \menge{ N \in \N_0 : \abs{\xi} \le 2^N}$. Dann ist
$$ \sum_{j=0}^\infty \lilwood{j}(\xi) = \sum_{j=0}^N \lilwood{j}(\xi) = \lilwood{0}{2^{-N} \xi} = 1 .$$ 
Folgerung \ref{cor:introLilwoodOP} liefert ein $j \in \N$, sodass
$$ \partial^\alpha_\xi  \sum^\infty_{j=0} \lilwood{j}(\xi) = \partial^\alpha_\xi \left( \lilwood{j-1}(\xi) + \lilwood{j}(\xi) + \lilwood{j+1}(\xi) \right) = 0 .$$
\end{proof}
\end{lemma}

\begin{lemma}\label{lemma:introPartialLilwoodAbsch}
Für alle $\alpha \in \N^n_0$ gibt es ein $c_\alpha > 0$, sodass für jedes $j \in \N_0$ und alle $\xi \in \R^n$
\begin{equation}
\abs{ \partial^\alpha_\xi \lilwood{j}(\xi) } \le c_\alpha \min \left( 2^{-j \abs{\alpha}}, \piket{\xi}{-\abs{\alpha}} \right)
.\end{equation}
\begin{proof}
Wir können $\lilwood{j}$ in Abhängigkeit von $\lilwood{0}$ ausdrücken und erhalten
\begin{align*}
\partial_\xi^\alpha \lilwood{j}(\xi) &= \partial_\xi^\alpha \lilwood{0}(2^{-j}\xi) - \partial_\xi^\alpha \lilwood{0}(2^{-(j-1)} \xi) \\
&= \lilwood{0}^{(\alpha)}(2^{-j}\xi) 2^{-j\abs{\alpha}} - \lilwood{0}^{(\alpha)}(2^{-(j-1)}\xi) 2^{-(j-1)\abs{\alpha}} \\
&= 2^{-j\abs{\alpha}} \left( \lilwood{0}^{(\alpha)}(2^{-j}\xi) - 2^{\abs{\alpha}} \lilwood{0}^{(\alpha)}(2^{-(j-1)} \xi) \right) 
.\end{align*}
Also ist
$
\abs{ \partial_\xi^\alpha \lilwood{j}(\xi) } \le
2^{-j\abs{\alpha}} \left( \abs{\lilwood{0}^{(\alpha)}(2^{-j}\xi)} + 2^{\abs{\alpha}} \abs{\lilwood{0}^{(\alpha)}(2^{-(j-1)} \xi)} \right)
$
Falls $\xi \notin D_j$ ist wegen $\supp \lilwood{0}^{(\alpha)} \subset \supp \lilwood{0} \subset B_2(0)$ bereits $\partial_\xi^\alpha \lilwood{j}(\xi) = 0$.
Für den Fall, dass $\xi \in D_j$ ist, ergibt sich
$$
\abs{ \partial_\xi^\alpha \lilwood{j}(\xi) } \le
\abs{ \sup_{\xi \in D_j} \lilwood{0}^{(\alpha)}(\xi) } (1 + 2^{\abs{\alpha}}) 2^{-j\abs{\alpha}} \le
c_\alpha 2^{-j\abs{\alpha}}
,$$
da $D_j \subset\subset \R^n$ und $\lilwood{0}^{(\alpha)} \in C^\infty_0(\R^n)$.

Insbesondere gilt auf dem Dyadischen Ring $D_j = \menge{ \xi \in \R^n : 2^{j-1} \le \abs{\xi} \le 2^{j+1} }$, dass
$$2^{-j\abs{\alpha}} = \left( 2^{j-1} \right)^{-\abs{\alpha}} 2^{-\abs{\alpha}} \le c'_\alpha \piket{\xi}{-\abs{\alpha}} .$$
Insgesamt ist also
$$\abs{ \partial_\xi^\alpha \lilwood{j}(\xi) } \le C_\alpha \min \left\{ 2^{-j\abs{\alpha}}, \piket{\xi}{-\abs{\alpha}} \right\} .$$
\end{proof}
\end{lemma}

\begin{lemma}\label{lemma:lilwoodSchwarzKonv}
Für $f \in \Schwarzi$ ist $\sum^N_{j=0} \lilwood{j}(D_x)f \longrightarrow f$ für $N \rightarrow \infty$ in $\Schwarzi$ konvergent.
\begin{proof}
Mit $ \lilwood{j}(D_x) f(x) = \FourierB{\xi}{x}{ \lilwood{j}(\xi) \hat{f}(x) }$ und der in Folgerung \ref{cor:SchwartzEquivNorm} definierten Halbnorm ist
$$
\SchwarzNorm{ \FourierB{\xi}{\hold}{\sum_{j=0}^N \lilwood{j}(\xi)\hat{f}(\xi) - \hat{f}(\xi) } }{k} \le
c \SchwarzzNorm{ \sum_{j=0}^N \lilwood{j}\hat{f} - \hat{f} }{k+3n+1}
. $$ 
Für $\alpha \in \N^n_0, l \in \N_0$ haben wir
\begin{align*}
&\norm{x \mapsto \piket{x}{l} D_x^\alpha \left( \sum_{j=0}^N \lilwood{j} \hat{f} - \hat{f} \right) }{L^2(\R^n)} \\
&= \norm{x \mapsto \piket{x}{l} \left( \sum_{j=0}^N \lilwood{j}(x)D_x^\alpha \hat{f}(x) - D_x^\alpha \hat{f}(x) + \sum_{\mu \le \alpha, \mu \ne 0} {\alpha \choose \mu} \sum_{j=0}^{N} D_x^\mu \lilwood{j}(x) D_x^{\alpha - \mu} \hat{f}(x) \right) }{L^2(\R^n)} \\
&\le \norm{ \piket{x}{l} \left( \sum_{j=0}^N \lilwood{j}(x) D_x^\alpha \hat{f}(x) - D_x^\alpha \hat{f}(x) \right)}{L^2(\R^n)} \\
&\quad + \norm{x \mapsto \piket{x}{l} \sum_{\mu \le \alpha, \mu \ne 0} {\alpha \choose \mu} \sum_{j=0}^N D_x^\mu \lilwood{j}(x) D^{\alpha - \mu} \hat{f}(x) }{L^2(\R^n)}
.\end{align*}
Wir folgern mit dem Satz von Lebesgue die Konvergenz von
$$ \norm{x \mapsto \piket{x}{l} \left( \sum_{j=0}^N \lilwood{j}(x) D_x^\alpha \hat{f}(x) - D_x^\alpha \hat{f}(x) \right)}{L^2(\R^n)} \longrightarrow 0 \textrm{ für } N \rightarrow \infty $$ 
mit Majorante $ 2 \piket{x}{l} D^\alpha_x \hat{f}(x) $
und von
$$ \norm{x \mapsto \piket{x}{l} \sum_{j=0}^N D_x^\mu \lilwood{j}(x) D_x^{\alpha - \mu} \hat{f}(x) }{L^2(\R^n)} \longrightarrow 0 \textrm{ für } N \rightarrow \infty ,$$
da nach Lemma \ref{lemma:introPartialLilwoodAbsch} 
$D_x^{\alpha-\mu_1} \left[ f(x) x^\alpha \right] \sum_{j=0}^\infty 2^{-j\abs{\mu}}$
eine Majorante ist.
Schließlich bilden wir das Supremum über alle $\alpha,l$ mit $\abs{\alpha} + l \le k+3n+1$ und erhalten
$$\SchwarzzNorm{ \sum_{j=0}^N \lilwood{j}\hat{f} - \hat{f} }{k+3n+1} \rightarrow 0 \textrm{ für } N \rightarrow \infty .$$
\end{proof}
\end{lemma}

\begin{lemma}
Sei $f : \R^n \rightarrow \C, m \in \R$. Dann gilt folgende Äquivalenz:
\begin{equation}
\exists C_m > 0 : \abs{f(\xi)} \le C_m \piket{\xi}{m} \; \forall \xi \in \R^n 
\Leftrightarrow
\exists \tilde{C}_m > 0 : \sup_{\xi \in \R^n} \abs{\lilwood{j}(\xi) f(\xi) } \le \tilde{C}_m 2^{jm} \; \forall j \in \N_0 
.\end{equation}
\begin{proof}
Beginnen wir mit einer Konstante $C_m > 0$, für die $\abs{f(\xi)} \le C_m \piket{\xi}{m}$ für jedes $\xi \in \R^n$ ist.
Um die rechte Seite zu beweisen, wählen wir ein $j \in \N_0$ und haben
$$
\abs{ \lilwood{j}(\xi) f(\xi) } \le C_m \characteristic{D_j}(\xi) \piket{\xi}{m} ,$$
wobei auf dem dyadischen Ring die charakteristische Funktion gegeben ist durch
$$\characteristic{D_j}(\xi) := 
\begin{cases}
1 & \textrm{ für } \xi \in D_j, \\
0 & \textrm{ sonst. }
\end{cases}
$$
Für $\xi \in D_j$ folgt mit $\abs{\xi} \le 2^{j+1}$
$$
\piket{\xi}{m} = \left( 1 + \abs{\xi}^2 \right)^{\frac{m}{2}} 
\le 
\begin{cases}
\left( 1 + 2^{j+1} \right)^m & m \ge 0, \\
\left( 1 + 2^{j-1} \right)^m & m < 0, \\
\end{cases}
\le 
\begin{cases}
2^{(j+2)m} & m \ge 0, \\
2^{(j-2)m} & m < 0. \\
\end{cases}
$$
Für die Rückrichtung wollen wir ausgehend von der rechten Seite die linke beweisen.
Dazu haben wir für ein $\xi \in D_j$
$$
\abs{f(\xi)} = 
\abs{ \sum_{k=0}^{\infty} \lilwood{k}(\xi) f(\xi) } \le
\sum_{k=0}^{j+1} \abs{ \lilwood{k}(\xi) f(\xi)} \le
\sum_{k=0}^{j+1} \tilde{C}_m 2^{km} =
\sum_{k=0}^{j+1} \tilde{C}_m (2^m)^k
.$$
Mit der Darstellungsformel für die Partialsumme der geometrischen Reihe folgt:
$$
\abs{f(\xi)} 
= \tilde{C}_m \frac{1 - (2^m)^{j+2}}{1 - 2^m} 
\le \tilde{C}_m 2^{(j+2)m - m+1}
= \tilde{C}_m 2^{(j+1)m+1} 
\le \tilde{C}_m  2^{2m+1} \piket{\xi}{m}
,$$
wobei $2^{j-1} \le \abs{\xi}$ für $\xi \in D_j$ und wir Lemma \ref{lemma:japanese} benutzt haben.
\end{proof}
\end{lemma}
 
\subsection{Der Hölder-Raum}\label{sec:hoelder}

\begin{definition}[Hölder-Raum]
Für $t \in \N_0, \theta \in [0,1]$ wird der Hölder-Raum $C^{t,\theta}(\R^n)$ durch die Norm
\begin{align}
\norm{u}{C^{t,\theta}(\R^n)} := \norm{u}{C^t(\R^n)} + \hnorm{u}{C^{t,\theta}} \\
\textrm{mit}
\hnorm{u}{C^{t,\theta}(\R^n)} := 
\begin{cases}
\max_{|\alpha| = t} \hnorm{\partial^\alpha u}{\theta} & \theta \ne 0, \\
0 & \theta = 0,
\end{cases}\\
\textrm{wobei}
\hnorm{f}{\theta} := 
\sup_{\substack{x,y \in \R^n\\x \neq y}} \frac{\left| f(x) - f(y)\right|}{\left| x-y \right|^{\theta}}
\end{align}
induziert. Insbesondere ist $C^{t,0}(\R^n) = C^t(\R^n)$ für jedes $t \in \N_0$.
Wir sagen für ein $f \in C^{t,\theta}(\R^n)$, dass $f$ Hölder-stetig zum Grad $t,\theta$ ist.
\end{definition}
\nomenclature[f]{$C^{t,\theta}(\R^n)$}{Hölder-Raum der Ordnung $t,\theta$}
\nomenclature[f]{$C^{t}(\R^n)$}{Hölder-Raum der Ordnung $t\in\R_+$}
\nomenclature[m]{$\R_+$}{Das Intervall $(0,\infty) \subset \R$}
\nomenclature[n]{$\norm{\hold}{C^{t,\theta}(\R^n)}$ }{Norm des Hölder-Raums der Ordnung $t,\theta$}
\nomenclature[n]{$\norm{\hold}{C^{t}(\R^n)}$ }{Norm des Hölder-Raums der Ordnung $t$}
\nomenclature[n]{$\hnorm{\hold}{C^{t,\theta}}$}{Halbnorm des Hölder-Raums der Ordnung $t,\theta$}
\nomenclature[n]{$\hnorm{\hold}{\theta}$}{Halbnorm des Hölder-Raums der Ordnung $\theta\in(0,1)$}

\begin{remark}
Die Abbildung $\norm{\hold}{C^{t,\theta}(\R^n)} : C^{t,\theta}(\R^n) \longrightarrow \R_{\ge 0}$ ist eine Norm, 
da $\hnorm{\hold}{C^{t,\theta}(\R^n)}$ offensichtlich eine Halbnorm definiert, und wir bereits wissen, dass $\norm{u}{C^t(\R^n)}$ eine Norm ist.
\end{remark}

\begin{satz}
Der Hölder-Raum $C^{t,\theta}(\R^n)$ mit $t \in \N_0$ und $\theta \in [0,1]$ ist vollständig.
\begin{proof}
Sei $\folge{u_j}{j\in\N}$ eine Cauchy-Folge in $C^{t,\theta}(\R^n)$. 
Dann ist $\folge{u_j}{j\in\N}$ Cauchy-Folge in $C^t(\R^n)$ und wegen der Vollständigkeit von $C^t(\R^n)$ 
gibt es ein $u \in C^t(\R^n)$, sodass $u_j \longrightarrow u$ für $j \rightarrow \infty$ in $C^t(\R^n)$.
Da $\folge{u_j}{j\in\N}$ eine Cauchy-Folge ist, gibt es für alle $\epsilon > 0$ ein $N_\epsilon \in \N$, sodass 
$\hnorm{u_j - u_k}{C^{t,\theta}(\R^n)} < \epsilon$
für alle $j, k \ge N_\epsilon$.
Setze $f_j := D^\alpha u_j , f := D^\alpha u$ für ein beliebiges $\alpha \in \N^n_0$ mit $\abs{\alpha} \le k$.
Wir haben damit für alle $x, y \in \R^n$ mit $x \neq y$ die Abschätzung
$$
\abs{ \left( f_j(x) - f_k(x) \right) - \left( f_j(y) - f_k(y) \right) } < \epsilon \abs{x-y}^\theta
.$$
Nun konvergiert $f_j \longrightarrow f$ für $j \rightarrow \infty$ gleichmäßig, sodass wir den Limes $k \rightarrow \infty$ bilden können, der
$$
\abs{ \left( f_j(x) - f(x) \right) - \left( f_j(y) - f(y) \right) } \le \epsilon \abs{x-y}^\theta
$$
ergibt. Da dies für alle $x, y \in \R^n$ mit $x \neq y$ gilt, hält auch die Ungleichung, wenn wir das Supremum über $x,y$ bilden.
Also ist $\hnorm{f_j - f}{\theta} < \epsilon$ für alle $j \ge N_\epsilon$.
Da $\alpha$ beliebig war, konvergiert $\hnorm{u - u_j}{C^{t,\theta}(\R^n)} \longrightarrow 0$.

Insgesamt konvergiert $u_j \longrightarrow u$ für $j \rightarrow \infty$ in $C^{t,\theta}(\R^n)$.
Mit der Dreiecksungleichung
$$
\frac{\abs{ f(x) - f(y) }}{\abs{x-y}^\alpha} \le 
\underbrace{\frac{\abs{ f(x) - f_j(x) }}{\abs{x-y}^\alpha}}_{\rightarrow 0 \textrm{ für } j \rightarrow \infty} +
\frac{\abs{ f_j(x) - f_j(y) }}{\abs{x-y}^\alpha} +
\underbrace{\frac{\abs{ f_j(y) - f(y) }}{\abs{x-y}^\alpha}}_{\rightarrow 0 \textrm{ für } j \rightarrow \infty}
$$
folgt schließlich, dass $u \in C^{t,\theta}(\R^n)$ ist.
\end{proof}
\end{satz}

\begin{lemma}
Für $\alpha \in \N^n_0, k \in \N_0, \theta \in [0,1]$ ist $\norm{\partial^\alpha u}{C^{k,\theta}(\R^n)} \le \norm{u}{C^{k+\abs{\alpha},\theta}(\R^n)}$ für $u \in C^{k+\abs{\alpha},\theta}(\R^n)$.
Insbesondere ist deswegen für $k,K \in \N_0$ mit $k < K$ bereits $C^{K,\theta}(\R^n) \hookrightarrow C^{k,\theta}(\R^n)$ für alle $\theta \in [0,1]$ natürlich eingebettet.
\begin{proof}
Klar, da
\begin{align*}
\norm{\partial^\alpha u}{C^{k,\theta}(\R^n)} 
&= \norm{\partial^\alpha u}{C^k(\R^n)} + \max_{\abs{\beta} = k} \hnorm{\partial^\beta \partial^\alpha u}{\theta} 
\le \norm{u}{C^{k+\abs{\alpha}}(\R^n)} + \max_{\abs{\beta} = k+\abs{\alpha}} \hnorm{\partial^\beta u}{\theta} \\
&= \norm{u}{C^{k+\abs{\alpha},\theta}(\R^n)}
.
\end{align*}
\end{proof}
\end{lemma}

\begin{lemma}
Für $0 \le \theta < \Theta < 1$ ist $C^{k,\Theta}(\R^n) \subset C^{k,\theta}(\R^n)$ für alle $k \in \N_0$
\begin{proof}
Für $\theta = 0$ ist die Behauptung trivialerweise erfüllt. 
Für $\theta \neq 0$ betrachte ein $u \in C^{k,\theta}(\R^n)$. 
Setze $f := \partial^\alpha u \in C^{0,\theta}(\R^n)$ für ein $\alpha \in \N^n_0$ mit $\abs{\alpha} = k$.
Seien $x,y \in \R^n, x \neq y$ gegeben.
Falls $\abs{x-y} < 1$, ist
$$
\frac{\abs{f(x)-f(y)}}{\abs{x-y}^\theta} = 
\frac{\abs{f(x)-f(y)}}{\abs{x-y}^\Theta} \abs{x-y}^{\Theta-\theta} \le
\frac{\abs{f(x)-f(y)}}{\abs{x-y}^\Theta}
.$$
Falls $\abs{x-y} \ge 1$ ist bereits $\frac{\abs{f(x)-f(y)}}{\abs{x-y}^\vartheta} \le \abs{f(x)-f(y)}$ für alle $\vartheta \in (0,1)$.
Geht man also zum Supremum bzgl. aller $x,y \in \R^n$ über, so ist bereits $\hnorm{f}{t,\theta} \le \hnorm{f}{t,\Theta}$ für beliebiges $u \in C^{k,\Theta}(\R^n)$ für beliebiges $\alpha \in \N^n_0$.
\end{proof}
\end{lemma}

\subsection{Interpolation der Hölder-Räume}

Ziel dieses Abschnittes wird es sein, den Hölder-Raum $C^{k,\vartheta}(\R^n)$ zu interpolieren, also ein Abschätzungskriterium in der Form
des folgenden Theorems zu erhalten, das wir in diesem Abschnitt beweisen werden:
\begin{thm}\label{interpol:cor}
Für $t,k \in \N, \theta, \vartheta \in [0,1)$ mit $t+\theta \le k+\vartheta$ finden wir eine Konstante $c_{t+\theta,k+\vartheta} > 0$, die die Ungleichung
$$ \norm{f}{C^{t,\theta}(\R^n)} \le c_{t+\theta,k+\vartheta} \norm{f}{C^0(\R^n)}^{1-\frac{t+\theta}{k+\vartheta}} \norm{f}{C^{k,\vartheta}(\R^n)}^{\frac{t+\theta}{k+\vartheta}} \textrm{ für alle } f \in C^{k+\vartheta}(\R^n)$$
erfüllt.
\end{thm}

Um das Theorem zu beweisen, werden wir zunächst einige Spezialfälle beweisen, um damit auf den allgemeinen Fall zu schließen.

\begin{lemma}\label{interpol:theta<0*1}
Für alle $\theta \in (0,1)$ gibt es ein $c_\theta > 0$, sodass für alle $f \in C^0(\R^n) \cap C^1(\R^n)$ die Abschätzung
$$
\norm{f}{C^{0,\theta}(\R^n)} \le c_\theta \norm{f}{C^0(\R^n)}^{1-\theta} \norm{f}{C^1(\R^n)}^\theta
$$
gilt.
\begin{proof}
Wir haben für $f \in C^0(\R^n) \cap C^1(\R^n), \theta \in (0,1)$ und $x,y \in \R^n$ mit $x \neq y$
$$ \frac{ \abs{ f(x) - f(y) }}{\abs{x-y}^\theta} = \left( \frac{ \abs{f(x) - f(y)}}{\abs{x-y}} \right)^\theta \abs{ f(x) - f(y) }^{1-\theta} .$$
Mit den elementaren Abschätzungen $ \frac{ \abs{f(x) - f(y)}}{\abs{x-y}} \le \norm{ \nabla f }{\infty} \le \norm{f}{C^1(\R^n)}$ 
und $\abs{f(x) - f(y) } \le 2 \norm{f}{C^0(\R^n)}$ folgt sofort die Behauptung.
\end{proof}
\end{lemma}

\begin{cor}\label{interpol:t,theta<t*t+1}
Zu jedem $\theta \in (0,1)$, $t \in \N_0$ gibt es eine Konstante $c_\theta > 0$, die für alle $f \in C^{t,\theta}(\R^n)$ die Ungleichung
$$\norm{f}{C^{t,\theta}(\R^n)} \le c_{t\theta} \norm{f}{C^t(\R^n)}^{1-\theta} \norm{f}{C^{t+1}(\R^n)}^{\theta}$$
erfüllt.
\begin{proof}
Nach Lemma \ref{interpol:theta<0*1} ist für alle $\alpha \in \N^n_0$ mit $\abs{\alpha} \le t$
\begin{align*}
\norm{ \partial^\alpha f}{C^0(\R^n)} + \hnorm{\partial^\alpha f}{\theta} 
&= \norm{ \partial^\alpha f}{C^\theta(\R^n)} 
\le c_\theta \norm{ \partial^\alpha f}{C^0(\R^n)}^{1-\theta} \norm{\partial^\alpha f}{C^1(\R^n)}^{\theta} \\
&\le c_{\alpha\theta} \norm{f}{C^t(\R^n)}^{1-\theta} \norm{f}{C^{t+1}(\R^n)}^\theta
.
\end{align*}
Da dies für alle $\alpha$ gilt, hält
$ \norm{f}{C^{t,\theta}(\R^n)} \le c_{t\theta} \norm{f}{C^t(\R^n)}^{1-\theta} \norm{f}{C^{t+1}(\R^n)}^\theta $.
\end{proof}
\end{cor}

\begin{lemma}\label{interpol:t<t-1*t,theta}
Zu $t \in \N, \theta \in (0,1)$ gibt es eine Konstante $c_{t\theta} > 0$, sodass für alle $f \in C^{t,\theta}(\R^n)$ die Abschätzung
$$\norm{f}{C^t(\R^n)} \le c_{t\theta} \norm{f}{C^{t-1}(\R^n)}^{1-\frac{1}{\theta+1}} \norm{f}{C^{t,\theta}(\R^n)}^{\frac{1}{\theta+1}}$$
erfüllt ist.
\begin{proof}
Setze $\tau := \theta+1$ und $g := \partial^\alpha f \in C^{1,\theta}(\R^n)$ für $\alpha \in \N^n_0$ mit $\abs{\alpha} = t-1$.
Bezeichne die kanonische Basis des $\R^n$ mit $\menge{e_j}_{j\in\N}$.
Setze $r := \left( \frac{ \norm{g}{C^0(\R^n)} }{ \norm{g}{C^{1,\theta}(\R^n)}} \right)^{\frac{1}{\tau}}$.
Sei dazu $\xi := x + \mu y$ mit $\mu \in [0,1], y := r e_j$ für ein $j \in \menge{1,\ldots,n}$.
Der Mittelwertsatz im Mehrdimensionalen liefert nun, dass für ein $\mu \in [0,1]$ die Gleichung
$$g(x+y) = g(x) + \nabla g(\xi)\cdot y = g(x) + \nabla g(x) \cdot y + \left(\nabla g(\xi) - \nabla g(x) \right) \cdot y$$
erfüllt ist.
Damit ist
$$\abs{\nabla g(x)\cdot y} \le \abs{g(x+y)} + \abs{g(x)} + \abs{\nabla g(x) - \nabla g(\xi)} \abs{y} .$$
Da $\nabla g(x) - \nabla g(\xi) \in C^{0,\theta}(\R^n)$ ist
$$\abs{\partial_j g(x)} r 
= \abs{\nabla g(x)\cdot y} 
\le 2 \norm{g}{C^0(\R^n)} + \norm{g}{C^{1,\theta}(\R^n)} \abs{y}^\theta \abs{y} 
= 2 \norm{g}{C^0(\R^n)} + \norm{g}{C^{1,\theta}(\R^n)} r^{\theta+1} .$$
Also
$$\abs{\partial_j g(x)} 
\le 2 \frac{\norm{g}{C^0(\R^n)}}{r} + \norm{g}{C^{1,\theta}(\R^n)} r^\theta 
= 2 \norm{g}{C^0(\R^n)}^{1 - \frac{1}{\tau}} \norm{g}{C^{1,\theta}(\R^n)}^{\frac{1}{\tau}} + \norm{g}{C^0(\R^n)}^{1 - \frac{1}{\tau}} \norm{g}{C^{1,\theta}(\R^n)}^{\frac{1}{\tau}}
.$$
Insgesamt finden wir für alle $\abs{\alpha} \le t-1$ ein $c>0$, sodass
$$
\norm{\partial^\alpha \partial_j f}{C^0(\R^n)} \le c \norm{\partial^\alpha f}{C^0(\R^n)}^{1-\frac{1}{\tau}} \norm{\partial^\alpha f}{C^{1,\theta}(\R^n)}^{\frac{1}{\tau}}
,$$
und mit dem Übergang zum Supremum über alle $\abs{\alpha} = t-1$ erhalten wir
$$
\norm{\partial_j f}{C^{t-1}(\R^n)} \le c_{t\theta} \norm{f}{C^{t-1}(\R^n)}^{1-\frac{1}{\tau}} \norm{f}{C^{t,\theta}(\R^n)}^{\frac{1}{\tau}} 
$$
für alle $j \in \menge{1,\ldots,n}$, also hält auch
$
\norm{f}{C^{t}(\R^n)} \le c_{t\theta} \norm{f}{C^{t-1}(\R^n)}^{1-\frac{1}{\tau}} \norm{f}{C^{t,\theta}(\R^n)}^{\frac{1}{\tau}}
.$
\end{proof}
\end{lemma}

\begin{cor}\label{interpol:t<t-1*t+1}
Für $t \in \N$ finden wir eine Konstante $c_t > 0$, die die Abschätzung 
$$\norm{f}{C^t(\R^n)} \le c_t \norm{f}{C^{t-1}(\R^n)}^{\frac{1}{2}} \norm{f}{C^{t+1}(\R^n)}^{\frac{1}{2}} \textrm{ für alle } f \in C^{t+1}(\R^n)$$
erfüllt.
\begin{proof}
Mit Lemma \ref{interpol:t<t-1*t,theta} und Folgerung \ref{interpol:t,theta<t*t+1} erhalten wir für $\theta \in (0,1)$
$$
\norm{f}{C^t(\R^n)} \le 
c_{t\theta} \norm{f}{C^{t-1}(\R^n)}^{1-\frac{1}{1+\theta}} \norm{f}{C^{t,\theta}(\R^n)}^{\frac{1}{1+\theta}} \le
c_{t\theta}' \norm{f}{C^{t-1}(\R^n)}^{1-\frac{1}{1+\theta}} \norm{f}{C^t(\R^n)}^{\frac{1-\theta}{1+\theta}} \norm{f}{C^{t+1}(\R^n)}^{\frac{\theta}{1+\theta}}
,$$
sodass
$\norm{f}{C^t(\R^n)} \le c_{t\theta}' \norm{f}{C^{t-1}(\R^n)}^{\frac{1}{2}} \norm{f}{C^{t+1}(\R^n)}^{\frac{1}{2}}$
gleichmäßig in $\theta \in (0,1)$.
\end{proof}
\end{cor}

\begin{lemma}\label{interpol:t<0*t+1}
Zu jedem $t \in \N$ gibt es eine Konstante $c_t > 0$, 
sodass für jedes $f \in C^t(\R^n)$ die Ungleichung
$
\norm{f}{C^t(\R^n)} \le c_t \norm{f}{C^0(\R^n)}^{\frac{1}{t+1}} \norm{f}{C^{t+1}(\R^n)}^{\frac{t}{t+1}}
$
gilt.
\begin{proof}
Die Behauptung lässt sich durch eine Induktion über $t \in \N$ beweisen.
Dafür wurde in Folgerung \ref{interpol:t<t-1*t+1} bereits der Fall für $t=1$ gezeigt.
Wir nehmen desweiteren an, dass die Behauptung $\norm{f}{C^t(\R^n)} \le c_t \norm{f}{C^0(\R^n)}^{\frac{1}{t+1}} \norm{f}{C^{t+1}(\R^n)}^{\frac{t}{t+1}}$ für ein $t \in \N$ bereits gilt.
Wir werden nun die Behauptung für $t \mapsto t+1$ mit Folgerung \ref{interpol:t<t-1*t+1} und der Induktionsannahme zeigen:
$$ \norm{ f }{C^{t+1}(\R^n)} 
\le c_t \norm{f}{C^t(\R^n)}^{\frac{1}{2}} \norm{f}{C^{t+2}(\R^n)}^{\frac{1}{2}} 
\le c_t \norm{f}{C^0(\R^n)}^{\frac{1}{2(t+1)}} \norm{f}{C^{t+1}(\R^n)}^{\frac{1}{2(t+1)}} \norm{f}{C^{t+2}(\R^n)}^{\frac{1}{2}} 
.$$
Somit ist
$ \norm{f}{C^{t+1}(\R^n)} \le c_t \norm{f}{C^0(\R^n)}^{\frac{1}{t+2}} \norm{f}{C^{t+2}(\R^n)}^{\frac{t+1}{t+2}} $.
\end{proof}
\end{lemma}

\begin{lemma}\label{interpol:t<0*k}
Für $k, t \in \N$ mit $t \le k$ ist
$ \norm{ f }{C^t(\R^n)} \le c_{t,k} \norm{f}{C^0(\R^n)}^{1-\frac{t}{k}} \norm{f}{C^k(\R^n)}^{\frac{t}{k}}$ für alle $f \in C^k(\R^n)$.
\begin{proof}
Wir werden die Behauptung mit einer Induktion über $t \le k \in \N$ beweisen.
Für $t = k$ folgt die Behauptung trivialerweise.
Wir nehmen an, dass die Behauptung für ein $t$ bereits gelte. Für $t \mapsto t-1$ folgt dann mit Lemma \ref{interpol:t<0*t+1}
\begin{align*}
\norm{f}{C^{t-1}(\R^n)} 
&\le c_{t-1,k} \norm{f}{C^0(\R^n)}^{\frac{1}{t}} \norm{f}{C^t(\R^n)}^{\frac{t-1}{t}} 
\le c_{t,k} \norm{f}{C^0(\R^n)}^{\frac{1}{t} + (1-\frac{t}{k})(\frac{t-1}{t})} \norm{f}{C^k(\R^n)}^{\frac{t}{k}(\frac{t-1}{t})} \\
&\le c_{t,k} \norm{f}{C^0(\R^n)}^{1 - \frac{t-1}{k}} \norm{f}{C^k(\R^n)}^{\frac{t-1}{k}}
.\end{align*}
\end{proof}
\end{lemma}

\begin{cor}\label{interpol:t,theta<0*k}
Für $k, t \in \N$ mit $t < k$ und $\theta \in (0,1)$ finden wir eine Konstante $c_{t,k,\theta} > 0$, sodass
$$ \norm{f}{C^{t,\theta}(\R^n)} \le c_{t,k,\theta} \norm{f}{C^0(\R^n)}^{1-\frac{t+\theta}{k}} \norm{f}{C^k(\R^n)}^{\frac{t+\theta}{k}} \textrm{ für alle } f \in C^k(\R^n) .$$
\begin{proof}
Mit Folgerung \ref{interpol:t,theta<t*t+1} und Lemma \ref{interpol:t<0*k} erhalten wir
\begin{align*}
\norm{f}{C^{t,\theta}(\R^n)} 
&\le c_{t,\theta} \norm{f}{C^t(\R^n)}^{1-\theta} \norm{f}{C^{t+1}(\R^n)}^{\theta}
\le c_{t,k,\theta} \norm{f}{C^0(\R^n)}^{(1-\theta)(1-\frac{t}{k})} \norm{f}{C^k(\R^n)}^{(1-\theta)\frac{t}{k}} \norm{ f }{C^0(\R^n)}^{\theta(1-\frac{t+1}{k})} \norm{f}{C^k(\R^n)}^{\theta(\frac{t+1}{k})} \\
&\le c_{t,k,\theta} \norm{f}{C^0(\R^n)}^{1-\frac{t+\theta}{k}} \norm{f}{C^k(\R^n)}^{\frac{t+\theta}{k}}
.\end{align*}
\end{proof}
\end{cor}

\begin{lemma}\label{interpol:k<0*k,theta}
Für $k \in \N, \vartheta \in (0,1)$ finden wir eine Konstante $c_{k,\vartheta} > 0$, sodass
$$ \norm{f}{C^k(\R^n)} \le c_{k,\vartheta} \norm{f}{C^0(\R^n)}^{1-\frac{k}{k+\vartheta}} \norm{f}{C^{k,\vartheta}(\R^n)}^{\frac{k}{k+\vartheta}} \textrm{ für alle } f \in C^{k,\vartheta}(\R^n) .$$
\begin{proof}
Für $k = 1$ folgt die Behauptung aus Lemma \ref{interpol:t<t-1*t,theta}.
Der Induktionsschritt $k \mapsto k+1$ folgt aus Lemma \ref{interpol:t<t-1*t,theta} und Lemma \ref{interpol:t<0*t+1} mit
\begin{align*}
\norm{f}{C^{k+1}(\R^n)} 
&\le c_{k+1,\vartheta} \norm{f}{C^k(\R^n)}^{1-\frac{1}{\vartheta+1}} \norm{f}{C^{k+1,\vartheta}(\R^n)}^{\frac{1}{\vartheta+1}} \\
&\le c_{k+1,\vartheta} \norm{f}{C^0(\R^n)}^{(1-\frac{1}{\vartheta-1})(1-\frac{k}{k+1})} \norm{f}{C^{k+1}(\R^n)}^{(1-\frac{1}{\vartheta+1})\frac{k}{k+1}} \norm{f}{C^{k+1,\vartheta}(\R^n)}^{\frac{1}{\vartheta+1}}
.\end{align*}
Also ist
$ \norm{f}{C^{k+1}(\R^n)} \le c_{k\vartheta} \norm{f}{C^0(\R^n)}^{1-\frac{k+1}{\vartheta+k+1}} \norm{f}{C^{k+1,\vartheta}(\R^n)}^{\frac{k+1}{\vartheta+k+1}} $.
\end{proof}
\end{lemma}

\begin{cor}\label{interpol:t,theta<0*k,vartheta}
Für $k, t \in \N$ und $\theta, \vartheta \in (0,1)$ mit $t< k$ gibt es eine Konstante $c_{t+\theta,k+\vartheta} > 0$ mit
$$ \norm{f}{C^{t,\theta}(\R^n)} \le c_{t+\theta,k+\vartheta} \norm{f}{C^0(\R^n)}^{1-\frac{t+\theta}{k+\vartheta}} \norm{f}{C^{k,\vartheta}(\R^n)}^{\frac{t+\theta}{k+\vartheta}} .$$
\begin{proof}
Mit Lemma \ref{interpol:t,theta<0*k} und Lemma \ref{interpol:k<0*k,theta} folgt
\begin{align*}
\norm{f}{C^{t,\theta}(\R^n)} 
&\le c_{t+\theta,k} \norm{f}{C^0(\R^n)}^{1-\frac{t+\theta}{k}} \norm{f}{C^k(\R^n)}^{\frac{t+\theta}{k}} 
\le c_{t+\theta,k+\vartheta} \norm{f}{C^0(\R^n)}^{1-\frac{t+\theta}{k} + \frac{t+\theta}{k}(1-\frac{k}{k+\vartheta})} \norm{f}{C^{k,\vartheta}(\R^n)}^{\frac{k}{k+\vartheta} \frac{t+\theta}{k}} \\
&= c_{t+\theta,k+\vartheta} \norm{f}{C^0(\R^n)}^{1-\frac{t+\theta}{k+\vartheta}} \norm{f}{C^{k,\vartheta}(\R^n)}^{\frac{t+\theta}{k+\vartheta}}
.\end{align*}
\end{proof}
\end{cor}

\begin{cor}\label{interpol:t<0*k,vartheta}
Für $k,t \in \N, \vartheta \in (0,1)$ finden wir eine Konstante $c_{t,k+\vartheta} > 0$, sodass
$$ \norm{f}{C^t(\R^n)} \le c_{t,k+\vartheta} \norm{f}{C^0(\R^n)}^{1-\frac{t}{k+\vartheta}} \norm{f}{C^{k,\vartheta}(\R^n)}^{\frac{t}{k+\vartheta}} \textrm{ für alle } f \in C^{k,\vartheta}(\R^n).$$
\begin{proof}
Folgt direkt mit Lemma \ref{interpol:t<0*k} und Lemma \ref{interpol:k<0*k,theta}
\begin{align*}
\norm{f}{C^t(\R^n)} 
&\le c_{t,k} \norm{f}{C^0(\R^n)}^{1-\frac{t}{k}} \norm{f}{C^k(\R^n)}^{\frac{t}{k}} 
\le c_{t,k+\vartheta} \norm{f}{C^0(\R^n)}^{1-\frac{t}{k} + (1-\frac{k}{k+\vartheta})\frac{t}{k}} \norm{f}{C^{k,\vartheta}(\R^n)}^{\frac{k}{k+\vartheta}\frac{t}{k}} \\
&=   c_{t,k+\vartheta} \norm{f}{C^0(\R^n)}^{1-\frac{t}{k+\vartheta}} \norm{f}{C^{k,\vartheta}(\R^n)}^{\frac{t}{k+\vartheta}}
.\end{align*}
\end{proof}
\end{cor}

\begin{lemma}\label{interpol:t,theta<0*t,vartheta}
Für $t \in \N, \theta, \vartheta \in (0,1)$ mit $\theta < \vartheta$ gibt es eine Konstante $c_{t+\theta,k+\vartheta} > 0$, sodass
$$ \norm{f}{C^{t,\theta}(\R^n)} \le c_{t+\theta,k+\vartheta} \norm{f}{C^0(\R^n)}^{1-\frac{t+\theta}{t+\vartheta}} \norm{f}{C^{t,\vartheta}(\R^n)}^{\frac{t+\theta}{t+\vartheta}} \textrm{ für alle } f \in C^{t,\vartheta}(\R^n) .$$
\begin{proof}
Zunächst ist für $\alpha \in \N^n_0$ mit $\abs{\alpha} \le t$ und $x,y\in \R^n$
\begin{align*}
\frac{\abs{\partial_x^\alpha f(x) - \partial_y^\alpha f(y) }}{\abs{x-y}^{\theta}} &=
\left( \frac{ \abs{\partial_x^\alpha f(x) - \partial_y^\alpha f(y) }^{\frac{\vartheta}{\theta}-1} \abs{\partial_x^\alpha f(x) - \partial_y^\alpha f(y) } }{\abs{x-y}^\vartheta} \right)^{\frac{\theta}{\vartheta}} \\
&= \left( \frac{ \abs{\partial_x^\alpha f(x) - \partial_y^\alpha f(y) } }{\abs{x-y}^\vartheta} \right)^{\frac{\theta}{\vartheta}} \abs{\partial^\alpha f(x) - \partial^\alpha f(y) }^{1-\frac{\theta}{\vartheta}} \\
&\le 2 \hnorm{\partial^\alpha f}{\vartheta}^{\frac{\theta}{\vartheta}} \norm{f}{C^t(\R^n)}^{1-\frac{\theta}{\vartheta}}
\end{align*}
gleichmäßig für alle $x,y\in\R^n$. Also ist
$$ \norm{f}{C^{t,\theta}(\R^n)} \le 2 \hnorm{f}{C^{t,\vartheta}(\R^n)}^{\frac{\theta}{\vartheta}} \norm{f}{C^t(\R^n)}^{1-\frac{\theta}{\vartheta}} + \norm{f}{C^t(\R^n)}
\le 2 \norm{f}{C^{t,\vartheta}(\R^n)}^{\frac{\theta}{\vartheta}} \norm{f}{C^t(\R^n)}^{1-\frac{\theta}{\vartheta}} .$$
Mit Lemma \ref{interpol:k<0*k,theta} folgt
$$ \norm{f}{C^{t,\theta}(\R^n)} \le c_{t+\theta,k+\vartheta} \norm{f}{C^0(\R^n)}^{(1-\frac{k}{k+\vartheta})(1-\frac{\theta}{\vartheta})} \norm{f}{C^{t,\vartheta}(\R^n)}^{\frac{\theta}{\vartheta} + (1-\frac{\theta}{\vartheta})(\frac{k}{k+\vartheta})} =
c_{t+\theta,k+\vartheta} \norm{f}{C^0(\R^n)}^{1-\frac{k+\theta}{k+\vartheta}} \norm{f}{C^{t,\vartheta}(\R^n)}^{\frac{t+\theta}{t+\vartheta}}
.$$
\end{proof}
\end{lemma}

\begin{proof}[Beweis von Theorem \ref{interpol:cor}]
Falls $t = k$, erhalten wir für $\theta > 0$ mit Lemma \ref{interpol:t,theta<0*t,vartheta} die Behauptung, für $\theta = 0$ mit Lemma \ref{interpol:k<0*k,theta}.
Für $t < k$ erhalten wir die Behauptung mit

\begin{tabular}{llll}
Lemma \ref{interpol:t<0*k}                    & falls &$\theta = 0    $& und $ \vartheta = 0    $,\\ 
Folgerung \ref{interpol:t<0*k,vartheta}       & falls &$\theta = 0    $& und $ \vartheta \neq 0 $,\\ 
Folgerung \ref{interpol:t,theta<0*k}          & falls &$\theta \neq 0 $& und $ \vartheta = 0    $,\\ 
Folgerung \ref{interpol:t,theta<0*k,vartheta} & falls &$\theta \neq 0 $& und $ \vartheta \neq 0 $.
\end{tabular}
\end{proof}

\begin{konv}
Im restlichen Teil dieser Arbeit werden wir für $\tau > 0$ den Hölder-Raum $C^{\gauss{\tau},\tau-\gauss{\tau}}(\R^n)$ mit $C^\tau(\R^n)$ bezeichnen.
Die assoziierte Norm ist hierbei durch
$
\norm{u}{C^\tau(\R^n)} := \norm{u}{C^{\gauss{\tau},\tau-\gauss{\tau}}(\R^n)}
$
gegeben.
Die Lipschitzräume $C^{t,1}(\R^n)$ mit $t \in \N_0$ werden durch diese Notation außer Acht gelassen.
\end{konv}

\subsection{Der Bessel-Potential-Raum}\label{sec:bessel}
Eine Verallgemeinerung der Sobolev-Räume $W^s_q(\R^n)$ für $s \in \N_0, 1 < q < \infty$ stellt
der Bessel-Potential-Raum $H^s_q(\R^n)$ für $s \in \R$ mit $W^s_q(\R^n) = H^s_q(\R^n)$ für alle $s \in \N_0$ dar.
Wir werden diesen Raum später verwenden, um Abbildungseigenschaften von Pseudodifferentialoperatoren studieren zu können.
Dazu wollen wir in diesem Abschnitt die Bessel-Potential-Räume einführen, die wir mit Hilfe der komplexen Interpolationstheorie 
als ``Zwischenräume'' der Sobolev-Räume charakterisieren werden. 

Wir beginnen mit einer Einführung der komplexen Interpolation analog zu \cite{Lunardi}:
\nomenclature[f]{$W^s_q(\R^n)$}{Sobolev-Raum der Ordnung $s\in\N_0$ mit $1 \le q \le \infty$}

\begin{definition}[Interpolationspaar]
Seien $X, Y$ reelle (bzw. komplexe) Banachräume.
Falls $X$ und $Y$ sich stetig in einem topologischen Vektorraum $V$ einbetten lassen, 
bezeichnet man das Tupel $(X,Y)$ als reelles (bzw. komplexes) Interpolationspaar.
Die stetigen Einbettungen werden mit $i_X : X \rightarrow V$ und $i_Y : Y \rightarrow V$ bezeichnet.
\end{definition}

\begin{lemma}
Sei $(X,Y)$ ein Interpolationspaar.
\begin{enumerate}[(a)]
\item
Der lineare Untervektorraum $X \cap Y \subset V$ wird durch die Norm
$$
\norm{f}{X \cap Y} := \max \menge{ \norm{f}{X}, \norm{f}{Y} }$$
zum Banachraum.
\item
Auch $X+Y := \menge{a+b : a \in X, b \in Y} \subset V$ ist ein linearer Untervektorraum,
der mit der Norm
$$
\norm{f}{X+Y} := \inf_{\substack{f = a + b\\ a \in X, b \in Y}} \left( \norm{a}{X} + \norm{b}{Y} \right)
$$
zum Banachraum wird.
Da $\norm{\hold}{X} \le \norm{\hold}{X+Y}$ und $\norm{\hold}{Y} \le \norm{\hold}{X+Y}$, werden $X$ und $Y$ stetig in $X+Y$ eingebettet.
\end{enumerate}
\begin{proof}
\begin{enumerate}[(a)]
\item Offensichtlich definiert $\norm{\hold}{X \cap Y}$ eine Norm auf $X \cap Y$.
Sei $\folge{u_n}{n \in \N}$ Cauchy-Folge in $X \cap Y$.
Da $X$ und $Y$ Banachräume sind, folgt, dass es ein $x \in X$ und ein $y \in Y$ gibt, 
sodass die Cauchy-Folge $\folge{u_n}{n \in \N}$ in $X$ gegen $x$ und in $Y$ gegen $y$ konvergiert.
Nun sind die Einbettungen  $i_X : X \rightarrow V$ und $i_Y : Y \rightarrow V$ stetig, sodass $i_X |_{X \cap Y} = i_Y |_{X \cap Y}$ und damit
$$i_Y(y) = \lim_{n \rightarrow \infty} i_Y(u_n) = \lim_{n \rightarrow \infty} i_X(u_n) = i_X(x)$$
gilt, d.h. $x = y$ ist.
Insgesamt ist also
$$ \norm{u_n - u}{X \cap Y} = \left( \norm{u_n - x}{X} + \norm{u_n - y}{Y} \right) \rightarrow 0 \textrm{ für } n \rightarrow \infty .$$
\item 
Um die Normeigenschaften von $\norm{\hold}{X+Y}$ zu überprüfen, genügt es, dass für alle $u \in X+Y$ mit $\norm{u}{X+Y} = 0$ folgt, dass $u = 0$ ist (der Rest ist trivial).
Für solch ein $u$ finden wir nach Definition zwei Folgen $\folge{x_n}{n\in\N}, \folge{y_n}{n\in\N}$ mit $x_n \in X, y_n \in Y, x_n + y_n = u$ und
$\norm{x_n}{X} + \norm{y_n}{Y} < \frac{1}{n}$.
Also konvergiert $x_n \rightarrow 0$ in $X$ und $y_n \rightarrow 0$ in $Y$. Da aber $X,Y \hookrightarrow V$ konvergieren beide Folgen auch in $V$ gegen $0$.
Also konvergiert auch $x_n + y_n \rightarrow 0$ in $V$ und damit muss bereits $u = 0$ sein.
Für die Vollständigkeit betrachten wir eine Cauchy-Folge
$\folge{u_n}{n \in \N}$ Cauchy-Folge in $X+Y$, $u_n =: x_n + y_n$ mit $x_n \in X, y_n \in Y$ für alle $n \in \N$.
Somit ist $\folge{x_n}{n \in \N}$ Cauchy-Folge in $X$ und $\folge{y_n}{n \in \N}$ Cauchy-Folge in $Y$,
also gibt es $x \in X, y \in Y$, sodass $x_n \rightarrow x$ in $X$ und $y_n \rightarrow y$ in $Y$.
Damit konvergiert bereits $u_n \rightarrow x + y$ in $X+Y$:
\begin{align*}
\norm{u_n - (x+y)}{X+Y} 
&= \norm{(x_n-x)+(y_n-y)}{X+Y} \\
&\le \norm{x_n - x}{X+Y} + \norm{y_n - y}{X+Y} \rightarrow 0 \textrm{ für } n \rightarrow \infty
.\end{align*}
\end{enumerate}
\end{proof}
\end{lemma}

\begin{definition}\label{interpol:komplPaar}
Sei $(X,Y)$ ein komplexes Interpolationspaar.
Sei 
$$\Streifen := \menge{ z \in \C : \Real z \in [0,1] }$$
der Streifen in der komplexen Ebene.
$\StripSum{X,Y}$ sei der Raum aller Funktionen $g : \Streifen \rightarrow X+Y$, für die 
\begin{enumerate}[(a)]
\item $g$ holomorph im Inneren des Streifen, stetig auf dessen Rand, sowie
\item $\R \rightarrow X, t \mapsto g(it) \in C(\R;X)$ und $ \R \rightarrow Y, t \mapsto g(1+it) \in C(\R;Y)$ ist.
\end{enumerate}
Setze $\norm{g}{\StripSum{X,Y}} := \max \menge{ \sup_{t\in\R} \norm{g(it)}{X}, \sup_{t\in\R} \norm{g(1+it)}{Y} } < \infty$.
Für jedes $\theta\in(0,1)$ setze
$$ \interpolC{X,Y}{\theta} := \menge{ g(\theta) : g \in \StripSum{X,Y} } \textrm{ und }
\norm{a}{\interpolC{X,Y}{\theta}} := \inf_{g\in \StripSum{X,Y}, g(\theta) = a} \norm{g}{\StripSum{X,Y}}.$$
\end{definition}
\nomenclature[m]{$\Streifen$}{Streifen in der komplexen Ebene - Definition \ref{interpol:komplPaar}}
\nomenclature[f]{$\interpolC{X,Y}{\theta}$}{Raum erzeugt durch die komplexe Interpolation des Paars $(X,Y)$ bzgl. $\theta$}
\nomenclature[n]{$\norm{\hold}{\interpolC{X,Y}{\theta}}$}{Norm des komplex interpolierten Raums $\interpolC{X,Y}{\theta}$}

\begin{thm}\label{thm:interpolComplex}
Seien $(X_1,Y_1),(X_2,Y_2)$ komplexe Interpolationspaare.
Für jedes $z \in \Streifen$ sei für einen Operator $T_z \in \linearop{X_1 \cap Y_1}{ X_2 + Y_2}$ die
Zuordnung $\Streifen \rightarrow X_2 + Y_2, z \mapsto T_z x$ holomorph in $\Streifen$ und stetig und beschränkt in $\Streifen$ für jedes $x \in X_1 \cap Y_1$.
Setze zudem voraus, dass $t \mapsto T_{it} x \in C(\R; \linearop{X_1}{X_2}), t \mapsto T_{1+it}x \in C(\R; \linearop{Y_1}{Y_2})$ und
$ \norm{T_{it}}{\linearop{X_1}{X_2}}$ bzw. $\norm{T_{1+it}}{\linearop{Y_1}{Y_2}}$ durch eine Konstante $M_0 > 0$ bzw. $M_1 > 0$ unabhängig von $t\in\R$ beschränkt sind, also
$$ M_0 := \sup_{t\in\R} \norm{T_{it}}{\linearop{X_1}{X_2}} < \infty,\; M_1 := \sup_{t\in\R} \norm{T_{1+it}}{\linearop{Y_1}{Y_2}} < \infty .$$
Für jedes $\theta\in(0,1)$ gilt nun, dass
$$\norm{T_\theta x}{\interpolC{X_2,Y_2}{\theta}} \le M_0^{1-\theta} M_1^\theta \norm{x}{\interpolC{X_1,Y_1}{\theta}} ,$$
sodass wir eine Erweiterung von $T_\theta$ auf $\linearop{\interpolC{X_1,Y_1}{\theta}}{\interpolC{X_2,Y_2}{\theta}}$ finden können, 
die wir mit $\tilde{T}_\theta$ bezeichnen, sodass $\tilde{T}_\theta |_{X_1 \cap Y_1} = T_\theta$.
Diese Erweiterung genügt nun der Abschätzung
$$ \norm{\tilde{T}_\theta}{\linearop{\interpolC{X_1,Y_1}{\theta}}{\interpolC{X_2,Y_2}{\theta})}} \le M_0^{1-\theta} M_1^\theta .$$
\begin{proof} Siehe \cite[Theorem 2.1.7]{Lunardi}
\end{proof}
\end{thm}

\begin{definition}
Der \emph{Bessel-Potential-Raum} $H^s_q(\R^n)$ mit $q \in (1,\infty), s \in \R$
ist definiert durch
\begin{equation}
H^s_q(\R^n) := \menge{ u \in \Schwarzi' : \piket{D_x}{s} \in L^q(\R^n)}
\end{equation}
mit der Norm
\begin{equation}
\norm{u}{H^s_q(\R^n)} := \norm{\piket{D_x}{s} u}{L^q(\R^n)} .
\end{equation}
\end{definition}
\nomenclature[f]{$H^s_q(\R^n)$}{Bessel-Potential-Raum der Ordnung $s\in\R$ mit $1 < q < \infty$}
\nomenclature[n]{$\norm{\hold}{H^s_q(\R^n)}$}{Norm des Bessel-Potential-Raums $H^s_q(\R^n)$}

\begin{thm}\label{thm:interpolBessel}
Für $\theta \in (0,1), 1 < q < \infty, m \in \N$ ist $ \interpolC{L^q(\R^n), W^m_q(\R^n) }{\theta} = H^{m\theta}_q(\R^n)$.
\begin{proof} Siehe \cite[\textsection 2.4.2 Definition (d)]{Triebel} oder \cite[Theorem 6.4.5]{Bergh}
\end{proof}
\end{thm}

\begin{thm}\label{thm:interpolBesselSobolev}
Für $s < t \in \R$ ist $H^t_q(\R^n) \subset H^s_q(\R^n)$ für alle $1 \le q \le \infty$.
Außerdem ist für ganzzahlige $m \in \N$ und $1 < q < \infty$ der Sobolev-Raum $W_q^m(\R^n)$ gleich dem Bessel-Potential-Raum $H_q^m(\R^n)$.
Die Normen beider Räume sind also äquivalent.
Schließlich ist der Schwartz-Raum $\Schwarzi$ für $1 \le q < \infty$ dicht in $H^s_q(\R^n)$ für alle $s \in \R$ enthalten.
\begin{proof} Siehe \cite[Theorem 6.2.3]{Bergh}
\end{proof}
\end{thm}

\begin{thm}\label{thm:roomBesselStetig}
Für $1 < q < \infty$, $s \in \R$ mit $\frac{n}{s} < q$ bettet $H^s_q(\R^n)$ in den Raum $C^0(\R^n)$ ein.
\begin{proof}Siehe \cite[2.8.1 Remark 2]{Triebel}
\end{proof}
\end{thm}

\begin{lemma}
Für $1 < q, q' < \infty$ mit $\frac{1}{q} + \frac{1}{q'} = 1$ gibt es einen isometrischen Isomorphismus $J : H^s_q(\R^n)' \xlongrightarrow{\sim} H^{-s}_{q'}(\R^n)$ für alle $s \in \R$.
Für $f \in H^{-s}_{q'}(\R^n) \subset \SchwarziDual, \varphi \in \Schwarzi \subset H^s_p(\R^n)$ ist das Dualitätsprodukt gegeben durch
$$ \dualskp{  J^{-1} f }{ \varphi }{H^s_q(\R^n)'}{H^s_q(\R^n)} := \dualskp{ f }{ \varphi }{\SchwarziDual}{\Schwarzi} = \int_{\R^n} f(x) \varphi(x) dx .$$
\begin{proof}
Der Pullback des ordnungserniedrigenden Operators $\piket{D_x}{-s} : L^q(\R^n) \xlongrightarrow{\sim} H^s_q(\R^n)$ ist der adjungierte Operator
$\left(\piket{D_x}{-s}\right)^* : H^s_q(\R^n)' \xlongrightarrow{\sim} L^q(\R^n)'$. $\piket{D_x}{-s}$ ist aber nach Definition eine Isometrie, also ist $\left(\piket{D_x}{-s}\right)^*$ auch eine.
Der isometrische Isomorphismus $T : L^{q'}(\R^n) \xlongrightarrow{\sim} L^q(\R^n)'$ gegeben durch den Riesz'schen Darstellungssatz definiert uns
schließlich einen isometrischen Isomorphismus $J :  H^s_q(\R^n)' \xlongrightarrow{\sim} H^{-s}_{q'}(\R^n)$, der das Diagramm
$$
\begin{xy}
\xymatrix{
	H^s_q(\R^n)' \ar@{-->}[d]^{J} \ar[rr]_{\left(\piket{D_x}{-s}\right)^*} && L^q(\R^n)' \\
	H^{-s}_{q'}(\R^n)  \ar[rr]_{\piket{D_x}{-s}} && L^{q'}(\R^n) \ar[u]_{T} 
}
\end{xy}
$$
kommutieren lässt.
Siehe dazu auch \cite[Corollary 6.2.8]{Bergh}.
\end{proof}
\end{lemma}
\nomenclature{$\dualskp{\hold}{\hold}{X'}{X}$}{Dualitätsprodukt bzgl. des Vektorraums $X$ mit seinem Dualraum $X'$}
\nomenclature{$\skp{\hold}{\hold}{X}$}{Skalarprodukt eines Hilbertraums $X$}

\subsection{Der Hölder-Zygmund-Raum}\label{sec:hoelderzygmund}
\begin{definition}[Hölder-Zygmund-Raum]
Für $\tau > 0$ ist der {\em Hölder-Zygmund-Raum} $C^\tau_*(\R^n)$ definiert durch
\begin{equation}
C^\tau_*(\R^n) = \menge{ u \in C^0 (\R^n) : \exists C > 0 \textrm{ sodass } \norm{ \lilwood{j}(D_x) u }{\infty} \le c 2^{-j\tau} \textrm{ für jedes } j \in \N_0 }
\end{equation}
mit der Norm
\begin{equation}\label{norm:C*}
\norm{u}{C^\tau_*(\R^n)} := \sup_{j \in \N_0} 2^{j\tau} \norm{\lilwood{j}(D_x)u}{\infty} .
\end{equation}
\end{definition}
\nomenclature[f]{$C^t_*(\R^n)$}{Hölder-Zygmund-Raum der Ordnung $t$}
\nomenclature[n]{$\norm{\hold}{C^t_*(\R^n)}$}{Norm des Hölder-Zygmund-Raums der Ordnung $t$}

\begin{remark}
Für $\tau \notin \N$ ist $C^\tau_*(\R^n) = C^\tau(\R^n)$ und für jedes $\tau \in \N$ ist $C^\tau(\R^n) \subset C^\tau_*(\R^n)$.
\begin{proof}
Siehe \cite[(A.1.6)]{Nonlinear} bzw. \cite[\S 2.5.7]{TriebelFunction}
\end{proof}
\end{remark}

\section{Pseudodifferentialoperatoren}\label{sec:psdo}
\subsection{Die Hörmanderklasse $C^\tau S^m_{\rho,\delta}(\R^n,\R^n)$} \label{sec:psdointro}

\begin{definition}
Seien $0 \le \delta \le \rho \le 1$.
Wir bezeichnen eine Funktion $p(x,\xi) : \R^n \times \R^n \rightarrow \C$, 
die glatt bezüglich $x$ und Hölder-stetig zum Grad $\tau \in \R^+$ ist, 
d.h. $p(x,\hold) \in \Stetigi{\infty}, p(\hold,\xi) \in C^\tau(\R^n)$,
als Symbol der Klasse $C^\tau S^m_{\rho,\delta}(\R^n, \R^n)$ (mit $m \in \R$), falls
es zu jedem Multiindex $\alpha \in \N^n_0$ eine Konstante $C_{\alpha} > 0$ gibt, sodass
$$ \abs{ D_\xi^\alpha p(x,\xi) } \le C_\alpha \piket{\xi}{m-\rho\abs{\alpha}} ,$$
und für jedes $0 \le t \le \tau$
\begin{equation}\label{equ:introPSDOtNorm}
\norm{D_\xi^\alpha p(\hold,\xi) }{C^t(\R^n)} \le C_\alpha \piket{\xi}{m-\rho\abs{\alpha}+\delta t}.
\end{equation}
Für $p \in C^\tau S^m_{\rho,\delta}(\R^n, \R^n)$ definiere die Familie von Halbnormen
$$
\abs{p}_{l,\tau}^{(m)} := \max_{\alpha\le l} \sup_{x,\xi\in\R^n} \left( \abs{\partial_\xi^\alpha p(x,\xi)} + \norm{\partial_\xi^\alpha p(\hold,\xi)}{C^\tau(\R^n)} \piket{\xi}{-\delta\tau} \right)  \piket{\xi}{-m+\rho\abs{\alpha}}
.$$
Die Halbnormen induzieren auf $C^\tau S^m_{\rho,\delta}(\R^n, \R^n)$ einen Fréchet-Raum.
Insbesondere ist 
$$\abs{ D_\xi^\alpha p(x,\xi) } + \norm{D_\xi^\alpha p(\hold,\xi) }{C^t(\R^n)} \le \abs{p}_{\abs{\alpha+\beta},\tau}^{(m)} \piket{\xi}{m-\rho\abs{\alpha}+\delta t} .$$

\end{definition}
\nomenclature[s]{$C^\tau S^m_{\rho,\delta}(\R^n,\R^n)$}{Klasse der Symbole in $x$-Form mit Hölder-Stetigkeitsgrad $\tau$}
\nomenclature[n]{$\abs{\hold}_{l,\tau}^{(m)}$}{Halbnorm der Klasse $C^\tau S^m_{\rho,\delta}(\R^n,\R^n)$}

\begin{remark}
Anstatt für jedes $0 \le t \le \tau$ die Abschätzung (\ref{equ:introPSDOtNorm}) zu überprüfen, reicht es, dass die Ungleichungen
\begin{align*}
\norm{D_\xi^\alpha p(\hold,\xi) }{C^0(\R^n)} &\le C_\alpha \piket{\xi}{m-\rho\abs{\alpha}} \textrm{ und } \\
\norm{D_\xi^\alpha p(\hold,\xi) }{C^\tau(\R^n)} &\le C_\alpha \piket{\xi}{m-\rho\abs{\alpha}+\delta\tau}
\end{align*}
erfüllt sind.
Denn für ein $0 \le t \le \tau$ und $\alpha \in \N^n_0$ setze $\theta := \frac{t}{\tau}$ und verwende Theorem \ref{interpol:cor}
\begin{align*}
\norm{ D_\xi^\alpha p(\hold,\xi) }{C^t(\R^n)} 
&\le c_{t\tau\theta} \norm{ D_\xi^\alpha p(\hold,\xi) }{C^0(\R^n)}^{1-\theta} \norm{ D_\xi^\alpha p(\hold,\xi) }{C^\tau(\R^n)}^\theta \\
&\le c_{t\tau\theta} c_\alpha \piket{\xi}{m-\rho\abs{\alpha}+\delta\tau\theta} = c_{t\tau\theta} c_\alpha \piket{\xi}{m-\rho\abs{\alpha}+t\delta}
.\end{align*}
\end{remark}

\begin{cor}\label{cor:IntroAbleitung}
Für $p \in C^\tau S^m_{\rho,\delta}(\R^n,\R^n)$ ist 
$D_x^\beta D_\xi^\alpha p(x,\xi) \in C^{\tau-\abs{\beta}} S^{m-\rho\abs{\alpha}+\delta\abs{\beta}}_{\rho,\delta}(\R^n,\R^n)$ für Multiindizes $\alpha,\beta\in\N^n_0$ mit $\abs{\beta} < \tau$.
\begin{proof}
Zunächst haben wir
$$\abs{ D_x^\beta D_\xi^\alpha p(x,\xi) } \le \norm{ D_\xi^\alpha p(\hold,\xi)}{C^{\abs{\beta}}(\R^n)} \le 
C_{\alpha\beta} \piket{\xi}{m-\rho\abs{\alpha}+\delta\abs{\beta}}
.$$

Für $\abs{\beta} \le t \le \tau$ ist
\begin{align*}
\norm{ D_x^\beta D_\xi^\alpha p(\hold,\xi) }{C^{t-\abs{\beta}}(\R^n)} 
&\le \norm{ D_x^\beta D_\xi^\alpha p(\hold,\xi) }{C^{\gauss{t}-\abs{\beta}}(\R^n)} + \max_{\abs{\gamma} = \gauss{t}} \hnorm{ D_x^\gamma D_\xi^\alpha p(\hold,\xi) }{t-\gauss{t}} \\
&\le C_\beta \norm{ D_\xi^\alpha p(\hold,\xi) }{C^t(\R^n)} 
\le C_{\alpha\beta} \piket{\xi}{m-\rho\abs{\alpha}+\delta t}
.\end{align*}
Also ist für $0 \le t \le \tau-\abs{\beta}$
$$
\norm{ D_x^\beta D_\xi^\alpha p(\hold,\xi) }{C^{t}(\R^n)} 
\le C_{\alpha\beta} \piket{\xi}{m+\delta\abs{\beta}-\rho\abs{\alpha}+\delta t}
.$$
\end{proof}
\end{cor}

\begin{cor}
Für zwei Symbole $p(x,\xi) \in C^\tau S^m_{\rho,\delta}(\R^n,\R^n), q(x,\xi) \in C^\tau S^{m'}_{\rho',\delta'}(\R^n,\R^n)$ ist
$p(x,\xi) q(x,\xi) \in C^{\tilde{\tau}} S^{m+m'}_{\tilde{\rho},\tilde{\delta}}(\R^n,\R^n)$ mit
$\tilde{\rho} := \min(\rho,\rho'), \; \tilde{\delta} := \max(\delta,\delta'), \; \tilde{\tau} := \min(\tau,\tau')$.
\end{cor}

\begin{lemma}\label{lemma:IntroSymS->S}
Für $u \in \Schwarzi$ und $p \in C^\tau S^m_{\rho,\delta}(\R^n,\R^n)$ ist 
$ \xi \mapsto p(x,\xi) u(\xi) \in \Schwarzi$ für jedes $x \in \R^n$.
\begin{proof}
Seien $l \in \N_0, \alpha \in \N^n_0$ und $\xi \in \R^n$ fest.
Wir erhalten mit der Regel von Leibniz
\begin{align*}
\abs{ \piket{\xi}{l} \partial_\xi^\alpha \left( p(x,\xi) u(\xi) \right) }
&= \abs{ \sum_{\alpha_1+\alpha_2 = \alpha} {\alpha \choose \alpha_1} \piket{\xi}{l} \partial_\xi^{\alpha_1} p(x,\xi) \partial_\xi^{\alpha_2} u(\xi) } \\
&\le \sum_{\alpha_1+\alpha_2 = \alpha}  {\alpha \choose \alpha_1} \piket{\xi}{l} \abs{ \partial_\xi^{\alpha_1} p(x,\xi) } \abs{ \partial_\xi^{\alpha_2} u(\xi) } \\
&\le \sum_{\alpha_1+\alpha_2 = \alpha}  {\alpha \choose \alpha_1} C_{\alpha_1} \piket{\xi}{m-\rho\abs{\alpha_1}+l} \abs{ \partial_\xi^{\alpha_2} u(\xi) } \\
&\le \sum_{\alpha_1+\alpha_2 = \alpha}  {\alpha \choose \alpha_1} C_{\alpha_1} \SchwarzNorm{u}{\max\menge{0, m-\rho\abs{\alpha_1}+\abs{\alpha_2}+l}} \\
&\le C_{\alpha\beta} \SchwarzNorm{u}{\max\menge{0, m+\abs{\alpha}+l}}
.\end{align*}
Mit dem Übergang zum Supremum über alle $\xi \in \R^n$ erhalten wir die Beschränktheit von $\xi \mapsto p(x,\xi) u(\xi)$ in jeder Schwartz-Halbnorm.
\end{proof}
\end{lemma}

\begin{thm}\label{thm:introSchwarzBeschr}
Das Symbol $p(x,\xi) \in C^\tau S^m_{\rho,\delta}(\R^n,\R^n)$, $\tau \in \R_+$, $m \in \R$, $0 \le \delta \le \rho \le 1$ definiert mit
$$P u(x) := \int e^{ix\cdot\xi} p(x,\xi) \hat{u}(\xi) \dslash\xi \textrm{ für alle } u \in \Schwarzi$$
einen beschränkten und stetigen Operator $P : \Schwarzi \rightarrow C^\tau(\R^n)$.
Wir schreiben $p(X,D_x) := \OPSym{p(x,\xi)} := P$, und sagen
dass der Operator $p(X,D_x)$ der Klasse $\OP{C^\tau S^m_{\rho,\delta}}(\R^n, \R^n)$ angehört.
\begin{proof}
Für $u \in \Schwarzi$ ist mit (\ref{equ:SchwarzNorm}) und Theorem \ref{thm:SchwartzFourierNorm} 
$$\piket{\xi}{m} \abs{\hat{u}(\xi)} \le \piket{\xi}{-(n+1)} \SchwarzNorm{\hat{u}}{ m_+ + n + 1 } \le \piket{\xi}{-(n+1)} \SchwarzNorm{u}{m_+ + 2 (n+1)} ,$$
wobei $m_+ := \max\menge{ 0, m }$.
Demnach haben wir
\begin{align*}
\abs{ p(X,D_x) u(x) } 
&\le \int \abs{ p(x,\xi) } \abs{\hat{u}(\xi)} \dslash\xi \\
&\le \abs{p}_{0,0}^{(m)} \int \piket{\xi}{-(n+1)} \dslash\xi \SchwarzNorm{u}{m_+ + 2 (n+1)}
\le C \abs{p}_{0,0}^{(m)} \SchwarzNorm{u}{m_+ + 2 (n+1)} 
.\end{align*}
Insbesondere ist für $\tau < 1$
\begin{align*}
& \hnorm{ p(X,D_x) u }{\tau} \\
&\le 
\sup_{\substack{ x,y \in \R^n \\ x \neq y}} 
\frac{\abs{ \int e^{ix\cdot\xi} \piket{\xi}{-(n+1)} \left( p(x,\xi) - p(y,\xi) \right) \hat{u}(\xi) \dslash\xi}
+ \abs{ \int \left( e^{ix\cdot\xi} - e^{iy\cdot\xi}\right) p(y,\xi) \hat{u}(\xi) \dslash\xi}}
{ \abs{x-y}^\tau} \\
&\le C \SchwarzNorm{u}{m_+ + 2 (n+1)} \sup_{\substack{ x,y \in \R^n \\ x \neq y}} \frac{\abs{ \int e^{ix\cdot\xi} \piket{\xi}{-(n+1)} \left( p(x,\xi) - p(y,\xi) \right) \dslash\xi}}{ \abs{x-y}^\tau} + C \abs{p}_{0,0}^{(m)} \SchwarzNorm{u}{m_+ + 2 (n+1)} \\
&\le C \SchwarzNorm{u}{m_+ + 2 (n+1)} \sup_{\substack{ x,y \in \R^n \\ x \neq y}} \int \piket{\xi}{-(n+1)} \frac{ \abs{ p(x,\xi) - p(y,\xi) } }{ \abs{x-y}^\tau } \dslash\xi + C \abs{p}_{0,0}^{(m)} \SchwarzNorm{u}{m_+ + 2 (n+1)} \\
&\le C \norm{\xi \mapsto \piket{\xi}{-m} \hnorm{x \mapsto p(x,\xi)}{\tau}}{\infty} \SchwarzNorm{u}{m_+ + 2 (n+1)} + C \abs{p}_{0,0}^{(m)} \SchwarzNorm{u}{m_+ + 2 (n+1)} 
.\end{align*}
Für $\tau \ge 1$ ist mit der Produktregel
\begin{align*}
\partial_{x_j} p(X,D_x) u &= \int e^{ix\cdot\xi} \left( i \xi_j p(x,\xi) + \partial_{x_j} p(x,\xi) \right) \hat{u}(\xi) \dslash\xi \\
&= \int e^{ix\cdot\xi} p(x,\xi) \Fourier{x}{\xi}{\partial_{x_j} u(x)} \dslash\xi + \int e^{ix\cdot\xi} \partial_{x_j} p(x,\xi) \hat{u}(\xi) \dslash\xi \\
&= p(X,D_x)(\partial_{x_j} u)(x) + \left(\partial_{x_j} p\right)(X,D_x)u(x)
.\end{align*}
Nach Folgerung \ref{cor:IntroAbleitung} ist 
$\partial_{\xi_j} p(x,\xi) \in C^\tau S^{m-\rho}_{\rho,\delta}(\R^n,\R^n)$ und
$\partial_{x_j} p(x,\xi) \in C^{\tau-1} S^{m+\delta}_{\rho,\delta}(\R^n,\R^n)$ (falls $\tau \ge 1$).
Betrachten wir nun induktiv alle Ableitungen bzgl. $x$ zum Grad $\gauss{\tau}$ und alle Ableitungen bzgl. $\xi$, so erhalten wir die Behauptung.
\end{proof}
\end{thm}
\nomenclature[o]{$\OP{C^\tau S^m_{\rho,\delta}(\R^n,\R^n)}$}{Klasse der Pseudodifferentialoperatoren in $x$-Form mit Hölder-Stetigkeitsgrad $\tau$}
\nomenclature{\PSDO}{abgekürzte Schreibweise für ``Pseudodifferentialoperator''}

\begin{remark}
Für spätere Abbildungscharakterisierungen von Pseudodifferentialoperatoren auf anderen Räumen werden wir stillschweigend annehmen, 
dass ein \PSDO die Eigenschaften einer beschränkten linearen Abbildung bzgl. dieser Räume beibehält.
\end{remark}

\begin{lemma}\label{lemma:introOszi}
Der Operator von $p \in C^\tau S^m_{\rho,\delta}(\R^n,\R^n)$ kann als oszillatorisches Intgeral dargestellt werden durch
$$ p(X,D_x) u(x) = \osint e^{-ix\cdot\eta} p(x,\eta) u(x+y) dy \dslash\eta \textrm{ für alle } u \in \Schwarzi .$$
\begin{proof}
Mit der elementaren Abschätzung $\abs{ \partial_\eta^\alpha \partial_y^\beta \left( p(x,\eta) u(x+y) \right) } \le C_{\alpha,\beta} \piket{\eta}{m}$
folgern wir $(\eta,y) \mapsto p(x,\eta) u(x+y) \in \Oszillatory{0,0}{m}$.
Insbesondere ist $\eta \mapsto \int e^{-iy\cdot\eta} p(x,\eta) u(x+y) dy$ integrierbar,
wir können also Theorem \ref{thm:osziFubini}.\ref{item:osziFubiniC} verwenden und erhalten
$$ \os{\eta}{y}{p(x,\eta) u(x+y)} = \int \left( \int  e^{-iy\cdot\eta} p(x,\eta) u(x+y) dy \right) \dslash\eta 
.$$
Die Translation $y \mapsto y - x$ liefert schließlich
\begin{align*}
\int \left( \int  e^{-iy\cdot\eta} p(x,\eta) u(x+y) dy \right) \dslash\eta
&= \int \left( \int e^{-i(y-x)\cdot\xi} p(x,\eta) u(y) dy \right) \dslash\eta \\
&= \int e^{ix\cdot\eta} p(x,\eta) \hat{u}(\eta) \dslash\eta \\
&= p(X,D_x) u(x)
.\end{align*}
\end{proof}
\end{lemma}

\begin{thm}
Der Operator eines Symbols $p \in C^\tau S^m_{\rho,\delta}(\R^n,\R^n)$ ist ein beschränkter Operator mit der Abbildungseigenschaft
$$ p(X,D_x) : \Beschri \rightarrow C^\tau(\R^n) .$$ 
\begin{proof}
Wir betrachten $a_x(\eta,y) := p(x,\eta) u(x+y)$ für $u \in \Beschri$. Wir haben also $x \mapsto a_x(\eta,y) \in C^\tau(\R^n)$.
Für $\alpha,\beta \in \N^n_0$ ist 
$$\abs{ \partial_{\eta}^\alpha \partial_y^\beta a_x(\eta,y) } \le C_{\alpha\beta} \abs{p}_{\abs{\beta},0}^{(m)} \norm{u}{C^{\abs{\alpha}}(\R^n)} \piket{\eta}{m} .$$
Also ist $a_x \in \Oszillatory{0,0}{m}$ mit $\OszillatoryNorm{a_x}{0,0}{m}{k} \le C \abs{p}_{k,0}^{(m)} \norm{u}{C^k(\R^n)}$ gleichmäßig in $x \in \R^n$.
Für $\tau \in \N_0$ finden wir nach Theorem \ref{thm:osziTransform} zwei Ganzzahlen $l,l' \in \N_0$ mit $2l > m+n$ und $2l' > n$, sodass
$$ p(X,D_x) u = \iint e^{iy\cdot\eta} \piket{y}{-2l'} \piket{D_\eta}{2l'} \left[ \piket{\eta}{-2l} \piket{D_y}{2l} ( p(x,\eta) u(x+y)) \right] dy \dslash\eta .$$
Demnach ist $ \norm{ p(X,D_x) u}{C^\tau(\R^n)} \le C \abs{p}_{2(l+l'),\tau}^{(m)} \norm{u}{C^{2(l+l')+\tau}} $.
Für $\tau \in (0,1)$ ist $(x,x') \mapsto \frac{ a_x(y,\eta) - a_{x'}(y,\eta)}{\abs{x-x'}^\tau}$ wegen der Hölderstetigkeit von $x \mapsto a_x(y,\eta)$ beschränkt. Theorem \ref{thm:osziKonvergenz} liefert nun, dass
\begin{align*}
&\hnorm{ p(X,D_x) u }{\tau} \\
&\le \sup_{\substack{ x,x' \in \R^n \\ x \neq x'}} \frac{ \abs{ \osint e^{-ix\cdot\eta} \left( a_x(y,\eta) - a_{x'}(y,\eta) \right) dy\dslash\eta } + \abs{ \osint \left(e^{ix\cdot\eta} - e^{ix'\cdot\eta} \right) a_{x'}(y,\eta) dy\dslash\eta} }{ \abs{x-x'}^\tau } \\
&\le \abs{ \osint e^{-ix\cdot\eta} \hnorm{ a_.(y,\eta)}{\tau} dy\dslash\eta } + C \abs{p}_{2(l+l')}^{(m)} \norm{u}{C^{2(l+l')}} 
\end{align*}
gleichmäßig in $x \in \R^n$ abgeschätzt werden kann.
\end{proof}
\end{thm}

\begin{satz}\label{intro:satzSymOP}
Die Abbildung 
$ C^\tau S^m_{\rho,\delta}(\R^n,\R^n) \rightarrow \OP{ C^\tau S^m_{\rho,\delta}(\R^n,\R^n) } $
mit 
$$ \OP{ C^\tau S^m_{\rho,\delta}(\R^n,\R^n) } = \menge{ P : C^\infty(\R^n) \rightarrow C^\tau(\R^n) : P \textrm{ linear und beschränkt }}, p \mapsto p(X,D_x)$$
ist eine Bijektion, deren Inverse gegeben ist durch
$$ p(x,\xi) = e^{ix\cdot\xi} p(X,D_x) \left\{ x \mapsto e^{ix\cdot\xi}\right\}(x) \textrm{ für alle } \xi \in \R^n .$$
\begin{proof}
Für $\xi \in \R^n$ ist nach Lemma \ref{lemma:introOszi}
$$e^{-ix\cdot\xi} p(X,D_x) e^{i.\cdot\xi}(x) = 
\osint e^{-iy\cdot(\eta-\xi)} p(x,\eta) dy \dslash\eta =
\osint e^{-iy\cdot\eta'} p(x,\eta'+\xi) dy \dslash\eta' ,$$
wobei wir die Translation $\eta \mapsto \eta + \xi =: \eta'$ verwendet haben.
Mit (\ref{equ:inversionFormula}) folgt aber bereits
$ \osint e^{-iy\cdot\eta} p(x,\eta+\xi) \dslash\eta dy = p(x,\xi)$.
\end{proof}
\end{satz}

\begin{cor}\label{cor:introIdentisch}
Seien $p,q \in C^\tau S^m_{\rho,\delta}(\R^n,\R^n)$.
Die Bedingung $p(X,D_x) u = q(X,D_x) u$ für alle $u \in \Schwarzi$ impliziert bereits, dass
beide Symbole identisch sind.
\begin{proof}
Da $x \mapsto e^{-ix\cdot\xi} \in \Beschri$ können wir nach Satz \ref{intro:satzSymOP}
bereits aus der Bedingung $p(X,D_x) u(x) = q(X,D_x) u(x)$ für alle $x \in \R^n$ und jedes $u \in \Beschri$ folgern, dass $p = q$ gelten muss.
Wir müssen also lediglich die Gleichheit auf $\Beschri$ überprüfen.
Sei dazu ein $u \in \Beschri$ gegeben und setze $u_\epsilon(x) := \chi(\epsilon x) u(x)$ für $x \in \R^n, \epsilon > 0$ und $\chi \in \Schwarzi$ mit $\chi(0) = 1$.
Dann ist $u_\epsilon \in \Schwarzi$ und $\partial_x^\alpha u_\epsilon(x) \rightarrow \partial_x^\alpha u(x)$ mit $\epsilon \rightarrow 0$ punktweise für alle $x \in \R^n$ und für jedes $\alpha \in \N^n_0$.
Setzen wir also $a_{x,\epsilon}(y,\eta) := p(x,\eta) u_\epsilon(x+y)$, so erhalten wir
$a_{x,\epsilon} \in \Oszillatory{0,0}{m}$
analog zu Beweis von Satz \ref{intro:satzSymOP}.
Die Folge $\menge{ u_\epsilon}_{\epsilon\in(0,1)}$ ist beschränkt in $\Beschri$, 
demnach ist $\menge{a_{x,\epsilon}}_{\epsilon\in(0,1)}$ beschränkt in $\Oszillatory{0,0}{m}$.
Also konvergiert für alle $\alpha,\beta \in \N^n_0$
$$\partial_x^\alpha \partial_\xi^\beta a_{x,\epsilon}(y,\eta) \rightarrow \partial_x^\alpha \partial_\xi^\beta a_x(y,\eta) \textrm{ für } \epsilon \rightarrow 0 \textrm{ für alle } x, \eta \in \R^n$$
gleichmäßig auf jeder kompakten Teilmenge des $\R^n$.
Theorem \ref{thm:osziKonvergenz} liefert schließlich
$$ p(X,D_x) u_\epsilon(x) = \osint e^{-iy\cdot\eta} p(x,\eta) u_\epsilon(x+y) dy \dslash\eta \rightarrow p(X,D_x) u(x) \textrm{ für } \epsilon \rightarrow 0 .$$
Analog erhalten wir für $q$ die Aussage
$$ q(X,D_x) u_\epsilon(x) = \osint e^{-iy\cdot\eta} q(x,\eta) u_\epsilon(x+y) dy \dslash\eta \rightarrow q(X,D_x) u(x) \textrm{ für } \epsilon \rightarrow 0 .$$
Nach Voraussetzung ist aber $p(X,D_x) u_\epsilon(x) = q(X,D_x) u_\epsilon(x)$, also auch $p(X,D_x) u(x) = q(X,D_x) u(x)$ für alle $x \in \R^n$.
\end{proof}
\end{cor}

\begin{definition}
Wir definieren für $0 \le \delta \le \rho \le 1$ die Klasse $S^m_{\rho,\delta}(\R^n\times\R^n) := \bigcap_{\tau\in\R_+} C^\tau S^m_{\rho,\delta}(\R^n,\R^n)$.
Offensichtlich ist ein Symbol $p \in S^m_{\rho,\delta}(\R^n\times\R^n)$ eine glatte Funktion $p : \R^n \times \R^n \rightarrow \C$ für die es zu den Multiindizes $\alpha,\beta\in\N^n_0$ eine Konstante $C_{\alpha\beta} > 0$ gibt, sodass
$$ \abs{ D_\xi^\alpha D_x^\beta p(x,\xi) } \le C_{\alpha\beta} \piket{\xi}{m-\rho\abs{\alpha}+\delta\abs{\beta}} .$$
Die Klasse  $S^m_{\rho,\delta}(\R^n\times\R^n)$ wird mit den Halbnormen 
$$ \abs{p}_{l,\infty}^{(m)} := \max_{\abs{\alpha+\beta}\le l} \sup_{x,\xi \in \R^n} \left( \abs{ \partial_\xi^\alpha \partial_x^\beta p(x,\xi) } \piket{\xi}{-(m-\rho\abs{\alpha}+\delta\abs{\beta})} \right) $$
zum Fréchet-Raum.
Falls $p(x,\xi) \in S^m_{\rho,\delta}(\R^n\times\R^n)$ ein von $x$ unabhängiges Symbol ist, schreiben wir auch $p(x,\xi) =: p(\xi) \in S^m_{\rho,\delta}(\R^n)$.
\end{definition}
\nomenclature[s]{$S^m_{\rho,\delta}(\R^n,\R^n)$}{Klasse glatter Symbole in $x$-Form}
\nomenclature[n]{$\abs{\hold}_{l,\infty}^{(m)}$}{$l$.te Halbnorm der Klasse $S^m_{\rho,\delta}(\R^n,\R^n)$}

\begin{example}
Das Symbol $\xi \mapsto \piket{\xi}{m}$ gehört nach Lemma \ref{lemma:piketTech} der Klasse $S^m_{1,0}(\R^n)$ für $m \in \R$ an.
\end{example}

\begin{definition}
Für $0 \le \delta \le \rho \le 1$ beliebig setze 
$$C^\tau S^{-\infty}(\R^n, \R^n) := \bigcap_{m\in\R} C^\tau S_{\rho,\delta}^m(\R^n, \R^n) .$$
\end{definition}
\nomenclature[s]{$C^\tau S^{-\infty}(\R^n, \R^n)$}{Klasse beschränkter Symbole in $x$-Form mit Hölder-Stetigkeitsgrad $\tau$}

\begin{remark}
$C^\tau S^{-\infty}(\R^n, \R^n)$ ist eine von
$0 \le \delta \le \rho \le 1$ unabhängige Klasse, da für ein $p \in C^\tau S^{-\infty}(\R^n, \R^n)$ die Abschätzungen
$$
\abs{ D_\xi^\alpha p(x,\xi) } \le C_\alpha \piket{\xi}{m} , \quad
\norm{ D_\xi^\alpha p(\hold,\xi) }{C_*^\tau(\R^n)} \le C_\alpha \piket{\xi}{m} ,\quad
\norm{D_\xi^\alpha p(\hold,\xi) }{C^\tau(\R^n)} \le C_\alpha \piket{\xi}{m}
$$
für alle $m \in \R$ 
mit einer von $\rho,\delta$ unabhängigen Konstante $C_\alpha > 0$ gelten müssen.
\end{remark}

\begin{lemma}\label{lemma:smoothing}
Sei $p \in C^\tau S^m_{\rho,\delta}(\R^n,\R^n), f \in \Schwarzi$. 
Dann ist 
$$p_j(x,\xi) := p(x,\xi)\lilwood{j}(\xi) \in C^{\tau} S^{m}_{\rho,\delta}(\R^n,\R^n) \textrm{ für jedes } j \in \N_0 ,$$
und
$$p(X,D_x) f(x) = \sum^\infty_{j=0} p_j(X,D_x)f(x) \textrm{ für alle } x,\xi \in \R^n .$$
\begin{proof}
Da $p(X,D_x)$ stetig ist, folgt sofort
$$ \sum_{j=0}^{N} p_j(X,D_x) f(x) = \sum_{j=0}^{N} p(X,D_x) \lilwood{j}(D_x) f(x) \rightarrow p(X,D_x) f(x) \textrm{ für } N \rightarrow \infty .$$
\end{proof}
\end{lemma}

\begin{definition}[mikrolokalisierbare Banachräume]
\label{def:microlocal}
Eine Familie von Banachräumen $\menge{X^\tau}_{\tau \in \Sigma^{X^.} }$ heißt \emph{mikrolokalisierbar}, falls
\begin{enumerate}[(a)]
\item es stetige kanonische Einbettungen $\iota : \Schwarzi \hookrightarrow X^\tau$ und $\iota' : X^\tau \hookrightarrow \SchwarziDual$ gibt, die unabhängig von $\tau \in \Sigma^{X^.}$ sind \label{def:microlocal:Schwarzi},
\item der ordnungserniedrigende Operator $\piket{D_x}{t} : X^{\tau + t} \rightarrow X^\tau$ eine beschränkte lineare Abbildung für alle $\tau, \tau + t \in \Sigma^{X^.}$ \label{def:microlocal:D} ist, und
\item für $p \in S^0_{1,\delta}(\R^n\times\R^n)$ mit $0 \le \delta < 1$ bereits $p(X,D_x) : X^0 \rightarrow X^0$ gilt.
\end{enumerate}
Die Menge $\Sigma^{X^.}$ ist ein Intervall der Form $[\sigma_0, \infty)$ oder $(\sigma_0,\infty)$ mit $\sigma_0 \in \bar{\R}$.
\end{definition}
\nomenclature{$\menge{X^\tau}_{\tau \in \Sigma^{X^.} }$}{Mikrolokalisierbare Familie mit Index $\Sigma^{X^.}$}

\begin{lemma}\label{lemma:micro1}
Aus der Definition der mikrolokalisierbaren Banachräume folgt bereits, dass

\begin{enumerate}[(a)]
\item
\begin{equation}
X^\tau = \menge{ f \in \SchwarziDual : \piket{D_x}{\tau} \in X^0},
\end{equation}
\item
und es eine stetige kanonische Einbettung
\begin{equation}\label{mikrolokal:einbettung}
\psi_{\tau, t} : X^\tau \hookrightarrow X^t
\end{equation}
für alle $t \in \Sigma^{X^.}$ mit $ t \le \tau$ gibt.
\end{enumerate}
\begin{proof}
\begin{enumerate}[(a)]
\item
Sei $f \in X^\tau$.
Nach Definition \ref{def:microlocal}.\ref{def:microlocal:D} ist
$\piket{D_x}{\tau} : X^{\tau} \rightarrow X^0$, also offensichtlich
$\piket{D_x}{\tau} f \in X^0 \hookrightarrow \SchwarziDual$.
Andererseits ist für 
$g \in \menge{ f \in \SchwarziDual : \piket{D_x}{\tau} f \in X^0}$  
bereits
$\piket{D_x}{\tau} g \in X^0$, 
also $g =  \piket{D_x}{-\tau} \piket{D_x}{\tau} g \in X^\tau$ 
nach Definition \ref{def:microlocal}.\ref{def:microlocal:D}.
\item
Wegen $t \le \tau$ gilt $\piket{\xi}{t - \tau} \in S^{t-\tau}_{1,0}(\R^n) \subset S^0_{1,0}(\R^n)$.
Für $f \in X^\tau$ ist wegen
$$
\piket{D_x}{t} f = \underbrace{\piket{D_x}{t - \tau}}_{\in \OP{S^0_{1,0}(\R^n)}} \underbrace{\piket{D_x}{\tau} f}_{\in X^0} \in X^0
$$
$f$ bereits in $X^t$, also $X^\tau \subset X^t$.
Die Stetigkeit der Abbildung $\psi_{\tau, t}$ folgt nun aus der Eigenschaft der Teilraumtopologie 
$X^s \subset \SchwarziDual$ 
welche durch die kanonische Einbettung 
$\iota' : X^s \hookrightarrow \SchwarziDual$
aus Definition \ref{def:microlocal}.\ref{def:microlocal:Schwarzi}
induziert wird:
Sei dazu $U \subset X^t$ offen, so existiert ein $U' \subset \SchwarziDual$ offen mit $U = X^t \cap U'$. 
Nun ist wegen $\iota_\tau : X^\tau \hookrightarrow \SchwarziDual$ auch 
$X^\tau \cap U = X^\tau \cap X^t \cap U' = X^\tau \cap U' \subset X^\tau$
offen.
\end{enumerate}
\end{proof}
\end{lemma}

\begin{lemma} \label{beschr:S_1,delta}
Sei $p(x,\xi) \in S^m_{1,\delta}(\R^n\times\R^n)$ mit $m \in \R$ und $\menge{X^\tau}_{\tau \in \Sigma^{X^.} }$ eine Familie von mikrolokalisierbarer Banachräume.
Dann hat der Operator des Symbols bzgl. dieser Räume die Abbildungseigenschaft
\begin{equation}
p(X,D_x) : X^{s+m} \rightarrow X^s \textrm{ für alle } s,s+m \in \menge{X^\tau}_{\tau \in \Sigma^{X^.} }.
\end{equation}
\begin{proof}
Wir haben $\piket{\xi}{s} p(x,\xi) \piket{\xi}{-s-m} \in S^0_{1,\delta}(\R^n\times\R^n)$, und damit liefert Lemma \ref{lemma:micro1} für $f \in X^{s+m}$ 
$$
\piket{D_x}{s} \left( p(X,D_x) f \right) = 
\underbrace{\piket{D_x}{s} p(X,D_x) \piket{D_x}{-s-m}}_{\in \OP{ S^0_{1,\delta}(\R^n\times\R^n)}} \underbrace{\piket{D_x}{s+m} f}_{\in X^0} \in X^0
.$$
Insgesamt ist also $p(X,D_x) f \in X^s$.
\end{proof}
\end{lemma}

\begin{example}
\cite[\S 1.1]{Nonlinear} führt die Mengen 
$$\menge{ H^s_q(\R^n) : s \in \R} \textrm{ und } \menge{C^s_*(\R^n) : s \in (0,\infty)}$$ 
als mikrolokalisierbare Familien für Symbolklassen mit $\delta = 0$ auf. Andererseits lassen sich mit der Anmerkung aus \cite[Corollary 2.1.B]{Nonlinear} diese Familien auch für unsere Definition mit $0 \le \delta < 1$ verwenden.
\end{example}

\subsection{Adjungierte Operatoren}

\begin{definition}
Falls es zu einem Operator $P : \Schwarzi \rightarrow \BeschrZ$ einen Operator $P^* : \Schwarzi \rightarrow \BeschrZ$ gibt,
der für alle $u,v \in \Schwarzi$ die Gleichung
$ \skp{ P u}{v}{L^2(\R^n)} = \skp{u}{P^* v}{L^2(\R^n)}$ hält, dann nennen wir $P^*$ den zu $P$ formal adjungierten Operator.
\end{definition}
\nomenclature{$P(D_x,X)^*$}{formal adjungierter Operator des \PSDOS $p(X,D_x)$ in $x$-Form}

\begin{definition}
Ein Symbol $p \in C^\tau S^m_{\rho,\delta}(\R^n,\R^n)$ definiert uns den Operator
$$ p(D_x, X) u(x) = \iint e^{i(x-y)\cdot\xi} p(y,\xi) v(y) dy \dslash\xi .$$
Wir nennen den Pseudodifferentialoperator $p(D_x,X)$ einen Operator in $y$-Form.
\end{definition}

\begin{satz}\label{satz:adjunPSDO}
Für $u,v \in \Schwarzi$ induziert das Symbol $p(x,\xi)^* := \overline{ p(x,\xi) }$ den Operator $p(D_x,X)^*$ in $y$-Form, der die Eigenschaft
$$ \skp{ u }{ p(D_x,X)^* v }{L^2(\R^n)} = \skp{ p(X,D_x) u }{ v }{L^2(\R^n)} $$
erfüllt. Der Operator $p(D_x,X)^* $ ist also der formal adjungierte Operator von $p(X,D_x)$. 
\begin{proof}
Zunächst ist $(x,\xi) \mapsto e^{ix\cdot\xi} p(x,\xi) \hat{u}(\xi) \overline{v(x)} \in L^1(\R^n\times\R^n)$ da
$\hat{u}, v \in \Schwarzi$ nach Theorem \ref{thm:SchwartzFourierNorm}.
Außerdem ist $(x,y) \mapsto e^{i(x-y)\xi} p(x,\xi) \overline{v(x)} u(y) \in L^1(\R^n\times\R^n)$.
Mit dem Satz von Fubini erhalten wir 
\begin{align*}
\skp{ p(X,D_x) u }{v}{L^2(\R^n)} &= 
\int \left( \int e^{ix\cdot\xi} p(x,\xi) \hat{u}(\xi) \dslash\xi \right) \overline{v(x)} dx \\
&= \int \left( \int e^{ix\cdot\xi} p(x,\xi) \overline{v(x)} dx \right) \left( \int e^{-iy\cdot\xi} u(y) dy \right) \dslash\xi \\
&= \int \left( \iint e^{i(x-y)\cdot\xi} p(x,\xi) \overline{v(x)} u(y) dx dy \right) \dslash\xi \\
&= \int u(y) \left( \iint e^{i(x-y)\cdot\xi} p(x,\xi) \overline{v(x)} dx \dslash\xi \right) dy \\
&= \int u(y) \overline{ \iint e^{i(y-x) \cdot\xi} \overline{ p(x,\xi) } v(x) dx \dslash\xi } dy\\
&= \skp{ u }{p(D_x,X)^* v}{L^2(\R^n)}
,\end{align*}
wobei $(y,\xi) \mapsto e^{-iy\cdot\xi} u(y) \int e^{ix\cdot\xi} p(x,\xi) \overline{v(x)} dx \in L^1(\R^n\times\R^n)$, da $p(D_x,X) v \in L^\infty(\R^n)$.
\end{proof}
\end{satz}

\subsection{Operatoren mit Doppelsymbolen}

\begin{definition}
Eine Funktion $p(x,\xi,x',\xi') : \R^n \times \R^n \times \R^n \times \R^n \rightarrow \C$ gehört der Klasse $C^\tau S_{\rho,\delta}^{m,m'}(\R^n,\R^n\times\R^n\times\R^n)$ an ($-\infty < m, m' < \infty, 0 \le \delta \le \rho \le 1, \delta < 1, \tau \in \R_+ )$, 
falls es für alle Multiindizes $\alpha,\alpha',\beta'\in\N^n_0$ eine Konstante $C_{\alpha,\alpha',\beta'} >0$ gibt, die beide Ungleichungen
\begin{align*}
\abs{ \partial_\xi^\alpha \partial_{\xi'}^{\alpha'} \partial_{x'}^{\beta'} p(x,\xi,x',\xi')} 
&\le C_{\alpha,\alpha',\beta'} \piket{\xi}{m-\rho\abs{\alpha}} \piket{\xi;\xi'}{\delta\abs{\beta'}} \piket{\xi'}{m'-\rho\abs{\alpha'}}, \\
\norm{ \partial_\xi^\alpha \partial_{\xi'}^{\alpha'} \partial_{x'}^{\beta'} p(\hold,\xi,x',\xi')}{C^\tau(\R^n)} 
&\le C_{\alpha,\alpha',\beta'} \piket{\xi}{m+\delta\tau-\rho\abs{\alpha}} \piket{\xi;\xi'}{\delta\abs{\beta'}} \piket{\xi'}{m'-\rho\abs{\alpha'}} 
\end{align*}
erfüllt, 
wobei $\piketi{\xi;\xi'} := \sqrt{1 + \abs{\xi}^2 + \abs{\xi'}^2}$.
Ein Symbol der Klasse $C^\tau S_{\rho,\delta}^{m,m'}(\R^n,\R^n\times\R^n\times\R^n)$ nennen wir ein {\em Doppelsymbol}.

Für $p(x,\xi,x',\xi') \in C^\tau S_{\rho,\delta}^{m,m'}(\R^n,\R^n\times\R^n\times\R^n)$ sei die Familie von Seminormen \\
$\menge{ \abs{p}^{(m,m')}_{l,\tau} }_{l\in\N}$
gegeben durch
\begin{align*}
&\abs{p}^{(m,m')}_{l,\tau} \\
&= \max_{\abs{\alpha+\alpha'+\beta'}\le l } \sup_{x,\xi,x',\xi'\in\R^n} 
\left( \abs{\partial_\xi^\alpha \partial_{\xi'}^{\alpha'} \partial_{x'}^{\beta'} p(x,\xi,x',\xi')} + 
\norm{\partial_\xi^\alpha \partial_{\xi'}^{\alpha'} \partial_{x'}^{\beta'} p(\hold,\xi,x',\xi')}{C^\tau(\R^n)} \piket{\xi}{-\delta\tau} \right) \\
&\quad \piket{\xi}{-(m-\rho\abs{\alpha})} \piket{\xi;\xi'}{-\delta\abs{\beta'}} \piket{\xi'}{-(m'-\rho\abs{\alpha})}
.\end{align*}
Diese Familie induziert den Fréchet-Raum $C^\tau S_{\rho,\delta}^{m,m'}(\R^n,\R^n\times\R^n\times\R^n)$.
Insbesondere ist für $p(x,\xi,x',\xi') \in C^\tau S_{\rho,\delta}^{m,m'}(\R^n,\R^n\times\R^n\times\R^n)$
\begin{align*}
& \abs{ \partial_\xi^\alpha \partial_{\xi'}^{\alpha'} \partial_{x'}^{\beta'} p(x,\xi,x',\xi')} + \norm{\partial_\xi^\alpha \partial_{\xi'}^{\alpha'} \partial_{x'}^{\beta'} p(\hold,\xi,x',\xi')}{C^\tau(\R^n)} \piket{\xi}{-\delta\tau} \\
&\le \abs{p}^{(m,m')}_l \piket{\xi}{m-\rho\abs{\alpha}} \piket{\xi;\xi'}{\delta\abs{\beta'}} \piket{\xi'}{m'-\rho\abs{\alpha}}
.\end{align*}

\end{definition}
\nomenclature[s]{$C^\tau S_{\rho,\delta}^{m,m'}(\R^n,\R^n\times\R^n\times\R^n)$}{Klasse der Doppelsymbole mit Hölder-Stetigkeitsgrad $\tau$}

\begin{definition}
Für ein Doppelsymbol $p(x,\xi,x',\xi') \in C^\tau S^{m,m'}_{\rho,\delta}(\R^n,\R^n\times\R^n\times\R^n)$ definieren wir den Operator $P := p(X,D_x,X',D_{x'})$ durch
$$ Pu(x) = \osint e^{i(y\cdot\eta + y'\cdot\eta')} p(x,\xi,x+y,\xi') u(x+y+y') d(y,y') \dslash(\xi,\xi')$$
für $u \in \Beschri$.
Wir sagen für einen Operator dieser Art, dass $P \in \OP{C^\tau S^{m,m'}_{\rho,\delta}(\R^n,\R^n\times\R^n\times\R^n)}$.
Hierbei fassen wir das Integral $\int d(x,x')$ als Doppelintegral $\int_{\R^n} \int_{\R^n} dx dx'$ auf, und schreiben $\dslash(\xi,\xi')$ statt $\frac{1}{(2\pi)^{2n}} d(\xi,\xi')$.
\end{definition}
\nomenclature[o]{$\OP{C^\tau S^{m,m'}_{\rho,\delta}(\R^n,\R^n\times\R^n\times\R^n)}$}{Klasse der Operatoren von Doppelsymbolen mit Hölder-Stetigkeitsgrad $\tau$}

\begin{lemma}
Die obige Definition ist für alle $p \in C^\tau S^{m,m'}_{\rho,\delta}(\R^n,\R^n\times\R^n\times\R^n)$ wohldefiniert.
\begin{proof}
Wir zeigen, dass 
$$((\xi,\xi'),(y,y')) \mapsto p(x,\xi,x+y,\xi') u(x+y+y') \in \Oszillatorye{\delta,0}{m_+ + m_+'}{\R^{2n}\times\R^{2n}} ,$$
wobei $m_+ := \max\menge{0,m}, m'_+ := \max\menge{0,m'}, u\in\Schwarzi$ und $x\in\R^n, u \in \Beschri$ fest gewählt sind 
- dann folgt mit der Definition des oszillatorischen Integrals die Wohldefiniertheit.
Zunächst ist $\piketi{\xi} \le \piketi{\xi,\xi'}$ und $\piketi{\xi'} \le \piketi{\xi,\xi'}$ für beliebiges $\xi,\xi' \in \R^n$.
Für ein fest gewähltes $x\in\R^n$ und $\alpha,\alpha',\beta'\in\N^n_0$ haben wir
\begin{align*}
\abs{ \partial_\xi^\alpha \partial_{\xi'}^{\alpha'} \partial_{x'}^{\beta'} p(x,\xi,x+y,\xi') u(x+y+y') } 
&\le C_{\alpha,\alpha',\beta',u} \piket{\xi}{m-\rho\abs{\alpha}} \piket{\xi;\xi'}{\delta\abs{\beta'}} \piket{\xi'}{m'-\rho\abs{\alpha}}  \\
&\le C_{\alpha,\alpha',\beta',u} \piket{\xi,\xi'}{m_+ + m'_+ + \delta\abs{\beta'}}
.\end{align*}
\end{proof}
\end{lemma}

\begin{definition}
\begin{enumerate}[(a)]
\item Wir definieren die Klasse der glatten Doppelsymbole durch 
$$S^m_{\rho,\delta}(\R^n\times\R^n\times\R^n\times\R^n) := \bigcap_{\tau > 0} C^\tau S^m_{\rho,\delta}(\R^n,\R^n\times\R^n\times\R^n) .$$
\item
Für $0 \le \delta \le \rho < 1, \delta < 1$ beliebig setze
$$C^\tau S^{-\infty}(\R^n,\R^n\times\R^n\times\R^n)  := \bigcap_{m\in\R} C^\tau S_{\rho,\delta}^m(\R^n,\R^n\times\R^n\times\R^n) .$$
\end{enumerate}
\end{definition}
\nomenclature[s]{$C^\tau S^{-\infty}(\R^n,\R^n\times\R^n\times\R^n)$}{Klasse beschränkter Hölder-stetiger Doppelsymbole}
\nomenclature[s]{$S^m_{\rho,\delta}(\R^n\times\R^n\times\R^n\times\R^n)$}{Klasse glatter Doppelsymbole}

\begin{remark}
$C^\tau S^{-\infty}(\R^n,\R^n\times\R^n\times\R^n)$ ist eine von $\rho$ und $\delta$ unabhängige Klasse, 
da für ein $p \in C^\tau S^{-\infty}(\R^n,\R^n\times\R^n\times\R^n)$ die Abschätzungen
\begin{align*}
\abs{ \partial_\xi^\alpha \partial_{\xi'}^{\alpha'} \partial_{x'}^{\beta'} p(x,\xi,x',\xi')} &\le C_{\alpha,\alpha',\beta'} \piket{\xi}{m} \piket{\xi'}{m'} ,\\
\norm{ \partial_\xi^\alpha \partial_{\xi'}^{\alpha'} \partial_x^\beta \partial_{x'}^{\beta'} p(x,\xi,x',\xi')}{C^\tau(\R^n)} &\le C_{\alpha,\alpha',\beta'}  \piket{\xi}{m} \piket{\xi'}{m'}
\end{align*}
für alle $m,m' \in \R$
mit einer von $\rho,\delta$ unabhängigen Konstante $C_{\alpha\beta} > 0$ gelten müssen. 
Wir beachten, dass $\piketi{\xi;\xi'}$ durch eine Konstante $C>0$ mit $\piketi{\xi;\xi'} \le C \piketi{\xi} \piketi{\xi'}$ abgeschätzt werden kann.
\end{remark}

\begin{lemma}\label{lemma:doubleLimesPermut}
Sei $P = p(X,D_x,X',D_{x'}) \in \OP{C^\tau S^{m,m'}_{\rho,\delta}(\R^n,\R^n\times\R^n\times\R^n)}$ mit $\delta < 1$.
Dann ist für alle $u \in \Schwarzi$
$$ Pu(x) = 
\int e^{ix\cdot\xi} \left(
\int e^{-ix'\cdot\xi} \left(
\int e^{ix'\cdot\xi'} \left(
\int e^{-iz\cdot\xi'} p(x,\xi,x',\xi')u(z)dz \right)
\dslash\xi' \right) dx' \right) \dslash\xi
.$$
Hierbei verstehen wir obigen Ausdruck als eine Abfolge iterierter Integrale.
\begin{proof}
Sei $\chi(\xi,\xi',y,y') \in \Schwarz{\R^n\times\R^n\times\R^n\times\R^n}$ mit $\chi(0,0,0,0) = 1$ und setze
$$\chi_\epsilon(\xi,\xi',y,y') := \chi(\epsilon\xi,\epsilon\xi',\epsilon y,\epsilon y') .$$
Dann ist nach Definition des oszillatorischen Integrals
\begin{align*}
& p(X,D_x,X',D_{x'}) u(x) \\
&= \lim_{\epsilon\rightarrow 0} \iint e^{-i(y\cdot\xi+y'\cdot\xi')} \chi_\epsilon(\xi,\xi',y,y') p(x,\xi,x+y,\xi') u(x+y+y') d(y,y') \dslash(\xi,\xi')
.\end{align*}
Für $x' := x+y, z := x+y+y'$ erhalten wir
$$ -i(y\cdot\xi + y'\cdot\xi') = -i(x'-x)\cdot\xi - i(z-x')\cdot\xi' \textrm{ und } \chi_\epsilon(\xi,\xi',y,y') = \chi_\epsilon(\xi,\xi',x'-x,z-x') .$$
Wenn wir
$ p_{\epsilon,u}(x,\xi,x',\xi',z) := \chi_\epsilon(\xi,\xi',x'-x,z-x') p(x,\xi,x',\xi') u(z)$
setzen, liefern die Transformationen $y\mapsto x+y = x'$ und $y' \mapsto x+y+y' = z$
\begin{equation}\label{equ:doubleIntegral}
Pu(x) = \lim_{\epsilon\rightarrow 0} 
\int e^{ix\cdot\xi} \left(
\int e^{-ix'\cdot\xi} \left(
\int e^{ix'\cdot\xi'} \left(
\int e^{-iz\cdot\xi'} p_{\epsilon,u}(x,\xi,x',\xi',z) dz \right)
\dslash\xi' \right) dx' \right) \dslash\xi
.\end{equation}
Wir werden für jedes Integral die Bedingungen für Lebesgues Satz der dominanten Konvergenz beweisen.
Betrachte dazu die Integrale einzeln
\begin{align*}
Pu(x) &= 
\lim_{\epsilon\rightarrow 0} 
\int e^{ix\cdot\xi} \left(
\int e^{-ix'\cdot\xi} \left(
\int e^{ix'\cdot\xi'}
a_{\epsilon,u}(x,\xi,x',\xi')
\dslash\xi' \right) dx' \right) \dslash\xi \\
&= \lim_{\epsilon\rightarrow 0} 
\int e^{ix\cdot\xi} \left(
\int e^{-ix'\cdot\xi} b_{\epsilon,u}(x,\xi,x') dx'  
\right) \dslash\xi = \\
&= \lim_{\epsilon\rightarrow 0} \int e^{ix\cdot\xi} c_{\epsilon,u}(x,\xi) \dslash\xi
,\end{align*}
wobei
\begin{align*}
a_{\epsilon,u}(x,\xi,x',\xi') :=& \int e^{-iz\cdot\xi'} p_{\epsilon,u}(x,\xi,x',\xi',z) dz, \\
b_{\epsilon,u}(x,\xi,x') :=& \int e^{ix'\cdot\xi'} a_{\epsilon,u}(x,\xi,x',\xi') \dslash\xi' \textrm{ und }\\
c_{\epsilon,u}(x,\xi) :=& \int e^{-ix'\cdot\xi} b_{\epsilon,u}(x,\xi,x') dx'
.\end{align*}
Wählen wir zwei Zahlen $l,l'\in\N$ und eine gerade Zahl $n_0 \in 2\N$ mit
$$ -2l(1-\delta)+m < -n,\; -2l'+2l+m' < -n \textrm{ und } n_0 > n ,$$
dann können wir $Pu(x)$ nach der Formel (\ref{equ:doubleIntegral}) berechnen.
Betrachte dazu die Multiindizes $\alpha,\beta,\beta' \in \N^n_0$ mit $\abs{\alpha} \le n_0, \abs{\beta} \le 2l, \abs{\beta'} \le 2l$.
Für $ a_{\epsilon,u}$ folgt wegen $u \in \Schwarzi$, dass
$$ \lim_{\epsilon\rightarrow 0} a_{\epsilon,u}(x,\xi,x',\xi') = \int e^{-iz\cdot\xi'} p(x,\xi,x',\xi') u(z) dz .$$
Wir formen mit der Leibniz-Regel um:
\begin{align*}
& D_{\xi'}^\alpha D_{x'}^\beta a_{\epsilon,u}(x,\xi,x',\xi') \\
&= \sum_{\alpha_1+\alpha_2=\alpha} {\alpha \choose \alpha_1 } \int D_{\xi'}^{\alpha_1} e^{-iz\cdot\xi'} D_{\xi'}^{\alpha_2} D_{x'}^{\beta} p_{\epsilon,u}(x,\xi,x',\xi',z) dz \\
&= \sum_{\alpha_1+\alpha_2=\alpha} {\alpha \choose \alpha_1 } \int e^{-iz\cdot\xi'} (-z)^{\alpha_1} D_{\xi'}^{\alpha_2} D_{x'}^{\beta} p_{\epsilon,u}(x,\xi,x',\xi',z) dz \\
&= \piket{\xi'}{-2l'} \sum_{\alpha_1+\alpha_2=\alpha} {\alpha \choose \alpha_1 } \int e^{-iz\cdot\xi'} \piket{D_z}{2l'} \left(  (-z)^{\alpha_1} D_{\xi'}^{\alpha_2} D_{x'}^{\beta} p_{\epsilon,u}(x,\xi,x',\xi',z)  \right) dz 
,\end{align*}
wobei wir die Identität $\piket{D_y}{2l'} e^{-iy\cdot\xi} = \piket{\xi'}{2l'} e^{-iy\cdot\xi}$ ausnutzen und $2l'$-mal partiell integriert haben.
Unter Verwendung des Lemmas \ref{lemma:osziChi} ist folgende Abschätzung für $\mu \in \N^n_0$ und ein $C_{u,l'}>0$ abhängig von $u$ gültig:
\begin{align*}
& \abs{ D_z^\mu \left(  (-z)^{\alpha_1} D_{\xi'}^{\alpha_2} D_{x'}^{\beta} p_{\epsilon,u}(x,\xi,x',\xi',z)  \right) } \\
&= \abs{ \sum_{\mu_1+\mu_2=\mu} {\mu \choose \mu_1 } D_z^{\mu_1}  (-z)^{\alpha_1} D_z^{\mu_2} u(z) } 
\abs{ D_{\xi'}^{\alpha_2} D_{x'}^{\beta} \left( \chi_\epsilon(\xi,\xi',x'-x,z-x') p(x,\xi,x',\xi') \right) }  \\
&\le C_{u,l'} C_{\alpha,\beta} \piket{\xi}{m} \piket{\xi;\xi'}{\delta\abs{\beta}} \piket{\xi'}{m'} 
\le C_{u,l'} C_{\alpha,\beta} \piket{\xi}{m+\delta\abs{\beta}} \piket{\xi'}{m'+\delta\abs{\beta}}
.\end{align*}
Also gibt es eine Konstante $C_{l,l',n_0,u}>0$ unabhängig von $0<\epsilon<1$, sodass
\begin{equation}\label{equ:doubleAepsilon}
\abs{  D_{\xi'}^\alpha D_{x'}^\beta a_{\epsilon,u}(x,\xi,x',\xi') } \le 
C_{l,l',n_0,u} \piket{\xi}{m+\delta\abs{\beta}} \piket{\xi'}{m'+\delta\abs{\beta}-2l'}
.\end{equation}
Da $m'+\delta\abs{\beta}-2l' \le m'+2l\delta-2l' < -n$, folgt $\xi' \mapsto \piket{\xi'}{m'+\delta\abs{\beta}-2l'} \in L^1(\R^n)$.
Insgesamt folgt also
$$ \lim_{\epsilon\rightarrow 0} b_{\epsilon,u}(x,\xi,x') = \int e^{-ix'\cdot\xi'} \lim_{\epsilon\rightarrow 0} a_{\epsilon,u}(x,\xi,x',\xi') \dslash\xi' .$$

Betrachte nun $D_{x'}^{\beta'} b_{\epsilon,u}(x,\xi,x')$:
\begin{align*}
D_{x'}^{\beta'} b_{\epsilon,u}(x,\xi,x') &= \sum_{\beta_1'+\beta_2' = \beta'} { \beta' \choose \beta_1' } 
\int D_{x'}^{\beta_1'} e^{ix'\cdot\xi'} D_{x'}^{\beta_2'}  a_{\epsilon,u}(x,\xi,x',\xi') \dslash\xi' \\
&= \sum_{\beta_1'+\beta_2' = \beta'} { \beta' \choose \beta_1' } \int e^{ix'\cdot\xi'} (\xi')^{\beta_1'} D_{x'}^{\beta_2'}  a_{\epsilon,u}(x,\xi,x',\xi') \dslash\xi' \\
&= \piket{x'}{-n_0} \sum_{\beta_1'+\beta_2' = \beta'} { \beta' \choose \beta_1' } \int e^{ix'\cdot\xi'} \piket{D_{\xi'}}{n_0} \left(  (\xi')^{\beta_1'} D_{x'}^{\beta_2'}  a_{\epsilon,u}(x,\xi,x',\xi') \right) \dslash\xi'
,\end{align*}
wobei $\piket{x'}{-n_0} \piket{D_{\xi'}}{n_0} e^{ix'\cdot\xi'} = e^{ix'\cdot\xi'}$ und wir $n_0$-mal partiell integriert haben.
Für ein Multiindex $\mu \in \N^n_0$ mit $\abs{\mu} \le n_0$ betrachten wir
$$ \abs{ D_{\xi'}^\mu \left(  (\xi')^{\beta_1'} D_{x'}^{\beta_2'}  a_{\epsilon,u}(x,\xi,x',\xi') \right) } = 
 \abs{ \sum_{\mu_1+\mu_2=\mu} {\mu \choose \mu_1} D_{\xi'}^{\mu_1} (\xi')^{\beta_1'} D_{\xi'}^{\mu_2} D_x^{\beta_2'}  a_{\epsilon,u}(x,\xi,x',\xi') } .$$
Wir haben
$$
\abs{ D_{\xi'}^{\mu_1} (\xi')^{\beta_1'}  } \le 
\begin{cases}
\abs{ \mu_1! (-i)^{\mu_1} (\xi')^{\beta_1'-\mu_1} } & \textrm{ falls } \mu_1 \le \beta'_1 , \\
1 & \textrm{ sonst. }
\end{cases}
$$
also 
$\abs{ D_{\xi'}^{\mu_1} (\xi')^{\beta_1'}  } \le C_{\beta_1',\mu_1} \piket{\xi'}{\abs{\beta_1'}}$.
Andererseits liefert Gleichung (\ref{equ:doubleAepsilon})
$$\abs{ D_{\xi'}^{\mu_2} D_{x'}^{\beta_2'}  a_{\epsilon,u}(x,\xi,x',\xi') } \le C_{l',\mu_2,\beta_2',u} \piket{\xi}{m+\delta\abs{\beta_2'}} \piket{\xi'}{m'+\delta\abs{\beta_2'}-2l'} .$$
Insgesamt erhalten wir
$$
\abs{ D_{\xi'}^\mu \left( (\xi')^{\beta_1'} D_{x'}^{\beta_2'}  a_{\epsilon,u}(x,\xi,x',\xi') \right) } \le
C_{\mu,\beta_1',\beta_2',u} \piket{\xi}{m+\delta\abs{\beta_2'}} \piket{\xi'}{m'+\delta\abs{\beta_2'}-2l'}  \piket{\xi'}{\abs{\beta_1'}}  .$$
Also hält die Ungleichung
$$ \abs{ \piket{D_{\xi'}}{n_0} \left(  (\xi')^{\beta_1'} D_{x'}^{\beta_2'}  a_{\epsilon,u}(x,\xi,x',\xi') \right) } \le
C'_{l,l',n_0,u} \piket{\xi}{m+\delta\abs{\beta_2'}} \piket{\xi'}{m'+\abs{\beta'}-2l'} .$$
Da $m'+\abs{\beta'}-2l' \le m'+2l-2l' < -n$, ist $\xi' \mapsto \piket{\xi'}{m'+\delta\abs{\beta'}-2l'} \in L^1(\R^n)$.
Also gibt es eine Konstante $C''_{l,l',n_0,u}>0$ unabhängig von $0<\epsilon<1$, sodass
$$ \abs{ D_{x'}^{\beta'} b(x,\xi,x') } \le C''_{l,l',n_0,u} \piket{\xi}{m+\delta\abs{\beta'}} \piket{x'}{-n_0} .$$
Da $-n_0 < -n$, ist $x' \mapsto \piket{x'}{-n_0} \in L^1(\R^n)$. Es folgt also
$$ \lim_{\epsilon\rightarrow 0} c_{\epsilon,u}(x,\xi) = \int e^{-ix'\cdot\xi} \lim_{\epsilon\rightarrow 0} b_{\epsilon,u}(x,\xi,x') dx' .$$
Schließlich betrachte 
$$ c_{\epsilon,u}(x,\xi) = \int e^{-ix'\cdot\xi} b_{\epsilon,u}(x,\xi,x') dx' = 
\piket{\xi}{-2l} \int e^{-ix'\cdot\xi} \piket{D_{x'}}{2l} b_{\epsilon,u}(x,\xi,x') dx' ,$$
wobei $\piket{D_{x'}}{2l} e^{-ix'\cdot\xi} = \piket{\xi}{2l} e^{-ix'\cdot\xi}$ und wir $2l$-mal partiell integriert haben.
Also gibt es eine Konstante $C'''_{l,l',n_0,u}>0$ unabhängig von $0<\epsilon<1$, sodass
$$ \abs{  c_{\epsilon,u}(x,\xi) } \le C'''_{l,l',n_0,u} \piket{\xi}{m-2l(1-\delta)} .$$
Da $ m-2l(1-\delta) < -n$, folgt, dass $\xi \mapsto \piket{\xi}{m-2l(1-\delta)} \in L^1(\R^n)$ und damit
$$ Pu(x) = \lim_{\epsilon\rightarrow 0} \int e^{ix\cdot\xi} c_{\epsilon,u}(x,\xi) \dslash\xi = \int e^{ix\cdot\xi} \lim_{\epsilon\rightarrow 0} c_{\epsilon,u}(x,\xi) \dslash\xi .$$
Der Limes kommutiert also tatsächlich mit allen vier Integralen.

\end{proof}
\end{lemma}

\subsection{Das vereinfachte Symbol}\label{sec:simplify}
Wir stellen uns die Frage, unter welchen Bedingungen ein Doppelsymbol einen Operator definiert, der mit dem Operator eines einfachen Symbols übereinstimmt.
Dazu werden wir uns eine stetige Abbildung von der Klasse der Doppelsymbole in die Klasse der einfachen Symbole konstruieren, und zeigen,
dass das Bild ähnliche Abschätzungseigenschaften aufweist wie die Urbildklasse. Wir nennen die Elemente des Bildes auch vereinfachte Symbole.
Eine Anwendung dafür findet sich in der Komposition eines Hölder-stetigen Symbols mit einem glatten.
Das resultierende Doppelsymbol lässt sich dann als vereinfachtes Symbol durch eine asymptotische Entwicklung darstellen.

\begin{lemma}\label{tech:simplifyNormAbsch}
Sei $R \in \R, R > 0$ gegeben. Für alle $\eta \in \R^n$ mit $\abs{\eta} \ge R$ gilt
$$ \abs{\eta} \ge \frac{1}{2} \prod_{j=1}^n \left( \abs{\eta_j} + R \right)^{\frac{1}{n}} .$$
\begin{proof}
Für $\abs{\eta} = \sqrt{ \sum_{j=1}^n \abs{\eta_j}^2}$ gilt $\abs{\eta} \ge \abs{\eta_j}$ und damit 
$$ \abs{\eta} = \frac{1}{2} \abs{\eta} + \frac{1}{2} \abs{\eta} \ge \frac{1}{2} ( \abs{\eta} + R) \ge \frac{1}{2} ( \abs{\eta_j} + R)$$
für ein beliebiges $j = 1, \ldots, n$.
Insgesamt ist
$$\abs{\eta} \ge \frac{1}{2} ( \abs{\eta} + R) = \left( \left( \frac{1}{2} ( \abs{\eta} + R) \right)^{\frac{1}{n}} \right)^n \ge \prod_{j=1}^n \left( \frac{1}{2} ( \abs{\eta_j}  + R) \right)^{\frac{1}{n}} = \frac{1}{2} \prod_{j=1}^n \left( \abs{\eta_j} + R \right)^{\frac{1}{n}} .$$
\end{proof}
\end{lemma}

\begin{lemma}\label{tech:simplifyPiket}
Für $\gamma \in \N^n_0, l_0 \in \N_0$ gibt es eine Konstante $C_{\gamma,l_0} > 0$, sodass für alle $y,\xi \in \R^n$ die Ungleichung
$$\partial^\gamma_y (1 + \piket{\xi}{2\delta} \abs{y}^2 )^{-l_0} \le C_{\gamma,l_0} \piket{\xi}{\delta\abs{\gamma}} ( 1 + \piket{\xi}{2\delta} \abs{y}^2 )^{-l_0}$$
hält.
\begin{proof}
Wir beweisen mit Hilfe einer Induktion über $N := \abs{\gamma} \in \N$ (die Aussage ist für $N = 0$ trivialerweise erfüllt).
Für den Induktionsschritt wollen wir die Aussage vorerst elementar für $N = 1$ und ein beliebiges $j = 1,\ldots,n$ beweisen:
$$\abs{ \partial_{y_j} (1 + \piket{\xi}{2\delta} \abs{y}^2 )^{-l_0} } =
\abs{ -l_0 (1 + \piket{\xi}{2\delta} \abs{y}^2 )^{-l_0-1} \piket{\xi}{2\delta} 2y_i} \le
C_{l_0} \piket{\xi}{\delta} (1 + \piket{\xi}{2\delta} \abs{y}^2 )^{-l_0} ,$$
wobei 
$\abs{ 2 l_0 \piket{\xi}{\delta} y_i } \le C_{l_0} (1 + \piket{\xi}{2\delta} \abs{y}^2 )$.

Wir betrachten den Induktionsschritt $N \mapsto N+1$: 
Sei $\gamma \in \N^n_0$ mit $\abs{\gamma} = N$ gegeben. Für ein beliebiges $j = 1,\ldots,n$ gilt:
\begin{align*}
\abs{ \partial_y^\gamma \partial_{y_j} (1 + \piket{\xi}{2\delta} \abs{y}^2  )^{-2l_0} } &=
2 l_0 \piket{\xi}{\delta} \abs{ \partial_y^\gamma \left( (1+ \piket{\xi}{2\delta} \abs{y}^2 )^{-l_0-1} y_j \right) }\\
&= 2 l_0 \piket{\xi}{\delta} \abs{ \sum_{\alpha \le \gamma} {\alpha \choose \gamma} \partial_y^\alpha (1 + \piket{\xi}{2\delta} \abs{y}^2)^{-l_0-1} \partial_y^{\gamma-\alpha} y_j}
.\end{align*}
Nach Induktionsvoraussetzung ist für alle $\alpha \le \gamma$
$$\partial_y^\alpha (1 + \piket{\xi}{2\delta} \abs{y}^2)^{-l_0-1} \le C_{\alpha,l_0+1} \piket{\xi}{\delta\abs{\alpha}} (1+\piket{\xi}{2\delta} \abs{y}^2 )^{-l_0-1}$$
und
$$ \partial_y^{\gamma-\alpha} y_j = 
\begin{cases}
y_j & \gamma = \alpha ,\\
1 & \gamma = e_j + \alpha ,\\
0 & \textrm{ sonst, }
\end{cases}
$$
wobei wir jeden Vektor $e_j$ aus der kanonischen Basis $\menge{e_j}_{j=1,\ldots,n}$ des $\R^n$ als Multiindex auffassen.
Wir können insgesamt folgern:
\begin{align*}
\abs{ \partial_y^\gamma \partial_{y_j} (1 + \piket{\xi}{2\delta} \abs{y}^2  )^{-2l_0} }
&\le 2 l_0 \left( \piket{\xi}{\delta} \abs{ C_{\gamma,l_0+1} \piket{\xi}{\delta\abs{\gamma}} (1+\piket{\xi}{2\delta}\abs{y}^2)^{-l_0-1} y_j} \right. \\
&\quad + \left. \abs{ C_{\abs{\gamma}-1,l_0+1} \piket{\xi}{\delta(\abs{\gamma}-1)} (1+\piket{\xi}{2\delta}\abs{y}^2)^{-l_0-1}  } \right) \\
&\le C_{\gamma+j,l_0} \piket{\xi}{\delta\abs{\gamma+e_j}}  (1+\piket{\xi}{2\delta}\abs{y}^2)^{-l_0-1} \abs{y_j}  \\
&\le C_{\gamma+j,l_0} \piket{\xi}{\delta(\abs{\gamma}+1)} (1+\piket{\xi}{2\delta}\abs{y}^2)^{-l_0} 
.\end{align*}
\end{proof}
\end{lemma}

\begin{lemma}\label{lemma:simplify}
Sei $p(x,\xi,x',\xi') \in C^\tau S^{m,m'}_{\rho,\delta}(\R^n,\R^n\times\R^n\times\R^n)$ mit $m,m' \in \R$, $0 \le \delta \le \rho \le 1$, $\delta < 1$.
Wir setzen für feste $\alpha,\alpha',\beta,\beta'\in \N^n_0$ mit $\abs{\beta} \le \tau$
$$q(x,\xi,x',\xi') := D_\xi^\alpha D_{\xi'}^{\alpha'} D_x^\beta D_{x'}^{\beta'} p(x,\xi,x',\xi') .$$
Sei $\theta \in [-1,1]$. Setze
$$ q_\theta(x,\xi) := \osint e^{-iy\cdot\eta} q(x,\xi+\theta\eta, x+y,\xi) dy \dslash\eta .$$

\begin{enumerate}[(a)]
\item Dann ist $\menge{ q_\theta(x,\xi)}_{\abs{\theta}\le 1}$ in $C^{\tau-\abs{\beta}} S^\kappa_{\rho,\delta}(\R^n,\R^n)$ beschränkt für 
$$\kappa := m + m' + \delta \abs{\beta + \beta'} - \rho \abs{\alpha+\alpha'} .$$
\item  Für jedes $l \in \N$ gibt es eine Konstante $C > 0$ und ein $l' \in \N$, sodass für alle $\theta \in [-1,1]$ die Ungleichung 
\begin{equation}\label{equ:simplifyQtheta}
\abs{q_\theta}^{(\kappa)}_{l,\tau} \le C \abs{p}^{(m,m')}_{l',\tau}
\end{equation}
erfüllt ist.
\end{enumerate}

\begin{proof}

{\bf 1. Schritt}

Zunächst wähle $x,\xi'\in\R^n$ fest. Nun gibt es eine Konstante $C_{\alpha,\alpha',\beta,\beta',x,\xi'}>0$, sodass
$$\abs{q(x,\xi,x',\xi')} \le C_{\alpha,\alpha',\beta,\beta',x,\xi'} \piket{\xi}{m+\delta\abs{\beta+\beta'}} ,$$
und
$$\abs{D_y^\gamma D_\xi^\mu q(x,\xi+\theta\eta,x+y,\xi')} \le C_{\gamma,\mu} C_{\alpha,\alpha',\beta,\beta',x,\xi'} \piket{\xi}{m+\delta\abs{\beta+\beta'}+\delta\abs{\gamma}} .$$
Also ist
$$ \menge{ (y,\eta) \mapsto q(x,\xi+\theta\eta, x+y, \xi) : x, \xi \in \R^n, \theta \in [0,1]} \subset \Oszillatory{\delta,0}{m_+ + \delta \abs{\beta+\beta'}} ,$$
und damit $q_\theta(x,\xi)$ ein wohldefinierter Ausdruck.
Nach Theorem \ref{thm:osziKonvergenz} gilt für beliebige $\gamma, \mu \in \N^n_0$ mit $\abs{\gamma}+\abs{\beta} \le \tau$
$$ D_x^\gamma D_\xi^\mu q_\theta(x,\xi) = \osint e^{-iy\cdot\eta} D_x^\gamma D_\xi^\mu  q(x,\xi+\theta\eta, x+y, \xi) dy \dslash\eta .$$
Schreiben wir
\begin{align*}
D_x^\gamma D_\xi^\mu q(x,\xi, x', \xi') &= D_x^\gamma D_\xi^\mu  D_\xi^\alpha D_{\xi'}^{\alpha'} D_x^\beta D_{x'}^{\beta'} p(x,\xi,x',\xi') \\
&= \sum_{\substack{\alpha_1 + \alpha_1' = \alpha+\alpha'+\mu\\\beta_1 + \beta_1' = \beta + \beta' + \gamma}} C_{\alpha_1,\alpha_1',\beta,\beta_1'} D_x^{\alpha_1} D_{x'}^{\alpha_1'} D_\xi^{\beta_1} D_{\xi'}^{\beta_1'} p(x,\xi,x',\xi')
,\end{align*}
so sehen wir, dass die Behauptung (a) und (b) folgen, falls es ein $C_1 > 0$ gibt, sodass
$$ \abs{ q_\theta(x,\xi) } \le C_1 \abs{p}_{l,\tau}^{(m,m')} \piket{\xi}{\kappa} .$$ 

{\bf 2. Schritt}

Sei $l_0 > \frac{n}{2}, l_0 \in \N$ fest gewählt.
Mit $ - (\triangle_\eta)^{l_0} e^{-iy\cdot\eta} = \abs{y}^{2l_0} e^{-iy\cdot\eta}$ ist
$$ e^{-iy\cdot\eta} = \left(1 + \piket{\xi}{2\delta} \abs{y}^2 \right)^{-l_0} \left(1 + \piket{\xi}{2\delta} (-\triangle_\eta)\right)^{l_0} e^{-iy\cdot\eta} .$$
Wir können somit $q_\theta$ darstellen durch
$$ q_\theta(x,\xi) = \osint e^{-iy\eta} r_\theta(x,\xi,\eta,y) dy \dslash\eta ,$$
wobei
$ r_\theta(x,\xi,\eta,y) = (1+ \piket{\xi}{2\delta} \abs{y}^2 )^{-l_0} (1+\piket{\xi}{2\delta} (-\triangle_\eta))^{l_0} q(x,\xi+\theta\eta,x+y,\xi)$.
Wir teilen $\R^n$ in die drei Integrationsbereiche 
$$\Omega_1 := \menge{ \eta \in \R^n : \abs{\eta} < \frac{\piket{\xi}{\delta}}{2}}, 
\Omega_2 := \menge{ \eta \in \R^n : \frac{\piket{\xi}{\delta}}{2} \le \abs{\eta} < \frac{\piketi{\xi}}{2}  } $$
$$ \textrm{ und } \Omega_3 := \menge{ \eta \in \R^n : \abs{\eta} \ge \frac{\piketi{\xi}}{2}  }$$
bezüglich eines fest gewählten $\xi \in \R^n$ auf.

Da $2l_0 > n$, ist $y \mapsto e^{-iy\cdot\eta} r_\theta(x,\xi,\eta,y) \in L^1(\R^n)$.
Wir nehmen vorerst an, dass $\eta \mapsto e^{-iy\cdot\eta} r_\theta(x,\xi,\eta,y) \in L^1(\R^n)$, 
sodass wir nach Folgerung \ref{cor:osziLebesgue} das oszillatorische Integral aufspalten können in
$$ q_\theta(x,\xi) = 
\int_{\Omega_1} \varphi(x,\xi,\eta) \dslash\eta + 
\int_{\Omega_2} \varphi(x,\xi,\eta) \dslash\eta + 
\int_{\Omega_3} \varphi(x,\xi,\eta) \dslash\eta
,$$
wobei $\varphi(x,\xi,\eta) := \int e^{-iy\cdot\eta} r_\theta(x,\xi,\eta,y) dy$.

{\bf 3. Schritt}

Betrachten wir zunächst $\eta \in \Omega_1 \cup \Omega_2$, also $\abs{\eta} < \frac{\piketi{\xi}}{2}$.
Mit dem Mittelwertsatz im Mehrdimensionalen gilt
$$\piketi{\xi+\theta\eta} - \piketi{\xi} =
\theta\eta \cdot \int_0^1 \nabla_\xi \piketi{\xi + t \theta \eta} dt =
\theta \int_0^1 \sum_{j=1}^n \eta_j \partial_{\xi_j} \piketi{\xi + t \theta \eta} dt .$$
Da $\partial_{\xi_j} \piketi{\xi+t\theta\eta} = \partial_{\xi_j} \sqrt{ 1 + (\xi + t\theta\eta)^2} = \frac{\xi_j + t\theta\eta_j}{\piketi{\xi + t\theta\eta}}$
ist
$$\abs{ \piketi{\xi+\theta\eta} - \piketi{\xi}} \le
\theta \abs{ \int_0^1 \sum_{j=1}^n \eta_j \frac{\xi_j + t\theta\eta_j}{\piketi{\xi + t\theta\eta}} dt } \le
\abs{ \sum_{j=1}^n \eta_j } \le \frac{\piketi{\xi}}{2}
.$$
Wir erhalten damit die Abschätzungen
$$
\frac{1}{2} \piketi{\xi} \le \piketi{\xi+\theta\eta} \le \frac{3}{2} \piketi{\xi} 
\textrm{ und }
\piketi{\xi + \theta\eta; \xi} \le \piketi{\xi + \theta\eta} \piketi{\xi} \le \frac{5}{2} \piketi{\xi}
$$
für alle $\eta \in \Omega_1 \cup \Omega_2$ und $\abs{\theta} \le 1$.
Wegen $\delta \le \rho$ ist 
\begin{align*}
&\abs{ \piket{\xi}{2\delta l_0} (-\triangle_\eta)^{l_0} q(x,\xi+\theta\eta, x+y, \xi) } \\
&\le C_{\alpha\alpha'\beta\beta'} \piket{\xi+\theta\eta}{m_++\delta\abs{\beta}-\rho\abs{\alpha}} \piket{\xi+\theta\eta;\xi}{\delta\abs{\beta'}} \piket{\xi}{m_+'-\rho\abs{\alpha' }+ 2\delta l_0} \piket{\theta\eta}{-2\rho l_0} \\
&\le C'_{\alpha\alpha'\beta\beta'} \piket{\xi}{\kappa+ 2l_0(\delta-\rho)}  
 \le C'_{\alpha\alpha'\beta\beta'} \piket{\xi}{\kappa}
\end{align*}
für $\eta \in \Omega_1 \cup \Omega_2$.
Zusammenfassend ergibt sich
$$\abs{ r_\theta(x,\xi,\eta,y) } \le c_2 (1 + \piket{\xi}{2\delta} \abs{y}^2 )^{-l_0} \piket{\xi}{\kappa} \textrm{ für jedes } \eta \in \Omega_1 \cup \Omega_2 .$$

Unter Koordinatentransformation $y \mapsto \piket{\xi}{\delta} y =: w$ ist mit ``$dw = \piket{\xi}{\delta n} dy$'' und wegen $-2l_0 < -n$
\begin{align*}
\abs{ \int_{\Omega_1} \varphi(x,\xi,\eta) \dslash\eta} 
&\le c_2 \piket{\xi}{\kappa - \delta\eta} \int_{\Omega_1} \int \abs{ e^{-i \piket{\xi}{-\delta} w \cdot \eta}} (1+\abs{w}^2)^{-l_0} dw \dslash \eta \\
&\le c_2' \piket{\xi}{\kappa-\delta\eta} \frac{1}{(2\pi)^n} \vol \Omega_1 
\le c_2'' \piket{\xi}{\kappa}
,\end{align*}
wobei 
$$\vol \Omega_1 \le \prod_{j=1}^n \int_{\abs{\lambda} \le \frac{\piket{\xi}{\delta}}{2}} d\lambda = \piket{\xi}{\delta\eta} .$$

Für $\int_{\Omega_2} \varphi(x,\xi,\eta) \dslash\eta$ betrachten wir 
\begin{align*}
& D_y^\gamma r_\theta(x,\xi,\eta,y) \\
&= \sum_{\gamma_1+\gamma_2 = \gamma} {\gamma \choose \gamma_1} D_y^{\gamma_1} \left(1 + \piket{\xi}{2\delta} \abs{y}^2 \right)^{-l_0} \left(1 + \piket{\xi}{2\delta}(-\triangle_\eta)\right)^{l_0} D_y^{\gamma_2} q(x,\xi+\theta\eta, x+y, \xi)
.\end{align*}
Es ist einfach zu sehen, dass
$\abs{ \piket{\xi}{2\delta l_0} (-\triangle_\eta)^{l_0} D_y^{\gamma_2} q(x,\xi+\theta\eta, x+y, \xi) } \le C_{\kappa\gamma_2} \piket{\xi}{\kappa+\delta\abs{\gamma_2}}$.
Mit Lemma \ref{tech:simplifyPiket} folgt
$$\abs{ D_y^\gamma r_\theta(x,\xi,\eta,y) } \le C_3 \piket{\xi}{\kappa+\delta\abs{\gamma}} (1 + \piket{\xi}{2\delta} \abs{y}^2 )^{-l_0} \textrm{ für alle } \eta \in \Omega_2 .$$
Insbesondere haben wir für $\abs{\gamma} = 2l_0$
$$\abs{ (-\triangle_y)^{l_0} r_\theta(x,\xi,\eta,y) } \le C_3 \piket{\xi}{\kappa+2l_0\delta} (1 + \piket{\xi}{2\delta} \abs{y}^2 )^{-l_0} \textrm{ für alle } \eta \in \Omega_2 .$$

Mit der partieller Integration
$$ \int e^{-iy\cdot \eta} r_\theta(x,\xi,\eta,y) dy = \abs{\eta}^{-2l_0} \int e^{-iy\cdot\eta} (-\triangle_y)^{l_0} r_\theta(x,\xi,\eta,y) dy$$
und analogen Umformungen folgt
\begin{align*}
\abs{\int_{\Omega_2} \varphi(x,\xi,\eta) \dslash\eta} 
&\le C_3 \piket{\xi}{\kappa+2l_0\delta} \int_{\Omega_2} \abs{\eta}^{-2l_0} \int \abs{e^{-iy\cdot\eta}} (1 + \piket{\xi}{2\delta} \abs{y}^2 )^{-l_0} dy \dslash\eta \\
&\le C_4 \piket{\xi}{\kappa+(2l_0-n)\delta} \int_{\Omega_2} \abs{\eta}^{-2l_0} \int (1+\abs{w}^2)^{-2l_0} dw \dslash\eta \\
&\le C_4' \piket{\xi}{\kappa+(2l_0-n)\delta} \int_{\Omega_2} \abs{\eta}^{-2l_0} \dslash\eta \\
&\le C_4''  \piket{\xi}{\kappa}
,\end{align*}
wobei nach Lemma \ref{tech:simplifyNormAbsch} gilt, dass
$$\int_{\Omega_2} \abs{\eta}^{-2l_0} \dslash\eta \le \int_{\menge{\abs{\eta} \ge \frac{\piket{\xi}{\delta}}{2}}} \abs{\eta}^{-2l_0} \dslash\eta 
 \le 2^{2l_0} \prod_{j=1}^n \int_\R \left( \abs{\eta_j} + \frac{\piket{\xi}{\delta}}{2} \right)^{-\frac{2l_0}{n}} \dslash \eta_j \le C_5 \piket{\xi}{(-2l_0+n)\delta} .$$
Dazu haben wir folgende Umformung für $\frac{2l_0}{n} > 1$ verwendet:
\begin{align*}
\int_\R \left( \abs{\lambda} + \frac{\piket{\xi}{\delta}}{2} \right)^{-\frac{2l_0}{n}} d\lambda 
&= \int_\R \left( \frac{1}{ \abs{\lambda} + \frac{\piket{\xi}{\delta}}{2} } \right)^{\frac{2l_0}{n}} d\lambda 
= 2 \frac{ \left(\frac{\piket{\xi}{\delta}}{2}\right)^{\frac{-2l_0}{n} + 1}}{\frac{2l_0}{n}-1} \\
&\le \frac{ 2^{\frac{2l_0}{n}-1}}{\frac{2l_0}{n}-1} \piket{\xi}{\frac{\delta}{n}(-2l_0+n)}
.\end{align*}

Schließlich betrachten wir das Integral $\int_{\Omega_3} \varphi(x,\xi,\eta) \dslash\eta$:
Für $\eta \in \Omega_3$ ist $\abs{\eta} \ge \frac{\piketi{\xi}}{2} \ge \frac{1}{2}$. Lemma \ref{lemma:japanese} liefert ein $C>0$, sodass
$\piketi{\xi+\theta\eta} \le \piketi{\xi} + C \abs{\eta} \le 3 C \abs{\eta}$.

Für $m_+ := \max \menge{ m, 0}, \gamma \in \N^n_0$ liefert Lemma \ref{tech:simplifyPiket} die Abschätzung
\begin{align*}
&\abs{ \partial_y^\gamma r_\theta(x,\xi,\eta,y)} \\
&\le C_{\alpha,\alpha',\beta,\beta'} \piket{\xi+\theta\eta}{m_++\delta\abs{\beta}-\rho\abs{\alpha}} \piket{\xi+\theta\eta; \xi}{\delta\abs{\beta'}} \piket{\xi}{m'-\rho\abs{\alpha'}} C_{\gamma,l_0} \piket{\xi}{\delta\abs{\gamma} + 2l_0\delta} (1 + \piket{\xi}{2\delta} \abs{y}^2)^{-l_0} \\
&\le C_6 \abs{\eta}^{m_+ + \delta\abs{\beta+\beta'} + \delta\abs{\gamma}+2l_0\delta} (1 + \piket{\xi}{2\delta} \abs{y}^2 )^{-l_0} \piket{\xi}{m'-\rho\abs{\alpha'}} \\
&\le C_6 \abs{\eta}^{\kappa' + \delta\abs{\gamma} + 2l_0\delta} (1 + \piket{\xi}{2\delta} \abs{y}^2 )^{-l_0} \piket{\xi}{m'}
,\end{align*}
wobei $\kappa' := m_+ + \abs{\beta + \beta'}\delta$.
Insbesondere erhalten wir für $\abs{\gamma} = 2l$ mit $l \in \N$
$$ \abs{ (-\triangle_y)^l r_\theta(x,\xi,\eta,y)} \le
C_6 \abs{\eta}^{\kappa' + 2l\delta + 2l_0\delta} (1 + \piket{\xi}{2\delta} \abs{y}^2 )^{-l_0} \piket{\xi}{m'} .$$
Wähle nun $l$ so groß, dass
$$\kappa' + 2 l_0\delta - 2l(1-\delta) < -n \textrm{ und } m'+\kappa' + 2l_0\delta - 2l(1-\delta) + (1-\delta)n \le \kappa .$$
Dann erhalten wir durch analoge Abschätzungen, dass
\begin{align*}
\abs{\int_{\Omega_3} \varphi(x,\xi,\eta) \dslash\eta} 
&= \abs{ \int_{\Omega_3} \abs{\eta}^{-2l} \left( \int e^{-iy\cdot\eta} (-\triangle_y)^l r_\theta(x,\xi,\eta,y) dy \right) \dslash\eta } \\
&\le C_7 \piket{\xi}{m'} \int_{\Omega_3} \abs{\eta}^{\kappa' + 2l_0\delta + 2l\delta - 2l} \int (1 + \piket{\xi}{2\delta}\abs{y}^2)^{-l_0} dw \dslash\eta \\
&\le C_7 \piket{\xi}{m'-\eta\delta} \int_{\Omega_3} \abs{\eta}^{\kappa' + 2l_0\delta - 2l(1-\delta)} \int (1+\abs{w}^2)^{-l_0} dw \dslash\eta \\
&\le C_7' \piket{\xi}{m'-\eta\delta} \int_{\Omega_3} \abs{\eta}^{\kappa' + 2l_0\delta - 2l(1-\delta)} \dslash\eta \\
&\le C_7'' \piket{\xi}{\kappa}
,\end{align*}
wobei
$\int_{\Omega_3} \abs{\eta}^{\kappa' + 2l_0\delta + 2l(\delta - 1)} \dslash\eta \le C' \piket{\xi}{\kappa'+2l_0\delta - 2l(1-\delta)+n}$.
Insgesamt ist
$\abs{ q_\theta(x,\xi)} \le C_{ges} \piket{\xi}{\kappa}$.
Wir finden also ein $C_{\gamma\mu} > 0$, sodass
$$\abs{ \partial_\xi^\gamma \partial_x^\mu q_\theta(x,\xi)} \le C_{\gamma\mu} \piket{\xi}{\kappa-\rho\abs{\gamma}+\delta\abs{\mu}} \textrm{ gleichmäßig in } x \in \R^n .$$
Insbesondere folgt mit dem Übergang zum Supremum über alle $x \in \R^n$ mit $\abs{\mu} \le \tau$
$$\norm{ \partial_\xi^\gamma q_\theta(\hold,\xi)}{C^{\gauss{\tau}}(\R^n)} \le C_{\gamma\gauss{\tau}} \piket{\xi}{\kappa-\rho\abs{\gamma}+\delta\gauss{\tau}} .$$
Wir folgern den Fall $\tau \in \R_+ \setminus \N$ induktiv aus dem Spezialfall $\tau \in (0,1)$:
Definiere dazu $\triangle_{\chi,x} q_\theta(x,\xi) := q_\theta(x+\chi,\xi) - q_\theta(x,\xi)$ für $\chi \in \R^n$.
Nun ist $x \mapsto q(x,\xi+\theta\eta,x+y,\xi)$ Hölder-stetig zum Grad $\tau$ (da $\beta = 0$) mit Hölder-Stetigkeitskonstante $C_H > 0$,
sodass wir analog zu obigen Beweis erhalten:
$$ \abs{ \partial_\xi^\gamma \triangle_{\chi,x} q_\theta(x,\xi) } 
= \abs{ \triangle_{\chi,x} \partial_\xi^\gamma q_\theta(x,\xi) } 
\le C_\gamma C_H \piket{\xi}{\kappa-\rho\abs{\gamma}+\delta\tau} \abs{\chi}^\tau .$$
Mit dem Übergang zum Supremum über alle $\chi \in \R^n$ mit $\xi \neq -x$ folgt die Beschränktheit in der $C^\tau(\R^n)$-Norm.
\end{proof}
\end{lemma}

\begin{thm}\label{thm:simplify}
Sei $p(x,\xi,x',\xi') \in C^\tau S^{m,m'}_{\rho,\delta}(\R^n,\R^n\times\R^n\times\R^n)$ mit $m,m' \in \R$, $0 \le \delta \le \rho \le 1$, $\delta < 1$.
Dann besitzt das Symbol
\begin{equation}\label{equ:simplifyVereinfachtesSymbol}
p_L(x,\xi) := \osint e^{iy\cdot\eta} p(x,\xi+\eta,x+y,\eta) dy \dslash\eta
\end{equation}
die folgenden Eigenschaften:
\begin{enumerate}[(a)]
\item Das Symbol $p_L(x,\xi)$ gehört der Klasse $C^\tau S^{m+m'}_{\rho,\delta}(\R^n,\R^n)$ an.
\item Die Operatoren $P$ und  $p_L(X,D_x)$ sind identisch.
\item Die Abbildung $C^\tau S^{m,m'}_{\rho,\delta}(\R^n,\R^n\times\R^n\times\R^n) \rightarrow C^\tau S^{m+m'}_{\rho,\delta}(\R^n,\R^n), p \mapsto p_L$ ist stetig, da es zu jedem $l \in \N$ ein $l' \in \N$ und eine Konstante $c_{l,l'}>0$ gibt, mit
$\abs{p_L}_{l,\tau}^{(m+m')} \le c_{l,l'} \abs{p}_{l',\tau}^{(m,m')}$.
Die Abbildung induziert uns die eine weitere Abbildung 
$$\OP{C^\tau S^{m,m'}_{\rho,\delta}(\R^n,\R^n\times\R^n\times\R^n)} \rightarrow \OP{C^\tau S^{m+m'}_{\rho,\delta}(\R^n,\R^n)} .$$
Insbesondere ist $\OP{C^\tau S^{m,0}_{\rho,\delta}(\R^n,\R^n\times\R^n\times\R^n)} = \OP{C^\tau S^m_{\rho,\delta}(\R^n,\R^n)}$.
\end{enumerate}
$p_L(x,\xi)$ nennen wir das vereinfachte Symbol.
\begin{proof}
Setze $\theta = 1, \alpha = \beta = \alpha' = \beta' = 0$ in Lemma \ref{lemma:simplify} und erhalte dort
$p_L = q_1 \in C^\tau S^{m+m'}_{\rho,\delta}(\R^n,\R^n)$ und damit folgen (a) und (c) bereits aus (\ref{equ:simplifyQtheta}).
Sei $\chi(\eta,y) \in \Schwarz{\R^n\times\R^n}$ mit $\chi(0,0) = 1$.
Setze für $0 < \epsilon < 1$
$$p_{L,\epsilon}(x,\xi) := \iint e^{-iy\cdot\eta} \chi(\epsilon\eta,\epsilon y) p(x,\xi+\eta,x+y,\xi) dy \dslash\eta .$$
Dann ist
$ p_L(x,\xi) = \lim_{\epsilon\rightarrow 0} p_{L,\epsilon}(x,\xi)$.
Wir wenden nun Theorem \ref{thm:osziTransform} an und werden dabei die dort definierte Konstante $l_0 \in \N$ übernehmen, sodass wir für $x,\xi \in \R^n$ fest gewählt erhalten:
\begin{align*}
& \abs{ \os{y}{\eta}{p(x,\xi+\eta,x+y,\xi)}} \\
&\le \OszillatoryNorm{ (\eta,y) \mapsto p(x,\xi+\eta,x+y,\xi) }{\delta,0}{m}{l_0} \\
&= \max_{\abs{\alpha+\beta}\le l_0} \sup_{\xi,\eta\in\R^n} \abs{\partial_\eta^\alpha \partial_y^\beta p(x,\xi+\eta,x+y,\xi)} \piket{\eta}{-m-\delta\abs{\beta}}
.\end{align*}
Unter Verwendung der Ungleichung von Peetre \ref{ungl:peetre} ist
\begin{align*}
\abs{\partial_\eta^\alpha \partial_y^\beta p(x,\xi+\eta,x+y,\xi)} &\le
C_{\alpha\beta} \piket{\xi+\eta}{m-\rho\abs\alpha} \piket{\xi+\eta;\xi}{\delta\abs{\beta}} \piket{\xi}{m'} \\
&\le C_{\alpha\beta} 2^{\abs{m}+\delta\abs{\beta}} \piket{\xi}{m_+ + m' + \delta\abs{\beta}} \piket{\eta}{m+\delta\abs{\beta}}
,\end{align*}
wobei $m_+ := \max \menge{0,m}$.
Wir finden also eine von $\epsilon$ unabhängige Konstante $C_1 > 0$, sodass
$\abs{p_{L,\epsilon}(x,\xi)} \le C_1 \piket{\xi}{m_+ + m' + l_0\delta}$.
Für $\tilde{\chi} \in \Schwarzi$ mit $\tilde{\chi}(0) = 1$ setze
$$p_\epsilon(x,\xi,x',\xi') := \chi(\epsilon(\xi-\xi'),\epsilon(x-x')) \tilde{\chi}(\epsilon\xi') p(x,\xi,x',\xi') .$$
Wir wollen nun zeigen, dass der Operator des Doppelsymbol $p_\epsilon(x,\xi,x',\xi')$ mit dem Operator $p_L(X,D_x)$ des vereinfachten Symbols von $p$ übereinstimmt.
Dazu betrachten wir für ein $u \in \Schwarzi$ den Ausdruck $p_L(X,D_x)u(x)$
unter den Koordinatentransformation $y \mapsto x+y =: x'$ und $\eta \mapsto \xi'+\eta =: \xi$, indem wir das oszillatorische Integral wie in Definition \ref{def:osziOsziInt} auffassen:
\begin{align*}
p_L(X,D_x)u(x) 
&= \lim_{\epsilon\rightarrow 0} \int e^{ix\cdot\xi'} 
\left( \iint e^{-iy\cdot\eta} p_\epsilon(x,\xi'+\eta,x+y,\xi') dy \dslash\eta \right) \hat{u}(\xi') \dslash\xi' \\
&=  \lim_{\epsilon\rightarrow 0} \int e^{ix\cdot\xi'}
\left( \iint e^{-i(x'-x)\cdot(\xi-\xi')} p_\epsilon(x,\xi,x',\xi') dx' \dslash\xi' \right) \hat{u}(\xi') \dslash\xi'\\
&= \lim_{\epsilon\rightarrow 0} \int e^{ix\cdot\xi} \left(
\int e^{-ix'\cdot\xi} \left(
\int e^{-ix'\cdot\xi'} p_\epsilon(x,\xi,x',\xi') \hat{u}(\xi') \dslash\xi' \right) dx' \right) \dslash\xi 
.\end{align*}
Nun d\"{u}rfen analog zu Lemma \ref{lemma:doubleLimesPermut} Integral und Limes vertauscht werden, da $\menge{p_\epsilon}_{0<\epsilon<1} \subset C^\tau S^{m,m'}_{\rho,\delta}(\R^n,\R^n\times\R^n\times\R^n)$ beschränkt ist, sodass wir die Identifizierung
$p_L(X,D_x) u(x) = p(X,D_x,X',D_{x'}) u(x)$ erhalten.
\end{proof}
\end{thm}

\begin{thm}\label{thm:simplifyAsmytotic}
Sei $p(X,D_x,X',D_{x'}) \in \OP{C^\tau S^{m,m'}_{\rho,\delta}}(\R^n,\R^n\times\R^n\times\R^n)$ mit $m,m' \in \R$, $0 \le \delta \le \rho \le 1$, $\delta < 1$.
Für ein Multiindex $\alpha\in\N^n_0$ setze $p_\alpha(x,\xi) := D_\xi^\alpha D_{x'}^\alpha p(x,\xi,x',\xi') \mid_{x' = x , \xi' = \xi} $.
Das vereinfachte Symbol $p_L \in C^\tau S^{m+m'}_{\rho,\delta}(\R^n,\R^n)$ von $p$ lässt sich für eine beliebige Ganzzahl $N\in\N$ durch die Taylorreihe 
$$ p_L(x,\xi) = \sum_{\abs{\alpha} < N} \frac{1}{\alpha!} p_\alpha(x,\xi) + N \sum_{\abs{\gamma} = N} \int_0^1 \frac{(1-\theta)^{N-1}}{\gamma!} r_{\gamma,\theta}(x,\xi) d\theta$$
darstellen, wobei $r_{\gamma,\theta} \in C^\tau S_{\rho,\delta}^{m+m'-(\rho-\delta)\abs{\gamma}}(\R^n,\R^n)$ gegeben ist durch 
$$ r_{\gamma,\theta}(x,\xi) := \osint e^{-iy\cdot\eta} p_\gamma(x,\xi+\theta\eta,x+y,\xi) dy\dslash\eta .$$
\begin{proof}
Wenn wir eine Taylorreihenentwicklung von $p$ bezüglich $\xi$ wie in Theorem \ref{remind:taylorreihe} machen, bekommen wir
\begin{align*}
p(x,\xi+\eta,x+y,\xi) &= \sum_{\abs{\alpha} < N} \frac{\eta^\alpha}{\alpha!} D_\xi^\alpha p(x,\xi,x+y,\vartheta) \mid_{\vartheta=\xi} \\
&+ N \sum_{\abs{\gamma}=N} \frac{\eta^\gamma}{\gamma!} \int_0^1 (1-\theta)^{N-1} D_\xi^\gamma p(x,\xi+\theta\eta,x+y,\vartheta) \mid_{\vartheta=\xi} d\theta
.\end{align*}
Setzen wir dies in die Definition des vereinfachten Symbols (\ref{equ:simplifyVereinfachtesSymbol}) ein, so erhalten wir
\begin{align*}
p_L(x,\xi) &= \sum_{\abs{\alpha} < N} \frac{1}{\alpha!} \left( \osint e^{-iy\cdot\eta} \eta^\alpha D_\xi^\alpha p(x,\xi,x+y,\vartheta) \mid_{\vartheta=\xi} dy \dslash\eta \right) \\
&+ N \sum_{\abs{\gamma}=N} \frac{(1-\theta)^{N-1}}{\gamma!} \int_0^1 \left( \osint e^{-iy\cdot\eta} \eta^\gamma D_\xi^\gamma p(x,\xi+\theta\eta,x+y,\vartheta) \mid_{\vartheta=\xi} dy \dslash\eta \right) d\theta
.\end{align*}
Wir integrieren partiell unter Ausnutzung von $D_y^\alpha e^{-iy\cdot\eta} = (-\eta)^\alpha e^{-iy\cdot\eta}$ und Theorem \ref{thm:osziPartielle}:
\begin{align*}
p_L(x,\xi) &= \sum_{\abs{\alpha} < N} \frac{1}{\alpha!} \left( \osint e^{-iy\cdot\eta} p_{(0,\alpha)}^{(\alpha,0)}(x,\xi,x+y,\xi) dy \dslash\eta \right) \\
&+ N \sum_{\abs{\gamma}=N} \frac{(1-\theta)^{N-1}}{\gamma!} \int_0^1 \left( \osint e^{-iy\cdot\eta} p_{(0,\gamma)}^{(\gamma,0)}(x,\xi+\theta\eta,x+y,\xi) dy \dslash\eta \right) d\theta
,\end{align*}
wobei $ p_{(0,\alpha)}^{(\alpha,0)}(x,\xi,x',\xi') =  D_{x'}^\alpha D_\xi^\alpha p(x,\xi,x',\xi')$.
Wir erkennen nun anhand der Inversionsformel (\ref{equ:inversionFormula}), 
dass es sich beim ersten Term um
$$ \sum_{\abs{\alpha} < N} \frac{1}{\alpha!} \left( \osint e^{-iy\cdot\eta} p_{(0,\alpha)}^{(\alpha,0)}(x,\xi,x+y,\xi) dy \dslash\eta \right) = \sum_{\abs{\alpha} < N} \frac{1}{\alpha!} p_\alpha(x,\xi)$$
handelt. 

Nun ist $p_\alpha(x,\xi) \in C^\tau S_{\rho,\delta}^{m+m'- (\rho-\delta)\abs{\alpha}}(\R^n,\R^n)$, und nach Lemma \ref{lemma:simplify} ist $ \menge{r_{\gamma,\theta}(x,\xi) }_{\abs{\theta}\le1}$ eine beschränkte Teilmenge von $C^\tau S^{m+m'-(\rho-\delta)\abs{\gamma}}_{\rho,\delta}(\R^n,\R^n)$.
Also ist die Repräsentation des vereinfachten Symbols unter der Taylorreihendarstellung wohldefiniert.
\end{proof}
\end{thm}

\begin{lemma}\label{thm:simplifyComposition}
Seien $p_1(x,\xi) \in C^\tau S^{m_1}_{\rho,\delta}(\R^n,\R^n), p_2(x,\xi) \in S^{m_2}_{\rho,\delta}(\R^n\times\R^n)$ mit $m_1, m_2 \in \R, \tau > 0, 0 \le \delta < \rho \le 1$ und $\delta < 1$.
Dann ist $p_1(X,D_x) p_2(X,D_x) \in \OP{ S^{m_1+m_2}_{\rho,\delta}(\R^n,\R^n) }$.
Das zusammengesetzte Symbol
$p_1 \# p_2(x,\xi) := \osint e^{-iy\cdot\eta} p_1(x,\xi+\eta) p_2(x+y,\xi) dy \dslash\eta$
gehört der Klasse $C^\tau S^{m_1+m_2}_{\rho,\delta}(\R^n,\R^n)$ an, und induziert einen Operator $p_1 \# p_2(X,D_x) = p_1(X,D_x) p_2(X,D_x)$.
Wir erhalten durch eine asymptotische Entwicklung die Aussage
$$p_1 \# p_2(x,\xi) \sim \sum_{\alpha\in\N^n_0} \frac{1}{\alpha!} D_\xi^\alpha p_1(x,\xi) D_x^\alpha p_2(x,\xi)$$
im Sinne, dass für jedes $N \in \N_0$ das Symbol
$$p_1 \# p_2(x,\xi) - \sum_{\abs{\alpha} \le N} \frac{1}{\alpha!} D_\xi^\alpha p_1(x,\xi) D_x^\alpha p_2(x,\xi)$$
der Klasse $C^\tau S^{m_1+m_2-(\rho-\delta)(N+1)}_{\rho,\delta}(\R^n,\R^n)$ angehört.
\begin{proof}
Das zusammengesetzte Symbol ist das vereinfachte Symbol $p_L(x,\xi)$ des Symbols $p(x,\xi,x',\xi') := p_1(x,\xi) p_2(x',\xi') \in C^\tau S^{m_1,m_2}_{\rho,\delta}(\R^n,\R^n\times\R^n\times\R^n)$.
Dementsprechend übertragen sich die Resultate von Theorem \ref{thm:simplify} auf das zusammengesetzte Symbol, und mit Theorem \ref{thm:simplifyAsmytotic} erhalten wir die asymptotische Entwicklung.
\end{proof}
\end{lemma}
\nomenclature{$p_1\#p_2$}{Zusammengesetztes Symbol aus den Symbolen $p_1$ und $p_2$}

\subsection{Symbolglättung}\label{sec:smoothing}

Ziel der Symbolglättung ist es, ein Symbol $p \in C^\tau S^m_{1,\delta}(\R^n,\R^n)$ mit $\delta \in [0,1), m \in \R$ in 
$p(x,\xi) = p^\sharp(x,\xi) + p^\flat(x,\xi)$
mit
$p^\sharp \in S^m_{1,\gamma}(\R^n\times\R^n)$
und
$p^\flat \in C^{\tau} S^{m- (\gamma - \delta) \tau}_{1,\gamma}(\R^n,\R^n)$
für $\gamma \in (\delta, 1)$ aufzuspalten.
Die Vorgehensweise besteht darin,
zunächst einen Glättungskern $\phi \in C^\infty_0(\R^n)$ mit
$\phi(\xi) = 1$ für jedes $\xi \in \R^n$ mit $\abs{\xi} \le 1$ 
zu wählen (z.B. $\phi = \lilwood{0}$)
und damit eine Familie von Glättungsoperatoren $\menge{ J_\epsilon }_{\epsilon \in (0,1]}$ durch
\begin{equation}
J_\epsilon f(x) := \phi(\epsilon D_x) f(x)
\end{equation}
zu definieren.
Lassen wir nun $J_\epsilon$ auf $p(x,\xi)$ bzgl. $x$ wirken, so erhalten wir offensichtlich ein glattes Symbol.
Setzen wir 
$$
p^\sharp(x,\xi) := \sum^\infty_{j=0} J_{\epsilon_j} p(x,\xi) \lilwood{j}(\xi)
$$
mit 
$\epsilon_j := 2^{-j \gamma } $ für jedes $j \in \N_0$
, so werden wir sehen, dass $p^\sharp$ der Klasse $S^m_{1,\gamma}(\R^n\times\R^n)$ angehört.
Es gilt also zu zeigen, dass $p^\sharp$ und 
$p^\flat(x,\xi) := p(x,\xi) - p^\sharp(x,\xi)$
in den postulierten Symbolklassen liegen.
Wir werden dazu versuchen, analog wie \cite[\S 1.3]{Nonlinear} bzw. \cite{PDE} den Beweis in mehrere Schritte aufzugliedern.

\begin{remark}\label{smoothing:epsilon}
Sei $\epsilon \in \menge{2^{-\gamma j}}_{j \in \N_0}$.
Da $\phi$ einen kompakten Träger hat, gibt es ein $r > 0$, sodass $\supp \phi \subset B_r(0)$.
Für $\gamma \in (0,1)$ gibt es ein $j \in \N_0$, sodass
\begin{equation}
\epsilon \in 
E_j := 
\begin{cases}
\menge{1} & j = 0,\\
\menge{ \epsilon \in (0,1] : 2^{-j}r \le \epsilon \le 2^{-j+1}r } & j \in \N .\\
\end{cases}\;
\end{equation}
Dadurch erhält man einen Zusammenhang zwischen $\phi(\epsilon\xi)$ und $\lilwood{j}(\xi)$ gegeben durch
\begin{equation}
\supp \phi(\epsilon\xi) \subset B_{2^{j}}(0) \subset \bigcup_{l=0}^{j} \menge{ x \in \R^n : \lilwood{j}(x) = 1}.
\end{equation}
\end{remark}

\begin{lemma}
Sei $f \in C^\tau_*(\R^n), \tau > 0, \epsilon \in (0,1]$.
Dann gilt für jeden Multiindex $\beta \in \N_0^n$
\begin{equation}
\norm{D^\beta_x J_\epsilon f}{C^\tau_*(\R^n)} \le c_\beta \epsilon^{-\abs{\beta}} \norm{f}{C^\tau_*(\R^n)}.
\end{equation}
\begin{proof}
Da $\phi \in S^{-\infty}_{1,0}(\R^n)$, ist die Behauptung für $\beta = 0$ trivialerweise erfüllt.
Für $\beta \neq 0$ ist $\xi^\beta \in S^{\abs{\beta}}_{1,0}(\R^n)$, 
und deswegen die Komposition $D_x^\beta \phi(\epsilon D_x) \in \OP{S^0_{1,0}(\R^n)}$.
Folglich hat nach Lemma \ref{beschr:S_1,delta} der Operator die Abbildungseigenschaft
$D_x^\beta \phi(\epsilon D_x) : C^\tau_*(\R^n) \rightarrow C^\tau_*(\R^n)$.
$$
\norm{D_x^\beta J_\epsilon f(x)}{C^\tau_*(\R^n)} =
\norm{D_x^\beta \phi(\epsilon D_x) f(x)}{C^\tau_*(\R^n)} =
\sup_{j \in \N_0} 2^{j\tau} \norm{\lilwood{j}(D_x) \left( D_x^\beta \phi(\epsilon D_x) f(x) \right)}{\infty}
.$$
Betrachte für ein festes $j\in\N^n_0$ die Funktion
$\tilde{f}(x) := \left( \lilwood{j-1}(D_x) + \lilwood{j}(D_x) + \lilwood{j+1}(D_x) \right) f(x)$.
Folgerung \ref{cor:introLilwoodOP} liefert
\begin{align*}
2^{j\tau} \norm{\lilwood{j}(D_x) \left( D_x^\beta \phi(\epsilon D_x) \tilde{f}(x) \right)}{\infty}
&= 2^{j\tau} \norm{\check{\lilwood{j}} * \left( D_x^\beta \phi(\epsilon D_x) \tilde{f} \right)(x)}{\infty} \\
&\le 2^{j\tau} \norm{\check{\lilwood{j}}}{L^1(\R^n)} \norm{D_x^\beta \phi(\epsilon D_x) \tilde{f}}{\infty}
.\end{align*}
Für $\phi_\epsilon(x) := \phi(\epsilon x)$ haben wir nach Definition der Faltung
$$
\norm{D_x^\beta \phi(\epsilon D_x) \tilde{f}(x)}{\infty} =
\norm{D_x^\beta (\check{\phi}_\epsilon * \tilde{f})(x)}{\infty} \le
\norm{D_x^\beta \check{\phi}_\epsilon(x)}{L^1(\R^n)} \norm{\tilde{f}}{\infty}
,$$
und mit den Eigenschaften der Fourier-Transformation gegeben in Bemerkung \ref{remark:fourier} folgt
\begin{align*}
\norm{D_x^\beta \check{\phi}_\epsilon(x)}{L^1(\R^n)} &=
\norm{ \FourierB{\xi}{x}{\xi^\beta \phi_\epsilon(\xi)}}{L^1(\R^n)} = 
\epsilon^{-\abs{\beta}} \norm{ \epsilon^n \FourierB{\xi}{\frac{x}{\epsilon}}{\xi^\beta \phi(\xi)} }{L^1(\R^n)} \\
&= \epsilon^{-\abs{\beta}} \norm{ \FourierB{\xi}{x}{\xi^\beta \phi(\xi)}}{L^1(\R^n)} \le
c_\phi \epsilon^{-\abs{\beta}}
,\end{align*}
wobei die Transformation $\xi \mapsto \epsilon\xi$ bzgl. des inneren und $x \mapsto \epsilon^{-1} x$ bzgl. des äußeren Integrals verwendet wurde.
Wegen $\norm{\tilde{f}}{\infty} \le \left( 2^{-(j-1)\tau} + 2^{-j\tau} + 2^{-(j+1)\tau} \right) \norm{f}{C^\tau_*(\R^n)}$ ist insgesamt
$$
2^{j\tau} \norm{\lilwood{j}(D_x) \left( D_x^\beta \phi(\epsilon D_x) \tilde{f}(x) \right)}{\infty} \le
c_\phi' \epsilon^{-\abs{\beta}} \norm{f}{C^\tau_*(\R^n)}
$$
für ein $c_\phi' > 0$ unabhängig von $j \in \N_0$.
Mit dem Übergang zum Supremum über alle $j \in \N_0$ folgt die Behauptung.
\end{proof}
\end{lemma}

\begin{lemma}\label{sharp:c}
Sei $f \in C^\tau(\R^n), \tau > 0, \epsilon \in (0,1]$.
Dann gibt es für jeden Multiindex $\beta \in \N_0^n$ ein $C > 0$, sodass
\item
\begin{equation}
\norm{D^\beta_x J_\epsilon f}{\infty} \le 
\begin{cases}
c \norm{f}{C^\tau(\R^n)} & \abs{\beta} \le \tau , \\
c \epsilon^{-(\abs{\beta}-\tau)} \norm{f}{C^\tau_*(\R^n)} & \abs{\beta} > \tau .
\end{cases}
\end{equation}
\begin{proof}

Für $\abs{\beta} \le \tau$ folgt sofort
$$
\norm{D_x^\beta J_\epsilon f}{C^0(\R^n)} \le 
\norm{J_\epsilon f}{C^\tau(\R^n)}
\le c \norm{f}{C^\tau(\R^n)}
.$$
Wir betrachten ab nun den Fall $\abs{\beta} > \tau$. 
Nach Bemerkung \ref{smoothing:epsilon} gibt es ein $j \in \N_0$, sodass $\epsilon \in E_j$ und 
$\supp \phi \subset \bigcup_{l=0}^{j} \menge{ x \in \R^n : \lilwood{l}(x) = 1}$
\begin{align}\label{equ:smoothingSharpC}
\norm{D_x^\beta \phi(\epsilon D_x) f}{\infty} 
&= \norm{D_x^\beta \phi(\epsilon D_x) \sum_{l=0}^j \lilwood{l}(D_x) f}{\infty} \notag\\
&\le \sum_{l=0}^j \norm{D_x^\beta \phi(\epsilon D_x) \lilwood{l}(D_x) f}{\infty} \notag\\
&\le \norm{\phi(\epsilon D_x)}{\operator{X}} \sum_{l=0}^j \norm{D_x^\beta\lilwood{l}(D_x) f}{\infty}
.\end{align}
Wir finden eine Konstante $c>0$, sodass $\norm{\phi(\epsilon D_x)}{\operator{X}} \le c$.
Weiter gilt für jeden Summanden
$$
\norm{D_x^\beta \lilwood{l}(D_x) f(x)}{\infty} =
\norm{D_x^\beta \left(\check{\lilwood{l}}* f\right)(x)}{\infty} \le
\norm{D_x^\beta \check{\lilwood{l}}(x)}{L^1(\R^n)} \norm{f}{\infty}
.$$
Wir erhalten 
$$ \abs{ D_x^\beta \check{\lilwood{l}}(x)} 
= \abs{ \int_{D_l} e^{ix\cdot\xi} \xi^\beta \lilwood{j}(\xi) \dslash\xi } 
\le \norm{\xi^\beta}{\infty,D_l} \norm{\lilwood{l}}{L^1(\R^n)} 
\le 2^{l\abs{\beta}} C_{\lilwood{ }}
$$
für $\xi \in \supp \lilwood{l} \subset D_l$.
Wir verwenden nun die Identität
$$\lilwood{l}(D_x) \left( \lilwood{l-1}(D_x) + \lilwood{l}(D_x) + \lilwood{l+1}(D_x) \right) f(x) = \lilwood{l}(D_x) f(x) $$
aus Folgerung \ref{cor:introLilwoodOP}:
\begin{align*}
\norm{D_x^\beta \lilwood{l}(D_x) f}{\infty} 
&= c_{\lilwood{l}} 2^{l\abs{\beta}} \norm{\left( \lilwood{l-1}(D_x) + \lilwood{l}(D_x) + \lilwood{l+1}(D_x) \right) f}{\infty} \\
&\le c_{\lilwood{l}} 2^{l\abs{\beta}} \left( 2^{-(l-1)\tau} + 2^{-l\tau} + 2^{-(l+1)\tau} \right) \norm{f}{C^\tau_*(\R^n)}
\end{align*}
nach Definition der Norm von $C^\tau_*(\R^n)$ in (\ref{norm:C*}).
Schließlich ist 
$$
\norm{D_x^\beta \lilwood{l}(D_x) f}{\infty} \le
2^{-l\tau} 2^{l\abs{\beta}} c_{\lilwood{l}} \norm{f}{C^\tau_*(\R^n)}
.$$
Da die Summe in (\ref{equ:smoothingSharpC}) endlich ist, gibt es ein $c > 0$ mit 
$\max_{0 \le l \le j} c_{\lilwood{l}} \le c$.
Insgesamt ist
$$
\norm{D_x^\beta \phi(\epsilon D_x) f}{\infty} \le c_\tau \sum_{l=0}^j 2^{l\abs{\beta}} 2^{-l\tau} \norm{f}{C^\tau_*(\R^n)}
$$
und mit $\epsilon \in E_j$
$
\sum_{l=0}^j 2^{l (\abs{\beta} - \tau)} \le c'_\tau 2^{j(\abs{\beta} - \tau)} \le c'' \epsilon^{\abs{\beta}-\tau}
$
wegen $l (\abs{\beta} - \tau) > 0$ mit $\abs{\beta} > \tau$.
\end{proof}
\end{lemma}

\begin{cor}\label{flat:t_t}
Sei $f \in C^\tau(\R^n), \tau > 0, \epsilon \in (0,1]$.
Dann gilt für jedes $t \in [0,\tau]$, dass es eine Konstante $c>0$ gibt mit
\begin{equation}
\norm{f - J_\epsilon f}{C^t(\R^n)} \le c \norm{f}{C^t(\R^n)}.
\end{equation}
\begin{proof}
Nach Lemma \ref{sharp:c} ist mit der Dreiecksungleichung
$$
\norm{f - J_\epsilon f}{C^t(\R^n)} \le
\norm{f}{C^t(\R^n)} + \norm{J_\epsilon f}{C^t(\R^n)} \le
(1 + c) \norm{f}{C^t(\R^n)}
.$$
\end{proof}
\end{cor}

\begin{thm}\label{thm:smoothingSharp}
Für $p(x,\xi) \in C^\tau S^m_{1,\delta}(\R^n,\R^n), \gamma \in (\delta, 1)$ ist 
$$
\abs{ D_x^\beta D_\xi^\alpha p^\sharp(x,\xi) } \le
\begin{cases}
c_\beta c_\alpha \piket{\xi}{m - \abs{\alpha} + \delta\abs{\beta}} & \abs{\beta} \le \tau, \\
c_\beta c_\alpha \piket{\xi}{m - \abs{\alpha} + \delta\tau + \gamma(\abs{\beta}-\tau)} & \abs{\beta} > \tau \\
\end{cases}
$$
mit 
$\epsilon_j := 2^{-j \gamma  } \textrm{ für jedes } j \in \N_0 $.
Insbesondere erhalten wir
$$
p^\sharp(x,\xi) = \sum^\infty_{j=0} J_{\epsilon_j} p(x,\xi) \lilwood{j}(\xi) \in S^m_{1,\gamma}(\R^n\times\R^n)
.$$
\begin{proof}
Seien $\alpha, \beta \in \N^n_0$ gegeben. Wir setzen in die Definition von $p^\sharp$ ein und erhalten mit der Leibniz-Formel
$$
D_x^\beta D_\xi^\alpha p^\sharp(x,\xi) = 
\sum_{j \in \Sigma_\xi} D_x^\beta D_\xi^\alpha \left( J_{\epsilon_j} p(x,\xi) \lilwood{j}(\xi) \right) =
\sum_{j \in \Sigma_\xi} \sum_{\mu \le \alpha} {\alpha \choose \mu} D_x^\beta J_{\epsilon_j} \left( D_\xi^\mu p(x,\xi) \right) D_\xi^{\alpha - \mu} \lilwood{j}(\xi) 
.$$
Lemma \ref{sharp:c} liefert die Abschätzungen
\begin{align*}
\abs{ D_x^\beta J_{\epsilon_j} \left( D_\xi^\mu p(x, \xi) \right)} &\le
\begin{cases}
c_\beta \norm{D_\xi^\mu p(\hold,\xi)}{C^{\abs{\beta}}(\R^n)} & \textrm{ für } \abs{\beta} \le \tau, \\
c_\beta \epsilon_j^{-(\abs{\beta} - \tau)} \norm{ D_\xi^\mu p(\hold,\xi)}{C^\tau_*(\R^n)} & \textrm{ für } \abs{\beta} > \tau, \\
\end{cases}
\\
&\le
\begin{cases}
c_\beta c_\mu \piket{\xi}{m - \abs{\mu} + \abs{\beta}\delta} & \textrm{ für } \abs{\beta} \le \tau, \\
c_\beta c_\mu \piket{\xi}{m - \abs{\mu} + \tau\delta} \epsilon_j^{-(\abs{\beta} - \tau)} & \textrm{ für } \abs{\beta} > \tau. \\
\end{cases}
\end{align*}
Für $ \nu := \max \menge{ 0, \abs{\beta} - \tau}$ und $\rho_\mu := m - \abs{\mu} + \delta \min \menge{\abs{\beta}, \tau}$
erhalten wir
$$ \abs{ D_x^\beta J_{\epsilon_j} \left( D_\xi^\mu p(x, \xi) \right)} \le c_{\beta\mu} \piket{\xi}{\rho_\mu} \epsilon_j^{-\nu} .$$
Für ein fest gewähltes $\xi \in \R^n$ liefert Folgerung \ref{cor:introLilwoodOP} ein $j \in \N_0$, sodass $\xi \in \supp \lilwood{j}$.
Wir haben also 
\begin{align*}
\abs{ D_x^\beta D_\xi^\alpha p^\sharp(x,\xi) } 
&\le \abs{ \sum_{\mu \le \alpha} {\alpha \choose \mu} c_{\beta\mu} \piket{\xi}{\rho_\mu} \sum_{k=0}^{\infty} D_\xi^{\alpha-\mu} \lilwood{k}(\xi) \epsilon_k^{-\nu} } \\
&= \abs{ \sum_{\mu \le \alpha} {\alpha \choose \mu} c_{\beta\mu} \piket{\xi}{\rho_\mu} D_\xi^{\alpha-\mu} (\lilwood{j-1}(\xi) \epsilon_{j-1}^{-\nu} + \lilwood{j}(\xi) \epsilon_j^{-\nu} + \lilwood{j+1}(\xi) \epsilon_{j+1}^{-\nu}) }
.\end{align*}
Nun ist $\lilwood{j-1}(\xi) + \lilwood{j}(\xi) + \lilwood{j+1}(\xi) = 1$, also ist 
$$ \lilwood{j-1} \epsilon_{j-1}^{-\nu} + \lilwood{j} \epsilon_j^{-\nu} + \lilwood{j+1} \epsilon_{j+1}^{-\nu} \equiv \textrm{ const } \textrm{ auf } B_r(\xi) \textrm{ für } r>0 \textrm{ klein genug} .$$
Deswegen sind alle Summanden mit $\mu \neq \alpha$ gleich Null.
Stehen bleibt der Ausdruck
$$ \abs{ D_x^\beta D_\xi^\alpha p^\sharp(x,\xi) } \le c_{\beta\alpha} \piket{\xi}{\rho_\alpha} \abs{\lilwood{j-1}(\xi) \epsilon_{j-1}^{-\nu} + \lilwood{j}(\xi) \epsilon_j^{-\nu} + \lilwood{j+1}(\xi) \epsilon_{j+1}^{-\nu} } .$$
Wegen $\epsilon_j = 2^{-j(\gamma - \delta)}$ und $\supp D^{\alpha-\mu} \lilwood{j} \subset D_j$ gibt es ein $C > 0$, sodass 
$$\max_{k\in\menge{j-1,j,j+1}} \epsilon_k \le C \piket{\xi}{-(\gamma-\delta)} .$$
Damit folgt:
\begin{align*}
\abs{ D_x^\beta D_\xi^\alpha p^\sharp(x,\xi) }
&\le C c_{\beta\alpha} \piket{\xi}{\nu(\gamma-\delta)} \piket{\xi}{\rho_\alpha} \abs{ \lilwood{j-1}(\xi) + \lilwood{j}(\xi) + \lilwood{j+1}(\xi) } \\
&= C c_{\beta\alpha} \piket{\xi}{\nu(\gamma-\delta)+\rho_\alpha} \textrm{ unabhängig von } j \in \N_0
.\end{align*}
Insgesamt ergibt sich:
$$
\abs{ D_x^\beta D_\xi^\alpha p^\sharp(x,\xi) } \le
\begin{cases}
c_\beta c_\alpha \piket{\xi}{m - \abs{\alpha} + \delta\abs{\beta}} & \abs{\beta} \le \tau, \\
c_\beta c_\alpha \piket{\xi}{m - \abs{\alpha} + \delta\tau + \gamma(\abs{\beta} - \tau) } & \abs{\beta} > \tau. \\
\end{cases}
$$
Wir sehen sofort, dass $\piket{\xi}{m-\abs{\alpha}+\delta\abs{\beta}} \le \piket{\xi}{m-\abs{\alpha}+\gamma\abs{\beta}}$.
Für $\abs{\beta} > \tau$ haben wir
$$
m - \abs{\alpha} + \delta\tau + \gamma(\abs{\beta} - \tau) 
= m - \abs{\alpha} - \tau(\gamma - \delta) + \gamma\abs{\beta}
\le m - \abs{\alpha} + \gamma\abs{\beta}
.$$
Also finden wir eine Konstante $C_{\alpha\beta} > 0$, sodass wir für alle $\alpha,\beta \in \N_0^n$ das Resultat
$$
\abs{ D_x^\beta D_\xi^\alpha p^\sharp(x,\xi) } \le C_{\alpha\beta} \piket{\xi}{m - \abs{\alpha} + \gamma\abs{\beta}}
$$
erhalten. Insbesondere ist das Symbol $p^\sharp$ glatt, da $J_{\epsilon_j} p(\hold,\xi) \in \Stetigi{\infty}$.

\end{proof}
\end{thm}

\begin{lemma}\label{flat:t_tau*}
Sei $f \in C^\tau(\R^n), \tau > 0$. 
Dann ist
\begin{equation}
\norm{f - J_\epsilon f }{C^0(\R^n)} \le c_\tau \epsilon^{\tau} \norm{f}{C^\tau_*(\R^n)} .
\end{equation}
\begin{proof} 
Der Beweis verläuft analog zum Beweis von Lemma \ref{sharp:c}:
Nach Bemerkung \ref{smoothing:epsilon} gibt es ein $j \in \N_0$, sodass $\epsilon \in E_j$. Wir folgern, dass
$$
\norm{\left( 1 - \phi( \epsilon D_x) \right) f }{\infty} 
\le c_\phi \sum_{l=j}^{\infty} \norm{ \lilwood{l}(D_x) f }{\infty} 
\le c_\phi \sum_{l=j}^{\infty} 2^{-l\tau} \norm{f}{C^\tau_*(\R^n)}
,$$
wobei $\norm{\lilwood{l}(D_x) f }{\infty} \le 2^{-l\tau} \norm{f}{C^\tau_*(\R^n)}$ nach Definition der Norm von $C^\tau_*(\R^n)$
und
$$
\sum_{l=j}^{\infty} 2^{-l\tau} \le c_\tau 2^{-j\tau} \le
c_{\tau,E_j} \epsilon^\tau
,$$
da $-l\tau < 0$.
\end{proof}
\end{lemma}

\begin{thm}\label{thm:smoothingFlat}
Für $p(x,\xi) \in C^\tau S^m_{1,\delta}(\R^n,\R^n), \gamma \in (\delta, 1)$ ist 
\begin{equation}
p^\flat(x,\xi) = p(x,\xi) - p^\sharp(x,\xi) \in C^\tau S^{m-(\gamma - \delta)\tau}_{1,\gamma}(\R^n,\R^n) .
\end{equation}
\begin{proof}
Zu zeigen ist für jeden Multiindex $\alpha \in \N_0^n$, dass
\begin{enumerate}[(a)]
\item $\abs{D_\xi^\alpha p^\flat(x,\xi) } \le c_\alpha \piket{\xi}{m - \abs{\alpha} - \tau\gamma + \delta\tau}$ und
\item $\norm{D_\xi^\alpha p^\flat(\hold,\xi) }{C^t(\R^n)} \le c_\alpha \piket{\xi}{m - \abs{\alpha} - \gamma(\tau - t) + \delta\tau}\; \textrm{ für alle } t \in [0,\tau]$.
\end{enumerate}
Wenn wir $p^\flat$ umformen in
$ p^\flat(x,\xi) = \sum_{j=0}^\infty ( 1 - \phi(\epsilon_j D_x)) p(x,\xi) \lilwood{j}(\xi)$,
so sehen wir leicht, dass sich die Behauptung mit obigen Lemmata folgern lässt:
\begin{enumerate}[(a)]
\item Zunächst wenden wir für $\xi \in \N^n_0$ die Leibniz-Regel an auf
$$ \abs{ D_\xi^\alpha ( 1 - \phi(\epsilon_j D_x)) p(x,\xi) \lilwood{j}(\xi)} 
= \abs{ ( 1 - \phi(\epsilon_j D_x)) \sum_{\mu < \alpha} {\alpha \choose \mu} D_\xi^{\alpha-\mu} p(x,\xi) D_\xi^\mu \lilwood{j}(\xi)}
.$$
Für $\mu \neq 0$ konvergiert nach (\ref{lilwood:abl}) die Reihe 
$$ \sum_{j=1}^N ( 1 - \phi(\epsilon_j D_x)) D_\xi^{\alpha-\mu} p(x,\xi) D_\xi^\mu \lilwood{j}(\xi) \rightarrow 0 \textrm{ für } N \rightarrow \infty ,$$
wir müssen uns also nur den Term
$ \abs{ ( 1 - \phi(\epsilon_j D_x)) D_\xi^\alpha  p(x,\xi) \lilwood{j}(\xi) }$
beachten.
Mit Lemma \ref{flat:t_tau*} erhalten wir
\begin{align*}
\abs{ ( 1-\phi(\epsilon D_x) ) (D_\xi^\alpha p(x,\xi)) } 
&\le \norm{ ( 1-\phi(\epsilon D_x) ) (D_\xi^\alpha p(\hold,\xi)) }{C^0(\R^n)} \\
&\le c \epsilon^{\tau} \norm{D_\xi^\alpha p(\hold,\xi)}{C^\tau_*(\R^n)} 
\le c_\alpha \epsilon^{\tau} \piket{\xi}{m-\abs{\alpha}+\delta\tau}
,\end{align*}
und für $\xi \in \supp \lilwood{j}$ ist 
$\epsilon_j^{\tau} = 2^{-j\gamma\tau} \le c_j \piket{\xi}{-\gamma\tau}$, also erhalten wir
$$ \abs{ ( 1 - \phi(\epsilon_j D_x)) D_\xi^\alpha  p(x,\xi) \lilwood{j}(\xi) } \le c_\alpha \piket{\xi}{m-\abs{\alpha}-\tau(\gamma-\delta)} ,$$
und damit insgesamt
$$ \abs{ D_\xi^\alpha p^\flat(x,\xi) } \le c_\alpha \piket{\xi}{m-\abs{\alpha}-\tau(\gamma-\delta)}.$$

\item Den Fall $t = 0$ haben wir bereits gezeigt.
Für $t = \tau$ liefert Folgerung \ref{flat:t_t} bereits
$$
\norm{ ( 1-\phi(\epsilon D_x) ) (D_\xi^\alpha p(\hold,\xi)) }{C^\tau(\R^n)} 
\le c \norm{D_\xi^\alpha p(\hold,\xi)}{C^\tau(\R^n)} 
\le c_\alpha \piket{\xi}{m-\abs{\alpha}+\delta\tau}
.$$
Für alle $t < \tau$ verwenden wir Theorem \ref{interpol:cor} und kommen zu
\begin{align*}
\norm{ D_\xi^\alpha p^\flat(x,\xi) }{C^t(\R^n)} 
&\le c_{t\tau} \norm{ D_\xi^\alpha p^\flat(x,\xi) }{C^0(\R^n)}^{1-\frac{t}{\tau}} \norm{ D_\xi^\alpha p^\flat(x,\xi) }{C^\tau(\R^n)}^{\frac{t}{\tau}}  \\
&\le c_{t\tau} c_\alpha \piket{\xi}{m-\abs{\alpha}+\delta\tau\frac{t}{\tau}} 
= c_{t\tau} c_\alpha \piket{\xi}{m-\abs{\alpha}+t\delta}
.\end{align*}
\end{enumerate}
\end{proof}
\end{thm}

Im Anschluss wollen wir das hier kennengelernte Werkzeug der Glättung auch für die Klasse $C^\tau_* S^m_{\rho,\delta}(\R^n,\R^n)$ einführen,
da der Beweis sehr ähnlich verläuft.

\begin{definition}
Eine Funktion $p : \R^n \times \R^n \rightarrow \C$ gehört der Symbolklasse $C^\tau_* S^m_{\rho,\delta}(\R^n,\R^n)$ an,
falls $p(x,\hold) \in \Stetigi{\infty}$ und $p(\hold,\xi) \in C^\tau_*(\R^n)$ ist, und es für alle $\alpha \in \N^n_0$ eine Konstante $C_\alpha > 0$ gibt mit
\begin{align*}
\abs{ D_\xi^\alpha p(x,\xi) } &\le C_\alpha \piket{\xi}{m-\rho\abs{\alpha}} \textrm{ und } \\
\norm{D_\xi^\alpha p(\hold,\xi) }{C^\tau_*(\R^n)} &\le C_\alpha \piket{\xi}{m-\rho\abs{\alpha}+\delta\tau}
.\end{align*}
\end{definition}

Wiederum wollen wir nun für ein Symbol $p \in C^\tau_* S^m_{1,\delta}(\R^n,\R^n)$ mit $\delta \in [0,1), m \in \R$ in 
$p(x,\xi) = p^\sharp(x,\xi) + p^\flat(x,\xi)$
mit
$p^\sharp \in S^m_{1,\gamma}(\R^n\times\R^n)$
und
$p^\flat \in C^{\tau}_* S^{m- (\gamma - \delta) \tau}_{1,\gamma}(\R^n,\R^n)$
aufspalten. Wir erhalten mit der Beweisstrategie von Theorem \ref{thm:smoothingSharp} für $p^\sharp$ das gleiche Ergebnis wie bei den Hölder-stetigen Symbolen.
Wir müssen also nur noch die Symbolklasse von $p^\flat$ verifizieren.

\begin{lemma}\label{flat:t*_tau*}
Sei $f \in C^\tau_*(\R^n), \tau > 0, \epsilon \in (0,1]$. Dann gilt
\begin{equation}
\norm{f - J_\epsilon f }{C^{\tau - t}_*(\R^n)} \le c \epsilon^t \norm{f}{C^\tau_*(\R^n)} .
\end{equation}
für alle $\tau, \tau - t \in \Sigma^{C^._*(\R^n)} = (0,\infty)$ mit $t \ge 0$.
\begin{proof}
Seien $\tau, \tau - t \in (0,\infty)$.
Da $\menge{C^s_*(\R^n)}_{s \in (0,\infty)}$ eine mikrolokalisierbare Familie ist, ist nach Definition \ref{def:microlocal}.\ref{def:microlocal:D} der Operator
$ \piket{D_x}{t} : C^\tau_*(\R^n) \rightarrow C^{\tau - t}_*(\R^n)$
ein Isomorphismus mit der Umkehrabbildung 
$\piket{D_x}{-t} : C^{\tau-t}_*(\R^n) \rightarrow C^{\tau}_*(\R^n) $.
Wegen $\phi \in S^{-\infty}_{1,0}(\R^n)$ ist nach Lemma \ref{beschr:S_1,delta} der Operator
$\phi(\epsilon D_x) : C^{\tau - t}_*(\R^n) \rightarrow C^{\tau - t}_*(\R^n)$ 
beschränkt und damit gilt für festes $\epsilon \in (0,1]$, dass
$$
\epsilon^{-t} \piket{D_x}{-t} \left( 1 - \phi(\epsilon D_x) \right)  = 
\epsilon^{-t} \piket{D_x}{-t} -\epsilon^{-t}  \piket{D_x}{-t} \phi(\epsilon D_x) : C^{\tau - t}_*(\R^n) \rightarrow C^{\tau}_*(\R^n)
.$$
Für $\phi_\epsilon := \phi(\epsilon\xi)$ mit $\epsilon \in (0,1]$ ist
$\phi_\epsilon(\xi) = 1$ für alle $\xi \in B_{\frac{1}{\epsilon}}(0)$, 
also $\supp (1-\phi_\epsilon) \subset \R^n \setminus B_{\frac{1}{\epsilon}}(0)$.
Damit gibt es eine Konstante $c_\epsilon>0$, sodass
$\supp (1 - \phi_\epsilon) \subset \menge{ \xi \in \R^n : \piket{\xi}{} > \frac{c_\epsilon}{\epsilon}}$.
Für $\alpha \in \N^n_0$ mit $\alpha \neq 0$ gibt es wegen $D^\alpha_\xi \phi_\epsilon(\xi) \in C^\infty_0(\R^n)$ ein $R > 0$, sodass
$\supp D^\alpha_\xi \phi_\epsilon(\xi) \subset B_R(0) \setminus B_{\frac{1}{\epsilon}}(0)$.
Also gibt es ein zusätzliches $C_\epsilon > 0$, sodass
$\supp D^\alpha_\xi \phi_\epsilon(\xi) \subset \menge{ \xi \in \R^n : \frac{c_\epsilon}{\epsilon} < \piket{\xi}{} < \frac{C_\epsilon}{\epsilon} }$.

Betrachte nun $\abs{ D^\alpha_\xi \left( \piket{\xi}{-t} (1-\phi_\epsilon(\xi)) \right)}$ mit $\alpha \in \N^n_0$.
Für $\alpha = 0$ ist
$\abs{\piket{\xi}{-t} (1-\phi_\epsilon(\xi))} \le c_\phi c_\epsilon^{-t} \epsilon^t$.
Andernfalls ist
$$\abs{D^\alpha_\xi \piket{\xi}{-t} (1-\phi_\epsilon(\xi))} =
\sum_{\mu \le \alpha} {\alpha \choose \mu} D^\mu_\xi \piket{\xi}{-t} D^{\alpha-\mu}_\xi (1-\phi_\epsilon(\xi)) .$$
Für alle $\mu \in \N^n_0$ ist
$\abs{D^\mu_\xi \piket{\xi}{-t}} \le c_t \piket{\xi}{-t-\abs{\mu}}$ da $\piket{\xi}{-t} \in S^{-t}_{1,0}(\R^n) $.
Für $\mu = \alpha$ ist 
\begin{align*}
\abs{ (D^{\alpha}_\xi \piket{\xi}{-t}) (1-\phi_\epsilon(\xi)) } 
&\le c_t \piket{\xi}{-t-\abs{\alpha}} \abs{1-\phi_\epsilon(\xi) } \\
&\le c_t c_\epsilon^{-t} \epsilon^t \piket{\xi}{-\abs{\alpha}} \abs{1-\phi_\epsilon(\xi) } 
\le c_t c_\phi c_\epsilon^{-t} \epsilon^t \piket{\xi}{-\abs{\alpha}}
.\end{align*}
Ansonsten ist für $\mu < \alpha$
\begin{align*}
\abs{ D^\mu_\xi \piket{\xi}{-t} D^{\alpha-\mu}_\xi (1-\phi_\epsilon(\xi)) } 
&= \abs{ D^\mu_\xi \piket{\xi}{-t} D^{\alpha-\mu}_\xi \phi_\epsilon(\xi) }  \\
&\le c_{t\mu} \piket{\xi}{-t-\abs{\mu}} \abs{D^{\alpha-\mu}_\xi \phi_\epsilon(\xi)}  \\
&\le c_{t\mu} \piket{\xi}{-t-\abs{\mu}} \epsilon^{\abs{\alpha}-\abs{\mu}} \abs{(D^{\alpha-\mu}_\xi \phi)(\epsilon\xi)}  \\
&\le c_{t\mu} c_\epsilon^{-t} \epsilon^t C_\epsilon^{\abs{\alpha}-\abs{\mu}} \piket{\xi}{-\abs{\alpha}} \abs{(D^{\alpha-\mu}_\xi \phi)(\epsilon\xi)}  \\
&\le c_{t\mu} c_\epsilon^{-t} \epsilon^t C_\epsilon^{\abs{\alpha}-\abs{\mu}} \piket{\xi}{-\abs{\alpha}} c_\phi
.\end{align*}
Zusammenfassend ist
$$\abs{ D^\alpha_\xi \left( \piket{\xi}{-t} (1-\phi_\epsilon(\xi)) \right)} \le
c_t c_\phi c_\epsilon^{-t} \epsilon^t \piket{\xi}{-\abs{\alpha}} + 
\sum_{\mu < \alpha} {\alpha \choose \mu} c_t c_\epsilon^{-t} \epsilon^t C_\epsilon^{\abs{\alpha}-\abs{\mu}} \piket{\xi}{-\abs{\alpha}} c_\phi \le
c_\alpha \epsilon^t \piket{\xi}{-\abs{\alpha}}
.$$

Also ist für alle $\epsilon \in (0,1]$ das Symbol $\epsilon^{-t} \piket{\xi}{-t} \left( 1 - \phi(\epsilon \xi) \right) \in S^0_{1,0}(\R^n\times\R^n)$ beschränkt.
\end{proof}
\end{lemma}

\begin{cor}\label{flat:tau*_tau*}
Für $t=0$ ist
\begin{equation}
\norm{f - J_\epsilon f }{C^\tau_*(\R^n)} \le c \norm{f}{C^\tau_*(\R^n)} .
\end{equation}
\end{cor}

\begin{thm}
Für $p(x,\xi) \in C^\tau_* S^m_{1,\delta}(\R^n,\R^n), \gamma \in (\delta, 1)$ ist 
\begin{equation}
p^\flat(x,\xi) = p(x,\xi) - p^\sharp(x,\xi) \in C^\tau_* S^{m-(\gamma - \delta)\tau}_{1,\gamma}(\R^n,\R^n) .
\end{equation}
\begin{proof}
In Theorem \ref{thm:smoothingFlat} haben wir bereits die Abschätzung
$$\abs{D_\xi^\alpha p^\flat(x,\xi) } \le c_\alpha \piket{\xi}{m - \abs{\alpha} - \tau\gamma + \delta\tau}$$
bewiesen.
Der Rest folgt mit Folgerung \ref{flat:tau*_tau*}:
$$
\norm{ ( 1-\phi(\epsilon D_x) ) (D_\xi^\alpha p(\hold,\xi)) }{C^\tau_*(\R^n)} \le 
c \norm{D_\xi^\alpha p(\hold,\xi)}{C^\tau_*(\R^n)} \le
c_\alpha \piket{\xi}{m-\abs{\alpha}+\delta\tau}
.$$
\end{proof}
\end{thm}

\subsection{Beschränktheit von Pseudodifferentialoperatoren}

Sei $m \in \R, 1 < q < \infty$.

\begin{thm} \label{beschr:S_1,1}
Sei $p \in C^\tau S^m_{1,1}(\R^n,\R^n)$. Dann ist für alle $s \in (0,\tau)$
\begin{alignat}{4} \label{equ:bounded}
p(X,D_x) : & H^{s+m}_q(\R^n) &\rightarrow & H^s_q(\R^n) , \nonumber \\
           & C^{s+m}_*(\R^n) &\rightarrow & C^s_*(\R^n) 
.\end{alignat}
\begin{proof}
Siehe \cite[Theorem 2.1.A]{Nonlinear}
\end{proof}
\end{thm}

\begin{thm}\label{beschr:Stau_1,0}
Der Operator des Symbols $p(x,\xi) \in C^\tau S^m_{1,0}(\R^n,\R^n)$ erfüllt 
\begin{alignat}{4}\label{equ:boundedA}
p(X,D_x) : & H^{s+m}_q(\R^n) &\rightarrow& H^s_q(\R^n) \textrm{ für alle } s \in (-\tau,\tau) , \nonumber\\
           & C^{s+m}_*(\R^n) &\rightarrow& C^s_*(\R^n) \textrm{ für alle } s \in (0,\tau) \textrm{ falls } s+m > 0.
\end{alignat}
\begin{proof}
Sei ohne Einschränkungen $m = 0$, ansonsten betrachte
$$
q(X,D_x) := p(X,D_x) \piket{D_x}{-m} = \OPSym{ p(x,\xi) \piket{\xi}{-m} } \in \OP{ C^\tau S^0_{1,0}(\R^n,\R^n)}
.$$
Für $s \in (0,\tau)$ folgt die Behauptung bereits aus Lemma \ref{beschr:S_1,1}.
Sei also $s \in (-\tau, 0]$. 
Wende die Symbolglättung an, und erhalte für beliebiges $\gamma \in (0,1)$ die Symbole
$$p^\sharp(x,\xi) \in S^0_{1,\gamma}(\R^n\times\R^n) \textrm{ und } p^\flat(x,\xi) \in C^\tau S^{-\tau\gamma}_{1,\gamma}(\R^n,\R^n) \textrm{ mit }
p(x,\xi) = p^\sharp(x,\xi) + p^\flat(x,\xi) .$$
Nach Lemma \ref{beschr:S_1,delta} erfüllt $p^\sharp$ bereits die Abbildungseigenschaft (\ref{equ:bounded}) für jedes $s \in \R$.
Wende nun Theorem \ref{beschr:S_1,1} auf $p^\flat$ an ($m \leftarrow -\tau\gamma$) und erhalte
$$ p^\flat(X,D_x) : H^{\sigma-\tau\gamma}(\R^n) \rightarrow H^\sigma_q(\R^n) \textrm{ für alle } \sigma \in (0,\tau) .$$
Für $\gamma \nearrow 1$ erfüllt $p^\flat(x,\xi)$ die Bedingungen der Behauptung ($m \leftarrow 0, s \leftarrow (-\tau,0]$)
und damit hält $p$ die Gleichung (\ref{equ:boundedA}).
\end{proof}
\end{thm}

\begin{lemma}\label{beschr:S^0_1,delta}
Für $p(x,\xi) \in C^\tau S^0_{1,\delta}(\R^n,\R^n)$ gilt
\begin{alignat*}{4}
p(X,D_x) : & H^s_q(\R^n) &\rightarrow& H^s_q(\R^n) \textrm{ für alle } s \in \left(- (1-\delta)\tau , \tau \right) , \\
           & C^s_*(\R^n) &\rightarrow& C^s_*(\R^n) \textrm{ für alle } s \in \left( 0 , \tau \right) .
\end{alignat*}
\begin{proof}
Erhalte $p^\flat(x,\xi) \in C^\tau S^{-(\gamma-\delta)\tau}_{1,\gamma}(\R^n,\R^n)$ und $p^\sharp(x,\xi) \in S^0_{1,\gamma}(\R^n\times\R^n)$ für $\gamma \in (\delta, 1)$ durch die Symbolglättung.
Wende Theorem \ref{beschr:S_1,1} auf $p^\flat(x,\xi)$ an ($m \leftarrow -(\gamma-\delta)\tau, s \in (0,\tau$) und erhalte
\begin{align*}
p^\flat(X,D_x) : & H^{\sigma - (\gamma - \delta)\tau}_q(\R^n)  \rightarrow H^\sigma_q(\R^n)  \textrm{ für alle } \sigma \in (0,\tau) , \\
                 & C^{\sigma - (\gamma - \delta)\tau}_*(\R^n)  \rightarrow C^\sigma_*(\R^n)  \textrm{ für alle } \sigma \in \left((\gamma - \delta)\tau,\tau\right) , \\
				 & C^0_*(\R^n)  \rightarrow C^\sigma_*(\R^n) \textrm{ für alle } \sigma \in (0, (\gamma-\delta)\tau) .
\end{align*}
Wegen $(\gamma-\delta)\tau > 0$ sind nach (\ref{mikrolokal:einbettung}) die Räume $H^{\sigma}_q(\R^n)$ bzw. $C^\sigma_*(\R^n)$
in $H^{\sigma - (\gamma - \delta)\tau}_q(\R^n)$ bzw. $C^{\sigma - (\gamma - \delta)\tau}_*(\R^n)$ natürlich einbettet,
also können wir schreiben:
\begin{align*}
p^\flat(X,D_x) : & H^{\sigma - (\gamma - \delta)\tau}_q(\R^n) \rightarrow  H^{\sigma - (\gamma - \delta)\tau}_q(\R^n) \textrm{ für alle } \sigma \in (0,\tau) , \\
                 & H^{\sigma}_q(\R^n) \rightarrow H^\sigma_q(\R^n) \textrm{ für alle } \sigma \in \left(0 , \tau \right) , \\
                 & C^{\sigma}_*(\R^n) \rightarrow C^\sigma_*(\R^n) \textrm{ für alle } \sigma \in \left( 0 , \tau \right) ,
\end{align*}
und damit erfüllt $p(x,\xi)$ die Behauptung.
\end{proof}
\end{lemma}

\begin{cor}\label{beschr:S_1gamma}
Für $p(x,\xi) \in C^\tau S^m_{1,\delta}(\R^n,\R^n)$ gilt 
\begin{alignat*}{4}
p(X,D_x) : & H^{s+m}_q(\R^n) &\rightarrow& H^s_q(\R^n) \textrm{ für alle } s \in \left(- (1-\delta)\tau , \tau \right) , \\
           & C^{s+m}_*(\R^n) &\rightarrow& C^s_*(\R^n) \textrm{ für alle } s \in \left( 0 , \tau \right) .
\end{alignat*}
\begin{proof}
Das Symbol $p(x,\xi) \in C^\tau S^m_{1,\delta}(\R^n,\R^n)$ lässt sich durch zwei Symbole 
$p_1(x,\xi) \in C^\tau S^0_{1,\delta}(\R^n,\R^n), p_2(x,\xi) \in C^\tau S^m_{1,0}(\R^n,\R^n)$ mit
$p(x,\xi) = p_1(x,\xi) p_2(x,\xi)$ repräsentieren.
Wenden wir Theorem \ref{beschr:S_1,delta} auf $p_1$ und Lemma \ref{beschr:S^0_1,delta} auf $p_2$ an, erhalten wir sofort die Behauptung.
\end{proof}
\end{cor}

\subsection{Symbolkomposition im nicht-glatten Fall}\label{sec:composition}

Für die Komposition zweier Symbole haben wir bis jetzt immer vorausgesetzt, dass eines der beiden glatt sein muss.
Wir wollen nun mit Hilfe der asymptotischen Entwicklung aus Lemma \ref{thm:simplifyComposition} ein Resultat erhalten, 
falls wir die Komposition zweier nicht-glatter Symbole bilden. 
Dazu müssen wir die Entwicklung bis zum Hölder-Stetigkeitsgrad des zweiten Symbols abbrechen und den Restterm betrachen.
Wir werden feststellen, dass der Restterm ein Operator mit regulierender Eigenschaft bzgl. der Bessel-Potential-Räume ist.

\begin{konv}
Für 
$
k \in \N_0, 
p_i \in C^{\tau_i} S^{m_i}_{1,0}(\R^n,\R^n),
\tau_i > 0,
m_i \in \R, i \in \menge{1,2}$
bezeichnen wir
$$
(p_1 \#_k p_2)(x,\xi) := \sum_{\abs{\alpha} \le k} \frac{1}{\alpha!} \partial^\alpha_\xi p_1(x,\xi) D^\alpha_x p_2(x,\xi) \textrm{ für } k < \tau_2 ,
$$
und
$
R_\theta(p_1,p_2) := p_1(X,D_x) p_2(X,D_x) - (p_1 \#_{\gauss{\theta}} p_2)(X,D_x) \textrm{ für } \theta \in [0,\infty) .$
\end{konv}

Wir folgern nun die Kompositionsbedingungen für nicht-glatte Symbole aus \cite{AbelsBoundary}:
\begin{thm}
Seien 
$
p_i \in C^{\tau_i} S^{m_i}_{1,0}(\R^n,\R^n),
\tau_i > 0, m_i \in \R, i \in \menge{1,2}$.
Setze $\tau := \min\menge{ \tau_1, \tau_2 - \theta }$.
Dann gilt für alle $s \in (-\tau,\tau)$ und $\theta \in (0, \tau_2)$ mit
$s - \theta > - \tau_2$ und $-\tau_2 + \theta < s + m_1 < \tau_2$, 
dass
\begin{equation}
R_\theta(p_1,p_2) : H^{s+m_1+m_2-\theta}_q(\R^n) \rightarrow H^s_q(\R^n)
\end{equation}
ein beschränkter Operator ist.
\begin{proof}
Sei $\gamma := \frac{\theta}{\tau_2} \in (0,1)$. 
Glätte $p_2$ mit $\gamma$ und erhalte
$p_2(x,\xi) = p_2^\sharp(x,\xi) + p_2^\flat(x,\xi)$ mit 
$p_2^\sharp(x,\xi) \in S^{m_2}_{1,\gamma}(\R^n\times\R^n)$ und $p_2^\flat(x,\xi) \in C^{\tau_2} S^{m_2-\theta}_{1,\gamma}(\R^n,\R^n)$.
Genauer ist nach Theorem \ref{thm:smoothingSharp} für jeden Multiindex $\beta \in \N^n_0$
$$\partial^\beta_x p_2^\sharp(x,\xi) \in
\begin{cases}
S^{m_2}_{1,\gamma}(\R^n\times\R^n) & \abs{\beta} \le \gauss{\tau_2},\\
S^{m_2-\gamma(\tau_2-\abs{\beta})}_{1,\gamma}(\R^n\times\R^n) & \abs{\beta} > \gauss{\tau_2}.
\end{cases}
$$
Betrachte $ p_1(X,D_x) p_2(X,D_x) = p_1(X,D_x)p_2^\sharp(X,D_x) + p_1(X,D_x) p_2^\flat(X,D_x) $.
Zunächst schließen wir aus Folgerung \ref{beschr:S_1gamma}, dass
$$
p_1(X,D_x) p_2^\flat(X,D_x) : H^{s+m_1+m_2-\theta}_q(\R^n) \rightarrow H^s_q(\R^n)
$$
ein beschränkter Operator ist, da nach Voraussetzung $\abs{s} < \tau \le \tau_1$ und
$$-\tau_2(1 - \gamma) = 
-\tau_2 \left( 1-\frac{\theta}{\tau_2} \right) =
-\tau_2 + \theta < s + m_1 < \tau_2
.$$ 
Wir haben damit
\begin{align*}
& p_1(X,D_x) p_2(X,D_x) - p_1(X,D_x)p_2^\sharp(X,D_x) \\
&= p_1(X,D_x)p_2^\flat(X,D_x) : H^{s+m_1+m_2-\theta}_q(\R^n) \rightarrow H^s_q(\R^n)
.\end{align*}
Für den Operator $p_L(X,D_x) := p_1 \# p_2^\sharp(X,D_x)$ liefert Lemma \ref{thm:simplifyComposition} die asymptotische Entwicklung
$$ p_L(x,\xi) \sim \sum_{\alpha \in \N^n_0} \frac{1}{\alpha!} \partial_\xi^\alpha p_1(x,\xi) D_x^\alpha p_2^\sharp(x,\xi)
.$$
Wir wollen als nächstes zwei Symbole $r^\sharp,r^\flat$ bestimmen, sodass wir aus dieser Entwicklung die Gleichung 
$ p_L(x,\xi) = p_1\#_{\gauss{\theta}}p_2(x,\xi) + r^\sharp(x,\xi) + r^\flat(x,\xi)$ erhalten.
Glätte hierzu $p_1$ und erhalte
$p_1^\sharp \in S^{m_1}_{1.\gamma}(\R^n\times\R^n)$, $p_1^\flat \in C^{\tau_1} S^{m_1 - \gamma\tau_1}_{1,\gamma}(\R^n,\R^n)$.
Definiere
\begin{align*}
p^\sharp(X,D_x) :=& p_1^\sharp(X,D_x) p_2^\sharp(X,D_x) \in \OP{S^{m_1+m_2}_{1,\gamma}(\R^n\times\R^n)} , \textrm{ und } \\
p^\flat(X,D_x) :=& p_1^\flat(X,D_x) p_2^\sharp(X,D_x) \in 
\begin{cases}
\OP{C^\tau S^{m_1+m_2-\gamma\tau_1}_{1,\gamma}}(\R^n,\R^n) & \abs{\alpha} \le \gauss{\tau_2} ,\\
\OP{C^\tau S^{m_1+m_2-\gamma(\tau_1+\tau_2-\abs{\alpha})}_{1,\gamma}}(\R^n,\R^n) & \abs{\alpha} > \gauss{\tau_2} .
\end{cases}
\end{align*}
Für $\alpha > \gauss{\theta}$ machen wir folgende Beobachtungen:
\begin{itemize}
\item Falls sogar $\alpha > \gauss{\tau_2}$ gilt, haben wir $\partial_\xi^\alpha p_1^\sharp(x,\xi) \in S^{m_1-\abs{\alpha}}_{1,\gamma}(\R^n\times\R^n)$ und
$D_x^\alpha p_2^\sharp(x,\xi) \in S^{m_2-\gamma\tau_2+\gamma\abs{\alpha}}_{1,\gamma}(\R^n\times\R^n)$. Also ist
$$
\partial_\xi^\alpha p_1^\sharp(x,\xi) D_x^\alpha p_2^\sharp(x,\xi) \in S^{m_1+m_2-\gamma\tau_2 + \abs{\alpha}\gamma - \abs{\alpha}}_{1,\gamma}(\R^n\times\R^n) = 
S^{m_1+m_2-\theta - (1-\gamma)\abs{\alpha}}_{1,\gamma}(\R^n\times\R^n)
.$$
Ansonsten ist $D_x^\alpha p_2^\sharp(x,\xi) \in S^{m_2}_{1,\gamma}(\R^n\times\R^n)$ und damit
$$
\partial_\xi^\alpha p_1^\sharp(x,\xi) D_x^\alpha p_2^\sharp(x,\xi) \in S^{m_1+m_2-\abs{\alpha}}(\R^n\times\R^n) \subset S^{m_1+m_2-\theta}(\R^n\times\R^n)
.$$
\item Da $ \partial_\xi^\alpha p_1^\flat(x,\xi) \in C^{\tau_1} S^{m_1-\gamma\tau_1-\abs{\alpha}}_{1,\gamma}(\R^n,\R^n),
D_x^\alpha p_2^\sharp(x,\xi) \in S^{m_2 - \gamma\tau_2 + \gamma\abs{\alpha}}_{1.\gamma}(\R^n\times\R^n) $ ist
$$
\partial_\xi^\alpha p_1^\flat(x,\xi) D_x^\alpha p_2^\sharp(x,\xi) \in C^{\tau_1} S^{m_1+m_2 - \theta - \gamma\tau_1 - (1-\gamma)\abs{\alpha}}_{1,\gamma}(\R^n,\R^n)
.$$
\end{itemize}
Wegen $(1-\gamma)\abs{\alpha} > 0$ finden wir ein 
$$ r^\sharp \in S^{m_1+m_2-\theta}_{1,\gamma}(\R^n\times\R^n) \textrm{ und ein } r^\flat \in C^{\tau_1} S^{m_1+m_2-\theta-\gamma\tau_1}_{1,\gamma}(\R^n,\R^n) ,$$
sodass
$$
p_L(x,\xi) = \sum_{\abs{\alpha} \le \gauss{\theta}} \frac{1}{\alpha!} \partial_\xi^\alpha \left( p_1^\sharp(x,\xi) + p_1^\flat(x,\xi) \right) D_x^\alpha p_2^\sharp(x,\xi) +
r(x,\xi)
$$
mit $r(x,\xi) = r^\sharp(x,\xi) + r^\flat(x,\xi)$ gegeben in Theorem \ref{thm:simplifyAsmytotic}.
Nun folgt wieder mit Lemma \ref{beschr:S_1gamma}, dass für $\tilde{s} \in \R$ mit
\begin{itemize}
\item $-(1-\gamma)\tau_1 < \tilde{s} < \tau_1$
$$
r^\flat(X,D_x) : H^{\tilde{s}+m_1+m_2-\theta}_q(\R^n) \subset H^{\tilde{s}+m_1+m_2-\theta-\gamma\tau_1}_q(\R^n) \rightarrow H^{\tilde{s}}_q(\R^n)
,$$
und
\item $-(1-\gamma)\tau_1 = -\tau_1+\gamma\tau_1 < \tilde{s} + \gamma\tau_1 < \tau_1$ 
$$
r^\flat(X,D_x) : H^{\tilde{s}+m_1+m_2-\theta}_q(\R^n) \rightarrow H^{\tilde{s}+\gamma\tau_1}_q(\R^n) \subset H^{\tilde{s}}_q(\R^n)
.$$
\end{itemize}
Damit gilt für alle
$\abs{\tilde{s}} < \tau_1$, dass
$ r^\flat(X,D_x) : H^{\tilde{s}+m_1+m_2-\theta}_q(\R^n) \rightarrow H^{\tilde{s}}_q(\R^n) $.
Lemma \ref{beschr:S_1,delta} liefert
$$
r^\sharp(X,D_x) : H^{\tilde{s}+m_1+m_2-\theta}_q(\R^n) \rightarrow H^{\tilde{s}}_q(\R^n) \textrm{ für alle } \abs{\tilde{s}} < \tau_1
.$$
Schließlich wollen wir die Symbole der Summe $p_1\#_{\gauss{\theta}}p_2(x,\xi)$ klassifizieren.
Betrachte dazu einen Multiindex $\alpha \in \N^n_0$ mit $\abs{\alpha} \le \gauss{\theta} \le \gauss{\tau_2}$:
Spalte zunächst auf:
$$
\partial_\xi^\alpha p_1(x,\xi) D_x^\alpha p_2^\sharp(x,\xi) =
\partial_\xi^\alpha p_1(x,\xi) D_x^\alpha p_2(x,\xi) -
\partial_\xi^\alpha p_1^\sharp(x,\xi) D_x^\alpha p_2^\flat(x,\xi) -
\partial_\xi^\alpha p_1^\flat(x,\xi) D_x^\alpha p_2^\flat(x,\xi)
.$$
Es gilt nun, die Beschränktheit der letzten beiden Terme zu zeigen:
\begin{itemize}
\item Wegen $\partial_\xi^\alpha p_1^\sharp(x,\xi) \in S^{m_1-\abs{\alpha}}_{1,\gamma}(\R^n\times\R^n)$ und 
$D_x^\alpha p_2^\flat(x,\xi) \in C^{\tau_2-\abs{\alpha}} S^{m_2-\theta+\gamma\abs{\alpha}}_{1,\gamma}(\R^n,\R^n)$
ist
$$
\partial_\xi^\alpha p_1^\sharp(x,\xi) D_x^\alpha p_2^\flat(x,\xi) \in C^{\tau_2-\abs{\alpha}} S^{m_1+m_2 - \theta - (1-\gamma)\abs{\alpha}}_{1,\gamma}(\R^n,\R^n)
.$$
\item Wir erhalten für 
$$\partial_\xi^\alpha p_1^\flat(x,\xi) \in C^{\tau_1} S^{m_1-\gamma\tau_1 - \abs{\alpha}}_{1,\gamma}(\R^n,\R^n) \textrm{ und }
D_x^\alpha p_2^\flat(x,\xi) \in C^{\tau_2-\abs{\alpha}} S^{m_2 - \theta + \gamma\abs{\alpha}}_{1,\gamma}(\R^n,\R^n)$$
die Abbildungseigenschaft
$$
\partial_\xi^\alpha p_1^\flat(x,\xi) D_x^\alpha p_2^\flat(x,\xi) \in C^{\min\left\{ \tau_1, \tau_2 - \abs{\alpha}\right\} } S^{m_1+m_2-\theta-\gamma\tau_1-(1-\gamma)\abs{\alpha}}_{1,\gamma}(\R^n,\R^n)
.$$
Nun ist wegen $\min\left\{ \tau_1, \tau_2 - \abs{\alpha}\right\} \ge \min\left\{ \tau_1, \tau_2 - \gauss{\theta}\right\} = \tau$
und $(1-\gamma)\abs{\alpha} > 0$ bereits
\begin{align*}
C^{\min\left\{ \tau_1, \tau_2 - \abs{\alpha}\right\} } S^{m_1+m_2-\theta-\gamma\tau_1-(1-\gamma)\abs{\alpha}}_{1.\gamma}(\R^n,\R^n) 
&\subset C^\tau S^{m_1+m_2-\theta-\gamma\tau_1-(1-\gamma)\abs{\alpha}}_{1,\gamma}(\R^n,\R^n) \\
&\subset C^\tau S^{m_1+m_2-\theta-\gamma\tau_1}_{1,\gamma}(\R^n,\R^n)
.\end{align*}
\end{itemize}
Betrachte nun $\OPi\left( \partial_\xi^\alpha p_1^\sharp(x,\xi) D_x^\alpha p_2^\flat(x,\xi) \right)$. 
Für $\tilde{s} \in \R$ gilt
\begin{itemize}
\item mit $-(1-\gamma)(\tau_2-\abs{\alpha}) < \tilde{s} < \tau_2 - \abs{\alpha}$
$$
\OPi\left( \partial_\xi^\alpha p_1^\sharp(x,\xi) D_x^\alpha p_2^\flat(x,\xi) \right) : H^{\tilde{s}+m_1+m_2-\theta}_q(\R^n) \subset H^{\tilde{s}+m_1+m_2-\theta-(1-\gamma)\abs{\alpha}}_q(\R^n) \rightarrow H^{\tilde{s}}_q(\R^n)
,$$
\item und mit $-(1-\gamma)(\tau_2-\abs{\alpha}) < \tilde{s} + (1-\gamma)\abs{\alpha} < \tau_2 - \abs{\alpha}$
$$
\OPi\left( \partial_\xi^\alpha p_1^\sharp(x,\xi) D_x^\alpha p_2^\flat(x,\xi) \right) : H^{\tilde{s}+m_1+m_2-\theta}_q(\R^n) \rightarrow H^{\tilde{s}+(1-\gamma)\abs{\alpha}}_q\subset H^{\tilde{s}}_q(\R^n)
.$$
\end{itemize}
Aus $-(1-\gamma)(\tau_2-\abs{\alpha}) \le -(\tau_2-\theta)+(1-\gamma)\abs{\alpha} \le \tau + (1-\gamma)\abs{\alpha}$
erhalten wir für alle $\abs{\tilde{s}} < \tau ( \le \tau_2 - \theta \le \tau_2 - \abs{\alpha} )$ die Abbildungseigenschaft
$$
\OPi\left( \partial_\xi^\alpha p_1^\sharp(x,\xi) D_x^\alpha p_2^\flat(x,\xi) \right) : H^{\tilde{s}+m_1+m_2-\theta}_q(\R^n) \rightarrow H^{\tilde{s}}_q(\R^n)
.$$
Für $\OPi\left( \partial_\xi^\alpha p_1^\flat(x,\xi) D_x^\alpha p_2^\flat(x,\xi) \right)$ und $\tilde{s} \in \R$ mit
\begin{itemize}
\item $-(1-\gamma)\tau < \tilde{s} < \tau$ ist
$$
\OPi\left( \partial_\xi^\alpha p_1^\flat(x,\xi) D_x^\alpha p_2^\flat(x,\xi) \right) : H^{\tilde{s}+m_1+m_2-\theta}_q(\R^n) \subset H^{\tilde{s}+m_1+m_2-\theta-\gamma\tau_1}_q(\R^n) \rightarrow H^{\tilde{s}}_q(\R^n)
,$$
\item bzw. mit $-(1-\gamma)\tau < \tilde{s} + \tau_1\gamma < \tau$ ist
$$
\OPi\left( \partial_\xi^\alpha p_1^\flat(x,\xi) D_x^\alpha p_2^\flat(x,\xi) \right) : H^{\tilde{s}+m_1+m_2-\theta}_q(\R^n) \rightarrow H^{\tilde{s}+\gamma\tau_1}_q(\R^n) \subset H^{\tilde{s}}_q(\R^n)
.$$
\end{itemize}
Insgesamt gilt also für alle $\abs{\tilde{s}} < \tau$, dass
$$
\OPi\left( \partial_\xi^\alpha p_1^\flat(x,\xi) D_x^\alpha p_2^\flat(x,\xi) \right) : H^{\tilde{s}+m_1+m_2-\theta}_q(\R^n) \rightarrow H^{\tilde{s}}_q(\R^n)
.$$
Für alle $\abs{\tilde{s}} < \tau$ ist 
$
( p_1 \#_{\gauss{\theta}} p_2)(X,D_x) : H^{\tilde{s}+m_1+m_2-\theta}_q(\R^n) \rightarrow H^{\tilde{s}}_q(\R^n)
$
und damit
$$
R_\theta(p_1,p_2) : H^{s+m_1+m_2-\theta}_q(\R^n) \rightarrow H^s_q(\R^n)
.$$
\end{proof}
\end{thm}
 
\section{Glatte Koordinatentransformation}
\subsection{Symbole in $(x,\xi,y)$-Form}

\begin{definition}
Sei $0 \le \delta \le \rho \le 1, \delta < 1$.
Ein Symbol $p : \R^n \times \R^n \times \R^n \rightarrow \C$ gehört der Klasse $C^\tau S^{m}_{\rho,\delta}(\R^n, \R^n \times \R^n)$ an, 
falls $p(x,\xi,y)$ zum Grade $\tau$ Hölder-stetig bezüglich $x\in\R^n$ und glatt bezüglich $(y,\xi) \in \R^{2n}$ ist, 
d.h. $p(\hold,\xi,y) \in C^\tau(\R^n)$ und $p(x,\hold,\hold) \in \Stetig{\infty}{\R^n\times\R^n}$,
und es zu jedem Multiindex $\alpha,\beta\in\N^n_0$ eine Konstante $C_{\alpha\beta} > 0$ unabhängig von $y,\xi\in\R^n$ gibt, sodass
$$ \norm{ \partial_\xi^\alpha \partial_y^\beta p(\hold,\xi,y) }{C^\tau(\R^n)} \le C_{\alpha\beta} \piket{\xi}{m-\rho\abs{\alpha}+\delta(\tau+\abs{\beta})}, $$
$$ \abs{ \partial_\xi^\alpha \partial_y^\beta p(x,\xi,y) } \le C_{\alpha\beta} \piket{\xi}{m-\rho\abs{\alpha}+\delta\abs{\beta}} \textrm{ für alle } x \in \R^n. $$
Den zugehörigen Operator $p(X,D_x,Y) : C^\infty_0(\R^n) \rightarrow C^\tau(\R^n)$ definieren wir durch
\begin{equation}\label{equ:operatorXXiY}
p(X,D_x,Y) u(x) := \osint e^{i(x-y)\cdot\xi} p(x,\xi,y) u(y) dy \dslash\xi ,
\end{equation}
und bezeichnen diesen Operator als Pseudodifferentialoperator in $(x,\xi,y)$-Form.
Wir fassen diejenigen Symbole, die insbesondere bezüglich $x$ glatt sind, unter der Klasse
$$ S^{m}_{\rho,\delta}(\R^n \times \R^n \times \R^n) = \bigcap_{\tau > 0} C^\tau S^{m}_{\rho,\delta}(\R^n, \R^n \times \R^n) $$
zusammen.
\end{definition}
\nomenclature[s]{$C^\tau S^m_{\rho,\delta}(\R^n, \R^n \times \R^n)$}{Klasse der Symbole in $(x,\xi,y)$-Form mit Hölder-Stetigkeitsgrad $\tau$}
\nomenclature[s]{$S^m_{\rho,\delta}(\R^n \times \R^n \times \R^n)$}{Klasse glatter Symbole in $(x,\xi,y)$-Form}

\begin{remark}
Der Operator eines Symbols $p(x,\xi,y) \in C^\tau S^m_{\rho,\delta}(\R^n, \R^n \times \R^n)$ definiert durch (\ref{equ:operatorXXiY})
besitzt die gleichen Abbildungseigenschaften wie ein Operator in $x$-Form.
\begin{proof}
Wir erhalten analog zu Lemma \ref{lemma:introOszi} für alle $u \in \Beschri$ die Idenfizierung
$$ p(X,D_x,Y) u(x) = \osint e^{-ix'\cdot\eta} p(x,\xi, x+x') u(x+x') dx' \dslash\xi = p_L(X,D_x) u(x) ,$$
wobei $p_L$ das vereinfachte Symbol aus Theorem \ref{thm:simplify} ist. 
Also übertragen sich die Abbildungseigenschaften des Operators $p_L(X,D_x)$ in $x$-Form auf den Operator in $(x,\xi,y)$-Form.
\end{proof}
\end{remark}

\begin{remark}
Im besonderen Fall $p(x,\xi,y) = p(x,\xi) \in C^\tau S^m_{\rho,\delta}(\R^n,\R^n)$ erhalten wir für $u \in \Schwarzi$ wegen 
$(\xi,y) \mapsto p(x,\xi,y) \in \Oszillatory{\delta,0}{m}$ für $x \in \R^n$ fest und wegen $\xi \mapsto p(x,\xi) \hat{u}(\xi) \in L^1(\R^n)$
aus dem oszillatorischen Integral (\ref{equ:operatorXXiY}) nach Folgerung \ref{cor:osziLebesgue} den bekannten Operator in $x$-Form
$$ p(X,D_x,X')u(x) = \int e^{ix\cdot\xi} p(x,\xi) \hat{u}(\xi) \dslash\xi = p(X,D_x) u(x) .$$
\end{remark}

\begin{lemma}
Für ein Symbol $p \in C^\tau S^m_{\rho,\delta}(\R^n, \R^n \times \R^n)$ ist das Symbol
$$ \partial_\xi^\alpha \partial_y^\beta \partial_x^\mu p(x,\xi,y) \in C^{\tau-\abs{\mu}} S^{m-\rho{\abs{\alpha}}+\delta(\abs{\beta}+\abs{\mu})}_{\rho,\delta}(\R^n, \R^n \times \R^n)$$
für alle Multiindizes $\alpha,\beta,\mu \in \N^n_0$ mit $\abs{\mu} \le \tau$.
\begin{proof} Folgt analog zu Folgerung \ref{cor:IntroAbleitung}.
\end{proof}
\end{lemma}

\subsection{Kerndarstellung von Pseudodifferentialoperatoren}\label{sec:kern}

Eine Aussage aus der Theorie der nuklearen Räume ist die Identifizierung des Raumes der beschränkten linearen Abbildungen von $\Schwarzi$ nach $\SchwarziDual$ mit
dem Dualraum $\SchwarzDual{\R^n\times\R^n}$, beschrieben in \cite[Collorary after Theorem 51.6]{Treves}.
Wir finden also zu jeder beschränkten linearen Abbildung $P : \Schwarzi \rightarrow \BeschrZ \subset \SchwarziDual$ ein eindeutig bestimmtes
$k \in \SchwarzDual{\R^n\times\R^n}$, sodass 
\begin{equation}\label{equ:schwartzKern}
\dualskp{P u}{v}{\SchwarziDual}{\Schwarzi} = \dualskp{k}{u \otimes v}{\SchwarzDual{\R^n\times\R^n}}{\Schwarz{\R^n\times\R^n}} .
\end{equation}
Wir bezeichnen $k$ als Schwartz-Kern des Operators $P$.
\cite[Kapitel 6, Paragraph 4]{Stein} folgert, dass der Kern des Operators eines Symbols $p(x,\xi) \in S^m_{1,0}(\R^n\times\R^n)$ sogar unter gewissen Voraussetzungen eine glatte Funktion 
$k : \R^n \times \R^n \rightarrow \C$ ist.
Für Symbole der Klasse $C^\tau S^{m}_{\rho,\delta}(\R^n, \R^n \times \R^n)$ wollen wir ähnlich wie in \cite{AbelsCauchy} die Kerndarstellung genauer studieren.

\begin{definition}\label{def:kernAssoz}
Sei $p$ ein Symbol in $(x,\xi,y)$-Form. Der zu $p$ assoziierte Kern ist definiert durch
\begin{equation}\label{equ:kernFourier}
k(x,y,\hold) := \FourierB{\xi}{\hold}{p(x,\xi,y)} \in \SchwarziDual .
\end{equation}
Wir spalten $p$ mit einer Littlewood-Paley-Partition $\menge{\lilwood{j}}_{j\in\N}$ in die Symbole
$p_j(x,\xi,y) := p(x,\xi,y)\lilwood{j}(\xi)$ auf, sodass wir analog zu Lemma \ref{lemma:smoothing} die Konvergenz $\sum_j^N p_j(x,\xi,y) \rightarrow p(x,\xi,y)$ für $N \rightarrow \infty$ erhalten.
Wiederum definieren wir den assoziierten Kern des Symbols $p_j$ durch
$k_j(x,y,z) := \FourierB{\xi}{z}{p_j(x,\xi,y)}$. 
\end{definition}

\begin{remark}\label{remark:kern-n-1}
Für $p \in C^\tau S^{-n-1}_{\rho,\delta}(\R^n, \R^n \times \R^n)$ ist $\xi \mapsto p(x,\xi,y) \in L^1(\R^n)$. 
Mit Folgerung \ref{cor:osziLebesgue} und dem Satz von Fubini folgt für alle $u \in \Schwarzi$
\begin{align*}
p(X,D_x,Y) u(x) 
&= \iint e^{i(x-y)\cdot\xi} p(x,\xi,y) u(y) dy \dslash\xi 
= \int \left( \int e^{i(x-y)\cdot\xi} p(x,\xi,y) \dslash\xi \right) u(y) dy  \\
&= \int k(x,y,x-y) u(y) dy
.\end{align*}
Insbesondere erfüllt $(x,y) \mapsto k(x,y,x-y)$ die Gleichung (\ref{equ:schwartzKern}), d.h. $(x,y) \mapsto k(x,y,x-y)$ ist der Schwartz-Kern des Operators $p(X,D_x,Y)$.
\end{remark}

\begin{lemma}\label{lemma:kernlilwood}
Sei $p \in C^\tau S^m_{\rho,\delta}(\R^n,\R^n\times\R^n)$. Definiere die Kerne $k_j$ wie in Definition \ref{def:kernAssoz}.
Dann gibt es für jedes $\alpha,\beta \in \N^n_0$ und jedes $M \in \N_0$ eine von $j$ unabhängige Konstante $C_{\alpha\beta M} > 0$, 
sodass für jedes $j \in \N_0$ gilt 
$$ \norm{ \partial_z^\alpha \partial_y^\beta k_j(\hold,y,z) }{C^\tau(\R^n)} \le
C_{\alpha\beta M} \abs{z}^{-M} 2^{j(n+m+\abs{\alpha} + \delta(\tau+\abs{\beta}) - \rho M)} \textrm{ für alle } z \neq 0
.$$ \begin{proof}
Seien $\alpha, \beta, \mu \in \N^n_0$ mit $\abs{\mu}\le\tau$ und $M \in \N_0$.
Eine einfache Umformung mit partieller Integration unter Ausnutzung, dass $\xi \mapsto p_j(x,\xi,y)$ kompakt getragen ist, ergibt
$$
z^\gamma \partial_x^\mu \partial_y^\beta D^\alpha_z k_j(x,y,z) =
(-1)^{\abs{\gamma}} \int_{\R^n} e^{iz\cdot\xi} D_\xi^\gamma \left[ \xi^\alpha \partial_x^\mu \partial_y^\beta p_j(x,\xi,y) \right] \dslash\xi
.$$
Wir wollen nun das Volumen des Trägers des Integranten abschätzen. Dazu nutzen wir, dass
$
\supp \left(\xi \mapsto  e^{ix\cdot\xi} D_\xi^\gamma \left[ \xi^\alpha \partial_x^\mu \partial_y^\beta p_j(x,\xi,y) \right] \right)
\subset \menge{ 2^{j-1} \le \abs{\xi} \le 2^{j+1}}$
(o.B.d.A. für $j \neq 0$)
und
$$ \vol \menge{ 2^{j-1} \le \abs{\xi} \le 2^{j+1}} \le 2^{(j+1)n} \vol B_1(0) =: 2^{jn} C
.$$
Sei ohne Einschränkung $\xi \in \supp \lilwood{j}$, da $\supp \xi \mapsto p_j(x,\xi,y) \subset \supp \lilwood{j}$.
Wir sehen leicht, dass 
$\xi^\alpha \partial_x^\mu \partial_y^\beta p_j(x,\xi,y)$ der Klasse $C^\tau S^{m+\delta\abs{\beta+\mu}+\abs{\alpha}}_{\rho,\delta}(\R^n, \R^n \times \R^n)$
angehört. Deswegen erhalten wir die Abschätzung
\begin{align*}
\abs{ D_\xi^\gamma \left[ \xi^\alpha \partial_x^\mu \partial_y^\beta p_j(x,\xi,y) \right] } 
&\le C_{\alpha\beta\gamma} \piket{\xi}{m+\abs{\alpha}+\delta\abs{\beta+\mu}-\rho\abs{\gamma}} \\ &\le C_{\alpha\beta\gamma} c 2^{j (m+\delta\abs{\beta+\mu}+\abs{\alpha}-\rho\abs{\gamma}) } ,\end{align*}
wobei es $c,c' >0$ gibt, sodass  $ c' 2^j \le \piketi{\xi} \le c 2^j$ für jedes $\xi \in D_j = \menge{ 2^{j-1} \le \abs{\xi} \le 2^{j+1}}$ gilt.
Insgesamt ist also
$$ \abs{ z^\gamma \partial_x^\mu \partial_y^\beta \partial_\xi^\alpha k_j(x,y,z) } \le 
C_{\alpha\beta\gamma} c 2^{jn} 2^{j (m+\delta\abs{\beta+\mu}+\abs{\alpha}-\rho\abs{\gamma}) } .$$
Mit dem Übergang zum Supremum über alle $x \in \R^n$ mit $\abs{\gamma} = M$ und $\abs{\mu} \le \tau$ erhalten wir
$$ \norm{ \partial^\alpha_z \partial_y^\beta k_j(\hold,y,z) }{C^{\gauss{\tau}}(\R^n)} \le
C_{\alpha\beta M}  \abs{z}^{-M} 2^{j(n+m+\abs{\alpha} + \delta\abs{\beta+\mu} - \rho M)} .$$
Für $\theta := \tau - \gauss{\tau} \neq 0$ erhalten wir mit nach der Definition der Norm $\norm{\hold}{C^\theta(\R^n)}$ für ein $\gamma \in \N^n_0$ und $x'\in\R^n$ mit $x' \neq x$
\begin{align*}
\abs{ \partial_\xi^\gamma \left[ \xi^\alpha \partial_y^\beta ( p_j(x,\xi,y) - p_j(x',\xi,y) ) \right] } 
&\le \norm{\partial_\xi^\gamma \left[ \xi^\alpha \partial_y^\beta p_j(\hold,\xi,y) \right] }{C^\theta(\R^n)} \abs{x-x'}^\theta  \\
&\le C_{\alpha\beta\gamma} \piket{\xi}{m-\rho\abs{\gamma}+\abs{\alpha}+\delta(\theta+\abs{\beta})}  \abs{x-x'}^\theta .\end{align*}
Für die Abschätzung der Hölder-Halbnorm betrachten wir die Differenz
$$ z^\gamma \partial_x^\mu \partial_y^\beta D^\alpha_z ( k_j(x,y,z) - k_j(x',y,z)) =
\int_{\R^n} e^{iz\cdot\xi} D_\xi^\gamma \left[ \xi^\alpha \partial_x^\mu \partial_y^\beta ( p_j(x,\xi,y) - p_j(x',\xi,y)) \right] \dslash\xi$$
und erhalten mit analogen Vorgehen
$$ \frac{ \abs{ D^\alpha_z \partial_y^\beta ( k_j(x,y,z) - k_j(x',y,z))  } }{ \abs{x-x'}^\theta } \le 
C_{\alpha\beta \gamma} 2^{j(n+m-\rho\abs{\gamma}+\abs{\alpha}+ \delta(\theta+\abs{\beta}))} .$$
Mit dem Übergang zum Supremum über alle $x,x'\in\R^n$ mit $x \neq x'$ erhalten wir die Behauptung.
\end{proof}
\end{lemma}

\begin{cor}
Die Kerne $k_j : \R^n \times \R^n \times (\R^n\setminus\menge{0}) \rightarrow \C,  (x, y, z) \mapsto k_j(x,y,z)$ sind glatt bzgl. $(y,z) \in \R^{2n}$ und Hölder-stetig zum Grad $\tau$ bzgl. $x\in\R^n$.
Insbesondere fassen wir die Kerne $k_j$ nicht als temperierte Distribution auf.
Denn für jedes $u \in \Schwarzi$ erhalten wir analog zu Bemerkung \ref{remark:kern-n-1}, dass
\begin{equation}\label{equ:kernLilwood}
p_j(X,D_x,Y) u(x) = \int k_j(x,y,x-y) u(y) dy \textrm{ für alle } x \not\in \supp u ,
\end{equation} 
da $\xi \mapsto p_j(x,\xi,y)$ kompakt getragen ist. 
\end{cor}

\begin{thm}\label{thm:kern}
Sei $p \in C^\tau S^m_{\rho,\delta}(\R^n\times\R^n,\R^n)$ und $0 \le \delta \le \rho \le 1, \delta < 1$ und $\rho > 0$.
Dann gibt es einen Kern $k : \R^n \times \R^n \times \left(\R^n\setminus\menge{0}\right) \rightarrow \C,$ sodass wir für jedes $u \in \Schwarzi$ den Operator von $p$ durch
$$p(X,D_x,Y) u(x) = \int k(x,y,x-y) u(y) dy \textrm{ für alle } x \not\in \supp u$$
darstellen können
und für alle $\alpha,\beta\in\N^n_0$, $N\in\N_0$ mit $n+m+\abs{\alpha}+\delta(\tau+\abs{\beta})+N > 0$ eine Konstante $C_{\alpha\beta N} > 0$ existiert mit
\begin{equation}\label{equ:kernAbsch}
\norm{ \partial_z^\alpha \partial_y^\beta k(\hold,y,z) }{C^\tau(\R^n)} \le
C_{\alpha\beta N}  \abs{z}^{-\frac{1}{\rho}(n+m+\abs{\alpha}+\delta(\tau+\abs{\beta}))} \piket{z}{-\frac{N}{\rho}} .
\end{equation}
\begin{proof}
Nach (\ref{equ:kernLilwood}) und (\ref{equ:kernFourier}) ist
$$
p(X,D_x,Y) u(x) = \sum_{j=0}^\infty \int_{\R^n} k_j(x,y,x-y) u(y) dy \textrm{ für alle } x \not\in \supp u .$$
Zu zeigen ist, dass $\sum^\infty_{j=0} k_j(x,y,z)$ für alle $(x,y,z) \in \R^n \times \R^n \times ( \R^n \setminus B_\epsilon(0) )$ absolut und gleichmäßig gegen eine Funktion $k(x,y,z)$ konvergiert, die die Behauptung erfüllt.

Sei zunächst $0 < \abs{z} \le 1$.
Zerteile
$$ \sum_{j=0}^\infty \norm{\partial^\alpha_z \partial_y^\beta k_j(\hold,y,z)}{C^\tau(\R^n)} =
\sum_{2^j \le \abs{z}^{-1}} \norm{\partial^\alpha_z \partial_y^\beta k_j(\hold,y,z)}{C^\tau(\R^n)} +
\sum_{2^j > \abs{z}^{-1}} \norm{\partial^\alpha_z \partial_y^\beta k_j(\hold,y,z)}{C^\tau(\R^n)} .$$
Der erste Term wird mit Lemma \ref{lemma:kernlilwood} (mit $M \leftarrow 0$) abgeschätzt:
$$ \sum_{2^j \le \abs{z}^{-1}} \norm{\partial^\alpha_z \partial_y^\beta k_j(\hold,y,z)}{C^\tau(\R^n)} \le
C_{\alpha\beta} \sum_{j=0}^{\gauss{\log_2 \left(\abs{z^{-1}}\right) }} 2^{j(n+m+\abs{\alpha}+\delta(\tau+\abs{\beta}))} ,$$
und
$$\sum_{j=0}^{\gauss{\log_2 \left(\abs{z^{-1}}\right) }} 2^{j(n+m+\abs{\alpha}+\delta(\tau+\abs{\beta}))} \le
\begin{cases}
 C \abs{z}^{-n-m-\abs{\alpha}-\delta(\tau+\abs{\beta})} & \textrm{ für }  n+m+\abs{\alpha}+\delta(\tau+\abs{\beta}) > 0, \\
 \left( 1 + \log_2\left(\abs{z}^{-1}\right) \right)     & \textrm{ für }  n+m+\abs{\alpha}+\delta(\tau+\abs{\beta}) = 0, \\
 C                                                      & \textrm{ für }  n+m+\abs{\alpha}+\delta(\tau+\abs{\beta}) < 0. 
\end{cases}
$$
Im Vergleich zu Polynomen ist $\Landau{1 + \log_2\left(\abs{z}^{-1}\right)} \le \Landau{\abs{z}^{-r}}$ für alle $z$ mit $\abs{z}<1$ und jedes $r\in\R_+$.
Insbesondere ist deswegen $\Landau{1 + \log_2\left(\abs{z}^{-1}\right)} \le \Landau{\abs{z}^{-n-m-\abs{\alpha}-\delta(\tau+\abs{\beta})}}$.
Wir können also den ersten Term durch
$$ \sum_{j=0}^{\gauss{\log_2\left(\abs{z^{-1}}\right)  }} 2^{j(n+m+\abs{\alpha}+\delta(\tau+\abs{\beta}))} \le
\Landau{\abs{z}^{-n-m-\abs{\alpha}-\delta(\tau+\abs{\beta})}} \le 
\Landau{\abs{z}^{\frac{1}{\rho}(-n-m-\abs{\alpha}-\delta(\tau+\abs{\beta}))}}$$
abschätzen.

Für den zweiten Term verwende wieder Lemma \ref{lemma:kernlilwood} mit $\rho M > n+m+\abs{\alpha}+\delta(\tau+\abs{\beta})$.
Wir setzen dazu $M := \frac{2}{\rho} \left( n+m+\abs{\alpha}+\delta(\tau+\abs{\beta}) \right)$ und erhalten
$$
\sum_{2^j > \abs{z}^{-1}} \norm{\partial^\alpha_z \partial_y^\beta k_j(\hold,y,z)}{C^\tau(\R^n)} \le
C_{\alpha\beta M} \abs{z}^{-M} \sum_{2^j > \abs{z}^{-1}} 2^{j(n+m+\abs{\alpha}+\delta(\tau+\abs{\beta}) - \rho M)} .$$
Für $2^j > \abs{z}^{-1}$ ist $2^{j(n+m+\abs{\alpha}+\delta(\tau+\abs{\beta}) - \rho M)} \le \abs{z}^{-n-m-\abs{\alpha}-\delta(\tau+\abs{\beta})+\rho M}$.
Also ist 
\begin{align*}
\abs{z}^{-M} \sum_{2^j > \abs{z}^{-1}} 2^{j(n+m+\abs{\alpha}+\delta(\tau+\abs{\beta}) - \rho M)} 
&\le C_{\alpha\beta} \abs{z}^{-M(1-\rho)} \abs{z}^{-n-m-\abs{\alpha}-\delta(\tau+\abs{\beta})} \\
&\le C_{\alpha\beta} \abs{z}^{ - (\frac{2}{\rho}-1) \left(n+m+\abs{\alpha}+\delta(\tau+\abs{\beta}) \right)}  \\
&\le C_{\alpha\beta} \abs{z}^{ - \frac{1}{\rho} \left(n+m+\abs{\alpha}+\delta(\tau+\abs{\beta}) \right)} 
.\end{align*}
Unter Einbeziehung $\piket{z}{-\frac{N}{\rho}} = \Landau{1}$ für $\abs{z} < 1$ ist insgesamt
$$ \sum_{j=0}^\infty \norm{\partial^\alpha_z \partial_y^\beta k_j(\hold,y,z)}{C^\tau(\R^n)} \le  
C_{\alpha\beta N} \piket{z}{-\frac{N}{\rho}} \abs{z}^{ - \frac{1}{\rho} \left(n+m+\abs{\alpha}+\delta(\tau+\abs{\beta}) \right)}
.$$ 
Für $\abs{z} > 1$ wählen wir $M\rho > n+m+\abs{\alpha}+\delta(\tau+\abs{\beta})+N$ in Lemma \ref{lemma:kernlilwood} und schließen daraus, dass
\begin{align*}
\sum_{j=0}^\infty \norm{\partial^\alpha_z \partial_y^\beta k_j(\hold,y,z)}{C^\tau(\R^n)} 
&\le C_{\alpha\beta} \abs{z}^{-M} \sum_{j=0}^\infty 2^{j( n+m+\abs{\alpha}+\delta(\tau+\abs{\beta})-M\rho)} \\
&\le C_{\alpha\beta N} C \abs{z}^{-M}
.\end{align*}
Schließlich finden wir eine Konstante $C>0$, sodass $\abs{z}^{-\frac{N}{\rho}} \le C \piket{z}{-\frac{N}{\rho}}$ für $\abs{z} > 1$ und damit
\begin{align*}
\abs{z}^{-M} 
&\le \abs{z}^{-\frac{1}{\rho} (n+m+\abs{\alpha}+\delta(\tau+\abs{\beta})+N)} 
 \le \abs{z}^{-\frac{1}{\rho}(n+m+\abs{\alpha}+\delta(\tau+\abs{\beta})+N)} \\
&\le C \abs{z}^{-\frac{1}{\rho}(n+m+\abs{\alpha}+\delta(\tau+\abs{\beta}))} \piket{z}{-\frac{N}{\rho}}
.\end{align*}

Also konvergiert $\sum_{j=0}^\infty k_j(x,y,z)$ absolut und gleichmäßig bezüglich $(x,y,\xi) \in \R^n \times \R^n \times ( \R^n \setminus B_\epsilon(0) )$ für alle $\epsilon > 0$ gegen eine Funktion $k(x,y,z)$, 
die die gewünschten Eigenschaften erfüllt.
Nach der Definition der gleichmäßigen Konvergenz und der Kerne $k_j$ folgt die Behauptung für alle $x \in \R^n$ mit
$\dist{x}{\supp u} > \epsilon$ für beliebige $\epsilon > 0$.
\end{proof}
\end{thm}

\begin{remark}
Die Bedingung $x \not\in \supp u$ ist notwendig, da für ein Symbol $p(x,\xi,y) = 1 \in S^0_{1,0}(\R^n,\R^n\times\R^n)$ 
$$ 
\dualskp{p(X,D_x,Y) u}{v}{\SchwarziDual}{\Schwarzi} = \skp{u}{v}{L^2(\R^n)} = \dualskp{ k }{u \otimes v}{\SchwarzDual{\R^n\times\R^n}}{\Schwarz{\R^n\times\R^n}}
$$
gilt, falls formal ``$k(x,y) := \delta_0(x-y)$'' ist. Hierbei bezeichne $\delta_0$ die Delta-Distribution gegeben durch
$ \dualskp{\delta_0}{u}{\SchwarziDual}{\Schwarzi} = u(0)$ für alle $u \in \Schwarzi$.
Für $x \in \supp u$ können wir also im Allgemeinen den Schwartz-Kern nur im distributionellen Sinn auffassen.
\end{remark}

\begin{definition}\label{def:kern}
Sei eine Funktion $k : \R^n \times \R^n \times \R^n \rightarrow \C$ mit $k(\hold,y,z) \in C^\tau(\R^n)$ und $k(x,\hold,\hold) \in \Stetig{\infty}{\R^n\times\R^n}$ gegeben,
für die es zu allen Multiindizes $\alpha,\beta \in \N^n_0$ und jedem $N \in \N_0$ eine Konstante $C_{\alpha\beta N}$ gibt, für die die Ungleichung
\begin{equation}\label{equ:kernglatt} 
\sup_{y,z \in \R^n} \piket{z}{N} \norm{\partial_z^\alpha \partial_y^\beta k(\hold,y,z) }{C^\tau(\R^n)} < C_{\alpha\beta N} \textrm{ für alle } \alpha,\beta \in \N^n_0, N \in \N_0
\end{equation}
hält.
Falls sich zudem der Operator eines Symbol $p$ durch
$$ p(X,D_x,Y) u(x) = \int k(x,y, x-y) u(y) dy \textrm{ für jedes } u \in \Schwarzi \textrm{ und alle } x \in \R^n $$ darstellen lässt, sagen wir, dass $p$ eine $C^\tau$-Kerndarstellung (mit $k$) besitzt.
\end{definition}

\begin{lemma}\label{lemma:kernRankNegativ}
Für $p \in C^\tau S^m_{\rho,\delta}(\R^n\times\R^n,\R^n)$ und $0 \le \delta \le \rho \le 1, \delta < 1$ und $\rho > 0$ 
gibt es einen Kern $k : \R^n \times \R^n \times \left(\R^n\setminus\menge{0}\right) \rightarrow \C$ mit
$$ p(X,D_x,Y) u(x) = \int k(x,y, x-y) u(y) dy \textrm{ für jedes } u \in \Schwarzi \textrm{ und alle } x \in \R^n , $$
der die Bedingung (\ref{equ:kernAbsch}) erfüllt.
\begin{proof}
Setze  $M := \frac{N+1}{\rho}$ in Lemma \ref{lemma:kernlilwood} und erhalte
$$ \norm{k_j(\hold,\hold,z) }{C^0(\R^n\times\R^n)} \le C \abs{z}^{-\frac{n+1}{\rho}} 2^{j(m-1)} .$$
Also konvergiert
$$ \norm{ \sum_{j=0}^\infty \norm{k_j(\hold,\hold,z)}{C^0(\R^n\times\R^n)} }{L^1(\R^n)} \le C \norm{ \abs{z}^{-\frac{n+1}{\rho}} }{L^1(\R^n)} \le C' ,$$
und damit erhalten wir mit Hilfe der dominanten Konvergenz
\begin{align*}
\sum_{j=0}^N p_j(X,D_x,Y) u(y) 
&= \sum_{j=0}^N \int k_j(x,y,x-y) u(y) dy  \\
&= \int \sum_{j=0}^N k_j(x,y,x-y) u(y) dy  \\
&\rightarrow \int k(x,y,x-y) u(y) dy \textrm{ für } N \rightarrow \infty 
.\end{align*}
Hierbei ist das letzte Integral wohldefiniert, da $z \mapsto k(x,y,z) \in L^1(\R^n)$.
\end{proof}
\end{lemma}

\begin{cor}\label{cor:kernglatt}
Ein Symbol $p(x,\xi,y) \in C^\tau S^m_{\rho,\delta}(\R^n,\R^n\times\R^n)$ mit $\rho > 0$ besitzt genau dann eine $C^\tau$-Kerndarstellung,
wenn $m = -\infty$.
Wir erhalten also eine Bijektion
$$ \menge{ \textrm{Kerne, die } (\ref{equ:kernglatt}) \textrm{ erfüllen} } \longleftrightarrow C^\tau S^{-\infty}(\R^n,\R^n\times\R^n) .$$
\begin{proof}
Sei $p\in C^\tau S^{-\infty}(\R^n,\R^n\times\R^n)$.
Direktes Einsetzen in Theorem \ref{thm:kern} liefert einen Kern $k(x,y,z) : \R^n \times \R^n \times \left(\R^n\setminus\menge{0}\right) \rightarrow \C$, 
der die Bedingungen aus Definition \ref{def:kern} erfüllt.
Wir erhalten mit Bemerkung \ref{remark:kern-n-1} eine Funktion $K(x,y,\hold) :=  \FourierB{\xi}{\hold}{p(x,\xi,y)}$ mit
$$p(X,D_x,Y) u(x) = \int K(x,y, x-y) u(y) dy \textrm{ für jedes } u \in \Schwarzi \textrm{ und alle } x \in \R^n .$$
Andererseits sind nach Lemma \ref{lemma:kernRankNegativ} die Darstellungen von $p(X,D_x,Y)$ durch $k$ als auch durch $K$ identisch.
Deswegen ist $k(x,y,z) = K(x,y,z)$ für alle $x,y,z \in \R^n$ mit $z \neq 0$. Wir können also $k$ glatt in $z = 0$ fortsetzen.

Sei umgekehrt ein Kern $k$ gegeben, der die Eigenschaft einer $C^\tau$-Kerndarstellung aus Definition \ref{def:kern} erfüllt.
Setzen wir $p(x,\xi,y) := \Fourier{z}{\xi}{k(x,y,z)} = \int_{\R^n} e^{-iz\cdot\xi} k(x,y,z) dz$, so induziert der Kern
$k$ das Symbol $p$, das in der Klasse $C^\tau S^{-\infty}(\R^n,\R^n\times\R^n)$ liegt:
Denn für $\alpha,\beta \in \N^n_0$ wähle ein $N \in \N$ mit $N > \abs{\alpha}+n$, sodass
\begin{align*}
&\sup_{y,\xi \in \R^n} \norm{ \xi^\gamma D_\xi^\alpha \partial_y^\beta p(\hold,y,\xi) }{C^\tau(\R^n)}  \\
&\le \sup_{y,z \in \R^n} \norm{ \int_{\R^n} e^{-iz\cdot \xi} D^\gamma_z \left[ z^\alpha \partial_y^\beta k(\hold,y,z) \right] dz }{C^\tau(\R^n)} \\
&\le \sup_{y,z \in \R^n} \sum_{\mu \le \gamma} 
	{\gamma \choose \mu} \norm{ \int_{\R^n} e^{-iz\cdot \xi} D_z^\mu z^\alpha \piket{z}{-N} \piket{z}{N} D_z^{\gamma-\mu} \partial_y^\beta k(\hold,y,z) dz}{C^\tau(\R^n)} 
.\end{align*}
Nun ist $\sup_{y,z \in \R^n} \norm{\piket{z}{N} D_z^{\gamma-\delta} \partial_y^\beta k(\hold,y,z) }{C^\tau(\R^n)} < C_{\beta,\gamma-\delta,N} $ 
und
$ z \mapsto e^{-iz\cdot \xi} D_z^\mu z^\alpha \piket{z}{-N} \in L^1(\R^n)$ da $\deg(D^\mu_z z^\alpha) - N < -n$ für alle $\mu \le \gamma$.
Insgesamt gibt es also ein $C_{\alpha\beta\gamma} > 0$, sodass
$$\sup_{y,\xi \in \R^n} \norm{ \xi^\gamma D_\xi^\alpha \partial_y^\beta p(\hold,y,\xi) }{C^\tau(\R^n)} \le C_{\alpha\beta\gamma} .$$
Analog erhalten wir Aussage, dass
$$\abs{ \xi^\gamma D_\xi^\alpha \partial_y^\beta p(x,y,\xi) } \le C_{\alpha\beta\gamma} \textrm{ für alle } x,\xi,y \in \R^n .$$

\end{proof}
\end{cor}

\begin{thm}\label{thm:kerndisjunkt}
Seien $\varphi, \psi \in C^\infty(\R^n)$, sodass eine Konstante $\epsilon > 0$ existiert mit 
$$\dist{\supp \varphi}{\supp \psi} \ge \epsilon .$$
Dann gilt für jeden Pseudodifferentialoperator $p(X, D_x, X') \in \OP{C^\tau S^m_{\rho,\delta}}(\R^n,\R^n\times\R^n), \rho > 0$ in $(x,\xi,y)$-Form, dass 
$$\varphi(X) p(X,D_x,X') \psi(X') \in \OP{C^\tau S^{-\infty}(\R^n,\R^n\times\R^n)} .$$
Insbesondere folgt aus Korollar $\ref{cor:kernglatt}$, dass $\varphi(X) p(X,D_x,X') \psi(X')$ eine $C^\tau$-Kerndarstellung hat.
\begin{proof}
Setze $Q(X,D_x,X',D_{x'}) := \varphi(X) p(X,D_x,X') \psi(X') \in \OP{ C^\tau S^{m,0}_{\rho,\delta}}(\R^n,\R^n\times\R^n\times\R^n)$.
Das Doppelsymbol des Operator $Q(X,D_x,X',D_{x'})$ induziert sein vereinfachtes Symbol $q_L(x,\xi) \in C^\tau S^m_{\rho,\delta}(\R^n\times\R^n)$ gegeben in Theorem \ref{thm:simplify} durch
\begin{align*}
q_L(x,\xi) 
&= \osint e^{-iy\cdot\eta} \varphi(x) p(x, \xi + \eta,x+y) \psi(x+y) dy \dslash\eta \\
&= \osint  e^{-iy\cdot\eta} \varphi(x) \abs{y}^{-2N} \psi(x+y) (-\triangle_\eta)^N p(x, \xi+\eta, x+y) dy \dslash\eta
,\end{align*}
wobei 
$(-\triangle_\eta) e^{-iy\cdot\eta} = \abs{y}^2 e^{-iy\cdot\eta}$
für
$\triangle_\eta := \sum_{j=1}^n \frac{\partial^2}{\partial\eta_j^2} $.
Die obige Umformung ist wegen $ \menge{ y \in \R^n : \abs{y} < \epsilon} \subset \menge{y \in \R^n : \varphi(x) \psi(x+y) = 0 \textrm { für alle } x \in \R^n}$ wohldefiniert.
Wir haben $$\OP{ C^\tau S^{m,0}_{\rho,\delta}(\R^n,\R^n\times\R^n\times\R^n)} = \OP{ C^\tau S^m_{\rho,\delta}(\R^n,\R^n)} ,$$
also muss der Operator $Q$ mit einem Operator mit dem Doppelsymbol
$$q(x,\xi,x',\xi') := \varphi(x) \abs{x-x'}^{-2N} \psi(x') (-\triangle_\xi)^N p(x,\xi,x')$$ übereinstimmen.
Wegen $\dist{\supp \varphi}{\supp \psi} \ge \epsilon$ gibt es ein $C>0$, sodass
$$C^{-1} \piketi{x-x'} \le \abs{x-x'} \le C \piketi{x-x'} \textrm{ für alle } x \in \supp \varphi, x' \in \supp \psi .$$
Demnach haben wir
$$\varphi(x) \abs{x-x'}^{-2N} \psi(x') (-\triangle_\xi)^N p(x,\xi,x') \in C^\tau S^{m-2\rho N,0}_{\rho,\delta}(\R^n,\R^n\times\R^n\times\R^n) .$$
Insgesamt ist $\varphi(X) p(X,D_x,X') \psi(X') \in \OP{ C^\tau S^{m-2\rho N}_{\rho,\delta}(\R^n,\R^n\times\R^n)}$ für beliebiges $N \in \N$.
Für $N \rightarrow \infty$ erhalten wir $$\varphi(X) p(X,D_x,X') \psi(X') \in \OP{ C^\tau S^{-\infty}(\R^n,\R^n\times\R^n)} .$$
Nach Voraussetzung ist $\menge{ (x,x) : x \in \R^n} \subset \menge{(x,y) \in \R^n \times \R^n : \varphi(x) k(x,y,x-y) \psi(y) = 0}$
und damit gilt
$$ p(X,D_x,X') u(x) = \int \varphi(x) k(x,y,x-y) \psi(y) u(y) dy \textrm{ für jedes } x \in \R^n .$$
Schließlich liefert Folgerung \ref{cor:kernglatt}, dass das Symbol des Operators $\varphi(X) p(X,D_x,X') \psi(X')$ eine $C^\tau$-Kerndarstellung besitzt.
\end{proof}
\end{thm}

\subsection{Reguläre Diffeomorphimen}

Basierend auf \cite{Kumanogo} und \cite{Marschall} wollen wir eine Regel für eine Transformation von Pseudodifferentialoperatoren definieren.
Dazu betrachten wir im folgenden für $t \in \R$ mit $t \ge 1$ einen $C^t$-Diffeomorphismus $h : \Omega_y \rightarrow \Omega_x$, 
der auf offene Teilmengen $\Omega_y, \Omega_x \subseteq \R^n$ definiert ist.
Unter der kanonischen Basis $\menge{e_k}_{k=1,\ldots,n}$ des $\R^n$ ist
$$h(y) := \left( h_1(y), \ldots , h_n(y) \right) \textrm{ mit } $$
$$h_j(y) := h_j(y_1, \ldots, y_n) \in C^\infty(\Omega_y) \textrm{ für alle } y = (y_1, \ldots, y_n) \in \Omega_y .$$
Die Jacobimatrix bezeichnen wir mit
$$ \derive{h}{y} := \menge{ \partial_{e_k} g_j(y) }_{k\rightarrow 1,\ldots,n}^{j\downarrow 1,\ldots,n} .$$

\begin{definition}
Ein $C^t$-Diffeomorphismus $h : \Omega_y \rightarrow \Omega_x$ auf offene Teilmengen $\Omega_y, \Omega_x \subseteq \R^n$ heißt regulär,
falls $h \in \Beschr{1}{\Omega_y}$
und es eine Konstante $C > 0$ gibt, sodass
\begin{equation}\label{equ:trafoUngl}
C^{-1} \le \abs{\det \derive{h}{y} } \le C \textrm{ für alle } y \in \Omega_y .
\end{equation}
\end{definition}

\begin{remark}
Falls $h$ regulär ist, so ist auch $h^{-1}$ ein regulärer $C^t$-Diffeomorphismus, da $\derive{h}{y} \cdot \derive{g}{h(y)} = id$.
Wir finden demnach eine Konstante $\tilde{C} > 0$, sodass
$$ \tilde{C}^{-1} \le \abs{ \det \derive{g}{x} } \le \tilde{C} .$$
Damit folgt bereits, dass es ein $C_1 > 0$ gibt, sodass
\begin{equation}\label{equ:trafoEquiv}
C_1^{-1} \abs{y - \tilde{y}} = C_1^{-1} \abs{h(g(y)) - h(g(\tilde{y}))} \le \abs{ g(y) - g(\tilde{y}) } \le C_1 \abs{y - \tilde{y}} \textrm{ für alle } y, \tilde{y} \in \Omega_y .
\end{equation} 
\end{remark}

\subsection{Invarianz unter glatter Koordinatentransformation} \label{sec:smoothkoord}
In diesem Unterabschnitt sei $h$  stets ein glatter regulärer Diffeomorphismus (setze $t := \infty$).

Für $\tilde{\Omega}_y \subsubset \Omega_y$, $\tilde{\Omega}_x := h(\tilde{\Omega}_y)$ betrachten wir die Komposition
eines Symbols $p(x,\xi,x') \in C^\tau S^m_{\rho,\delta}(\R^n,\R^n\times\R^n)$ mit zwei glatten Funktionen $\varphi, \psi \in C^\infty_0(\tilde{\Omega}_x)$ mit der Abbildungseigenschaft
$$ P := \varphi(X) p(X,D_x,X') \psi(X') : C^\infty_0(\Omega_x) \rightarrow C^\tau(\Omega_x) .$$
Wir wollen einen Pseudodifferentialoperator $A : C^\infty_0(\Omega_y) \rightarrow C^\tau(\Omega_y)$ konstruieren, der für alle $w := u\circ h \in C^\infty_0(\Omega_y)$ die Gleichung
$ A w(y) = P \left(u(y)\right)$ 
erfüllt.
Wir werden feststellen, dass sich $A$ bis auf die Restklasse $\OP{C^\tau S^{-\infty}(\R^n,\R^n\times\R^n)}$ eindeutig durch ein Symbol
$a(x,\xi,y) \in C^\tau S^m_{\rho,\delta}(\R^n,\R^n\times\R^n)$ mit
$$A \equiv \varphi(h(Y)) a(Y,D_y, Y') \psi(h(Y')) \mod \OP{C^\tau S^{-\infty}(\R^n,\R^n\times\R^n)} $$
charakterisieren lässt.
Verwenden wir die Pullback-Schreibweise, so lässt sich unsere Bedingung in 
$ \left( h^{-1} \right)^* A h^* = P$
umformulieren.
Hierbei ist
$$ \left( \left( h^{-1} \right)^* A h^* \right)(u)(x) =  A (u\circ h) (h^{-1}(x)) \textrm{ für alle } u \in C^\infty_0(\Omega_x), x \in \Omega_x .$$
Genauer gesagt, soll der Operator $A$ das Diagramm
$$
\begin{xy}
\xymatrix{
	C^\infty_0(\Omega_x) \ar[r]^P \ar[d]_{h^*}  & C^\tau(\Omega_x) \ar[d]_{h^*} \\
	C^\infty_0(\Omega_y) \ar[r]^A  & C^\tau(\Omega_y) 
}
\end{xy}
$$
kommutieren lassen.

Zunächst beschränken wir uns auf den Fall $\tilde{\Omega}_x = \Omega_x = \Omega_y = \R^n$ und wollen für $\varphi, \psi = 1$ einen \PSDO $A : C^\infty_0(\R^n) \rightarrow C^\tau(\R^n)$ finden, für den die Gleichung
\begin{equation}\label{equ:trafoPostulateRn}
A (u \circ h) (y) = \left( p(X,D_x,X') u \right) (h(y)) \textrm{ für alle } u \in \Schwarzi, y \in \R^n
\end{equation}
hält.

\begin{lemma}\label{lemma:trafoSinfty}
Sei $p(x,\xi,y) \in C^\tau S^{-\infty}(\R^n, \R^n\times\R^n)$. 
Der zu $p$ assoziierte Kern $k_p$ 
definiert das Symbol
$$a(y,\eta) := \int e^{-iz\cdot\eta} k_p(h(y), h(y-z), h(y) - h(y-z)) \abs{ \det \derive{h}{y-z} } dz \in C^\tau S^{-\infty}(\R^n, \R^n) .$$
$a$ induziert uns einen Operator $A \in \OP{C^\tau S^{-\infty}(\R^n, \R^n)}$, der (\ref{equ:trafoPostulateRn}) erfüllt.
\begin{proof}
Da $p(x,\xi,x')$ der Klasse $C^\tau S^{-\infty}(\R^n,\R^n\times\R^n)$ angehört, hat $p$ nach Korollar \ref{cor:kernglatt} eine $C^\tau$-Kerndarstellung gegeben durch
$$k_p(x,y,z) := \int e^{iz\cdot\xi} p(x,\xi,y) \dslash\xi ,$$
sodass sich der Operator von $p$ darstellen lässt durch
$$p(X,D_x,X') u(x) = \int k_p(x,\tilde{x}, x-\tilde{x}) u(\tilde{x}) d\tilde{x} \textrm{ für alle } u \in \Schwarzi .$$
Also ist für $w = u \circ h$ mit der Transformationsformel
$$A w(y) = p(X,D_x,X') u(h(y)) = 
\int k_p(h(y), h(\tilde{y}), h(y) - h(\tilde{y})) \abs{ \det \derive{h}{\tilde{y}} } w(\tilde{y}) d\tilde{y} .$$
Setzen wir
$ k_a(y,z) := k_p(h(y), h(y-z), h(y) - h(y-z)) \abs{ \det \derive{h}{y-z} }$,
so erhalten wir eine Darstellung des Operators $A$ definiert durch den Kern $k_a$:
$$ A w(y) = \int k_a(y, y - \tilde{y}) w(\tilde{y}) d\tilde{y} .$$
Schließlich wollen wir noch zeigen, dass $k_a(y,z)$ die Bedingung (\ref{equ:kernglatt}) erfüllt, also $A \in \OP{C^\tau S^{-\infty}(\R^n, \R^n)}$ ist.
Für den eher technischen Beweis wollen wir uns auf ein Symbol in $x$-Form $p \in C^\tau S^{-\infty}(\R^n, \R^n)$ beschränken. 
Der Fall für ein Symbol in $(x,y)$-Form folgt dann analog mit den selben Mitteln.

Nach (\ref{equ:trafoUngl}) ist $\abs{ \det \derive{h}{y-\eta} } \le C$, 
und (\ref{equ:trafoEquiv}) liefert $$ \abs{\eta} C_1^{-1} \le \abs{ h(y) - h(y - \eta) } \le C_1 \abs{\eta} .$$

Wir setzen $\Xi_h(y-z,y) := - \int_0^1 \derive{h}{y - \theta z} d\theta$, sodass $\Xi_h(y-z,y) z = h(y) - h(y-z) $ und betrachten unter der Notation aus Lemma \ref{faaDiBruno} mit $\alpha,\beta \in \N^n_0, \abs{\beta} \le \tau$ die Abschätzung 
\begin{align*}
&\abs{ D_y^\beta D_z^\alpha k_p(h(y), \Xi_h(y-z,y) z) }  \\
&\le 
\Landau{ D_y^\beta \sum_{\sigma \in \N^n_{0,\abs{\beta}}} k^{(\sigma)}_p(h(y), \Xi_h(y-z,y) z)
\sum_{\Sigma(\alpha,\sigma)} \prod_{j=1}^{\abs{\alpha}}
\frac{1}{\nu^j! (\mu^j!)^{\abs{\nu^j}}} \left( D_z^{\mu^j} \left[ \Xi_h(y-z,y) z) \right] \right)^{\nu^j} }
,\end{align*}
wobei $k_{p(\beta)}^{(\alpha)}(y,z) := D_z^\alpha D_y^\beta k_p(y,z)$, und wir annehmen, dass $\alpha,\beta \neq 0$ (ansonsten können die ensprechenden Schritte entfallen).
Für ein $\sigma\in\N^n_0$ betrachte
\begin{align*}
&\abs{ D_y^\beta  k^{(\sigma)}_p(h(y), \Xi_h(y-z,y) z) }  \\
&\le
\Landaui \left( \sum_{\rho \in \N^{2n}_{0,\abs{\beta}}}  k^{(\sigma+\rho_2)}_{p(\rho_1)}(h(y), \Xi_h(y-z,y) z) \right. \\
&\quad \left.
\sum_{\Sigma(\beta,\rho)} \prod_{j=1}^{\abs{\beta}}
\frac{1}{\nu^j! (\mu^j!)^{\abs{\nu^j}}} \left( D_y^{\mu^j} \menge{ y \mapsto (h(y), \Xi_h(y-z,y) z) } \right)^{\nu^j} \right)
,\end{align*}
wobei $\rho_1,\rho_2\in\N^n_0$ mit $(\rho_1,\rho_2) = \rho$.
Mit den Abschätzungen
$$\abs{ D_y^\vartheta D_z^{\mu^j} \left[ (y,z) \mapsto \Xi_h(y-z,y) z \right] }\le C_{\vartheta,\mu^j} \piket{z}{\abs{\mu^j}+\abs{\vartheta}} \textrm{ für ein beliebiges } \vartheta\in\N^n_0
$$ und 
$\abs{D_y^{\mu^j} \menge{ y \mapsto (h(y), \Xi_h(y-z,y) z) } } \le C_{\mu^j} \piket{z}{\abs{\mu^j}}$
erhalte
$$ \abs{ D_y^\beta D_z^\alpha k_p(h(y), \Xi_h(y-z,y) z) } \le
\Landau{ \sum_{\sigma \in \N^n_{0,\abs{\beta}}} \sum_{\rho \in \N^n_{0,\abs{\beta}}} k^{(\sigma+\rho_2)}_{p(\rho_1)}(h(y), \Xi_h(y-z,y) z) \piket{z}{\abs{\alpha}+\abs{\beta}} }
.$$
Da $k_p(y,z)$ der Ungleichung (\ref{equ:kernglatt}) genügt, können wir ein $N \ge \abs{\alpha}+\abs{\beta}+M$ für ein $M \in \N_0$ wählen, und folgern, dass
$$ \abs{ D_y^\beta D_z^\alpha k_p(h(y), \Xi_h(y-z,y) z) } \le C_{\alpha,\beta,N} \piket{z}{-M} .$$
Mit dem Übergang zum Supremum über alle $y \in \R^n$ erhalten wir
$$ \norm{ k_a(\hold,z) }{C^\tau(\R^n)} \le C_{\alpha,\tau,N} \piket{z}{-M} .$$
Also genügt auch $k_a(y,z)$ der Ungleichung (\ref{equ:kernglatt}). 
\end{proof}
\end{lemma}

\nomenclature{$\Landaui$}{Landau-Symbol für eine asymptotische obere Schranke}

\begin{figure}[ht]
\caption{Träger von $\menge{\left( \varphi_j, \psi_j \right)}_{j\in\N}$ aus Lemma \ref{koord:lemmaPartition} für $r = 1, n = 2$}
\centering
\begin{tikzpicture}
\draw[very thin,color=gray] (0,0) grid(10,10);
\draw[->] (5,0) -- (5,10.2) node[above] {$y$};
\draw[->] (0,5) -- (10.2,5) node[right] {$x$};
\foreach \pos in {-1,1}
	\draw[shift={(\pos+5,5)}] (0pt,2pt) -- (0pt,-2pt) node[below] {$\pos$};
\foreach \pos in {-1,1}
	\draw[shift={(5,\pos+5)}] (2pt,0pt) -- (-2pt,0pt) node[left] {$\pos$};

\draw[densely dotted] (5,5) circle(1.4142135623730950488);
\draw[densely dotted] (5,5) circle(1.4142135623730950488*2);
\draw[densely dotted] (5,5) circle(1.4142135623730950488*3);
\draw (5,5) circle(1) node[anchor=center, fill=white] {$\varphi_0$};
\draw[dashed] (6,6) circle(1) node[anchor=center, fill=white] {$\varphi_1$};
\draw (5,5) circle(3.5);
\draw (5,5+3.5) node[anchor=south, fill=white] {$\psi_0$};
\draw (5-1.4142135623730950488*3,5) node[fill=white,anchor=east] {$B_{3\sqrt{2}}(0)$};
\draw (5-1.4142135623730950488*2,5) node[anchor=center] {$B_{2\sqrt{2}}(0)$};
\draw (5,-1.4142135623730950488+5) node[fill=white,anchor=north] {$B_{\sqrt{2}}(0)$};
\end{tikzpicture}
\end{figure}
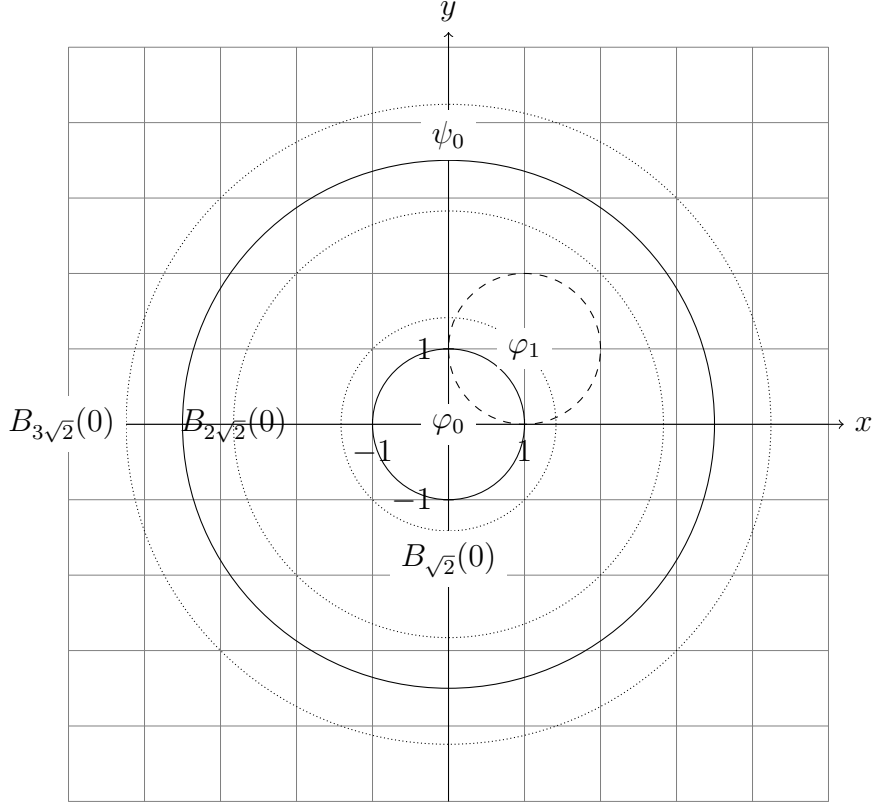

\begin{lemma}\label{koord:lemmaPartition}
Sei $\Gamma := \menge{ \gamma_j}_{j\in\N} := \Z^n \subset \R^n$ das Gitter der ganzen Zahlen auf $\R^n$.
Für jedes $r \in \R_+$ gibt es eine Familie $\menge{\left( \varphi_j, \psi_j \right)}_{j\in\N}$ von $C^\infty_0(\R^n)$-Funktionen auf $\R^n$ mit den Eigenschaften
\begin{enumerate}[(a)]
\item $ \sum_{j=0}^\infty \varphi_j(x) = 1 \textrm{ für alle } x \in \R^n$,
\item $ \supp \varphi_j \subset B_{r\sqrt{n}}(r\gamma_j) \subset \supp \psi_j \subset B_{3r\sqrt{n}}(r\gamma_j)$,
\item $ \psi(x) = 1 \textrm{ für alle } x \in B_{2\sqrt{n}}(r\gamma_j) $,
\item $ \abs{\partial_x^\beta \varphi_j(x)} \le c_{\varphi,\beta} r^{-\abs{\beta}} \textrm{ und }
\abs{\partial_x^\beta \psi_j(x)} \le c_{\psi,\beta} r^{-\abs{\beta}} \textrm{ für alle } x \in \R^n $,
\end{enumerate}
wobei $ c_{\varphi,\beta}$ und $ c_{\psi,\beta}$ unabhängig von $j$ und $r$ sind.
\begin{proof}
Seien $\varphi, \psi \in C^\infty_0(\R^n)$ mit $\varphi \ge 0$, $\psi \ge 0$ und
$$ \varphi > 0 \textrm{ auf } B_{\frac{2}{3}\sqrt{n}}(0) \textrm{ und } \psi = 1 \textrm{ auf } B_{2\sqrt{n}}(0) ,$$
sowie
$ \supp \varphi \subset B_{\sqrt{n}}(0) \subset \supp \psi \subset B_{3\sqrt{n}}(0)$
vorgegeben.
Wegen 
$$\inf_{j \in \N, i \neq j} \dist{\gamma_i}{\gamma_j} \le \sqrt{n} \textrm{ zu jedem } i \in \N ,$$
und wegen der äquidistanten Verteilung des Gitters gilt bereits
$\bigcup_{j\in\N} B_d(\gamma_j) = \R^n$ für jedes $d > \frac{\sqrt{n}}{2}$.
Wegen $\supp \varphi \subset B_{\sqrt{n}}(0)$ existiert ein $l \in \N$, sodass
$$ M(x) := \menge{k \in \N : \varphi( \frac{x}{r} - \gamma_k) \neq 0} < l \textrm{ gleichmäßig in } x \in \R^n .$$
Da $B_{\frac{2}{3}\sqrt{n}}(0) \subset \supp \varphi$ ist $M(x) > 0$ für alle $x \in \R^n$.
Setzen wir demnach für $r > 0$
$$ \psi_j(x) := \psi(\frac{x}{r} - \gamma_j) \textrm{ und } \varphi_j(x) := \frac{ \varphi( \frac{x}{r} - \gamma_j) }{ \sum_{k=1}^\infty \varphi( \frac{x}{r} - \gamma_k) } ,$$
so sind $\varphi_j, \psi_j$ wohldefiniert und erfüllen die gewünschten Eigenschaften.
\end{proof}
\end{lemma}

\begin{thm}\label{thm:trafoRn}
Sei $h : \R^n \rightarrow \R^n$ ein regulärer Diffeomorphismus und $0 \le 1-\rho \le \delta \le \rho, \delta < 1$.
Dann ist für $p(x, \xi, x') \in C^\tau S^m_{\rho,\delta}(\R^n, \R^n\times\R^n)$
der Operator 
$$A : \Schwarzi \rightarrow \BeschrZ, \left(A (u \circ h)\right)(y) = (P u)(h(y))$$
ein Pseudodifferentialoperator der Klasse $\OP{C^\tau S^m_{\rho,\delta}(\R^n,\R^n\times\R^n)}$.
Für $r > 0$ genügend klein ist das Symbol in $(x,y)$-Form
\begin{equation}\label{equ:trafoq}
a(y,\eta,y') = \sum_{j=1}^\infty \varphi_j(h(y)) p\left(h(y), \Xi_h(y,y')^{-T} \eta, h(y')\right) \psi_j(h(y')) \abs{\det \Xi_h(y,y')}^{-1} \abs{\det \derive{h}{y'} }
\end{equation}
wohldefiniert und erfüllt 
$ A \equiv a(Y, D_y, Y') \mod{\OP{C^\tau S^{-\infty}(\R^n,\R^n\times\R^n)}}$,
wobei
\begin{equation}\label{equ:trafoA_h}
\Xi_h(y,y') = \int_0^1 \derive{h}{y'+\theta(y-y')} d\theta .
\end{equation}
\begin{proof}
Für $r>0$ klein genug gibt es wegen (\ref{equ:trafoUngl}) nach Definition von $\Xi_h$ ein $\epsilon > 0$, sodass
\begin{equation}\label{equ:trafoA_hBed}
\frac{1}{C+\epsilon} \le \abs{\det \Xi_h(y,y')} \le C+\epsilon \textrm{ auf } \supp \left( (\varphi_j \psi_j) \circ h\right) ,
\end{equation}
wobei $C$ wie in (\ref{equ:trafoUngl}) definiert ist.
Setze
$$p_j(x,\xi,x') := \varphi_j(x) p(x,\xi, x') \psi_j(x') \textrm{ und } 
p_{\infty,j}(x,\xi,x') := \varphi_j(x) p(x,\xi,x') (1 - \psi_j(x')) .$$
Wegen $\dist{\supp \varphi_j}{\supp (1-\psi_j)} > 0$ ist nach Theorem \ref{thm:kerndisjunkt} das Symbol $p_{\infty,j}(x,\xi,x') \in C^\tau S^{-\infty}(\R^n,\R^n\times\R^n)$.
Nach Lemma \ref{lemma:trafoSinfty} stimmt $p_{\infty,j}(X,D_x,X')$ nach der Koordinatentransformation mit einem Operator aus $\OP C^\tau S^{-\infty}(\R^n,\R^n\times\R^n)$ überein.
Unser Operator lässt sich also in zwei Terme
$$p(X, D_x, X') = \sum_j p_j(X,D_x, X') + \sum_j p_{\infty,j}(X,D_x,X')$$
zerteilen, wovon der letzte Term nach Koordinatentransformation mit einem Operator $A_\infty$ der Klasse $\OP{C^\tau S^{-\infty}(\R^n,\R^n\times\R^n)}$ übereinstimmt.
Wir müssen also noch die Transformationseigenschaft für den ersten Term überprüfen.
Sei dazu $\chi \in \Schwarzi$ mit $\chi(0) = 1$ und definiere $\chi_\epsilon(\xi) := \chi(\epsilon\xi)$ für $\epsilon>0, \xi \in \R^n$.
Theorem \ref{thm:osziFubini}.\ref{item:osziFubiniA} liefert für jedes $u \in \Schwarzi$ wegen $x' \mapsto p_j(x,\xi,x') u(x') \in L^1(\R^n)$ nach Lemma \ref{lemma:IntroSymS->S} die Aussage, dass wir das oszillatorische Integral umformen können in
\begin{align*}
& \left( p_j(X,D_x, X') u \right) \left(h(y)\right) \\
&= \lim_{\epsilon\rightarrow 0} \iint e^{i(h(y)-x')\cdot\xi} \chi_\epsilon(\xi) p_j(h(y), \xi, x') u(x') dx' \dslash \xi \\
&= \lim_{\epsilon\rightarrow 0} \iint e^{i(h(y)-h(y'))\cdot\xi} \chi_\epsilon(\xi) p_j(h(y), \xi, h(y')) \abs{\det \derive{h}{y'} } u(h(y')) dy' \dslash\xi
,\end{align*}
wobei wir mit dem Diffeomorphismus $h : x' \mapsto h(x') =: h(y')$ das innere Integral transformiert haben.
Andererseits gilt nach dem Mittelwertsatz im Mehrdimensionalen, dass
$$ h(y) - h(y') = \left( \int_0^1 \derive{h}{y' + \theta(y-y')} d\theta \right) (y-y') = \Xi_h(y,y') (y-y') .$$
$\Xi_h(y,y')$ ist damit ein wohldefinierter Homöomorphismus in einer Umgebung der Diagonalen von $\R^n\times\R^n$, 
und lokal durch
$$ \left( h(y) - h(y') \right) \cdot \xi = 
\sum_{j=1}^n \left( h_j(y) - h_j(y') \right) \xi_j = 
\sum_{j,k} \Xi_h(y,y')_{kj} (y_k - y_k') \xi_j$$
gegeben.
Für $\eta := \Xi_h(y,y')^T \xi$ ist also $\left(h(y) - h(y')\right)\cdot\xi = (y-y')\cdot\eta$, und damit
\begin{align*}
\left( p_j(X,D_x,X') u\right)(h(y)) 
=& \lim_{\epsilon\rightarrow 0} \iint e^{i(y-y')\cdot\eta} \chi_\epsilon\left(\Xi_h(y,y')^{-T}\eta\right) p_j\left(h(y), \Xi_h(y,y')^{-T}\eta, h(y')\right) \\
& \abs{\det \Xi_h(y,y')}^{-1} \abs{\det \derive{h}{y'} } u(h(y')) dy' \dslash\eta \\
=& \lim_{\epsilon\rightarrow 0} \iint  e^{i(y-y')\cdot\eta} \chi_\epsilon\left(\Xi_h(y,y')^{-T}\eta\right) \tilde{a}_j(y,\eta,y') u(h(y')) dy' \dslash\eta
,\end{align*}
wobei 
$\tilde{a}_j(y,\eta,y') := p_j(h(y), \Xi_h(y,y')^{-T} \eta, h(y')) \abs{\det \Xi_h(y,y')}^{-1} \abs{\det \derive{h}{y'} }$.
Nun ist
$$\menge{ \chi_\epsilon( \Xi_h(y,y+z)^{-T}\eta) : \epsilon \in (0,1) } \subset \Oszillatory{0,0}{0}$$
beschränkt bezüglich $(z,\eta)$ und Lemma \ref{lemma:osziChi} liefert die beiden Konvergenzeigenschaften
$$\chi_\epsilon\left( \Xi_h(y, y+z)^{-T} \eta \right) \rightarrow 1 \textrm{ für } \epsilon \rightarrow 0 \textrm{ für alle } y,z,\eta \in \R^n ,$$
$$ \partial_z^\alpha \partial_\eta^\beta \left( \chi_\epsilon\left( \Xi_h(y, y+z)^{-T} \eta \right) \right) \rightarrow 0 \textrm{ für } \epsilon \rightarrow 0 \textrm{ für alle } y,z,\eta \in \R^n, \alpha,\beta \in \N^n_0 \textrm{ mit } \abs{\alpha}+\abs{\beta}>0 .$$
Nach Theorem \ref{thm:osziKonvergenz} können wir den Limes in das Integral hineinziehen und erhalten mit $w := u \circ h$ ein Oszillationsintegral der Form 
\begin{align*}
\left(p_j(X,D_x,X') u\right)(h(y))
&= \osint e^{i(y-y')\cdot\eta} \tilde{a}_j(y,\eta,y') w(y') dy' \dslash\eta \\
&=: a_j(Y, D_y, Y') w(y)
.\end{align*}
Wegen $a(y,\eta,y') = \sum_{j=1}^\infty a_j(y,\eta,y')$ und $A = a(Y,D_y, Y') + A_\infty$ folgt die Behauptung, 
falls $a_j(x,\xi,y) \in C^\tau S^m_{\rho,\delta}(\R^n,\R^n\times\R^n)$.
Dies gilt es also noch zu zeigen. 
Zunächst betrachten wir den Fall $\tau \in \N$.
Dafür beweisen wir, dass $\tilde{a}_j(x,\xi,y) \in C^\tau S^{m}_{\rho,\delta}(\R^n,\R^n\times\R^n)$.
Seien $\alpha,\beta,\gamma \in \N^n_0$ mit $\abs{\beta} \le \tau$. Wir haben
\begin{align*}
\abs{ D_y^\gamma D_x^\beta D_\xi^\alpha \tilde{a}_j(x,\xi,y) } 
&= \abs{ D_y^\gamma D_x^\beta D_\xi^\alpha p_j\left(h(x), \Xi_h(x,y)^{-T}\xi, h(y)\right) \abs{\det \Xi_h(x,y)}^{-1} \abs{\det \derive{h}{y} } }  \\
&= \abs{ D_y^\gamma D_x^\beta \left[ D_\xi^\alpha p_j\left(h(x), \Xi_h(x,y)^{-T}\xi, h(y)\right) \zeta(x,y) \right] } 
,\end{align*}
wobei 
$\zeta(x,y) := \abs{\det \Xi_h(x,y)}^{-1} \abs{\det \derive{h}{y} }$ mit $\abs{ D_y^\gamma D_x^\beta \zeta(x,y) } \le C_{\alpha\beta}$ nach (\ref{equ:trafoUngl}).
Mit Kettenregel folgt
\begin{align*}
&\abs{ \deriveii_\xi \left( p_j\left(h(x), \Xi_h(x,y)^{-T}\xi, h(y)\right) \right) }\\
&= \abs{ \deriveii_{\Xi_h(x,y)^{-T}\xi} p_j\left(h(x), \Xi_h(x,y)^{-T}\xi, h(y)\right) \cdot \derive{\Xi_h(x,y)^{-T}}{\xi} } \\
&\le C \abs{ \deriveii_{\Xi_h(x,y)^{-T}\xi} p_j\left(h(x), \Xi_h(x,y)^{-T}\xi, h(y)\right) }
.\end{align*}
Wir erhalten also die Abschätzung 
$$ \abs{ D_y^\gamma D_x^\beta D_\xi^\alpha \tilde{a}_j(x,y,\xi) } \le
\Landau{ D_y^\gamma D_x^\beta p^{(\alpha)}_j\left(h(x), \Xi_h(x,y)^{-T}\xi, h(y)\right) } .$$
Für den nächsten Schritt teilen wir ein Multiindex $\sigma\in\N^{2n}_0$ in die Hälften $\sigma_1,\sigma_2\in\N^n_0$ auf, sodass $\sigma = (\sigma_1,\sigma_2)$.
Unter Verwendung von Formel \ref{faaDiBruno} und dessen Notation ist
\begin{align*}
&\abs{ D_y^\gamma D_x^\beta p^{(\alpha)}_j\left(h(x), \Xi_h(x,y)^{-T}\xi, h(y)\right) } \\
&\le 
\Landaui \left( D_y^\gamma \sum_{\sigma \in \N^{2n}_{0,\abs{\beta}}} p_{j(\sigma_1)}^{(\alpha+\sigma_2)}\left(h(x), \Xi_h(x,y)^{-T}\xi, h(y)\right) \right. \\
&\quad \left.
\sum_{\Sigma(\beta,\sigma)} \prod_{j=1}^{\abs{\beta}}
\frac{1}{\nu^j! (\mu^j!)^{\abs{\nu^j}}} 
\left( D_x^{\mu^j} \left( x \mapsto (h(x), \Xi_h(x,y)^{-T} \xi ) \right) \right)^{\nu^j} \right) \\
&\le 
\Landau{ D_y^\gamma \sum_{\beta_1+\beta_2 = \beta} p_{j(\beta_1)}^{(\alpha+\beta_2)}\left(h(x), \Xi_h(x,y)^{-T}\xi, h(y)\right) \piket{\xi}{\abs{\beta_2}} }
,\end{align*}
wobei $p^{(\alpha)}_{j(\beta,\beta')}(x,\xi,y) := D_y^{\beta'} D_\xi^\alpha D_x^\beta p_j(x,\xi,y)$. 
Für den Summanden haben wir
\begin{align*}
& \abs{ D_y^\gamma p_{j(\beta_1)}^{(\alpha+\beta_2)}\left(h(x), \Xi_h(x,y)^{-T}\xi, h(y)\right) } \\
&= \gamma! \left| \sum_{\sigma \in \N^n_{0,\abs{\gamma}}} p_{j(\beta_1,\sigma_1)}^{(\alpha+\beta_2+\sigma_2)}\left(h(x), \Xi_h(x,y)^{-T}\xi, h(y)\right) \right. \\
& \quad \left.
\sum_{\Sigma(\gamma,\sigma)} \prod_{j=1}^{\abs{\gamma}}
\frac{1}{\nu^j! (\mu^j!)^{\abs{\nu^j}}}
\left( D_y^{\mu^j} \left( y \mapsto (\Xi_h(x,y)^{-T} \xi, h(y) ) \right) \right)^{\nu^j} \right| \\
&\le
\Landau{ p_{j(\beta_1)}^{(\alpha+\beta_2+\gamma)}\left(h(x), \Xi_h(x,y)^{-T}\xi, h(y)\right) \piket{\xi}{\abs{\gamma}}}
.\end{align*}
Für die Berechnungen haben wir die Voraussetzung $\beta,\gamma \neq 0$ für Lemma \ref{faaDiBruno} fallen gelassen, 
da wir für $\beta = 0$ bzw. $\gamma = 0$ einfach den jeweiligen Schritt auslassen können.
Insgesamt ist also 
\begin{align*}
\abs{ D_y^\gamma D_x^\beta D_\xi^\xi \tilde{a}_j(x,\xi,y) }
& \le C_{\alpha\beta\gamma} \piket{\xi}{m-\rho\abs{\alpha}+(1-\rho)\abs{\gamma}} \sum_{\beta_1+\beta_2 = \beta} \piket{\xi}{\delta\abs{\beta_1} + (1-\rho)\abs{\beta_2}} \\
& \le C_{\alpha\beta\gamma} \piket{\xi}{m-\rho\abs{\alpha}+(1-\rho)\abs{\gamma} + \abs{\beta} \max\menge{\delta,1-\rho}}
.\end{align*}
Nach Voraussetzung ist $1-\rho < \delta$, also $\max\menge{\delta,1-\rho} = \delta$ und damit
$$ \abs{ D_y^\gamma D_x^\beta D_\xi^\xi \tilde{a}_j(x,y,\xi) } \le
C_{\alpha\beta\gamma} \piket{\xi}{m-\rho\abs{\alpha}+\delta\abs{\beta}+\delta\abs{\gamma}} .$$
Also ist $\tilde{a}_j(x,\xi,y) \in C^\tau S^m_{\rho,\delta}(\R^n,\R^n\times\R^n)$.

Für $\tau \in (0,1)$ und $x, \chi \in \R^n$ definiere den Operator $\triangle_{\chi,x}$ durch
$\triangle_{\chi,x} p(x,\xi,y) := f(x+\chi) - f(x)$ für eine stetige Funktion $f : \R^n \rightarrow \C$.
Wir sehen sofort, dass
\begin{align*}
\triangle_{\chi,x} p\left(h(x), \Xi_h(x,y)^{-T}\xi, h(y)\right) 
&= \triangle_{\chi,x} p\left(h(x), \Xi_h(z,y)^{-T}\xi, h(y)\right) \mid_{z=x} \\
&+ \triangle_{\chi,x} p\left(h(z), \Xi_h(x,y)^{-T}\xi, h(y)\right) \mid_{z=x}
.\end{align*}
Insbesondere kommutiert der Operator $\triangle_{\chi,x}$ mit dem Ableitungsoperator, sodass
wir wegen der Hölder-Stetigkeit im ersten Argument von $p$ 
\begin{align*}
\abs{ D_y^\gamma D_\xi^\alpha \triangle_{\chi,x} p\left(h(x), \Xi_h(z,y)^{-T}\xi, h(y)\right) \mid_{z=x} } 
&= \abs{ \triangle_{\chi,x} D_y^\gamma D_\xi^\alpha p\left(h(x), \Xi_h(z,y)^{-T}\xi, h(y)\right) \mid_{z=x} }  \\
&\le C_{\alpha\gamma} \piket{\xi}{m-\rho\abs{\alpha}+\delta\abs{\gamma}+\delta\tau} \abs{\triangle_{x,\chi} h(x)}^\tau  \\
&\le C_{\alpha\gamma} C \piket{\xi}{m-\rho\abs{\alpha}+\delta\abs{\gamma}+\delta\tau} \abs{\chi}^\tau
\end{align*}
erhalten, wobei $\abs{ \triangle_{x,\chi} h(x) } \le C \abs{\chi}$ mit $C>0$ Lipschitzkonstante von $h \in C^1(\R^n)$.
Genauso finden wir für $x \mapsto p\left(h(z), \Xi_h(x,y)^{-T}\xi, h(y)\right) \in C^1(\R^n)$ eine Lipschitzkonstante $C_L > 0$, sodass
\begin{align*}
\abs{D_y^\gamma D_\xi^\alpha \triangle_{\chi,x} p\left(h(z), \Xi_h(x,y)^{-T}\xi, h(y)\right) \mid_{z=x} } 
&\le C_L C_{\alpha\gamma} \piket{\xi}{m-\rho\abs{\alpha}+\delta\abs{\gamma}} \abs{\triangle_{\chi,x} \Xi_h(x,y)^{-T}\xi} \\
&\le C_L C_{\alpha\gamma} C_1 \piket{\xi}{m-\rho\abs{\alpha}+\delta\abs{\gamma}} \abs{\chi}
\end{align*}
mit $C_1$ aus (\ref{equ:trafoUngl}).
Beide Abschätzungen zusammengefasst ergeben
$$
\abs{ D_y^\gamma D_\xi^\alpha \triangle_{\chi,x} p\left(h(x), \Xi_h(x,y)^{-T}\xi, h(y)\right) } \le
C_{\alpha\gamma} \piket{\xi}{m-\rho\abs{\alpha}+\delta\abs{\gamma}+\delta\tau} \abs{\chi}^\tau
$$
für $\abs{\chi} < 1$.
Wir erhalten also 
$\hnorm{D_y^\gamma D_\xi^\alpha p\left(h(x), \Xi_h(x,y)^{-T}\xi, h(y)\right) }{\tau} \le C_{\alpha\gamma} \piket{\xi}{m-\rho\abs{\alpha}+\delta\abs{\gamma}+\delta\tau}$.

Für allgemeines $1 < \tau \not\in \N$ betrachte das Symbol 
$\partial_x^{\beta} p(x,\xi,y) \in C^{\tau-\gauss{\tau}} S^{m+\delta\gauss{\tau}}(\R^n,\R^n\times\R^n)$ 
für $\abs{\beta} = \gauss{\tau}$.
\end{proof}
\end{thm}

\begin{remark}
Die Bedingung $1-\rho \le \delta$ ist für den Beweis notwendig.
Betrachte dazu $x,y\in\R^n$ mit $\abs{x-y}$ klein genug, sodass es eine Konstante $C_0 > 0$ gibt mit
$$ C_0^{-1} \abs{\xi} \le \abs{ \partial_{x_j} \Xi_h(x,y)^{-T}\xi} \le C_0 \abs{\xi} \textrm{ für alle } \xi \in \R^n \textrm{ und jedes } j = 1,\ldots,n .$$
Nun wählen wir unter vereinfachten Annahmen ein Symbol $p(x,\xi,y) = p(\xi) \in S^m_{\rho,\delta}(\R^n)$, 
und betrachten wie im Beweis das Symbol
\begin{align*}
\tilde{a}_j(x,\xi,y) :&= p\left(h(x), \Xi_h(x,y)^{-T}\xi, h(y)\right) \abs{\det \Xi_h(x,y)}^{-1} \abs{\det \derive{h}{y} } \\
&= p(\Xi_h(x,y)^{-T}\xi) r(x,y)
\end{align*}
mit $r(x,y) := \abs{\det \Xi_h(x,y)}^{-1} \abs{\det \derive{h}{y} } \in S^{-\infty}(\R^n\times\R^n)$.
Eine Anwendung der Kettenregel liefert
$$ \partial_{x_j} p\left(\Xi_h(x,y)^{-T}\xi\right) = \sum_{k=1}^n \partial_{\eta_k} p(\eta) \mid_{\eta = \Xi_h(x,y)^{-T}\xi} \frac{\partial \left(\Xi_h(x,y)^{-T}\xi\right)_k}{\partial x_j} .$$
Also finden wir zu jedem $j = 1,\ldots,n$ eine Konstante $C_j > 0$, sodass
$$ \abs{ \partial_{x_j} p\left(\Xi_h(x,y)^{-T}\xi\right) } \le C_0 C_j \piket{\xi}{m-\rho} \abs{\xi} \le C_0 C_j \piket{\xi}{m+1-\rho} .$$
Die Voraussetzung $\tilde{a}_j(x,\xi,y) \in S^m_{\rho,\delta}(\R^n\times\R^n\times\R^n) \subset C^\tau S^m_{\rho,\delta}(\R^n,\R^n\times\R^n)$ impliziert jedoch, dass es eine Konstante $\tilde{C}_j > 0$ geben muss, sodass
$$ \abs{ \partial_{x_j} p\left(\Xi_h(x,y)^{-T}\xi\right) } < \tilde{C}_j \piket{\xi}{m+\delta} .$$
Da $\tilde{C}_j, C_j$ und $C_0$ unabhängig von $\xi$ sind, muss $\piket{\xi}{m+1-\rho} \le \piket{\xi}{m+\delta}$ für alle $\xi\in\R^n$ gelten, 
d.h. die Bedingung $\delta \ge 1-\rho$ muss erfüllt sein.
\end{remark}

\begin{cor}\label{cor:koordAsympt}
Falls $p(X,D_x)$ ein Operator mit einfachen Symbol $p \in C^\tau S^m_{\rho,\delta}(\R^n,\R^n)$ mit $0 \le 1-\rho \le \delta < \rho \le 1, \delta < 1$ ist, dann finden wir zu $a(y,\eta,y) = p(h(y),\derive{h}{y}^{-T} \eta)$ ein Symbol $a_L(y,\eta)$ in $x$-Form, das die  asymptotische Abschätzung
$$a_L(y,\eta) \sim \sum_{\alpha\in\N^n} \frac{1}{\alpha!} \partial_\eta^\alpha \partial_y^\alpha a(x,\eta,y) |_{x=y}$$
genügt in dem Sinne, dass für jedes $N \in \N_0$
$$ a_L(y,\eta) - \sum_{\abs{\alpha} \le N} \frac{1}{\alpha!} \partial_\eta^\alpha \partial_y^\alpha a(x,\eta,y) |_{x=y} \in C^\tau S^{m-(\rho-\delta)(N+1)}_{\rho,\delta}(\R^n,\R^n) .$$
Insbesondere ist
$$ a_L(y,\eta) = p(h(y), \derive{h}{y}^{-T} \eta) + r(y,\eta) \textrm{ mit } r(y,\eta) \in C^\tau S^{m-(\rho-\delta)}_{\rho,\delta}(\R^n,\R^n) .$$
\begin{proof}
Da
$ \Xi_h(y,y) = \int_0^1 \derive{h}{y} d\theta = \derive{h}{y}$ 
und
$\psi_j \circ h = 1$ auf $ \supp \varphi_j \circ h$
folgt
\begin{align*}
a(y,\eta,y) 
&= \sum_{j=1}^\infty \varphi_j(h(y)) p(h(y), \derive{h}{y}^{-T}\eta) \abs{\det \derive{h}{y}^{-1}} \abs{\det \derive{h}{y} } \\
&= p(h(y), \derive{h}{y}^{-T} \eta)
.\end{align*}
Für das vereinfachte Symbol $a_L$ von $a$ erhalten wir nach Theorem \ref{thm:simplifyAsmytotic} die asymptotische Abschätzung.
\end{proof}
\end{cor}

\begin{thm}\label{thm:trafo}
Seien $\tilde{\Omega}_y \subset\subset \Omega_y$, $0 \le 1-\rho \le \delta \le \rho$ mit $\delta < 1$ gegeben. Setze 
$\tilde{\Omega}_x := \menge{ h(y) : y \in \tilde{\Omega}_y} \subset\subset \Omega_x$.
Für $p(x,\xi,x') \in C^\tau S^m_{\rho,\delta}(\R^n,\R^n\times\R^n)$ existiert ein 
$A \in \OP{ C^\tau S^m_{\rho,\delta}(\R^n,\R^n \times \R^n)}$, sodass
für alle $\varphi, \psi \in C^\infty_0(\tilde{\Omega}_x)$ gilt
$$ \varphi(h(Y)) A \psi(h(Y')) w(y) = (\varphi(X) p(X,D_x,X') \psi(X') u) (h(y)) \textrm{ für alle } w := u \circ h \in C^\infty_0(\Omega_y) .$$
Desweiteren finden wir nach (\ref{equ:trafoq}) ein $a(y,\eta,y') \in C^\tau S^m_{\rho,\delta}(\R^n,\R^n \times \R^n)$, das für ein genügend kleines $r > 0$ die Gleichung
$$ \varphi(Y) A \equiv \varphi(Y) a(Y,D_y, Y') \mod{\OP{C^\tau S^{-\infty}(\R^n,\R^n\times\R^n)}}$$
erfüllt.
\begin{proof}
Wir spalten wie in Theorem \ref{thm:trafoRn} das Symbol $p$ in $p_j$ und $p_{\infty,j}$ auf.
Diese Aufspaltung können wir auch mit dem Symbol $\varphi(x) p(x,\xi,y) \psi(y)$ machen.
Dann erhalten wir in Operatorschreibweise mit einer Familie $\menge{\left( \varphi_j, \psi_j \right)}_{j\in\N}$ aus Lemma \ref{koord:lemmaPartition} :
\begin{equation}\label{equ:trafoThmAufspaltung}
\varphi(X) p(X,D_x,X') \psi(X') = \varphi(X) \sum_{j \in \N} p_j(X,D_x,X') \psi(X') + \varphi(X) \sum_{j \in \N}  p_{\infty,j}(X,D_x,X') \psi(X') .
\end{equation}
Setze $M := \menge{ j \in \N : \supp \psi_j \cap \tilde{\Omega}_y \neq \emptyset}$.
Da $\tilde{\Omega}_y$ relativ kompakt ist, ist $\#M < \infty$, die erste Summe ist also endlich.
Genauso wie in Theorem \ref{thm:trafoRn} erhalten wir nach Koordinatentransformation von $p_j$ einen Operator
$$a_j(Y,D_y,Y') \equiv h^* p_j(X,D_x,X') \left(h^{-1}\right)^* .$$
Nach der Transformation der Operatoren $\varphi(X)$ und $\psi(X')$ erhalten wir
$$ \varphi(X) p_j(X,D_x,X') \psi(X') \equiv \varphi(Y) h^* \left( a_j(Y,D_y,Y') \psi(h(Y')) \right) \left(h^{-1}\right)^* .$$
Schließlich finden wir wegen $\#M < \infty$ ein $r > 0$, sodass die Bedingung (\ref{equ:trafoA_hBed}) für $\Xi_h$ definiert in (\ref{equ:trafoA_h}) für jedes $j \in M$ erfüllt wird.
Insgesamt bekommen wir die Transformation des Operators $\sum_{j\in M} p_j(X,D_x,X') \psi(X')$ durch
$$
\varphi(X) \sum_{j\in M} p_j(X,D_x,X') \psi(X') \equiv \varphi(Y) \sum_{j\in M} h^* \left( a_j(Y,D_y,Y') \psi(h(Y')) \right) \left(h^{-1}\right)^*  
.$$
Wir erhalten also tatsächlich wie in (\ref{equ:trafoq}) ein Symbol
\begin{align*}
&a(y,\eta,y') \\
&= \sum_{j=1}^\infty \varphi_j(h(y)) p(h(y), \Xi_h(y,y')^{-T} \eta, h(y')) \psi(h(y')) \psi_j(h(y')) \abs{\det \Xi_h(y,y')}^{-1} \abs{\det \derive{h}{y'} }
\end{align*}
mit $ \varphi(Y) A \equiv \varphi(Y) a(Y,D_y, Y') \mod \OP{C^\tau S^{-\infty}(\R^n,\R^n\times\R^n)} $ falls die zweite Summe von (\ref{equ:trafoThmAufspaltung}) in der Klasse $\OP{C^\tau S^{-\infty}(\R^n,\R^n\times\R^n)}$ liegt.

Wir wissen bereits, dass $p_{\infty,j}(X,D_x,X') \in \OP{C^\tau S^{-\infty}(\R^n,\R^n\times\R^n)}$.
Transformieren wir den Operator $p_{\infty,j}(X,D_x,X')$, so induziert der zu $p_{\infty,j}$ assoziierte Kern $k_{p_{\infty,j}}$ nach Lemma \ref{lemma:trafoSinfty} einen Kern $k_{a_{\infty,j}}$, der die Bedingung (\ref{equ:kernglatt}) erfüllt.
Für ein $\phi \in C^\infty_0(\Omega_y)$ mit $\phi = 1$ auf $\tilde{\Omega}_y$ haben wir nach Lemma \ref{lemma:trafoSinfty}
\begin{align*}
k_{a_{\infty,j}}(y,z)
&= \phi(y) k_{a_{\infty,j}}(y,z) \phi(y-z) \\
&= \phi(y) k_{p_{\infty,j}}(h(y), h(y-z), h(y) - h(y-z)) \abs{ \det \derive{h}{y-z} } \phi(y-z) 
.\end{align*}
Insbesondere ist der Kern von $p_{\infty,j}(X,D_x,X') \psi(X')$ koordinatentransformiert gegeben durch 
\begin{align*}
\varphi(y) k_{a_{\infty,j}}(y,z) \psi(z) 
=& \varphi(h(y)) k_{p_{\infty,j}}(h(y), h(y-z), h(y) - h(y-z)) \\
 & \abs{ \det \derive{h}{y-z} } \psi(h(y-z)) \\
=& \varphi(h(y)) \phi(y) k_{p_{\infty,j}}(h(y), h(y-z), h(y) - h(y-z)) \\
 & \abs{ \det \derive{h}{y-z} } \psi(h(y-z)) \phi(y-z)
.\end{align*}
Der Operator $\varphi(X) \sum_{j\in\N} p_{\infty,j}(X,D_x,X') \psi(X')$ transformiert sich also zu 
$$\varphi(h(Y)) \sum_{j\in\tilde{M}} a_{\infty,j}(Y,D_y,Y') \psi(h(Y')) ,$$
wobei die Menge
$\tilde{M} := \menge{ j \in \N : \supp \phi \cap \supp \psi_j \circ h \neq \emptyset}$ endlich ist.
Also ist nach Folgerung \ref{cor:kernglatt} der Operator
$\sum_{j\in\tilde{M}} \varphi(h(Y)) a_{\infty,j}(Y,D_y,Y') \psi(h(Y'))$ in $\OP{ C^\tau S^{-\infty}(\R^n,\R^n\times\R^n)}$ enthalten.
\end{proof}
\end{thm}

\begin{remark}\label{remark:trafo} 
Verstärken wir mit der Bedingung $\delta < \rho$ die Voraussetzungen aus Theorem \ref{thm:trafo}, so erhalten wir für 
$$p(x,\xi) \in C^\tau S^m_{\rho,\delta}(\R^n,\R^n) \textrm{ ein Symbol } a_L(y,\eta) \in C^\tau S^m_{\rho,\delta}(\R^n,\R^n) ,$$
gegeben in Folgerung \ref{cor:koordAsympt}, sodass
$$\varphi(Y) A \equiv \varphi(Y) a_L(Y,D_y)  \mod \OP{C^\tau S^{-\infty}(\R^n,\R^n\times\R^n)} .$$
\end{remark}

\begin{thm}\label{thm:koordBesselInv} 
Sei $\Omega_x = \Omega_y = \R^n$, $h$ ein glatter Diffeomorphismus auf dem $\R^n$.
Dann ist der Bessel-Potential-Raum $H^s_q(\R^n)$ für $s \in \R, 1 \le q < \infty$ invariant unter Koordinatentransformation.
Genauer ist die Abbildung
$ h^* : H^s_q(\R^n) \rightarrow H^s_q(\R^n), u \mapsto u \circ h$
eine Bijektion und es existiert ein $C > 0$, sodass
$$ \frac{1}{C} \norm{u}{H^s_q(\R^n)} \le \norm{u \circ h}{H^s_q(\R^n)} \le C \norm{u}{H^s_q(\R^n)} .$$
\begin{proof}
Für $\piket{D_x}{s} \in \OP{S^s_{1,0}(\R^n)}$ existiert nach Theorem \ref{thm:trafoRn} und Folgerung \ref{cor:koordAsympt} ein Symbol $a_s \in S^s_{1,0}(\R^n\times\R^n)$, sodass
$ \left(\piket{D_x}{s} u\right)(h(y)) = a_s(Y,D_y)(u \circ h(y))$.
Da $\piket{\xi}{s}$ und $a_s$ der selben Symbolklasse angehören, finden wir eine Konstante $C_{sq} > 0$ sodass
$ C_{sq}^{-1} \abs{a_s}_{l,\infty}^{(s)} \le \abs{\piket{\xi}{s}}_{l,\infty}^{(s)} \le C_{sq} \abs{a_s}_{l,\infty}^{(s)}$.
Mit Koordinationtransformation und der Abschätzung (\ref{equ:trafoUngl}) erhalten wir
$$\norm{u}{H^s_q(\R^n)}^q = 
\int \abs{ \piket{D_x}{s} u(x) }^q dx = 
\int \abs{ a_s(Y,D_y) (w(y)) }^q \abs{\det \derive{h}{y} } dy \le 
C_{sq} \norm{w}{H^s_q(\R^n)}^q ,$$
wobei $w := u \circ h$.
Die andere Richtung kann analog bewiesen werden, da $h^{-1}$ als regulärer Diffeomorphismus ebenso (\ref{equ:trafoUngl}) erfüllt.
\end{proof}
\end{thm}

\begin{definition} 
Sei $p(x,\xi) \in C^\tau S^m_{\rho,\delta}(\R^n,\R^n)$ mit $0 \le 1-\rho \le \delta \le \rho \le 1, \delta <1$.
Das Prinzipalsymbol $p_r(x,\xi)$ von $p(x,\xi)$ ist ein Vertreter der Klasse 
$$\quotientenraum{C^\tau S^m_{\rho,\delta}(\R^n,\R^n)}{C^\tau S^{m-(\rho-\delta)}_{\rho,\delta}(\R^n,\R^n\times\R^n)}$$
mit
$ p(x,\xi) \equiv p_r(x,\xi) \mod C^\tau S^{m-(\rho-\delta)}(\R^n,\R^n\times\R^n)$.
\end{definition}

\begin{remark}
Für $\rho = 1$ ist das Prinzipalsymbol eines Symbols $p \in C^\tau S^m_{1,\delta}(\R^n\times\R^n)$ unter Koordinatentransformation invariant.
\begin{proof}
Mit Hilfe der Symbolglättung können wir $p$ aufspalten in $p = p^\sharp + p^\flat$ mit
$p^\sharp \in S^m_{1,\gamma}(\R^n\times\R^n)$
und
$p^\flat \in C^{\tau} S^{m- (\gamma - \delta) \tau}_{1,\gamma}(\R^n,\R^n)$ für ein $\gamma \in (\delta,1)$.
Ein regulärer glatter Diffeomorphismus $h : \Omega_y \rightarrow \Omega_x$ mit $\Omega_x, \Omega_y \subsubset \R^n$ 
transformiert das glatte Symbol $p^\sharp$ nach Bemerkung \ref{remark:trafo} zu
$$ h^* p^\sharp(X,D_x) (h^{-1})^* = A^\sharp + R^\sharp$$
mit $R^\sharp \in \OP{ S^{-\infty}(\R^n\times\R^n\times\R^n)}$ und $A \in \OP{ S^m_{1,\delta}(\R^n\times\R^n\times\R^n) }$.
Wir erhalten also
$$h^* p(X,D_x) (h^{-1})^* = A^\sharp + R^\sharp + h^* p^\flat(X,D_x) (h^{-1})^* .$$
Andererseits finden wir ein Symbol $a \in C^\tau S^m_{1,\delta}(\R^n,\R^n)$ mit
$h^* p(X,D_x) (h^{-1})^* \equiv a(Y,D_y) \mod C^\tau S^{-\infty}(\R^n,\R^n\times\R^n)$
Nun können wir wiederum das Symbol $a$ aufspalten in $a = a^\sharp + a^\flat$ mit
$a^\sharp \in S^m_{1,\gamma}(\R^n\times\R^n)$
und
$a^\flat \in C^{\tau} S^{m- (\gamma - \delta) \tau}_{1,\gamma}(\R^n,\R^n)$.
Wir erhalten insgesamt also
\begin{align*}
h^* p^\sharp(X,D_x) (h^{-1})^* 
&\equiv a^\sharp(X,D_x) + a^\flat(X,D_x) \mod C^\tau S^{-\infty}(\R^n,\R^n\times\R^n)  \\
&\equiv A^\sharp + h^* p^\flat(X,D_x) (h^{-1})^* \mod C^\tau S^{-\infty}(\R^n,\R^n\times\R^n)
.\end{align*}
Damit ist $a^\sharp(X,D_x) \equiv A^\sharp \mod C^\tau S^{m-(\gamma-\delta)\tau}_{1,\gamma}(\R^n,\R^n\times\R^n)$.
\end{proof}
\end{remark}

\subsection{Differenzierbare Mannigfaltigkeiten}
\cite[Chapter 2 \S 7]{Kumanogo} beschreibt die Invarianz des Koordinatenwechsels von Pseudodifferentialoperatoren der Klasse  $S^m_{\rho,\delta}(\R^n\times\R^n)$
auf parakompakten differenzierbaren Mannigfaltigkeiten, indem er einen Kartenwechsel auf den lokalen Fall im $\R^n$ zurückführt.
Wir wollen mit einer ähnlichen Vorgehensweise die Invarianz für Operatoren in $x$-Form studieren.
Dazu wiederholen wir zunächst die Definition einer parakompakten differenzierbaren Mannigfaltigkeit:

\begin{definition}\label{def:parakompakt}
Ein topologischer Raum $M$ heißt parakompakt, falls jede offene Überdeckung $\menge{\Omega_j}_{j\in\N}$ von $M$ eine lokal endliche Verfeinerung besitzt.
Eine Überdeckung $\menge{V_j}_{j\in\N}$ von $M$ heißt lokal endliche Verfeinerung, falls
zu jedem $j\in\N$ ein $k\in\N$ existiert mit $V_j \subset \Omega_k$, und es zu jedem $x \in M$ eine Umgebung $U \subset M$ gibt mit
$$\# \menge{j \in \N : V_j \cap U \neq \emptyset} < \infty .$$
\end{definition}

\begin{definition}
Ein lokaler euklidischer Raum $M$ der Dimension $n$ ist ein topologischer Hausdorff-Raum, für den zu jedem Punkt
eine Umgebung existiert, die homöomorph zu einer offenen Teilmenge des euklidischen Raums $\R^n$ ist.
Falls $h$ ein Homöomorphismus ist, der ein Gebiet $\Omega \subseteq M$ auf eine offene Teilmenge $U \subseteq \R^n$ abbildet,
nennen wir $h$ eine Karte und $U$ die Koordinatenumgebung von $h$.
\end{definition}

\begin{definition}
Eine differenzierbare Struktur $S$ der Klasse $C^t$ mit $t \in [1,\infty]$
auf einem lokalen euklidischen Raum $M$ ist eine Familie von Karten
$\menge{ h : \Omega_j \rightarrow U_j }_{j\in A}$ mit einer Indexmenge $A$. 
Zusätzlich erfüllt $S$ die folgenden Eigenschaften:
\begin{enumerate}[(a)]
\item $\bigcup_{j\in\N} \Omega_j = M$.
\item Setze $\Omega_{jk} := \Omega_j \cap \Omega_k$ für  alle $j,k \in \N$. Für $\Omega_{jk} \neq \emptyset$ gelte
$$h_{jk} := h_j \circ h_k^{-1} : h_k(\Omega_{jk}) \subseteq U_k \rightarrow h_j(\Omega_{jk}) \subseteq U_j \in C^t(h_k(\Omega_{jk}))^n .$$
\item Die Menge $S$ ist maximal bezüglich der Eigenschaft (b), d.h., falls es $U \subseteq \R^n, \Omega \subseteq M$ offen und eine Funktion
$h : U \rightarrow \Omega$ gibt mit $h \circ h_j^{-1} \in C^t(h_j(\Omega_j\cap\Omega))^n, h_j \circ h^{-1} \in C^t(h(\Omega_j \cap \Omega))^n$, dann ist bereits $h \in S$.
\end{enumerate}
\end{definition}

\begin{remark}
Falls $S_0 := \menge{ h : \Omega_j \rightarrow U_j }_{j\in A}$ eine beliebige Familie von Karten mit einer Indexmenge $A$ ist, die (a) und (b) erfüllt,
so lässt sich diese Menge zu einer eindeutigen differenzierbaren Struktur erweitern.
Diese ist gegeben durch
\begin{align*}
S :=& \left\{ h : \Omega \subseteq M \rightarrow U \subseteq \R^n : h \circ h_j^{-1} \in C^t(h_j(\Omega_j\cap\Omega))^n, \right. \\
    & \left. h_j \circ h^{-1} \in C^t(h(\Omega_j \cap \Omega))^n \textrm{ für jedes } j \in A \right\} 
.\end{align*}
\end{remark}

\begin{definition}
Sei $M$ ein topologischer, parakompakter, lokal euklidischer Raum mit einer Struktur $S$.
Nach Definition \ref{def:parakompakt} gibt es eine lokal endliche, relativ kompakte Überdeckung $\menge{\Omega_j}_{j\in\N}$, d.h. dass
$M = \bigcup_{j\in\N} \Omega_j$, $\bar{\Omega}_j$ für jedes $j \in \N$ kompakt ist und
$$\# \menge{j \in \N : \Omega_j \cap \Omega_{j_0} \neq \emptyset} < \infty \textrm{ für jedes } j_0 \in \N \textrm{ gilt.}$$
Wir nennen den Raum $M$ eine $C^t$-Mannigfaltigkeit, 
falls wir zu einer lokal endliche, relativ kompakte Überdeckung $\menge{\Omega_j}_{j\in\N}$ eine Menge von Karten
$\menge{h_j : \Omega_j \rightarrow U_j \subset \R^n}_{j\in\N} \subseteq S$ mit den folgenden Eigenschaften finden:
\begin{enumerate}[(a)]
\item Für alle $j,k \in \N \textrm{ für } \Omega_{jk} := \Omega_j \cap \Omega_k \neq \emptyset$ sind $h_j$ und $h_k$ verträglich, d.h. dass sich die Abbildung
$$ h_{jk} := h_j \circ h_k^{-1} : h_k(\Omega_{jk}) \subseteq U_k \rightarrow h_j(\Omega_{jk}) \subseteq U_j$$
zu einer Transformation $h_{jk} : V_k \rightarrow V_j$ für 
$h_k(\Omega_{jk}) \subsubset V_k \subset \R^n$ offen und $h_j(\Omega_{jk}) \subsubset V_j \subset \R^n$ offen
erweitern lässt.
\end{enumerate}
Die Familie $\menge{h_j : \Omega_j \rightarrow U_j}_{j\in\N}$ nennen wir einen Atlas.
Für $t = \infty$ sagen wir, dass $M$ eine glatte Mannigfaltigkeit ist.
\end{definition}

\begin{lemma}
Sei eine $C^t$-Mannigfaltigkeit $M$ mit Atlas $\menge{h_j : \Omega_j \rightarrow U_j}_{j\in\N}$ gegeben.
Für eine offene Teilmenge $U \subseteq M$ definiert $U$ eine $C^t$-Mannigfaltigkeit mit Atlas
$$\menge{h_j |_{U\cap\Omega_j} : \Omega_j \cap U \rightarrow U_j}_{j\in\N } .$$
\end{lemma}

\begin{definition}
Sei $M$ eine $C^t$-Mannigfaltigkeit, $\Omega \subseteq M$ offen.
Eine Funktion $u : \Omega \rightarrow \C$ ist in $C^t(\Omega)$, falls für jedes $j \in \N$ gilt, dass $u_j := u \circ h_j^{-1} \in C^t(h_j(\Omega_j \cap \Omega))$
Wir definieren $$ C^t_0(\Omega) := \menge{ u \in C^t(M) : \supp u \subsubset \Omega} .$$
Setzen wir die Funktionen mit kompakten Träger durch 0 fort, so erhalten wir eine Einbettung $C^t_0(\Omega) \hookrightarrow C^t_0(M)$.
\end{definition}

\begin{definition}
Eine Folge $\menge{f_j}_{j\in\N}, f_j \in C^t(M)$ heißt konvergent bzgl. $C^t(M)$, falls 
die Folge gleichmäßig konvergiert und jede Ableitung auf einer kompakten Teilmenge von $M$.
Dadurch erhält $C^t(M)$  die Struktur eines Fréchet-Raumes.
\end{definition}

\begin{definition}
Sei $\menge{\Phi_j}_{j\in\N}$ eine $C^t$-Partition der Eins, die $\Omega_j$ untergeordnet ist,
d.h. 
$$0 \le \Phi_j \in C^t_0(\Omega_j) \textrm{ und } \sum_{j=1}^\infty \Phi_j(x) = 1 \textrm{ für alle } x \in M .$$
Außerdem sei für jedes $j \in \N$ ein $\Psi_j \in C^t_0(\Omega_j)$ mit $\Phi_j \subsubset \Psi_j$ gegeben.
Wir bezeichnen eine Familie von Abschneidefunktionen ${(\Phi_j,\Psi_j)}_{j\in\N}$ mit diesen Eigenschaften als $C^t$-$\Psi$DO-Abschneide\-funktions\-familie.
Insbesondere gilt für jede $C^t$-$\Psi$DO-Abschneide\-funktions\-familie ${(\Phi_j,\Psi_j)}_{j\in\N}$, dass $\supp \Phi_j \cap \supp (1 - \Psi_j) = \emptyset$.
Wenn nicht anders erwähnt, sind mit $\varphi_j, \psi_j$ stets die Abbildungen 
$$ \varphi_j := \Phi_j \circ h_j^{-1} \in C^t_0(U_j), \; \psi_j := \Psi_j \circ h_j^{-1} \in C^t_0(U_j) $$
gemeint. 
\end{definition}

\begin{example}
Die Familie $\menge{\left( \varphi_j, \psi_j \right)}_{j\in\N}$ aus Lemma \ref{koord:lemmaPartition} ist eine glatte $\Psi$DO-Abschneide\-funktions\-familie des $\R^n$.
\end{example}

\begin{notation}
Sei $M$ eine $C^t$-Mannigfaltigkeit, $\Omega \subseteq M$ offen, $\Phi, \Psi \in C^t_0(\Omega)$.
Mit $\Phi \subsubset \Psi$ meinen wir, dass
$ \supp \Phi \subset \menge{x \in \Omega : \Psi(x) = 1 }$.
\end{notation}

\subsection{Pseudodifferentialoperatoren auf glatten Mannigfaltigkeiten} \label{sec:smoothmanifold}

\begin{remark}
Eine Karte $h : \Omega \rightarrow U$ einer glatten Mannigfaltigkeit induziert uns einen Isomorphismus
$h^* : C^t(U) \rightarrow C^t(\Omega), u \mapsto u \circ h$ für jedes $t \in \bar{\R_+}$.
Unter der natürlichen Einschränkung $r_\Omega : C^\infty(M) \rightarrow C^\infty(\Omega), f \mapsto f \mid_\Omega$ und der natürlichen Einbettung $i_\Omega : C^\infty(\Omega) \hookrightarrow C^\infty(M)$ können wir zu einem linearen Operator $P : C^\infty_0(M) \rightarrow C^0(M)$ einen linearen Operator $Q_h : C^\infty_0(U) \rightarrow C^0(U)$ finden, für den das Diagramm
$$
\begin{xy}
\xymatrix{
	C^\infty_0(\Omega) \ar[rr]^{r_\Omega \circ P \circ i_\Omega} & & C^\infty(\Omega) \\
	C^\infty_0(U) \ar[rr]^{Q_h} \ar[u]^{\simv}_{h^*}  & & C^\infty(U) \ar[u]^{\simv}_{h^*}
}
\end{xy}
$$
kommutiert.
\end{remark}

\begin{definition}
Sei $M$ eine glatte Mannigfaltigkeit und $P : C^\infty_0(M) \rightarrow C^0(M)$ ein linearer Operator.
Wir sagen, $P$ hat eine $C^\tau$-Kerndarstellung, falls für alle $j,l \in \N$ eine Funktion $k_{lj}(x,\eta) \in C^\tau(U_l) \times \Stetig{\infty}{U_j}$ existiert,
die bzgl. $x \in \R^n$ Hölder-stetig zum Grad $\tau$ und glatt bzgl. $\xi \in \R^n$ ist,
und für jedes $u \in C^\infty_0(\Omega_j)$ gilt
$$ (P u)(h_l^{-1}(x)) = \int k_{lj}(x,\eta) u_j(\eta) d\eta \textrm{ für alle } x \in U_l ,$$
wobei $u_j := u \circ h_j^{-1} \in C^\infty_0(U_j)$.
\end{definition}

\begin{definition}\label{def:mfgPSDO}
Sei $M$ eine glatte Mannigfaltigkeit mit Atlas $\menge{ h_j : \Omega_j \rightarrow U_j}_{j\in\N}$.
Unter einem Pseudodifferentialoperator der Klasse $\OPMfg{m}{\rho,\delta}{M}$ mit $0 \le 1-\rho \le \delta < \rho \le 1, \delta < 1$ verstehen wir einen linearen Operator
$ P : C^\infty_0(M) \rightarrow C^0(M)$, der sich bezüglich einer $\Psi$DO-Abschneide\-funktions\-familie ${(\Phi_j,\Psi_j)}_{j\in\N}$ in der Form
$$ ( P u)(y) = \sum_{j\in\N} (\Phi_j P \Psi_j u)(y) + \sum_{j\in\N} ( \Phi_j P (1-\Psi_j) u )(y)$$
schreiben lässt, wobei
\begin{enumerate}[(a)]
\item $\Phi_j P \Psi_j u$ in der Form
\begin{equation}\label{equ:mfgPsdoSymbol}
( \Phi_j P \Psi_j u)(h_j^{-1}(x)) = \varphi_j(X) p_j(X,D_x) (\psi_j u_j)(x) \textrm{ für alle } u \in C^\infty_0(\Omega_j)
\end{equation}
für ein $p_j(x,\xi) \in C^\tau S^m_{\rho,\delta}(\R^n, \R^n)$ mit $u_j := u \circ h_j^{-1}$ geschrieben werden kann und
\item für jedes $j \in \N$ der Operator $\Phi_j P (1-\Psi_j)$ eine $C^\tau$-Kerndarstellung besitzt.
\end{enumerate}
Die Familie $\menge{p_j(x,\xi)}_{j\in\N}$ heißt Symbol von $P$.
\end{definition}
\nomenclature[o]{$\OPMfg{m}{\rho,\delta}{M}$}{Klasse der \PSDOS auf einer glatten Mannigfaltigkeit $M$}

\begin{lemma}\label{lemma:mfgdisjunkt}
Sei $M$ eine glatte Mannigfaltigkeit mit Atlas $\menge{ h_j : \Omega_j \rightarrow U_j}_{j\in\N}$.
Sei $P \in \OPMfg{m}{\rho,\delta}{M}$.
Seien $\Psi, \Phi \in C^\infty_0(M)$ mit $\supp \Phi \cap \supp \Psi = \emptyset$ gegeben.
Dann hat $\Psi P \Phi$ eine $C^\tau$-Kerndarstellung.
\begin{proof}
Wir schreiben
$$ \Phi P \Psi = \sum_{j\in\N} \Phi (\Phi_j P \Psi_j) \Psi + \sum_{j\in\N} \Phi (\Phi_j P (1-\Psi_j)) \Psi .$$
Die beiden Summe sind endlich, da $\Phi$ und $\Psi$ kompakten Träger haben.
Da $\Phi_j P(1-\Psi_j)$ eine $C^\tau$-Kerndarstellung hat, besitzt auch $\Phi( \Phi_j P (1 - \Psi_j) ) \Psi$ eine.
Betrachten wir $\Phi (\Phi_j P \Psi_j) \Psi$ lokal auf $U_j$, so sehen wir, dass für alle $u \in C^\infty_0(\Omega_j)$
$$ \Phi \Phi_j P(\Psi_j \Psi u)(h_j^{-1}(x)) = \varphi(X) \varphi_j(X) p_j(X,D_x) (\psi_j \psi u_j)(x) \textrm{ für alle } x \in U_j$$
gilt, wobei $u_j := u \circ h_j^{-1}$, $\varphi := \Phi \circ h_j^{-1}, \psi := \Psi \circ h_j^{-1}$ in $U_j$.
Da $\supp(\psi_j\psi) \cap \supp(\varphi_j\varphi) = \emptyset$ folgt aus Theorem \ref{thm:kerndisjunkt}, 
dass $\Phi (\Phi_j P \Psi_j) \Psi$ die $C^\tau$-Kerndarstellung
$$\Phi\Phi_l P(\Psi_l\Psi u)(h_l^{-1}(x)) = \int_{\R^n} \varphi(x)\varphi_l(x) k_{lj}(x,\eta) \psi_l(\eta) \psi(\eta) u_j(\eta) d\eta \textrm{ für alle } x \in U_l$$
mit $ \varphi(x)\varphi_l(x) k_{lj}(x,\eta) \psi_l(\eta) \psi(\eta) \in C^\tau(U_l) \times \Stetig{\infty}{U_j}$ besitzt. 
\end{proof}
\end{lemma}

\begin{thm}\label{thm:mfgInvarianz}
Sei $\menge{ h_j : \Omega_j \rightarrow U_j}_{j\in\N}$ ein Atlas der glatten Mannigfaltigkeit $M$.
Sei $\menge{ h'_{j'} : \Omega'_{j'} \rightarrow U'_{j'}}_{j'\in\N}$ ein weiterer Atlas der selben $C^\infty$-Struktur,
sodass für alle $j,j' : \Omega_j \cap \Omega'_{j'} \neq \emptyset$ sich 
$$h_{j,j'} := h_j \circ (h'_{j'})^{-1} : h'_{j'}(\Omega_j \cap \Omega'_{j'}) \rightarrow h_j(\Omega_j \cap \Omega'_{j'})$$
zu einem regulären $C^\infty$-Diffeomorphismus auf einer offenen Umgebung von $\overline{ h'_{j'}(\Omega_j \cap \Omega'_{j'})}$ fortsetzen lässt, 
die auf eine offene Umgebung von $\overline{ h_j(\Omega_j \cap \Omega'_{j'})}$ abgebildet wird.
Dann ist ein Pseudodifferentialoperator $P : C^\infty_0(M) \rightarrow C^0(M)$ der Klasse $\OPMfg{m}{\rho,\delta}{M}$ bezüglich 
des Atlanten $\menge{ h_j : \Omega_j \rightarrow U_j}_{j\in\N}$ wiederum ein Pseudodifferentialoperator bezüglich 
des Atlanten $\menge{ h'_{j'} : \Omega'_{j'} \rightarrow U'_{j'}}_{j'\in\N}$ .
\begin{proof}
Für $\Psi, \Phi \in C^\infty(M)$ mit $\supp \Phi \cap \supp \Psi = \emptyset$ liefert Lemma \ref{lemma:mfgdisjunkt}, 
dass $\Phi P \Psi$ eine $C^\tau$-Kerndarstellung bezüglich des Atlanten $\menge{ h_j : \Omega_j \rightarrow U_j}_{j\in\N}$ besitzt.
Da die $h_{j,j'}$ für alle $j,j' \in \N$ glatte Koordinatentransformationen sind, hat $\Phi P \Psi$ auch eine  $C^\tau$-Kerndarstellung bezüglich des Atlanten $\menge{ h'_{j'} : \Omega'_{j'} \rightarrow U'_{j'}}_{j'\in\N}$ nach Lemma \ref{lemma:trafoSinfty}.
Ist ${(\Phi'_{j'},\Psi'_{j'})}_{j'\in\N}$ eine weitere glatte $\Psi$DO-Abschneide\-funktions\-familie, so können wir schreiben:
\begin{equation}\label{equ:thmMfgInvarianzAufsp}
P = \sum_{j'\in\N} \Phi'_{j'} P \Psi'_{j'} + \sum_{j'\in\N} \Phi'_{j'} P (1 - \Psi'_{j'}) .
\end{equation}
Hier hat der zweite Term nach Lemma \ref{lemma:mfgdisjunkt} eine $C^\tau$-Kerndarstellung, da $\supp \Phi'_{j'} \cap \supp (1-\Psi'_{j'}) = \emptyset$.
Für den ersten Term von (\ref{equ:thmMfgInvarianzAufsp}) fassen wir $\Phi'_{j'} P \Psi'_{j'}$ als \PSDO bzgl. der Partition $\menge{(\Phi_j,\Psi_j)}_{j\in\N}$ auf, sodass
\begin{equation}\label{equ:thmMfgInvarianzAnderePart}
\Phi'_{j'} P \Psi'_{j'} = \sum_{j\in\N} \Phi'_{j'} ( \Phi_j P \Psi_j) \Psi'_{j'} + \sum_{j\in\N}  \Phi'_{j'} ( \Phi_j P (1 - \Psi_j)) \Psi'_{j'} .
\end{equation}
Im Folgenden betrachte nur die $j,j' \in \N$ mit $\Omega_j \cap \Omega'_{j'} \neq \emptyset$.
Unter dieser Einschränkung erfüllt $\Phi'_{j'} ( \Phi_j P \Psi_j) \Psi'_{j'}$ in $U_j$ lokal
$$ \Phi'_{j'} \Phi_j P ( \Psi_j \Psi'_{j'} u) (h_j^{-1}(x)) = 
\varphi'_{j'}(X) \varphi_j(X) p_j(X,D_x) (\psi_j \psi'_{j'} u_j)(x) \textrm{ für alle } x \in U_j ,$$
wobei $u_j := u \circ h_j^{-1}, \varphi'_{j'} := \Phi'_{j'} \circ h_j^{-1}, \psi'_{j'} = \Psi'_{j'} \circ h_j^{-1}$.

Wir werden nun mit $h_{j,j'}^*$ den Operator von einer Teilmenge von $U_j$ nach $U'_{j'}$ transformieren.
Dann sehen wir wegen der Invarianz der Koordinatentransformation (Theorem $\ref{thm:trafo}$), 
dass $\Phi'_{j'} ( \Phi_j P \Psi_j) \Psi'_{j'}$ ein Pseudodifferentialoperator auf einer Teilmenge von $U'_{j'}$ ist:

Zunächst wissen wir, dass sich $h_{j,j'} := h_j \circ (h'_{j'})^{-1}$ zu einem $C^\infty$-Diffeomorphismus 
$h_{j,j'} : V \rightarrow U$
auf den offenen Teilmengen $U, V$ von $\R^n$ mit 
$\derive{h_{j,j'}}{y}, \derive{h_{j,j'}^{-1}}{y} \in C^\infty(U)$ für alle $y \in V$ fortsetzen lässt.
Wir transformieren lokal auf $U_j$ mit $h_{j,j'}$:
\begin{align*}
& \left(\Phi'_{j'} \Phi_j P(\Psi_j \Psi'_{j'} u)\right)\left( (h'_{j'})^{-1}(x)\right)  
\\&=
\left( \Phi'_{j'} \circ (h'_{j'})^{-1}(x) \right) \\& \quad
\left(
(\Phi_j \circ h_j^{-1})
p_j(X,D_x)
\left(
(\Psi_j \circ h_j^{-1})
(\Psi'_{j'} \circ (h'_{j'})^{-1} \circ h'_{j'} \circ h_j^{-1})
(u \circ (h'_{j'})^{-1} \circ h'_{j'} \circ h_j^{-1}) 
\right)
\right) \\& \quad
\left( h_j \circ (h'_{j'})^{-1}(x) \right)
\\ &=
\tilde{\varphi}'_{j'}
\left(
\left( h_j \circ (h'_{j'})^{-1} \right)^*
\left(
\varphi_j p_j(X,D_x) \left( \psi_j \left( h'_{j'} \circ h_j^{-1} \right)^* (\tilde{\psi}'_{j'} u'_{j'}) \right)
\right)
\right)
(x) \textrm{ für alle } x \in U_j \cap U_{j'}
,\end{align*}
wobei
$u'_{j'} := u \circ (h'_{j'})^{-1}, 
\tilde{\psi}'_{j'} := \Psi'_{j'} \circ (h'_{j'})^{-1},
\tilde{\varphi}'_{j'} := \Phi'_{j'} \circ (h'_{j'})^{-1}$.

Nach Bemerkung \ref{remark:trafo} (setze $h := h_{j,j'}, \tilde{\Omega}_y := h'_{j'}(\Omega'_{j'} \cap \Omega_j)$) gibt es ein $q_{j,j'} \in C^\tau S^m_{\rho,\delta}(\R^n,\R^n)$ mit 
$$ q_{j,j'}(X,D_x) u = h_{j,j'}^* \left( \varphi_j(X) p_j(X,D_x) \left( \psi_j ( h_{j,j'}^{-1} )^* u \right) \right) .$$
Für
$$ q_{j'}(x,\xi) := \sum_{j \in \N : \Omega'_{j'} \cap \Omega_j \neq \emptyset} q_{j,j'}(x,\xi)$$
sehen wir, dass unser erster Term von (\ref{equ:thmMfgInvarianzAnderePart}) in der Form
$$ ( \Phi'_{j'}  ( \Phi_j P \Psi_j) \Psi'_{j'} u)( (h'_{j'})^{-1}(x)) = \tilde{\varphi}'_{j'} q_{j'}(X,D_x) (\tilde{\psi}'_{j'} u'_{j'})(x) $$
für alle $u'_{j'} = u \circ \left(h'_{j'}\right)^{-1} \in C^\infty_0(U'_{j'} \cap U_j)$
geschrieben werden kann, unabhängig von $j \in \N$ - also gilt obige Gleichung auch für alle $u'_{j'} \in C^\infty_0(U'_{j'})$.

Nun hat von (\ref{equ:thmMfgInvarianzAnderePart}) der zweite Term $ \sum_j  \Phi'_{j'} ( \Phi_j P (1 - \Psi_j)) \Psi'_{j'}$ eine $C^\tau$-Kerndarstellung,
da $\Phi_j P (1 - \Psi_j)$ eine hat, es also ein $r_{j',j}(x,\xi) \in C^\tau S^{-\infty}(\R^n,\R^n)$ gibt, sodass 
für jedes $u \in C^\infty_0(\Omega_j)$ und für alle $x \in U'_{j'}$
\begin{align*}
\Phi'_{j'} \Phi_j P \left( (1 - \Psi_j) \Psi'_{j'} u_j \right) \left((h'_{j'})^{-1}(x)\right) 
&=  \tilde{\varphi}'_{j'}(x) r_{j',j}(X,D_x) (\tilde{\psi}'_{j'}u)(x) \\
&= \tilde{\varphi}'_{j'}(x) \int k_{r_{j',j}}(x,y) (\tilde{\psi}'_{j'} u_j)(y) dy
\end{align*}
mit $k_{r_{j',j}} \in C^\tau(\R^n) \times \Stetigi{\infty}$ gilt.
Da nun $\tilde{\varphi}'_{j'}, \tilde{\psi}'_{j'}$ glatte Abbildungen sind, ist 
$$\tilde{\varphi}'_{j'}(x) k_{r_{j',j}}(x,y) \tilde{\psi}'_{j'}(y)$$ ein Kern, der die Bedingungen in Definition \ref{def:kern} erfüllt.
Somit definiert $P$ bezüglich des Atlanten $\menge{ h'_{j'} : \Omega'_{j'} \rightarrow U'_{j'}}_{j'\in\N}$ wiederum einen Pseudodifferentialoperator.

\end{proof}
\end{thm}

\begin{definition}
Sei $P$ ein Pseudodifferentialoperator der Klasse $\OPMfg{m}{\rho,\delta}{M}$ mit Symbol $\menge{p_j(x,\xi)}_{j\in\N}$.
Wir sagen, $\sigma_N(P) := \menge{p_j(x,\xi)}_{j\in\N}$ ist das Symbol von $P$ mit der Genauigkeit $N \in \N_0 \cup \menge{\infty}$, 
falls für jede glatte $\Psi$DO-Abschneide\-funktions\-familie  ${(\Phi_j,\Psi_j)}_{j\in\N}$ der Operator
$ P - \sum_j \Phi_j P_j \Psi_j$ Element der Operatorklasse $\OPMfg{m - (\rho-\delta)(N+1)}{\rho,\delta}{M}$ ist,
wobei $\Phi_j P_j \Psi_j$ ein Pseudodifferentialoperator ist, der lokal auf $U_j$ durch 
$$ \left( \Phi_j P_j (\Psi_j u) \right) \left( h_j^{-1}(x)\right) = \varphi_j(X) p_j(X,D_x) (\psi_j u_j)(x) \textrm{ für alle } x \in U_j$$
mit $u_j := u \circ h_j^{-1}$ definiert ist.
Für $N = \infty$ folgt $ P - \sum_j \Phi_j P_j \Psi_j \in \OPMfgInf$, wobei $\OPMfgInf$ die Klasse der Operatoren ist, deren Symbole eine $C^\tau$-Darstellung besitzen.
Wir nennen $\sigma_0(P)$ das Prinzipalsymbol, $\sigma_\infty(P)$ das vollständige Symbol.
\end{definition}

\begin{thm}\label{thm:mfgSymbolSim}
Hat ein Pseudodifferentialoperator $P$ der Klasse $\OPMfg{m}{\rho,\delta}{M}$ mit Symbol  $\menge{p_j(x,\xi)}_{j\in\N}$ ein anderes Symbol $\menge{\tilde{p}_j(x,\xi)}_{j\in\N}$,
dann gilt für jedes $\varphi \in C^\infty_0(U_j)$, dass
\begin{equation}\label{equ:mfgSymbolSim}
\varphi(x) \left( \tilde{p}_j(x,\xi) - p_j(x,\xi) \right) \in C^\tau S^{-\infty}(\R^n,\R^n\times\R^n) .
\end{equation}
Insbesondere ist damit $\sigma_\infty(P)$ eindeutig bestimmt.
\begin{proof}
Sei $j_0 \in \N$ fest gewählt.
Für $\varphi \in C^\infty_0(U_{j_0})$ wählen wir $\psi,\varphi_{j_0} \in C^\infty_0(U_{j_0})$ mit 
$$\varphi \subsubset \psi \subsubset \varphi_{j_0} \textrm{ in } U_{j_0} .$$
Wähle eine passende glatte $\Psi$DO-Abschneide\-funktions\-familie ${(\Phi_j,\Psi_j)}_{j\in\N}$, sodass $\Phi_{j_0} = \varphi_{j_0} \circ h_{j_0}$.
Die Symbole  $\menge{p_j(x,\xi)}_{j\in\N}$ und $\menge{\tilde{p}_j(x,\xi)}_{n\in\N}$ sollen mit dieser glatten $\Psi$DO-Abschneide\-funktions\-familie kompatibel sein, also (\ref{equ:mfgPsdoSymbol}) erfüllen.
Definiere $\Phi := \varphi \circ h_{j_0}, \Psi := \psi \circ h_{j_0}$.
Dann ist
$ \Phi \subsubset \Psi \subsubset \Phi_{j_0} \subsubset \Psi_{j_0}$
und damit $\Phi\Phi_{j_0} = \Phi$ sowie $\Psi\Psi_{j_0} = \Psi$, also
$$\Phi P \Psi u = \Phi \Phi_{j_0} P \Psi_{j_0} \Psi u \textrm{ auf } U_{j_0} \textrm{ für alle } u \circ h_{j_0}^{-1} \in C^\infty_0(U_{j_0}) .$$
Da $\varphi \subsubset \varphi_{j_0}$ und $\psi \subsubset \psi_{j_0}$, gilt 
\begin{align*}
& \varphi(X) \left( \tilde{p}_{j_0}(X,D_x) - p_{j_0}(X,D_x) \right) \psi(X') u_{j_0}(x) \\
&= \varphi(X) \varphi_{j_0}(X) \left( \tilde{p}_{j_0}(X,D_x) - p_{j_0}(X,D_x) \right) \psi_{j_0}(X') \psi(X') u_{j_0}(x) \\
&= \Phi \Phi_{j_0} P (\Psi_{j_0} \Psi u)(h_{j_0}^{-1}(x)) - \Phi \Phi_{j_0} P (\Psi_{j_0} \Psi u)(h_{j_0}^{-1}(x)) 
= 0
\end{align*}
lokal für jedes $u \in C^\infty_0(\Omega_{j_0})$ und für alle $x \in U_{j_0}$, wobei $u_{j_0} := u \circ h_{j_0}^{-1}$. 
Da $\psi \in C^\infty_0(U_{j_0})$, hält die Gleichung
$$ \varphi(X) \left( \tilde{p}_{j_0}(X,D_x) - p_{j_0}(X,D_x) \right) \psi(X') u = 0 \textrm{ für alle } u \in \Schwarzi .$$ 
Folgerung \ref{cor:introIdentisch} liefert die Gleichheit $\varphi(X) \tilde{p}_{j_0}(X,D_x) \psi(X') = \varphi(X) p_{j_0}(X,D_x) \psi(X')$.
Nach Theorem \ref{thm:kerndisjunkt} sehen wir also, dass bei der Aufspaltung
\begin{align*}
\varphi(X) \left( \tilde{p}_{j_0}(X,D_x) - p_{j_0}(X,D_x) \right) 
&= \varphi(X) \left( \tilde{p}_{j_0}(X,D_x) - p_{j_0}(X,D_x) \right) \psi(X') \\
&+ \varphi(X) \left( \tilde{p}_{j_0}(X,D_x) - p_{j_0}(X,D_x) \right) (1 - \psi(X'))
\end{align*}
der zweite Term wegen $\varphi \subsubset \psi$ der Klasse $\OP{C^\tau S^{-\infty}(\R^n,\R^n\times\R^n)}$ angehört.
Der erste Term fällt aber nach obiger Rechnung bereits weg.
\end{proof}
\end{thm}

\begin{cor}
Sei $P \in \OPMfg{m}{\rho,\delta}{M}$ mit Symbol $\menge{p_j(x,\xi)}_{j\in\N}$ bezüglich der glatte $\Psi$DO-Abschneide\-funktions\-familie ${(\Phi_j,\Psi_j)}_{j\in\N}$.
Dann gibt es für jedes $j \in \N$ ein Symbol $\tilde{p}_j \in C^\tau S^m_{\rho,\delta}(\R^n,\R^n)$, für das gilt:
\begin{enumerate}[(a)]
\item $\tilde{p}_j$ erfüllt die Gleichung (\ref{equ:mfgSymbolSim}).
\item Auf $U_j$ gilt 
$ \Psi_j P \Phi_j = \psi_j(X) \tilde{p}_j(X,D_x) \varphi_j(X')$.
\item Der Operator $\sum_j (1-\Psi_j) P \Phi_j$ hat eine $C^\tau$-Kerndarstellung.
\end{enumerate}

Insbesondere können wir $P$ darstellen durch
$$ P = \sum_j \Psi_j P \Phi_j + \sum_j (1-\Psi_j) P \Phi_j .$$
\begin{proof}
Setzen wir $\Psi_j P \Phi_j$ in Definition \ref{def:mfgPSDO} für $P$ ein, erhalten wir 
$$ \Psi_j P \Phi_j = \sum_k \Psi_j ( \Phi_k P \Psi_k ) \Phi_j + \sum_k \Psi_j \Phi_k P(1-\Psi_k) \Phi_j .$$

Dann ist klar, dass der zweite Term $\Psi_j \Phi_k P (1-\Psi_k) \Phi_j \in \OPMfgInf$ auf $U_j$, 
da $\Phi_j P (1 - \Psi_j)$ nach Lemma \ref{lemma:mfgdisjunkt} eine $C^\tau$-Kerndarstellung besitzt.
Deshalb existiert nach Folgerung \ref{cor:koordAsympt} für jedes $j \in \N$ ein $r_j(x,\xi) \in C^\tau S^{-\infty}(\R^n,\R^n)$, sodass
$$\sum_k \Psi_j \Phi_k P ( (1-\Psi_k) \Phi_j u) (h_j^{-1}(x)) = \psi_j(x) r_j(X,D_x) (\varphi_j u_j)(x) ,$$
wobei $u_j = u \circ h_j^{-1}$.
Für den ersten Term folgt mit Theorem \ref{thm:trafo} (setze $h := h_k \circ h_j^{-1}, \tilde{\Omega}_y := h_k(\Omega_k \cap \Omega_j)$), 
dass es ein $p_j'(x,\xi) \in C^\tau S^m_{\rho,\delta}(\R^n,\R^n)$ gibt, sodass
$$ \sum_k \Psi_j \Phi_k P (\Psi_k \Phi_j u) (h_j^{-1}(x)) = \psi_j p_j'(X,D_x) (\varphi_j u_j)(x) \textrm{ für alle } x \in U_j .$$
Setzen wir $\tilde{p}_j := p_j' + r_j$, so ist
$$ \Psi_j P (\Phi_j u)(h_j^{-1}(x)) = \psi_j(X) \tilde{p}_j(X,D_x) (\varphi_j u_j)(x) \textrm{ für alle } x \in U_j .$$
Wähle wie im letzten Beweis für $j_0 \in \N$ fest zwei glatte Funktionen $\varphi, \psi \in C^\infty_0(U_{j_0})$, sodass
$$ \varphi \subsubset \psi \subsubset \varphi_{j_0} \textrm{ auf } U_{j_0} ,$$
wobei $\varphi_{j_0} := \Phi_{j_0} \circ h_{j_0}^{-1}$.
Also ist $\varphi \subsubset \psi \subsubset \varphi_{j_0} \subsubset \psi_{j_0}$.
Damit folgt
\begin{align*}
\varphi_{j_0}(X) \varphi(X) p_{j_0}(X,D_x) \psi(X') \psi_{j_0}(X') 
&= \varphi(X) p_{j_0}(X,D_x) \psi(X') \\
&= \psi_{j_0}(X) \varphi(X) p_j(X,D_x) \psi(X') \varphi_{j_0}(X') \\
&= \psi_{j_0}(X) \varphi(X) \tilde{p}_{j_0}(X,D_x) \psi(X') \varphi_{j_0}(X') \\
&= \varphi(X) \tilde{p}_{j_0}(X,D_x) \psi(X') \\
&= \varphi_{j_0}(X) ( \varphi(X) \tilde{p}_{j_0}(X,D_x) \psi(X') ) \psi_{j_0}(X')
.\end{align*}
Mit Theorem \ref{thm:mfgSymbolSim} auf den Pseudodifferentialoperator $\Phi P \Psi$ angewandt, ergibt
$$ \varphi(X) \tilde{p}_{j_0}(X,D_x) \psi(X') \equiv \varphi(X) p_{j_0}(X,D_x) \psi(X') \mod \OP{C^\tau S^{-\infty}(\R^n,\R^n\times\R^n)} ,$$
also
$ \varphi(x) \tilde{p}_{j_0}(x,\xi) \equiv \varphi(x) p_{j_0}(x,\xi) \mod{C^\tau S^{-\infty}(\R^n,\R^n\times\R^n)}$.
Da $\varphi \in C^\infty_0(U_{j_0})$ und die glatte $\Psi$DO-Abschneide\-funktions\-familie beliebig waren, folgt die Behauptung.
\end{proof}
\end{cor}

\section{Koordinatentransformation auf nicht-glatten Mannigfaltigkeiten}\label{sec:nonsmooth}
\subsection{Symbole in $(x,y,\xi)$-Form}\label{sec:x,y,xi}

\begin{definition}\label{def:xyxi}
Ein Symbol $p : \R^n \times \R^n \times \R^n \rightarrow \C$ gehört der Klasse $C^\tau S^m_{1,0}(\R^n \times \R^n, \R^n)$ an, 
falls $p(x,y,\xi)$ zum Grad $\tau$ Hölder-stetig bezüglich $(x,y)\in\R^{2n}$ und glatt bezüglich $\xi \in \R^n$ ist,
d.h. $p(\hold,y,\xi) \in C^\tau(\R^n)$ und $p(x,\hold,\hold) \in \Stetig{\infty}{\R^n\times\R^n}$,
und es zu jedem Multiindex $\beta\in\N^n_0$ eine Konstante $C_\beta > 0$ unabhängig von $\xi\in\R^n$ gibt, sodass
\begin{align*}
\norm{ \partial_\xi^\beta p(\hold,\hold,\xi) }{C^\tau(\R^n\times\R^n)} &\le C_\beta \piket{\xi}{m-\abs{\beta}} \textrm{, und }\\
\abs{ \partial_\xi^\beta p(x,y,\xi) } &\le C_\beta \piket{\xi}{m-\abs{\beta}} \textrm{ für alle } x,y,\xi\in \R^n
.\end{align*}
Schließlich definieren wir zu einem $N \in \N$ das Symbol $p_N(x,y,\xi) := \sum_{j=0}^N p(x,y,\xi) \lilwood{j}(\xi)$, 
wobei $\menge{\lilwood{j}}_{j\in\N_0}$ ein Vertreter der Littlewood-Paley-Partition aus Definition \ref{def:littlewoodpaley} ist.
\end{definition}
\nomenclature[s]{$C^\tau S^m_{1,0}(\R^n \times \R^n, \R^n)$}{Klasse der Symbole in $(x,y,\xi)$-Form mit Hölder-Stetigkeitsgrad $\tau$}

\begin{thm}\label{thm:operatorXY}
Ein Symbol $p \in C^\tau S^m_{1,0}(\R^n \times \R^n, \R^n)$
definiert zu einem $u \in H^{s+m}_q(\R^n)$ den Operator
\begin{equation}\label{equ:operatorXYXi}
p(X,Y,D_x) u(x) := \lim_{N\rightarrow\infty} \iint e^{i(x-y)\cdot\xi} p_N(x,y,\xi) u(y) dy \dslash\xi
\end{equation}
für $s \in (-\tau,\tau), 1 < q < \infty$.
Wir nennen diesen Operator einen Pseudodifferentialoperator in $(x,y,\xi)$-Form.
\begin{proof} Zu zeigen ist die Wohldefiniertheit des Operators.
Für ein festes $x_0 \in \R^n$ haben wir 
$$ p_N(X,Y,D_x) u(x_0) = \iint e^{i(x_0-y)\xi} p_N(x_0,y,\xi) u(y) dy \dslash\xi \textrm{ für alle } u \in \Schwarzi .$$
Für $p_{N,x_0}(y,\xi) := p_N(x_0,\xi,y) \in C^\tau S^m_{1,0}(\R^n,\R^n)$ ist
$ p_{N,x_0}(D_x,X) u(x_0) = p_N(X,Y,D_x) u(x_0) $.
Nach Satz \ref{satz:adjunPSDO} definiert das Symbol $b_{N,x_0}(y,\xi) = \overline{ p_{N,x_0}(y,\xi) }$ einen Operator
$b_{N,x_0}(X,D_x)$ in $x$-Form mit $b_{N,x_0}(X,D_x)^* = p_{N,x_0}(D_x,X)$.
Setzen wir $b_{x_0}(y,\xi) := \overline{ p(x_0,y,\xi) }$, so haben wir $b_{N,x_0}(y,\xi) = b_{x_0}(y,\xi) \sum_{j=0}^N \lilwood{j}(\xi)$.
Lemma \ref{lemma:lilwoodSchwarzKonv} liefert für alle $u \in \Schwarzi$
$$b_{N,x_0}(X,D_x) u(x_0) 
= b_{x_0}(X,D_x) \sum_{j=0}^N \lilwood{j}(D_x) u(x_0) 
\rightarrow b_{x_0}(X,D_x) u(x_0) \textrm{ für } N \rightarrow \infty .$$
Nach Theorem \ref{beschr:Stau_1,0} ist
$ b_{N,x_0}(X,D_x) : H^{s+m}_q(\R^n) \rightarrow H^s_q(\R^n) \textrm{ für alle } s \in (-\tau,\tau), 1 < q < \infty$
ein beschränkter linearer Operator. 
Nun ist $\Schwarzi \subset H^{s+m}_q(\R^n)$ dicht erhalten.
Wir folgern mit dem Dichtheitsargument für alle $x_0 \in \R^n$ fest gewählt, dass
$$b_{N,x_0}(X,D_x) u(x_0) \rightarrow b_{x_0}(X,D_x) u(x_0) \textrm{ für } N \rightarrow \infty \textrm{ für alle } u(x_0) \in H^s_q(\R^n) .$$
Wegen $b_{N,x_0}(X,D_x)^* = p_{N,x_0}(D_x,X)$ konvergiert auch $p_{N,x_0}(D_x,X) u(x_0)$ gegen einen Operator $p_{x_0}(D_x,X) u(x_0)$.
Wir erhalten schließlich die punktweise Konvergenz
$$ p_N(X,Y,D_x) u(x_0) = p_{x_0,N}(D_x,X) u(x_0) \rightarrow p_{x_0}(D_x,X) u(x_0) = p(X,Y,D_x) u(x_0) $$
für $N \rightarrow \infty$ für jedes $u \in H^{s+m}_q(\R^n)$.
\end{proof}
\end{thm}

\begin{cor}\label{cor:adjunXY} 
Der formal adjungierte Operator eines \PSDO in $(x,y,\xi)$-Form ist ein \PSDO in $(x,y,\xi)$-Form.
Genauer erfüllt für ein Symbol $p \in C^\tau S^m_{1,0}(\R^n \times \R^n, \R^n)$ der formal adjungierte Operator gegeben durch das Symbol
$p(x,y,\xi)^* := p(y,x,-\xi)$ die Gleichheit
$$ \skp{ p(X,X',D_x) u}{ v }{L^2(\R^n)} = \skp{ u }{ p(X,X',D_x)^* v }{L^2(\R^n)} .$$
\begin{proof}
Wir werden analog zu Beweis von Satz \ref{satz:adjunPSDO} schrittweise mit Hilfe des Satzes von Fubini das linke Skalarprodukt umformen.
Zunächst ist
$ (y,\xi) \mapsto e^{-iy\cdot\xi} p_N(x,y,\xi) u(y) \in L^1(\R^n\times\R^n)$ da der Träger von $\xi \mapsto p_N(x,y,\xi)$ kompakt ist.
Außerdem ist für $p_{x,N}(y,\xi) := p_N(x,y,\xi)$ der Operator $p_{x,N}(X,D_x) u \in L^\infty(\R^n)$, also
$$ (x,y) \mapsto e^{ix\cdot\xi} \overline{v(x)} \int e^{-iy\cdot\xi} p_N(x,y,\xi) u(y) \dslash\xi \in L^1(\R^n\times\R^n) .$$
Insbesondere haben wir $ (x,\xi) \mapsto e^{i(x-y)\cdot\xi} p_N(x,y,\xi) \overline{v(x)}  \in L^1(\R^n\times\R^n)$.
Setzen wir 
$$p_N(x,y,\xi)^* := p_N(y,x,-\xi) ,$$ 
so ergibt sich mit Fubini's Satz
\begin{align*}
\skp{ p_N(X,X',D_x) u }{v}{L^2(\R^n)} &= 
\int \left( \iint e^{i(x-y)\cdot\xi} p_N(x,y,\xi) u(y) dy \dslash\xi \right) \overline{v(x)} dx \\
&= \int \left( \iint e^{i(x-y)\cdot\xi} p_N(x,y,\xi) \dslash\xi u(y) dy \right) \overline{v(x)} dx \\
&= \int u(y) \left( \int \left( \int e^{i(x-y)\cdot\xi} p_N(x,y,\xi) \dslash\xi \right) \overline{v(x)} dx \right) dy \\
&= \int u(y) \left( \iint e^{i(x-y)\cdot\xi} p_N(x,y,\xi) \overline{v(x)} dx \dslash\xi \right) dy \\
&= \int u(y) \overline{ \iint e^{i(y-x)\cdot\xi} \overline{p_N(x,y,\xi)} v(x) dx \dslash\xi } dy \\
&= \int u(y) \overline{ \iint e^{i(x-y)\cdot\xi} \overline{p_N(x,y,-\xi)} v(x) dx \dslash\xi } dy \\
&= \skp{ u }{p_N(X,X',D_x)^* v}{L^2(\R^n)}  
,\end{align*} 
wobei wir im letzten Schritt die Spiegelung $\xi \mapsto -\xi$ verwendet haben.
Schließlich folgt
\begin{align*}
\skp{ p(X,X',D_x) u}{ v }{L^2(\R^n)} 
&= \lim_{N\rightarrow\infty} \skp{ p_N(X,X',D_x) u}{ v }{L^2(\R^n)} \\
&= \lim_{N\rightarrow\infty} \skp{ u }{p_N(X,X',D_x)^* v}{L^2(\R^n)} \\
&= \skp{ u }{ p(X,X',D_x)^* v }{L^2(\R^n)}
.\end{align*} 
\end{proof}
\end{cor}

\begin{remark}\label{remark:badform}
Sei ein Symbol $p \in C^\tau S^m_{1,0}(\R^n \times \R^n, \R^n)$ gegeben.
\begin{enumerate}[(a)]
\item \label{remark:badformA}
Falls ein $p(x,y,\xi)$ glatt bzgl. $y$ ist, also $p(x,\hold,\xi) \in \Stetigi{\infty}$, dann
lässt sich der Operator $p(X,Y,D_x)$ als oszillatorisches Integral darstellen, und wir erhalten
\begin{align*}
p(X,Y,D_x) u(x) 
&= \lim_{N\rightarrow\infty} \osint e^{i(x-y)\cdot\xi} p_N(x,y,\xi) u(y) dy \dslash\xi \\
&= \osint e^{i(x-y)\cdot\xi} p(x,y,\xi) u(y) dy \dslash\xi \\
&= p(X,D_x,Y) u(x)
.\end{align*}
Operatoren in $(x,\xi,y)$-Form sind also Spezialfälle von Operatoren in $(x,y,\xi)$-Form.
\item \label{remark:badformB}
Falls der Träger von $\xi \mapsto p(x,y,\xi)$ kompakt ist, haben wir 
$ (y,\xi) \mapsto p(x,y,\xi) u(y) \in L^1(\R^n\times\R^n)$ für $u \in \Schwarzi$.
Dominante Konvergenz liefert
$$ p(X,Y,D_x) u(x) = \iint e^{i(x-y)\cdot\xi} p(x,y,\xi) u(y) dy \dslash\xi .$$
\end{enumerate}
\end{remark}

Ab nun werden wir \PSDO mit geringerer Stetigkeit betrachten. Dazu werden wir die Notation der Stetigkeit mit $\tau$ fallen lassen, und stattdessen 
die Hölder-Stetigkeit eines Symbols mit $\theta \in (0,1)$ klassifizieren.
Symbole dieser Klasse weisen tatsächlich einige interessante Abbildungseigenschaften auf, die wir hier mit Hilfe von \cite{Tools} studieren werden:

\begin{thm}\label{thm:result1}
Sei $p \in C^\theta S^0_{1,0}(\R^n\times\R^n,\R^n)$, dann ist
$$ p(X,Y,D_x) : L^q(\R^n) \rightarrow L^q(\R^n) \textrm{ für alle } 1 < q < \infty .$$
\begin{proof} Siehe \cite[Proposition 9.1]{Tools}
\end{proof}
\end{thm}

\begin{cor}\label{thm:result3}
Sei $p \in C^{1,\theta} S^1_{1,0}(\R^n\times\R^n,\R^n)$ und $p(x,x,\xi) = 0$ für alle $x,\xi\in\R^n$.
Dann ist
$$ p(X,Y,D_x) : L^q(\R^n) \rightarrow L^q(\R^n) \textrm{ für alle } 1 < q < \infty .$$
\begin{proof}
Wegen $p(x,x,\xi) = 0$ finden wir $b^1, \ldots, b^n \in C^\theta S^1_{1,0}(\R^n\times\R^n,\R^n)$ mit
$$p(x,y,\xi) = \sum_{j=1}^n b^j(x,y,\xi) (x_j - y_j) .$$
Wir verwenden analog zu Definition \ref{def:xyxi} die Bezeichnung $b^j_N(x,y,\xi) := \sum_{k=0}^N b^j(x,y,\xi) \lilwood{k}(\xi)$.
Partielle Integration unter Verwendung von $D_{\xi_j} e^{i(x-y)\cdot\xi} = (x_j - y_j) e^{i(x-y)\cdot\xi}$ liefert
$$ p(X,Y,D_x) u(x) = - \lim_{N \rightarrow \infty} \sum_{j=0}^n \iint e^{i(x-y)\cdot\xi} D_{\xi_j} b^j_N(x,y,\xi) u(y) dy \dslash\xi \textrm{ für alle } u \in L^q(\R^n) .$$
Nun ist $D_{\xi_j} b^j_N(x,y,\xi) = ( D_{\xi_j} b^j(x,y,\xi) ) \sum_{k=0}^N \lilwood{k}(\xi) + b^j(x,y,\xi) \sum_{k=0}^N D_{\xi_j} \lilwood{k}(\xi)$ und
$$\iint e^{i(x-y)\cdot\xi} b^j(x,y,\xi) \left( \sum_{k=0}^N D_{\xi_j} \lilwood{k}(\xi) \right) u(y) dy \dslash\xi \rightarrow 0 \textrm{ für } N \rightarrow \infty .$$
Für $D_{\xi_j} b^j(x,y,\xi) \in C^\theta S^0_{1,0}(\R^n\times\R^n,\R^n)$ folgern wir mit Theorem \ref{thm:result1} die Behauptung.
\end{proof}
\end{cor}

\begin{cor}\label{thm:result4}
Sei $p \in C^{1,\theta} S^0_{1,0}(\R^n\times\R^n,\R^n)$ und $p(x,x,\xi) = 0$ für alle $x,\xi\in\R^n$.
Dann ist
$$ p(X,Y,D_x) : L^q(\R^n) \rightarrow W^1_q(\R^n) \textrm{ für alle } 1 < q < \infty .$$
\begin{proof}
Produktregel liefert für jedes $j = 1,\ldots, n$ die Aussage
\begin{align*}
\partial_{x_j} p(X,Y,D_x) u(x) 
&= i \lim_{N \rightarrow \infty} \iint e^{i(x-y)\cdot\xi} \xi_j p_N(x,y,\xi) u(y) dy \dslash\xi  \\
&+ \lim_{N \rightarrow \infty} \iint e^{i(x-y)\cdot\xi} \partial_{x_j}  p_N(x,y,\xi) u(y) dy \dslash\xi
.\end{align*}
Für $\partial_{x_j} p(x,y,\xi) \in C^\theta S^0_{1,0}(\R^n\times\R^n,\R^n)$ liefert Theorem \ref{thm:result1} und
für $$\xi_j p(x,y,\xi) \in C^{1,\theta} S^1_{1,0}(\R^n\times\R^n,\R^n)$$ liefert Folgerung \ref{thm:result3} das Resultat
$$ \partial_{x_j} p(X,Y,D_x) : L^q(\R^n) \rightarrow  L^q(\R^n) \textrm{ für alle } 1 < q < \infty ,$$
und damit die Behauptung.
\end{proof}
\end{cor}

\begin{thm}\label{thm:resultx,x,xi}
Sei $p \in C^\theta S^0_{1,0}(\R^n\times\R^n,\R^n)$.
Falls $p(x,x,\xi) = 0$ für alle $x,\xi \in \R^n$, dann folgt, dass
\begin{equation}\label{equ:resultx,x,xia}
p(X,Y,D_x) : L^q(\R^n) \rightarrow H^s_q(\R^n) \textrm{ für alle } 1 < q < \infty, 0 \le s < \theta .
\end{equation}
Insbesondere folgt daraus, dass
\begin{equation}\label{equ:resultx,x,xib}
p(X,Y,D_x) : H^\sigma_q(\R^n) \rightarrow H^{\sigma+s}_q(\R^n) \textrm{ für alle } 1 < q < \infty, 0 \le s < \theta, -\theta < \sigma \le 0 .
\end{equation}
\begin{proof}
Wir definieren den Operator $\Lambda := (1 - \triangle_x - \triangle_{y} )^{\frac{1}{2}}$ 
mit Doppelsymbol 
$$\lambda(\xi,\xi') := \piketi{\xi,\xi'} = (1 + \abs{\xi}^2 + \abs{\xi'}^2)^{\frac{1}{2}} \in S^{1,1}_{1,0}(\R^n\times\R^n\times\R^n\times\R^n) ,$$
und setzen
$b_z(x,y,\xi) := \Lambda^{-(z-s)} p(x,y,\xi)$ für $z \in \Streifen$.
Zusätzlich setzen wir $p_z(x,y,\xi) := b_z(x,y,\xi) - b_z(x,x,\xi)$.
Wir haben die Abschätzung
$$\lim_{\abs{\Imag z} \rightarrow \infty} \abs{ \lambda(\xi,\xi')^{-(z-s)} } \le C \piket{\xi}{s -\Real z} \piket{\xi'}{s - \Real z}$$
für alle $z \in \Streifen$ unabhängig von $\Imag z$.
Zum einen erfüllt damit nach Theorem \ref{thm:result1} der Operator von $p_z$ die Abbildungseigenschaft
$$p_z(X,Y,D_x) : L^p(\R^n) \rightarrow L^p(\R^n) \textrm{ für alle } z \in \Streifen . $$
Zum anderen können wir das Doppelsymbol $\lambda(\xi,\xi')$ als vereinfachtes Symbol $\lambda_L(\eta) := \lambda(\xi,\xi') \in S^1_{1,0}(\R^{2n})$ mit $\eta = (\xi,\xi') \in \R^{2n}$
auffassen,
für das wir wegen obiger Abschätzung eine Konstante $C_l>0$ für beliebiges $l \in \N_0$ finden mit
$$ \abs{ \lambda_L(\xi) }_{l,\infty}^{(m)} \le C_l \textrm{ gleichmäßig in } z \in \Streifen \textrm{ mit } \Real z = 1 ,$$
wobei $m := s - \Real z = s - 1 < 0$.
Dieses Symbol induziert uns für $\vartheta := \theta - s$ mit Theorem \ref{beschr:S_1,delta} die Eigenschaft 
$\Lambda^{s-1} : C^\theta(\R^n\times\R^n) \rightarrow C^{1,\vartheta}(\R^n\times\R^n)$, 
sodass
$\Lambda^{-(z-s)} p(x,y,\xi) \in C^{1,\vartheta} S^0_{1,0}(\R^n\times\R^n,\R^n)$ für $\Real z = 1$, da
$$ \norm{ \partial_\xi^\alpha \Lambda^{-(z-s)} p(\hold,\hold,\xi) }{C^{1,\vartheta}(\R^n\times\R^n)} \le C \norm{ \partial_\xi^\alpha p(\hold,\hold,\xi) }{C^\theta(\R^n\times\R^n)} $$
mit einer Konstante $C>0$ unabhängig von $z \in \Streifen$ mit $\Real z = 1$.
Hierbei haben wir verwendet, dass Ableitungsoperatoren mit Pseudodifferentialoperatoren kommutieren, deren Doppelsymbole unabhängig von $x,x'$ sind.
Folgerung \ref{thm:result4} liefert $p_z(X,Y,D_x) :  L^q(\R^n) \rightarrow H^1_q(\R^n)$ für alle $\Real z = 1$.
Aus der Beschränktheit von $\lambda_L$ in den Halbnormen $\abs{\hold}_{l,\infty}^{(\Real z - s)}$ folgern wir
$$ \sup_{z \in \Streifen} \norm{p_z(X,Y,D_x) }{\linearop{L^q(\R^n)}{L^q(\R^n)}} < \infty \textrm{ und }
\sup_{\Real z = 1} \norm{p_z(X,Y,D_x) }{\linearop{L^q(\R^n)}{H^1_q(\R^n)}} < \infty .$$
Mit komplexer Interpolation (Theorem \ref{thm:interpolComplex}) folgt $p_s(X,Y,D_x) : L^q(\R^n) \rightarrow H^s_q(\R^n)$.
Da $p(x,x,\xi) = 0$ ist $p_s(X,Y,D_x) = p(X,Y,D_x)$. Damit ist (\ref{equ:resultx,x,xia}) bewiesen.
Für (\ref{equ:resultx,x,xib}) betrachten wir den in Folgerung \ref{cor:adjunXY} definierten dualen Operator
$$p(X,Y,D_x)^* u(x) = \lim_{N \rightarrow \infty} \iint e^{i(x-y)\xi} p_N(x,y,\xi)^* u(y) dy \dslash\xi$$
mit Symbol $p(x,y,\xi)^* := p(y,x,-\xi)$. 
Der Operator erfüllt wegen $p(x,x,\xi)^* = 0$ ebenso (\ref{equ:resultx,x,xia}).
Wir haben also für $u, v \in \Schwarzi$ und $1 < q,q' < \infty$ mit $\frac{1}{q} + \frac{1}{q'} = 1$
\begin{align*}
\dualskp{ p(X,Y,D_x) u }{v}{L^q(\R^n)}{L^{q'}(\R^n)} &= 
\dualskp{ p(X,Y,D_x) u }{v}{H^{-s}_q(\R^n)}{H^s_{q'}(\R^n)} \\
&= \dualskp{  u }{ p(X,Y,D_x)^* v}{H^{-s}_q(\R^n)}{H^s_{q'}(\R^n)} \\
&\le \norm{u}{H^{-s}_q(\R^n)} \norm{ p(X,Y,D_x)^* v }{H^s_{q'}(\R^n)} \\
&\le C \norm{ u }{H^{-s}_q(\R^n)} \norm{ v }{L^{q'}(\R^n)}
.\end{align*}
Die Aussage gilt für alle $v \in L^{q'}(\R^n)$, also ist
$ \norm{ p(X,Y,D_x) u }{L^q(\R^n)} \le C \norm{ u }{H^{-s}_q(\R^n)}$.
Wir erhalten
\begin{equation}\label{equ:resultx,x,xic}
p(X,Y,D_x) : H^{-s}_q(\R^n) \rightarrow L^q(\R^n) \textrm{ für alle } 1 < q < \infty, 0 \le s < \theta .
\end{equation}
Die komplexe Interpolation bzgl. (\ref{equ:resultx,x,xia}) und (\ref{equ:resultx,x,xic}) liefert die Eigenschaft
$$ p(X,Y,D_x) : H^{-s\vartheta}_q(\R^n) \rightarrow H^{(1-\vartheta)s}_q(\R^n) \textrm{ für alle } \vartheta \in [0,1] .$$
Für $\sigma := -s\vartheta \in [-s,0]$ erhalten wir
$$p(X,Y,D_x) : H^\sigma_q(\R^n) \rightarrow H^{\sigma+s}_q(\R^n) \textrm{ für alle } 1 < q < \infty, 0 \le s < \theta,  -s \le \sigma \le 0 .$$
(\ref{equ:resultx,x,xib}) folgt nun aus der Tatsache, dass für $-\theta < \sigma \le -s$ nach (\ref{equ:resultx,x,xic})
$$p(X,Y,D_x) : H^{\sigma}_q(\R^n) \rightarrow L^q(\R^n) \hookrightarrow H^{\sigma+s}_q(\R^n) .$$
\end{proof}
\end{thm}

\begin{thm}
Ein \PSDO in $(x,y,\xi)$-Form $p(X,Y,D_x) \in \OP C^\theta S^0_{1,0}(\R^n\times\R^n,\R^n)$ besitzt die Abbildungseigenschaft
\begin{align*}
p(X,Y,D_x) :& H^s_q(\R^n) \rightarrow H^s_q(\R^n) \textrm{ für alle } s \in (-\theta,\theta), 1 < q < \infty\\
& C^\infty_0(\R^n) \rightarrow C^0(\R^n)
.\end{align*}
\begin{proof}
\cite[Proposition 9.8]{Tools} liefert per Dualität die erste Behauptung.
Für die zweite Abbildungseigenschaft liefert \cite[Proposition 9.17]{Tools} sogar
$$ p(X,Y,D_x) : C^s(\R^n) \rightarrow C^s(\R^n) \textrm{ für alle } s \in (0,\theta) .$$
\end{proof}
\end{thm}

\subsection{Nicht-glatter Koordinatenwechsel}

Für den Fall einer nicht-glatten Mannigfaltigkeit betrachten wir zunächst wieder einen Diffeomorphismus auf offenen Mengen des $\R^n$.
Die daraus gezogenen Resultate werden wir für den lokalen Kartenwechsel auf Mannigfaltigkeiten betrachten, deren Struktur nur noch Hölder-stetig ist.
Leider werden wir bei der Koordinatentransformation nicht mehr eine Restklasse wie die Pseudodifferententialoperatorklasse $C^\tau S^{-\infty}(\R^n,\R^n\times\R^n)$ erhalten.
Stattdessen bekommen wir als Restmenge all jene Operatoren, die zu einem bestimmten Grad (siehe Theorem \ref{thm:resultx,x,xi} und Definition \ref{def:restOP}) regulierend sind.

Für ein Gebiet $\Omega$ wollen wir Pseudodifferentialoperatoren auf den Bessel-Potential-Räumen 
$H^s_q(\Omega)$ beschrieben in \cite[Definition 3.2.2]{TriebelFunction} betrachten, 
und den Operator nach einer Koordinatentransformation studieren. 

\begin{notation}
Für ein $1 < q < \infty$ bezeichne $L^{q'}(\R^n)$ den zu $L^q(\R^n)$ dualen Raum, d.h. wir wählen $1 < q' < \infty$ so, dass $\frac{1}{q} + \frac{1}{q'} = 1$.
\end{notation}

\begin{definition}\label{def:besselOmega}
Sei $\Omega \subseteq \R^n$ ein Gebiet.
Für $f,g \in \distrodual{\Omega}$ definiert die Äquivalenzrelation 
$$f \sim g :\Leftrightarrow \textrm{ es existiert ein } H \in H^s_q(\R^n) \textrm{ mit } H\mid_\Omega = 0 \textrm{ und } f = g + H$$
den Quotientenraum
$H^s_q(\Omega) := \menge{ f \in \distrodual{\Omega} : \textrm{ es gibt ein } F \in H^s_q(\R^n) \textrm{ mit } f = F\mid_\Omega }$
mit induzierter Norm
$$\norm{f}{H^s_q(\Omega)} := \inf_{\substack{F \in H^s_q(\R^n)\\ F\mid_\Omega = f}} \norm{F}{H^s_q(\R^n)} .$$
Insbesondere ist $H^s_q(\Omega) = H^s_q(\R^n)$ für $\Omega = \R^n$.
Weiter sei der Unterraum $H^s_{q,c}(\Omega) \subset H^s_q(\R^n)$ definiert durch
$$ H^s_{q,c}(\Omega) := \menge{ f \in H^s_q(\R^n) : \supp f \subsubset \Omega} .$$
\end{definition}
\nomenclature[f]{$H^s_q(\Omega)$}{Bessel-Potential-Raum der Ordnung $s \in \R$ mit $1 < q < \infty$ auf $\Omega$}
\nomenclature[n]{$\norm{f}{H^s_q(\Omega)}$}{Norm des Bessel-Potential-Raums auf $\Omega$}
\nomenclature[f]{$H^s_{q,c}(\Omega)$}{Raum aller Abbildungen aus dem Bessel-Potential-Raum $H^s_q(\R^n)$ mit kompakten Träger in $\Omega$}

\begin{remark}
Für $\Omega_1, \Omega_2 \subset \R^n$ ist für jedes $f \in H^s_q(\Omega_1) \cup H^s_q(\Omega_2)$ bereits $H^s_q(\Omega_1 \cup \Omega_2)$,
da wir nach Definition zwei $F_1, F_2 \in H^s_q(\R^n)$ finden mit $F_1\mid_{\Omega_1} = f\mid_{\Omega_1}$, $F_2\mid_{\Omega_2} = f\mid_{\Omega_2}$
und damit für 
$$ F(x) := \begin{cases} F_1(x) & x \in \R^n \setminus \Omega_2, \\ F_2(x) & x \in \Omega_2 , \end{cases}$$
gilt, dass 
$$ \norm{ F }{H^s_q(\R^n)} \le \norm{ F_1 }{H^s_q(\R^n)} + \norm{ F_2}{H^s_q(\R^n)} < \infty .$$
Also ist $F \in H^s_q(\R^n)$ mit $F\mid_{\Omega_1\cup\Omega_2} = f\mid_{\Omega_1\cup\Omega_2}$.
\end{remark}

Ein glatter regulärer Diffeomorphismus $h : \Omega_y \rightarrow \Omega_x$ induziert uns nach Theorem \ref{thm:koordBesselInv} einen Isomorphismus
$ h^* : H^s_q(\R^n) \rightarrow H^s_q(\R^n)$.
Wir wollen nun ein analoges (wenn auch weitaus schwächeres) Resultat erhalten, falls $h$ nicht glatt ist:

\begin{lemma}\label{lemma:koordNonKarte}
Ein regulärer $C^{1,\theta}$-Diffeomorphismus $h : \R^n \rightarrow \R^n$ definiert uns einen Isomorphismus
$$ h^* : H^s_q(\R^n) \xlongrightarrow{\sim} H^s_q(\R^n) \textrm{ für alle } s \in (-\theta,1], 1 < q < \infty .$$
\begin{proof}
{\bf 1. Fall: $s \in [0,1]$ }.
Die Abbildung $h^* : L^q(\R^n) \rightarrow L^q(\R^n)$ ist ein beschränkter linearer Operator, da nach Transformationsformel
$$ \norm{h^*(f)}{L^q(\R^n)}^q = \int_{\R^n} \abs{ h^*(f)(y) }^q dy = \int_{\R^n} \abs{f(x)}^q \abs{ \det \derive{h^{-1}}{x} } dx \le C \norm{f}{L^q(\R^n)}^q .$$
Außerdem können wir für $f \in C^\infty_0(\R^n)$ die distributionelle Ableitung mit der klassischen ausdrücken, da
$$\partial_{x_j} h^*(f)(x) = \derive{f}{h(x)} \cdot \partial_{x_j} h(x) \textrm{ für jedes } j = 1,\ldots,n .$$
Da $\partial_{x_j} h(x) \in L^\infty(\R^n)$, ist insbesondere $\partial_{x_j} h^*(f)(x) \in L^q(\R^n)$, also ist 
$\partial_{x_j} h^* : C^\infty_0(\R^n) \rightarrow L^q(\R^n)$ für alle $1 < q < \infty$ beschränkt.
Wegen $C^\infty_0(\R^n) \subset W^1_q(\R^n)$ dicht 
erhalten wir durch Approximation von $C^\infty_0(\R^n)$-Funktionen einen beschränkten linearen Operator
$h^* : W^1_q(\R^n) \rightarrow W^1_q(\R^n)$.
Nach Theorem \ref{thm:interpolBesselSobolev} können wir $W^1_q(\R^n)$ mit $H^1_q(\R^n)$ bzw. $L^q(\R^n)$ mit $H^0_q(\R^n)$ identifizieren.
Mit Hilfe der komplexen Interpolation (Theorem \ref{thm:interpolComplex}) erhalten wir schließlich einen Operator
$$ h^* : \interpolC{H^0_q(\R^n), H^1_q(\R^n) }{\vartheta} \rightarrow \interpolC{H^0_q(\R^n), H^1_q(\R^n) }{\vartheta} \textrm{ für alle } \vartheta \in (0,1) .$$
Nach Theorem \ref{thm:interpolBessel} ist aber $\interpolC{H^0_q(\R^n), H^1_q(\R^n) }{\vartheta} = H^\vartheta_q(\R^n)$.

{\bf 2. Fall: $s \in (-\theta,0)$ }.
Setze $\sigma := -s \in (0,\theta)$.
Zunächst haben wir $\det \derivei{h} \in C^\theta(\R^n)$, und für $v \in H^\sigma_q(\R^n)$ nach obiger Berechnung $h^{-1,*} v \in H^\sigma_{q'}(\R^n)$.
Fassen wir $a(x) := \abs{ \det \derive{h}{x}} \in C^\theta S^0_{1,0}(\R^n)$ als \PSDO auf, so erhalten wir nach Lemma \ref{beschr:S^0_1,delta} die Abschätzung
$ \norm{ a(X) v }{H^\sigma_{q'}(\R^n)} \le C_{s,q} \norm{v}{H^\sigma_{q'}(\R^n)}$.
Für eine Testfunktion $\varphi \in \Schwarzi \subset H^\sigma_{q'}(\R^n)$ ist
$$ \norm{ \abs{\det \derivei{h}^{-1}} h^{-1,*} \varphi }{H^\sigma_{q'}(\R^n)} \le C \norm{ \varphi }{H^\sigma_{q'}(\R^n)} .$$
Also ist für ein $u \in H^{-\sigma}_q(\R^n)$ das Dualitätsprodukt
\begin{align*}
\dualskp{ u }{ \abs{\det \derivei{h}^{-1}} h^{-1,*} \varphi }{H^{-\sigma}_q(\R^n)}{H^\sigma_{q'}(\R^n)} 
&= \int_{\R^n} u(x) \left(h^{-1}\right)^* \varphi(x) \abs{ \det \derive{h^{-1}}{x}} dx \\
&= \int_{\R^n} h^*u(x) \varphi(x) dx
\end{align*}
wohldefiniert.
Dies motiviert die Definition
$$ \dualskp{ h^* u }{ \varphi }{\SchwarziDual}{\Schwarzi} := \dualskp{ u }{ \abs{\det \derivei{h}^{-1}} h^{-1,*} \varphi }{H^{-\sigma}_q(\R^n)}{H^\sigma_{q'}(\R^n)} .$$
Mit dem Dichtheitsargument aus Theorem \ref{thm:interpolBesselSobolev} folgt, dass
$$\abs{ \dualskp{ h^* u}{ v }{H^{-\sigma}_q(\R^n)}{H^\sigma_{q'}(\R^n)} } \le C_{s,q} \norm{u}{H^{-\sigma}_q(\R^n)} \norm{v}{H^\sigma_{q'}(\R^n)} \textrm{ für alle } u \in H^{-\sigma}_q(\R^n), v \in H^\sigma_{q'}(\R^n) .$$
\end{proof}
\end{lemma}

\begin{cor}\label{cor:koordNonKarte}
Ein regulärer $C^{1,\theta}$-Diffeomorphismus $h : \Omega_y \rightarrow \Omega_x$ definiert uns einen Operator
$$ h^* : H^s_{q,c}(\Omega_x) \xlongrightarrow{\sim} H^s_{q,c}(\Omega_y) \textrm{ für alle } s \in (-\theta,1], 1 < q < \infty .$$
\begin{proof}
Die Behauptung folgt direkt aus Lemma \ref{lemma:koordNonKarte}, da $H^s_{q,c}(\Omega_x) \subset H^s_q(\R^n)$ und
$\supp u \circ h \subsubset \Omega_y$.
\end{proof}
\end{cor}

\begin{lemma}\label{lemma:diffeoErweitern}
Sei $h : \Omega_y \rightarrow \Omega_x$ ein regulärer $C^{1,\theta}$-Diffeomorphismus.
Dann gibt es zu jedem $x_0 \in \Omega_y$ ein $r > 0$ und einen globalen Diffeomorphismus $h_{r,x_0} : \R^n \rightarrow \R^n$,
der auf $B_r(x_0)$ mit $h$ übereinstimmt. Der Diffeomorphismus $h\mid_{B_r(x_0)}$ lässt sich also zu einem globalen erweitern.
\begin{proof}
Sei $c_0$ die Hölder-Stetigkeitskonstante von $\derivei{h}$, 
d.h. sei $c_0 > 0$ sodass für alle $x,y \in \R^n$ gilt, dass $\opnorm{ \derive{h}{x} - \derive{h}{y} } \le c_0 \abs{x-y}^\theta$, 
wobei $\opnorm{A} := \sup_{v \in \R^n, \abs{v} = 1} \abs{Av}$ für jede lineare Abbildung $A : \R^n \rightarrow \R^n$.
Wähle ein $x_0 \in \Omega_y$.
Wegen (\ref{equ:trafoUngl}) können wir ein $0 < r^\theta \le \frac{1}{2 c_0} \opnorm{ \derive{h}{x_0} }$ finden, sodass $B_{2r}(x_0) \subset \Omega_y$.
Eine Taylorentwicklung um $x_0$ bei $x \in B_r(x_0)$ liefert
$$ h(x) = h(x_0) + \Xi_h(x,x_0) (x - x_0) + \Landau{ \abs{x - x_0}^2} ,$$
wobei $ \Xi_h(x,x_0) = \derive{h}{x_0} - \int_0^1 \left( \derive{h}{x_0} - \derive{h}{ (1-t)x_0 + tx} \right) dt $ aus Lemma \ref{mittelwertsatz}.
Wir setzen $B_h(x,x_0) := \int_0^1 \left( \derive{h}{x_0} - \derive{h}{ (1-t)x_0 + tx} \right) dt$ und verwenden die Hölder-Stetigkeit von $\derivei{h}$:
\begin{align*}
\opnorm{ B_h(x,x_0) } 
&\le \int_0^1 \opnorm{ \derive{h}{x_0} - \derive{h}{ (1-t)x_0 + tx} } dt \\
&\le c_0 \abs{ x_0 - x }^\theta \le c_0 r^\theta \le \frac{1}{2} \opnorm{\derive{h}{x_0}}
.\end{align*}
Setze $\tilde{\Xi}_h(x,x_0) := \derive{h}{x_0} + \phi(x) B_h(x,x_0)$ mit $\phi \in C^\infty_0\left( B_{2r}(x_0)\right), 0 \le \phi \le 1$ und $\phi = 1$ auf $B_r(x_0)$.
Dann ist $$ \opnorm{ \phi(x) B_h(x,x_0) \derive{h}{x_0}^{-1} } \le \frac{1}{2} < 1 .$$
Also konvergiert die Neumann-Reihe 
$\sum_{j=0}^\infty \left( \phi(x) B_h(x,x_0) \derive{h}{x_0}^{-1} \right)^j$ 
absolut und gleichmäßig, 
sodass wir eine Konstante $C>0$ finden mit
$$ \opnorm{ \left( \einheitsmatrix - \phi(x) B_h(x,x_0) \derive{h}{x_0}^{-1} \right)^{-1} } \le C .$$
Schließlich folgern wir mit der einfachen Umformung
\begin{align*}
\opnorm{ \left( \tilde{\Xi}(x,x_0) \right)^{-1} }
&= \opnorm{ \derive{h}{x_0}^{-1} \left( \einheitsmatrix - \phi(x) B_h(x,x_0) \derive{h}{x_0}^{-1} \right)^{-1}} \\
&\le \opnorm{ \derive{h}{x_0}^{-1} } \opnorm{ \left( \einheitsmatrix - \phi(x) B_h(x,x_0) \derive{h}{x_0}^{-1} \right)^{-1}} \\
&\le C \opnorm{ \derive{h}{x_0}^{-1} }
,\end{align*}
dass $x \mapsto \tilde{\Xi}_h(x,x_0)$ auf ganz $\R^n$ invertierbar ist.
Setzen wir $h_{r,x_0}(x) := h(x_0) + \tilde{\Xi}_h(x,x_0)$, so ist $h\mid_{B_r(x_0)} = h_{r,x_0}\mid_{B_r(x_0)} + \Landau{ \abs{x - x_0}^2}$.
\end{proof}
\end{lemma}

\begin{thm}
Sei $h : \Omega_y \rightarrow \Omega_x$ ein regulärer $C^{1,\theta}$-Diffeomorphismus, $\tilde{\Omega}_y \subsubset \Omega_y$, $\tilde{\Omega}_x := h(\tilde{\Omega}_y)$.
Dann definiert uns $h$ den Diffeomorphismus
$$ h^* : H^s_q(\tilde{\Omega}_x) \xlongrightarrow{\sim} H^s_q(\tilde{\Omega}_y) \textrm{ für alle } s \in (-\theta,1], 1 < q < \infty .$$
\begin{proof}
Zunächst finden wir wegen $\tilde{\Omega}_y \subsubset \R^n$ ein $N \in \N$ und damit eine endliche Überdeckung von $\tilde{\Omega}_y$ der Form
$\bigcup_{j=1}^N B_{r_j}(x_j) = \tilde{\Omega}_y$ mit $r_j > 0$, $x_j \in \tilde{\Omega}_y$.
Zudem gibt es eine endliche $C^{1,\theta}$-$\Psi$DO-Abschneide\-funktions\-familie ${(\varphi_j, \psi_j)}_{j=1,\ldots,N}$ auf $\tilde{\Omega}_y$, die den Mengen $\menge{h(B_{r_j}(x_j))}_{j=1,\ldots,N}$
untergeordnet ist.
Wir haben also für ein $u \in H^s_q(\Omega_x)$ die Identität $u(x) = \sum_{j=1}^N \varphi_j(x) u(x)$.
Wegen $\varphi_j \subsubset \psi_j$ erhalten wir mit der Linearität des Pullbacks für $u_j := u \psi_j$
$$ h^* u(x) = \sum_{j=0}^N \varphi_j(h(x)) u(h(x)) = \sum_{j=0}^N \varphi_j(h(x)) u_j(h(x)) .$$
Nach Definition \ref{def:besselOmega} finden wir zu $u_j$ ein $f_j \in H^s_q(\R^n)$ mit $f_j \circ h\mid_{B_{r_j}(x_j)} = u \circ h\mid_{B_{r_j}(x_j)}$.
Lemma \ref{lemma:diffeoErweitern} liefert für $r_j$ genügend klein einen regulären $C^{1.\theta}$-Diffeomorphismus $h_{j,x_j} : \R^n \rightarrow \R^n$ mit
$h_{j,x_j}\mid_{B_{r_j}(x_j)} = h\mid_{B_{r_j}(x_j)}$, und wegen $\supp \varphi_j \subset h(B_{r_j}(x_j))$ haben wir
$$ \varphi_j(h(x)) u_j(h(x)) = \varphi_j(h_{j,x_j}(x)) u_j(h_{j,x_j}(x)) = \varphi_j(h_{j,x_j}(x)) f_j(h_{j,x_j}(x)) .$$
Nach Lemma \ref{lemma:koordNonKarte} ist aber $h_{j,x_j}^* f_j \in H^s_q(\R^n)$.
Insgesamt erhalten wir
$$ \norm{ h^* u }{H^s_q(\tilde{\Omega}_y)} \le C \sum_{j=1}^N \norm{ h_{j,x_j}^* f_j}{H^s_q(\R^n)} < \infty .$$
Insbesondere ist $f_j \circ h_{j,x_j} \mid_{B_{r_j}(x_0)} = u_j \circ h \mid_{B_{r_j}(x_0)}$, also 
finden wir ein $f \in H^s_q(\R^n)$ mit $f \mid_{\Omega_y} = h^* u$.
\end{proof}
\end{thm}

\begin{definition}\label{def:restOP}
$\restOP$ sei die Menge aller Operatoren $R$ mit der Abbildungseigenschaft
$$R : H^{-s}_q(\R^n) \rightarrow H^{t-s}_q(\R^n) \textrm{ für alle } 1 < q < \infty, 0 \le t < \theta \textrm{ und } 0 \le s < \theta .$$
Die Operatoren der Menge $\restOP$ sind wegen $H^{t-s}_q(\R^n) \subset H^{-s}_q(\R^n)$ regulierender als Operatoren der Klasse $\OP C^\tau S^0_{1,0}$.
Wir werden im Folgenden sehen, dass sich der Operator eines transformierten Pseudodifferentialoperators in einen \PSDO der gleichen Klasse und einem Operator der Restklasse $\restOP$ aufspalten lässt.
\end{definition}
\nomenclature[o]{$\restOP$}{Menge aller Operatoren, die Definition \ref{def:restOP} erfüllen}

\begin{lemma}\label{lemma:NonBesselWohlDef}
Für $\tilde{\Omega}_x \subsubset \Omega_x \subset \Image h$ erhalten wir durch Komposition
eines Symbols $p(x,x',\xi) \in C^\theta S^0_{1,0}(\R^n\times\R^n,\R^n)$ mit zwei Funktionen $\varphi, \psi \in C^{1,\theta}_0(\tilde{\Omega}_x)$
einen Operator
$$ P := \varphi(X) p(X,X',D_x) \psi(X') : H^s_{q,c}(\Omega_x) \rightarrow H^s_q(\tilde{\Omega}_x) \textrm{ für alle } -\theta < s < \theta .$$
\begin{proof}
Theorem \ref{thm:operatorXY} liefert
$
\varphi(X) p(X,X',D_x) \psi(X') : H^s_{q,c}(\Omega_x) \rightarrow H^s_q(\R^n)
$.
Außerdem ist $ \supp \left( \varphi(X) p(X,X',D_x) \psi(X') u \right) \subset \tilde{\Omega}_x$ für alle $u \in H^s_{q,c}(\Omega_x)$.
Wir finden also ein $f \in H^s_q(\R^n)$ mit $f\mid_{\tilde{\Omega}_x} = \varphi(X) p(X,X',D_x) \psi(X') u$.
\end{proof}
\end{lemma}

Ab nun seien $\tilde{\Omega}_y \subsubset \Omega_y$, $\tilde{\Omega}_x := h(\tilde{\Omega}_y)$.
Wir wollen ein Symbol $a \in C^\theta S^0_{1,0}(\R^n\times\R^n,\R^n)$ konstruieren, das für alle $u \in H^s_{q,c}(\Omega_x)$ die Gleichung
$ \left( h^{-1} \right)^* A h^* u = P u$
erfüllt,
wobei 
$$A := \varphi(h(Y)) \left( a(Y,Y',D_y) + R \right) \psi(h(Y')) : H^s_{q,c}(\Omega_y) \rightarrow H^s_q(\tilde{\Omega}_y) \textrm{ für alle } -\theta < s < \theta$$
und $R \in \restOP$.
Der Operator $A$ soll also das Diagramm
$$
\begin{xy}
\xymatrix{
	H^s_{q,c}(\Omega_x) \ar[r]^P \ar[d]_{h^*}  & H^s_q(\tilde{\Omega}_x) \ar[d]_{h^*} \\
	H^s_{q,c}(\Omega_y) \ar[r]^A  & H^s_q(\tilde{\Omega}_y) 
}
\end{xy}
$$
kommutieren lassen.

\begin{lemma}\label{lemma:restOPinvariant}
$\restOP$ ist unter einer globalen $C^{1,\theta}$-Koordinatentransformation invariant.
\begin{proof}
Wir erhalten für $1 < q < \infty, 0 \le t < \theta \textrm{ und } 0 \le s < \theta$ aus dem Diagramm
$$
\begin{xy}
\xymatrix{
	H^{-s}(\R^n) \ar[d]_{h^*} \ar[rrr]^{ R } &&& H^{t-s}_q(\R^n) \ar[d]_{h^*} \\
	H^{-s}(\R^n) \ar@{-->}[rrr]^{ \tilde{R} }  &&& H^{t-s}_q(\R^n) 
}
\end{xy}
$$
nach Folgerung \ref{cor:koordNonKarte} einen Operator $\tilde{R} \in \restOP$.
\end{proof}
\end{lemma}

\begin{lemma}\label{lemma:koordNonLilwoodKonv}
Sei $\Xi_h(x,y)$ die lineare Abbildung definiert in Lemma \ref{mittelwertsatz} für einen regulären $C^{1,\theta}$-Diffeomorphismus $h : \R^n \rightarrow \R^n$
und $U \subseteq \R^n$ offen, sodass (\ref{equ:trafoA_hBed}) in $U$ erfüllt ist.
Dann erhalten wir für eine Funktion $f(x,y,\xi) \in \Stetig{0}{\R^n\times\R^n\times\R^n}$ mit $x,y \in U$ die Aussage,
dass für $f_j(x,y,\xi) := f(x,y,\xi) \lilwood{j}(\Xi_h(x,y)^{-T}\xi)$ die Reihe
$\sum_{j=0}^N f_j(x,y,\xi) \rightarrow f(x,y,\xi)$ für $N \rightarrow \infty$ konvergiert.
\begin{proof}
Zunächst finden wir für $\tilde{\lilwood{j}}(\xi) := \lilwood{j}(\Xi_h(x,y)^{-T}\xi)$ wegen (\ref{equ:trafoUngl}) ein $C_0 > 1$ mit
\begin{align*}
\supp \tilde{\lilwood{j}} 
&\subset \menge{ \xi \in \R^n : \abs{\Xi_h(x,y)^{-T}\xi}^{-1} 2^{j-1} \le \abs{\xi} \le \abs{\Xi_h(x,y)^{-T}\xi}^{-1} 2^{j+1}} \\
&\subset \menge{ \xi \in \R^n : C_0^{-1} 2^{j-1} \le \abs{\xi} \le C_0 2^{j+1}}
.\end{align*}
Wir suchen uns das kleinste $\tilde{l} \in \N_0$ mit $2^{j+1} \le C_0 2^{j+\tilde{l}-1}$, d.h. 
$$\tilde{l} := \min \menge{ \tilde{l} \in \N_0 : \tilde{l} \ge 2 - \log_2 C_0}, $$
wobei $\abs{\log_2 C_0} < \infty$ nach (\ref{equ:trafoA_hBed}).
Zusätzlich wählen wir ein $L \in \N$ so, dass $C_0 2^{j+1} \le 2^{L-1}$, dann gilt für alle $l \ge L$, dass $\supp \lilwood{l} \cap \supp \tilde{\lilwood{j}} = \emptyset$.
Diese Bedingung lässt sich umformulieren zu $L \ge j + 2 + \log_2 C_0$.
Setzen wir also $l := 1+\log_2 C_0$, so erhalten wir aus Symmetriegründen
$\tilde{\lilwood{j}}(\xi) = \sum_{k = j - l}^{j+l} \tilde{\lilwood{j}}(\xi) \lilwood{k}(\xi)$.
Für ein festes $\xi \in \R^n$ finden wir ein $j \in \N$, sodass $\lilwoodt{j-1}{\xi} + \lilwoodt{j}{\xi} + \lilwoodt{j+1}{\xi} = 1$.
Damit folgt, dass für alle $N > j+1$
$$ \sum_{i=1}^N \lilwoodt{i}{\xi} 
= \sum_{i=j-1}^{j+1} \lilwoodt{i}{\xi}
= \sum_{i=j-1}^{j+1} \lilwoodt{i}{\xi} \sum_{k=i-l}^{i+l} \lilwood{k}(\xi) 
= \sum_{k=j-l}^{j+l} \lilwood{k}(\xi) \sum_{m=k-\tilde{l}}^{k+\tilde{l}} \lilwoodt{m}{\xi} ,$$
da $\supp \sum_{i=j-1}^{j+1} \tilde{\lilwood{i}} \sum_{k=i-l}^{i+l} \lilwood{k} \subset 
\supp \sum_{k=j-l}^{j+l} \lilwood{k} \sum_{m=k-\tilde{l}}^{k+\tilde{l}} \tilde{\lilwood{m}} $.
Also finden wir zwei Ganzzahlen $M' \ge M \ge 0$, sodass 
\begin{align*}
\sum_{j=0}^N f_j(x,y,\xi)
&= f(x,y,\xi) \sum_{j=0}^N \lilwoodt{j}{\xi}
= f(x,y,\xi) \sum_{k=-l}^l \sum_{j=0}^N \lilwoodt{j}{\xi} \lilwood{k+j}(\xi)\\
&= f(x,y,\xi) \sum_{j=0}^{N+M} \lilwood{j}(\xi) \sum_{k=-\tilde{l}}^{\tilde{l}} \lilwoodt{k+j}{\xi}
= f(x,y,\xi) \sum_{j=0}^{N+M'} \lilwood{j}(\xi)\\
&\rightarrow f(x,y,\xi) \textrm{ für } N \rightarrow \infty \textrm{ punktweise nach Lemma } \ref{lemma:lilwoodSchwarzKonv}
.\end{align*}
\end{proof}
\end{lemma}

\begin{thm}\label{thm:trafoNonRn}
Sei $h : \R^n \rightarrow \R^n$ ein regulärer $C^{1,\theta}$-Diffeomorphismus.
Definiere zu $p(x,x', \xi) \in C^\theta S^0_{1,0}(\R^n\times\R^n,\R^n)$
den Operator 
\begin{equation}\label{equ:trafoNonQRn}
A : H^s_q(\R^n) \rightarrow H^s_q(\R^n), \left(A (u \circ h)\right)(y) = \left( p(X,D_x,X') u \right) (h(y)) 
\end{equation}
für alle $u \in H^s_q(\R^n), y \in \R^n$.
Für $r > 0$ genügend klein erfüllt das Symbol $a \in C^\theta S^0_{1,0}(\R^n\times\R^n,\R^n)$ in $(x,y,\xi)$-Form gegeben durch (\ref{equ:trafoq}) die Bedingung
$ A \equiv a(Y, Y', D_y) \mod{\restOP}$
\begin{proof}
Mit Hilfe der Partition $\menge{\left( \varphi_j, \psi_j \right)}_{j\in\N}$ aus Lemma \ref{koord:lemmaPartition} spalten wir das Symbol $p$ auf in
$$p_j(x,x',\xi) := \varphi_j(x) p(x,x',\xi) \psi_j(x') \textrm{ und } p_{\infty,j}(x,x',\xi) := \varphi_j(x) p(x,x',\xi) (1 - \psi_j(x')) ,$$
sodass 
$p(X, X', D_x) = \sum_j p_j(X,X',D_x) + \sum_j p_{\infty,j}(X,X',D_x)$.
Für diese Partition wählen wir ein $r>0$ klein genug, sodass die Bedingung (\ref{equ:trafoA_hBed}) erfüllt ist.
Wegen $$\dist{\supp \varphi_j}{\supp (1-\psi_j)} > 0$$ ist $p_{j,\infty}(x,x,\xi) = 0$. 
Wir erhalten also mit Theorem \ref{thm:resultx,x,xi} einen Operator $p_{j,\infty}(X,X',D_x) \in \restOP$.
Das Symbol $p_j(x,y,\xi)$ zerlegen wir wie in Definition \ref{def:xyxi} mit der Littlewood-Paley-Partition in $p_j(x,y,\xi) =: \lim_{N\rightarrow\infty} p_{j,N}(x,y,\xi)$.
Bemerkung \ref{remark:badform}.\ref{remark:badformB} liefert
$$ \left(p_{j,N}(X,X',D_x) u\right)(h(y)) = \iint e^{i(h(y)-x')\cdot\xi} p_{j,N}(h(y), x', \xi) u(x') dx' \dslash \xi .$$
Mit analogen Umformungen wie im Beweis von Theorem \ref{thm:trafoRn} erhalten wir für $w = u \circ h$
\begin{align*}
\tilde{a}_{j,N}(Y,Y',D_y) w(y) 
:&= \left( p_{j,N}(X,X',D_x) u\right)(h(y)) \\
&= \iint  e^{i(y-y')\cdot\eta} \tilde{a}_{j,N}(y,y',\eta) w(y') dy' \dslash\eta ,
\end{align*}
mit 
$\tilde{a}_{j,N}(y,y',\eta) := p_{j,N}(h(y),h(y'), \Xi_h(y,y')^{-T} \eta) \abs{\det \Xi_h(y,y')}^{-1} \abs{\det \derive{h}{y'} }$.
Wir setzen $a_j(y,y',\eta) := p_j(h(y),h(y'), \Xi_h(y,y')^{-T} \eta) \abs{\det \Xi_h(y,y')}^{-1} \abs{\det \derive{h}{y'} }$.
Schließlich konvergiert $\tilde{a}_{j,N}(Y,Y',D_y) w(y) \rightarrow a_j(Y,Y',D_y) w(y)$ für $N \rightarrow \infty$ nach Lemma \ref{lemma:koordNonLilwoodKonv} punktweise, sodass
Theorem \ref{thm:operatorXY} liefert
\begin{align*}
a_j(Y,Y',D_y) w(y) &= \lim_{N\rightarrow \infty} \tilde{a}_{j,N}(Y,Y',D_y) w(y) = \lim_{N\rightarrow \infty} \left( p_{j,N}(X,X',D_x) u\right)(h(y)) \\
&= \left( p_j(X,X',D_x) u\right)( h(y))
.\end{align*}

Schließlich wollen will zeigen, dass $a_j$ der Klasse $C^\theta S^0_{1,0}(\R^n\times\R^n,\R^n)$ angehört.
Da $h \in C^{1,\theta}(\R^n)$ ist $\det \Xi_h \in C^\theta(\R^n \times \R^n)$. Das Symbol $a_j$ ist also bzgl. $(x,y) \in \R^{2n}$ Hölder-stetig zum Grad $\theta$ und glatt bzgl. $\xi \in \R^n$.
Wir betrachten
\begin{align*}
\norm{ a_j(\hold,\hold,\xi) }{C^\theta(\R^n\times\R^n)} 
&= \norm{ D_\xi^\alpha p_j\left(h(\hold), h(\hold), \Xi_h(\hold,\hold)^{-T}\xi\right) \abs{\det \Xi_h(\hold,\hold)}^{-1} \abs{\det \derive{h}{\hold} } }{C^\theta(\R^n\times\R^n)} \\
&\le C \norm{ D_\xi^\alpha p_j(h(\hold), h(\hold),\Xi_h(\hold,\hold)^{-T}\xi) }{C^\theta(\R^n\times\R^n)} \\
&\le C \norm{ p_j^{(\alpha)}(h(\hold), h(\hold), \Xi_h(\hold,\hold)^{-T}\xi) }{C^\theta(\R^n\times\R^n)} 
\le C \piket{\xi}{m-\abs{\alpha}}
,\end{align*}
wobei $p_j^{(\alpha)}(x,x',\xi) = D_\xi^\alpha p_j(x,x',\xi)$.
\end{proof}
\end{thm}

\begin{cor}\label{cor:trafoNonRn}
Das Symbol $a_L(y,\eta) := a(y,y,\eta)$ definiert uns einen Operator $a_L(Y,D_y)$ in $x$-Form, der ebenso die Gleichung (\ref{equ:trafoNonQRn}) für ein $R \in \restOP$ erfüllt.
\begin{proof}
Definiere zum Symbol $a(x,y,\xi) \in C^\theta S^0_{1,0}(\R^n\times\R^n,\R^n)$ aus Theorem \ref{thm:trafoNonRn} das Symbol
$b(x,y,\xi) := a(x,y,\xi) - a_L(x,\xi)$.
Dann folgt wegen $b(x,x,\xi) = 0$ für alle $x,\xi\in\R^n$ mit Theorem \ref{thm:resultx,x,xi}, dass $b(X,X',D_x) \in \restOP$.
Wir erhalten also
$a(X,X',D_x) \equiv a(X,D_x) \mod \restOP$.
\end{proof}
\end{cor}

\begin{thm}\label{thm:trafoNon}
Für $p(x,x',\xi) \in C^\theta S^0_{1,0}(\R^n\times\R^n,\R^n), \varphi, \psi \in C^{1,\theta}_0(\tilde{\Omega}_x)$ finden wir ein $R \in \restOP$, sodass
für alle $w := u \circ h \in H^s_{q,c}(\Omega_y)$
$$ \varphi(h(Y)) ( a(Y,Y',D_y) + R ) \psi(h(Y')) w(y) = (\varphi(X) p(X,X',D_x) \psi(X') u) (h(y))$$
gilt, wobei das Symbol $a(y,y',\eta) \in C^\theta S^0_{1,0}(\R^n\times\R^n,\R^n)$ durch die Gleichung (\ref{equ:trafoq}) für ein genügend kleines $r > 0$ gegeben ist.
\begin{proof}
Wie im Beweis von Theorem \ref{thm:trafo} erhalten wir die Darstellung 
$$\varphi(X) p(X,X',D_x) \psi(X') = \varphi(X) \sum_{j \in M} p_j(X,X',D_x) \psi(X') + \varphi(X) \sum_{j \in \N}  p_{\infty,j}(X,X',D_x) \psi(X') .$$
Analog zum Beweis von Theorem \ref{thm:trafoNonRn} erhalten wir nach Koordinatentransformation von $p_j(X,X',D_x)$ einen Operator
$a_j(Y,Y',D_y) = h^* p_j(X,X',D_x) \left(h^{-1}\right)^* $.
Transformieren wir die Operatoren $\varphi(X)$ und $\psi(X')$, so erhalten wir
$$ \varphi(X) p_j(X,X',D_x) \psi(X') = \varphi(Y) h^* \left( a_j(Y,Y',D_y) \psi(h(Y')) \right) \left(h^{-1}\right)^* .$$
Schließlich finden wir wegen $\#M < \infty$ ein $r > 0$, sodass die Bedingung (\ref{equ:trafoA_hBed}) für $\Xi_h$ definiert in (\ref{equ:trafoA_h}) für jedes $j \in M$ erfüllt wird.
Insgesamt bekommen wir die Transformation des Operators $\sum_{j\in M} p_j(X,X',D_x) \psi(X')$ durch
$$\varphi(X) \sum_{j\in M} p_j(X,X',D_x) \psi(X') = \varphi(Y) \sum_{j\in M} h^* \left( a_j(Y,Y',D_y) \psi(h(Y')) \right) \left(h^{-1}\right)^* .$$
Für den obigen zweiten Term $\sum_{j \in \N}  p_{\infty,j}(X,X',D_x) \psi(X')$ folgern wir wieder aus 
$$p_{\infty,j}(x,x,\xi) = 0, \textrm{ dass } p_{\infty,j}(X,X',D_x) \psi(X') \in \restOP$$
nach Theorem \ref{thm:resultx,x,xi}.
Für ein $\phi \in C^{1,\theta}_0(\Omega_x)$ mit $\phi = 1$ auf $\tilde{\Omega}_x$ haben wir für alle $u \in H^s_{q,c}(\tilde{\Omega}_x)$
$$ \left( \phi(X) \varphi(X) p_{\infty,j}(X,X',D_x) \psi(X') \phi(X') \right) u = \left( \varphi(X) p_{\infty,j}(X,X',D_x) \psi(X') \right) u .$$
Die Summe $ \varphi(X) \sum_{j \in \N} \phi(X) p_{\infty,j}(X,X',D_x) \psi(X') \phi(X')$ ist aber endlich, also erhalten wir auf $H^s_{q,c}(\tilde{\Omega}_x)$ einen Operator 
$$\varphi(X) \sum_{j \in \N}  p_{\infty,j}(X,X',D_x) \psi(X') = \varphi(X) \sum_{j \in \tilde{M}} \phi(X) p_{\infty,j}(X,X',D_x) \psi(X') \phi(X') ,$$
mit $\sum_{j \in \tilde{M}} \phi(X) p_{\infty,j}(X,X',D_x) \psi(X') \phi(X') \in \restOP$
für 
$$\tilde{M} := \menge{ j \in \N : \supp \phi \cap \supp \psi_j \circ h \neq \emptyset} \textrm{ endlich.} $$
\end{proof}
\end{thm}

\begin{remark} 
Unter den gleichen Voraussetzungen wie in Theorem \ref{thm:trafoNon} liefert Folgerung \ref{cor:trafoNonRn}, dass
wir für $p(x,\xi) \in C^\theta S^0_{1,0}(\R^n,\R^n)$
ein Symbol $a(y,\eta) \in C^\theta S^0_{1,0}(\R^n,\R^n)$ erhalten, sodass
$$ \varphi(h(Y)) ( a(Y,D_y) + R ) \psi(h(Y')) w(y) = (\varphi(X) p(X,X',D_x) \psi(X') u) (h(y))$$
für ein $R \in \restOP$.
\end{remark}

\clearpage 
\subsection{Pseudodifferentialoperatoren auf nicht-glatten Mannigfaltigkeiten}

\begin{definition}
Sei $M$ eine $C^{1,\theta}$-Mannigfaltigkeit, $\Omega \subseteq M$ offen, $-\theta < s \le 1+\theta$, $1 < q < \infty$ und eine Funktion  $u : \Omega \rightarrow \C$ gegeben.
\begin{itemize}
\item $u$ ist in $H^s_q(\Omega)$ falls $u \circ h_j^{-1} \in H^s_q(h_j(\Omega_j \cap \Omega))$ für jedes $j \in \N$.
\item $u$ ist in $H^s_{q,c}(\Omega)$ falls $u \in H^s_q(\Omega)$ und der Träger von $u$ kompakt in $\Omega$ enthalten ist.
\end{itemize}
\end{definition}

\begin{definition}\label{def:mfgNonRestOP}
Sei $M$ eine $C^{1,\theta}$-Mannigfaltigkeit mit Atlas $\menge{ h_j : \Omega_j \rightarrow U_j}_{j\in\N}$.
Ein linearer Operator $R : H^s_{q,c}(M) \rightarrow H^s_q(M)$ für $s \in (-\theta,\theta)$ gehört der Klasse $\restOPM$ an, 
falls es zu zwei beliebig gewählten Partitionen $\menge{\Phi_j}_{j\in\N}$ und $\menge{\Psi_j}_{j\in\N}$, die den offenen Mengen $\menge{\Omega_j}_{j\in\N}$ untergeordnet sind,
eine Familie von Operatoren $\menge{R_j}_{j\in\N}$ mit $R_j \in \restOP$ gibt mit
$$ ( \Phi_j R \Psi_j u)(h_j^{-1}(x)) = \varphi_j(X) R_j (\psi_j u_j)(x) \textrm{ für alle } u \in H^s_{q,c}(\Omega_j) .$$
Lemma \ref{lemma:restOPinvariant} liefert die Unabhängigkeit dieser Definition bezüglich des gewählten Atlanten.
\end{definition}
\nomenclature[o]{$\restOPM$}{Operatoren auf der nicht-glatten Mannigfaltigkeit $M$, die Definition \ref{def:mfgNonRestOP} erfüllen}

Für den Rest dieses Abschnitts wählen wir nun ein festes $s \in (-\theta,\theta)$.

\begin{definition}\label{def:mfgNonPSDO}
Sei $M$ eine $C^{1,\theta}$-Mannigfaltigkeit mit Atlas $\menge{ h_j : \Omega_j \rightarrow U_j}_{j\in\N}$.
Unter einem Pseudodifferentialoperator der Klasse $\OPMfgN{\theta}{M}$ verstehen wir einen linearen Operator
$ P : H^s_{q,c}(M) \rightarrow H^s_q(M)$, der sich bezüglich einer $C^{1,\theta}$-$\Psi$DO-Abschneide\-funktions\-familie ${(\Phi_j,\Psi_j)}_{j\in\N}$ in der Form
$ ( P u)(y) = \sum_j (\Phi_j P \Psi_j u)(y) + R$
für alle $y \in M$ und jedes $u \in H^s_{q,c}(M)$
schreiben lässt, wobei
\begin{enumerate}[(a)]
\item $\Phi_j P \Psi_j u$ in der Form
\begin{equation}\label{equ:mfgNonPsdoSymbol}
( \Phi_j P \Psi_j u)(h_j^{-1}(x)) = \varphi_j(X) p_j(X,D_x) (\psi_j u_j)(x) \textrm{ für alle } x \in U_j \textrm{ und jedes } u \in H^s_{q,c}(\Omega_j)
\end{equation}
für ein $p_j(x,\xi) \in C^\theta S^0_{1,0}(\R^n , \R^n)$ mit $u_j := u \circ h_j^{-1}$ geschrieben werden kann und
\item der Operator $R$ der Klasse $\restOPM$ angehört.
\end{enumerate}
Die Familie $\menge{p_j(x,\xi)}_{j\in\N}$ heißt Symbol von $P$.
\end{definition}
\nomenclature[o]{$\OPMfgN{\theta}{M}$}{Klasse der \PSDOS auf einer nicht-glatten Mannigfaltigkeit $M$}

\begin{lemma}\label{lemma:mfgNondisjunkt}
Sei $M$ eine $C^{1,\theta}$-Mannigfaltigkeit mit Atlas $\menge{ h_j : \Omega_j \rightarrow U_j}_{j\in\N}$.
Sei $P \in \OPMfgN{\theta}{M}$.
Seien $\Psi, \Phi \in C^{1,\theta}_0(M)$ mit $\supp \Phi \cap \supp \Psi = \emptyset$ gegeben.
Dann ist $\Psi P \Phi \in \restOPM$.
\begin{proof}
Zunächst ist die Summe endlich, da $\Phi$ und $\Psi$ beide einen kompakten Träger haben.
Wir schreiben
$$ \Phi P \Psi = \sum_j \Phi (\Phi_j P \Psi_j) \Psi + R .$$
Nach Voraussetzung ist $R \in \restOPM$, also ist auch $\Phi R \Psi \in \restOPM$.
Für $\Phi (\Phi_j P \Psi_j) \Psi$ haben wir lokal auf $U_j$ die Darstellung
$$ \Phi \Phi_j P(\Psi_j \Psi u)(h_j^{-1}(x)) = \varphi(X) \varphi_j(X) p_j(X,D_x) (\psi_j \psi u_j)(x) \textrm{ für alle } x \in U_j, u \in H^s_{q,c}(\Omega_j) ,$$
wobei $u_j := u \circ h_j^{-1}$, $\varphi := \Phi \circ h_j^{-1}, \psi := \Psi \circ h_j^{-1}$ in $U_j$.
Da $\supp \psi \cap \supp \varphi = \emptyset$ folgt für $b(x,y,\xi) := \varphi(x) p_j(x,\xi) \psi(y) \in C^\theta S^0_{1,0}(\R^n\times\R^n,\R^n)$, dass
$b(x,x,\xi) = 0$, also nach Theorem \ref{thm:resultx,x,xi} ist $b(X,X',D_x) \in \restOP$ und damit
$$\Phi_j\Phi P(\Psi\Psi_j u)(h_j^{-1}(x)) = \varphi_j(X) b(X,X',D_x) (\psi_j u_j) (x) \textrm{ für alle } u \in H^s_{q,c}(\Omega_j) .$$
Wir folgern $\Phi\Phi_j P \Psi_j\Psi \in \restOPM$.
\end{proof}
\end{lemma}

\begin{thm}[Invarianz unter Kartenwechsel]
Sei $\menge{ h_j : \Omega_j \rightarrow U_j}_{j\in\N}$ Atlas der $C^{1,\theta}$-Mannigfaltigkeit $M$.
Sei $\menge{ h'_{j'} : \Omega'_{j'} \rightarrow U'_{j'}}_{j'\in\N}$ ein weiterer Atlas der selben $C^{1,\theta}$-Struktur,
sodass für alle $j,j' : \Omega_j \cap \Omega'_{j'} \neq \emptyset$ sich 
$$h_{j,j'} := h_j \circ (h'_{j'})^{-1} : h'_{j'}(\Omega_j \cap \Omega'_{j'}) \rightarrow h_j(\Omega_j \cap \Omega'_{j'})$$
zu einem $C^{1,\theta}$-Diffeomorphismus auf einer offenen Umgebung von $\overline{ h'_{j'}(\Omega_j \cap \Omega'_{j'})}$ fortsetzen lässt,
der auf eine offene Umgebung von $\overline{ h_j(\Omega_j \cap \Omega'_{j'})}$ abbildet.
Dann ist ein Pseudodifferentialoperator $P : H^s_{q,c}(M) \rightarrow H^s_q(M)$ der Klasse $\OPMfgN{\theta}{M}$ bezüglich 
des Atlanten $\menge{ h_j : \Omega_j \rightarrow U_j}_{j\in\N}$ auch ein \PSDO bezüglich 
des Atlanten $\menge{ h'_{j'} : \Omega'_{j'} \rightarrow U'_{j'}}_{j'\in\N}$.
\begin{proof}
Unter einer anderen $C^{1,\theta}$-$\Psi$DO-Abschneide\-funktions\-familie ${(\Phi'_{j'},\Psi'_{j'})}_{j'\in\N}$ lässt sich $P$ schreiben als
$$ P = \sum_{j'\in\N} \Phi'_{j'} P \Psi'_{j'} + \tilde{R} ,$$
wobei $\tilde{R} \in \restOPM$.
Wir setzen $\Phi'_{j'} P \Psi'_{j'}$ in die Definition \ref{def:mfgNonPSDO} bzgl. der $C^{1,\theta}$-$\Psi$DO-Abschneide\-funktions\-familie ${(\Phi_{j},\Psi_{j})}_{j\in\N}$ ein und erhalten
\begin{equation}\label{equ:nonsmoothInv}
\Phi'_{j'} P \Psi'_{j'} = \sum_{j\in\N} \Phi'_{j'} ( \Phi_j P \Psi_j) \Psi'_{j'} + \Phi'_{j'} R \Psi'_{j'} ,
\end{equation}
wobei nach Definition $\Phi'_{j'} R \Psi'_{j'} \in \restOPM$.

Für den ersten Term folgern wir analog zum Beweis von Theorem \ref{thm:mfgInvarianz}, dass
\begin{align*}
&\left(\Phi'_{j'} \Phi_j P(\Psi_j \Psi'_{j'} u)\right)\left( (h'_{j'})^{-1}(x)\right) \\
&= 
\tilde{\varphi}'_{j'}
\left(
\left( h_j \circ (h'_{j'})^{-1} \right)^*
\left(
\varphi_j p_j(X,D_x) \left( \psi_j \left( h'_{j'} \circ h_j^{-1} \right)^* (\tilde{\psi}'_{j'} u'_{j'}) \right)
\right)
\right)
(x)
\end{align*}
für alle $x \in U_j \cap U_{j'}$
wobei
$u'_{j'} := u \circ (h'_{j'})^{-1}, 
\tilde{\psi}'_{j'} := \Psi'_{j'} \circ (h'_{j'})^{-1},
\tilde{\varphi}'_{j'} := \Phi'_{j'} \circ (h'_{j'})^{-1}$ 

Nach Theorem \ref{thm:trafoNon} (setze $h := h_{j,j'}, \tilde{\Omega}_y := h'_{j'}(\Omega'_{j'} \cap \Omega_j)$) 
gibt es ein $a_{j,j'} \in C^\theta S^0_{1,0}(\R^n,\R^n)$ und ein $R_{j,j'} \in \restOP$ mit 
\begin{align*}
&\tilde{\varphi}'_{j'}(X) \left(  a_{j,j'}(X,D_x) + R_{j,j'} \right) \tilde{\psi}'_{j'}(X') u_{j'}'(x) \\
&= \tilde{\varphi}'_{j'}(X) \left( h^*_{j,j'} \left( \varphi(X) p_j(X,D_x) \left( \psi_j(X') \left(h^{-1}_{j,j'}\right)^* (\tilde{\psi}'_{j'} u'_{j'}) \right)\right)\right)(x)
.\end{align*}
Definieren wir nun
$$ a_{j'}(x,\xi) := \sum_{\substack{j \in \N : \\ \Omega'_{j'} \cap \Omega_j \neq \emptyset}} a_{j,j'}(x,\xi), \textrm{ und }
   R_{j'} := \sum_{\substack{j \in \N : \\ \Omega'_{j'} \cap \Omega_j \neq \emptyset}} R_{j,j'} \in \restOP
,$$
so sehen wir, dass unser erster Term aus (\ref{equ:nonsmoothInv}) in der Form
$$ ( \Phi'_{j'} ( \Phi_j P \Psi_j) \Psi'_{j'} u)( (h'_{j'})^{-1}(x)) = \tilde{\varphi}'_{j'} \left( a_{j'}(X,D_x) + R_{j'} \right) (\tilde{\psi}'_{j'} u'_{j'})(x)$$
für alle $u'_{j'} = u \circ \left(h'_{j'}\right)^{-1} \in H^s_{q,c}(U'_{j'} \cap U_j)$
geschrieben werden kann, unabhängig von $j \in \N$ - also gilt obige Gleichung auch für alle $u'_{j'} \in H^s_{q,c}(U'_{j'})$.

Somit definiert $P$ bezüglich des Atlanten $\menge{ h'_{j'} : \Omega'_{j'} \rightarrow U'_{j'}}_{j'\in\N}$ wiederum einen Pseudodifferentialoperator.
Wegen der Transformationsinvarianz der Klasse $\restOPM$ erhalten wir die Behauptung.
\end{proof}
\end{thm}

\clearpage
\addcontentsline{toc}{section}{\nomname}
\markboth{\nomname}{\nomname}
\printnomenclature

\clearpage
\nocite{*} \bibliographystyle{abbrvnat}
\bibliography{bibliography}

\end{document}